\documentclass[12pt]{article}
{\begin{Sbox}\begin{minipage}}%
{\end{minipage}\end{Sbox}\fbox{\TheSbox}}%

\usepackage[psamsfonts]{amssymb}
\usepackage{amsmath, amsthm, anysize, enumerate, color, url}
\usepackage{graphicx}
\usepackage{epstopdf}
\usepackage{epsfig}
\usepackage{subfigure}
\usepackage[ruled,linesnumbered]{algorithm2e}
\usepackage[sectionbib]{natbib}

\usepackage{threeparttable}

\usepackage[colorlinks, breaklinks,
                   linkcolor=red,
                  citecolor=blue
                  ]{hyperref}

\usepackage{breakurl}

\newtheorem{theorem}{Theorem} 
\newtheorem{corollary}{Corollary}

\newtheorem{lemma}{Lemma} 

\theoremstyle{definition}

\newcommand{\hh}{\mathbf{h}}
\newcommand{\dd}{\mathbf{d}}
\newcommand{\E}{\mathbb{E}}

\newcommand{\R}{\mathbb{R}}

\renewcommand{\P}{\mathbb{P}}

\newcommand{\diag}{\mbox{diag}}
\newcommand{\bs}{\boldsymbol}

\usepackage{setspace}

\begin{document}

\title{Wilks' theorems in the $\beta$-model}

\author{Ting Yan\thanks{Department of Statistics, Central China Normal University, Wuhan, 430079, China.
\texttt{Email:} tingyanty@mail.ccnu.edu.cn.}
\hspace{4mm}
Yuanzhang Li\thanks{Walter Reed Army Institute of Research,
503 Robert Grant Ave., Silver Spring, Maryland, 20910, USA.
\texttt{Email:} Liy.Li@us.army.mil}
\hspace{4mm}
Jinfeng Xu\thanks{Department of Biostatistics, City University of Hong Kong, Hong Kong.
\texttt{Email:} jinfengxu@gmail.com}
\hspace{4mm}
Yaning Yang\thanks{Department of Statistics and Finance,
University of Science and Technology of China, Anhui, 230026, China.
\texttt{Email:} ynyang@ustc.edu.cn}
\hspace{4mm}
Ji Zhu\thanks{ Department of Statistics,
University of Michigan, Ann Arbor, Michigan, 48109-1107, USA.
\texttt{Email:} jizhu@umich.edu}
}
\date{}

\maketitle

\begin{abstract}
\begin{spacing}{1.2}

Likelihood ratio tests and the Wilks theorems have been pivotal in statistics but have rarely been explored in network models with an increasing dimension.
We are concerned here with likelihood ratio tests in the $\beta$-model for undirected graphs.
For two growing dimensional null hypotheses including
a specified null $H_0: \beta_i=\beta_i^0$ for $i=1,\ldots, r$ and a homogenous null $H_0: \beta_1=\cdots=\beta_r$,
we reveal high dimensional Wilks' phenomena  that the normalized log-likelihood ratio statistic,
$[2\{\ell(\widehat{\bs{\beta}}) - \ell(\widehat{\bs{\beta}}^0)\} -r]/(2r)^{1/2}$,
converges in distribution to the standard normal distribution as $r$ goes to infinity.
Here, $\ell( \bs{\beta})$ is the log-likelihood function on the vector parameter $\bs{\beta}=(\beta_1, \ldots, \beta_n)^\top$,
$\widehat{\bs{\beta}}$ is its maximum likelihood estimator (MLE) under the full parameter space,
and $\widehat{\bs{\beta}}^0$ is the restricted MLE under the null parameter space.
For the corresponding fixed dimensional null $H_0: \beta_i=\beta_i^0$ for $i=1,\ldots, r$ and the homogenous null
$H_0: \beta_1=\cdots=\beta_r$ with a
fixed $r$, we establish Wilks type of results that $2\{\ell(\widehat{\bs{\beta}}) - \ell(\widehat{\bs{\beta}}^0)\}$ converges in distribution to a Chi-square distribution with respective $r$ and $r-1$ degrees of freedom, as the total number of parameters, $n$, goes to infinity.
The Wilks type of results are further extended into a closely related Bradley--Terry model for paired comparisons,
where we discover a different phenomenon that the log-likelihood ratio statistic under the fixed dimensional specified null
asymptotically follows neither a Chi-square nor a rescaled Chi-square distribution.
Simulation studies and an application to NBA data illustrate the theoretical results.
\end{spacing}

\vskip 5 pt \noindent
\begin{spacing}{1.4}
\textbf{Key words}:   $\beta$-model; Bradley--Terry model; Growing dimensional hypothesis; Likelihood ratio statistic; Wilks' theorem.
\end{spacing}

\end{abstract}

\vskip 5pt

\footnote{This paper supersedes arxiv article arXiv:1201.0058 by Yan et al. Several significant improvements are:
\begin{itemize}
\item We remove the previous condition $n/r\ge c$ in increasing dimensional hypothesis testing problems, where $c$ is a constant, $r$ is the number of being tests parameters and $n$ is the total number of parameters.
Our results here state that high dimensional Wilks' type of results hold as long as $r/(\log n)^6 \to\infty$, given that all parameters are bounded above by a constant.
\item We remove the previous condition $b_n^3 \left( \frac{\log n}{n} \right)^{3/2} \sum_{i\neq j}\left|\frac{e^{\beta_i+\beta_j} -1}{ 1 + e^{\beta_i+\beta_j}}  \right|=o(1)$, which requires
a large number of parameters being equal to zeros. This is a very strong condition. This condition was discarded by establishing a very small bound of a weighted cubic sum $\sum_i (\widehat{\beta}_i - \beta_i)^3$ in Lemma \ref{lemma:beta3:err}.
\item We present rigorous proofs for all supported lemmas and theorems.
We also present rigorous proofs of approximate error bound of using simple matrices to approximate Fisher information matrices under null spaces, where we discover an interesting phenomenon that
the error bound does not dependent $r$ in the $\beta$-model while it is not true in the Bradley--Terry model.
\item We develop a new method for proving the central limit theorem for the sum of weighted centered degrees.
We divide $\sum_{i=1}^r \bar{d}_i^{\,2}/v_{ii}$ into two parts:
\begin{eqnarray}\label{eq-lemma1-a}
\sum_{i=1}^r \frac{ (\bar{d}_i^{\,2} - \E \bar{d}_i^{\,2}) }{ v_{ii} }  & = &  \sum_{i=1}^r \sum_{j=1}^n \frac{ (\bar{a}_{ij}^2 - \E \bar{a}_{ij}^2 ) }{ v_{ii} }
+ \sum_{i=1}^r \sum_{j=1,j\neq i}^n \sum_{k=1, k\neq i,j}^n \frac{ \bar{a}_{ij} \bar{a}_{ik} }{ v_{ii} }.
\end{eqnarray}
The first summation in the right-hand side of the above equation scaled by $r^{1/2}$ varnishes while the second summation can be represented as
a sum of martingale differences with a delicate construction. Then we use Martingale theory to show its central limit theorem, leading to a weaker condition.
\item We have rewritten the paper and the proofs of Theorems. Moreover, we have added more explanations and simulation results to illustrate the Wilk' type of result does not hold in the Bradley--Terry model for the fixed dimensional specified null hypothesis.
\end{itemize}
}
The $\beta$-model, a name coined by \cite{Chatterjee:Diaconis:Sly:2011}, is an exponential family distribution on an undirected graph
with the degree sequence as the sufficient statistic. Specifically, the model assigns each node $i$ with its intrinsic
degree parameter $\beta_i$ and postulates that random edges, $a_{ij}\in\{0,1\}$ for $1\le i< j \le n$, occur independently with connection probabilities
\begin{equation}\label{model-beta}
\P(a_{ij}=1) = \frac{ e^{\beta_i + \beta_j} }{ 1 + e^{\beta_i + \beta_j} },
\end{equation}
where $n$ is the number of nodes in the graph.
The $\beta$-model can be viewed as the undirected version of an earlier $p_1$-model [\cite{Holland:Leinhardt:1981}]
and has been widely used to model degree heterogeneity in realistic networks [e.g., \cite{Park:Newman:2004,Blitzstein:Diaconis:2011,Chen:2020}].

Since the number of parameters grows with the number of nodes and the sample is only one realized graph,
asymptotic inference is nonstandard and turns out to be challenging [\cite{Goldenberg2010,Fienberg2012}].
This stimulate great interests in exploring theoretical properties of the $\beta$-model and some are known now, including
consistency of the maximum likelihood estimator (MLE) [\cite{Chatterjee:Diaconis:Sly:2011}], its central limit theorems [\cite{Yan:Xu:2013}]
and conditions of the MLE existence [\cite{Rinaldo2013}].
Asymptotic theories are also established in generalized $\beta$-models [e.g., \cite{Perry:Wolfe:2012,Hillar:Wibisono:2013,Yan:Leng:Zhu:2016,Graham2017,mukherjee2018,Chen:2020}]. 
However, likelihood ratio tests have not yet been explored in these works and their theoretical
properties are still unknown.

The likelihood ratio statistics play a very important role in parametric hypothesis testing problems. 
Under the large sample framework that the dimension of parameter space is fixed and the size of samples
goes to infinity, one of the most celebrated results is the Wilks theorem [\cite{wilks1938}].
That says minus twice log-likelihood ratio statistic under the null converges in distribution to
a Chi-square distribution with $k$ degrees of freedom independent of nuisance parameters,
where $k$ is equal to the difference between the dimension of the full parameter space and the dimension of null parameter space.
This appealing property 
was referred to as the Wilks phenomenon by \cite{fan2001}.
Since the dimension of parameter space often increases with the size of samples,
it is interesting to see whether the Wilks type of results continue to hold in high dimension settings.
In this paper, we investigate Wilks' theorems for both
increasing and fixed dimensional parameter testing problems in the $\beta$-model.
Our contributions are as follows.
\begin{itemize}
\item
For two increasing dimensional null hypotheses $H_0: \beta_i = \beta_i^0$ for $i=1, \ldots, r$ and $H_0: \beta_1=\cdots=\beta_r$,
we show that the normalized log-likelihood ratio statistic, $[2\{\ell(\widehat{\bs{\beta}}) - \ell(\widehat{\bs{\beta}}^0)\} -r]/(2r)^{1/2}$,
converges in distribution to the standard normal distribution as $r\to\infty$, where $\beta_i^0$ is a known number.
Here, $\ell( \bs{\beta})$ is the log-likelihood function on the vector parameter $\bs{\beta}=(\beta_1, \ldots, \beta_n)^\top$, $\widehat{\bs{\beta}}$ is its MLE under the full parameter space $\Theta= \R^n$,
and $\widehat{\bs{\beta}}^0$ is the restricted MLE under the null parameter space. 
 In other words,
$2(\ell(\widehat{\bs{\beta}}) - \ell(\widehat{\bs{\beta}}^0))$ is approximately a Chi-square distribution with a large degree $r$ of freedom.

\item
For a fixed $r$, under the specified null $H_0: \beta_i=\beta_i^0$, $i=1,\ldots, r$, and the homogenous null $H_0$: $\beta_1=\cdots=\beta_r$, we show that $2\{\ell(\widehat{\bs{\beta}}) - \ell(\widehat{\bs{\beta}}^0)\}$ converges in distribution to a Chi-square distribution with respective $r$ and $r-1$ degrees of freedoms, as the number of nodes $n$ goes to infinity.
That says the high dimensional likelihood ratio statistics behave like classical ones as long as the difference between the dimension of the full space and the dimension of the null space is fixed.

\item
The Wilks type of results are further extended into a closely related Bradley--Terry model for
paired comparisons [\cite{bradley-terry1952}], which assumes that subject $i$ is preferred to (or wins) subject $j$ with probability
$\exp(\beta_i - \beta_j)/( 1 + \exp(\beta_i - \beta_j) )$.
However, when testing  $H_0: \beta_i=\beta_i^0$, $i=1,\ldots, r$ with a fixed $r$,
a different phenomenon is discovered, in which
$2[\ell(\widehat{\bs{\beta}}) - \ell(\widehat{\bs{\beta}}^0)]$ in the Bradley--Terry model neither follows asymptotically a Chi-square nor a rescaled Chi-square distribution as in \cite{sur2019the}.

\end{itemize}

To the best of our knowledge, this is the first time to explore Wilks' theorems in both models with an increasing dimension.
Our mathematical arguments depend on the asymptotic expansion of
the log-likelihood function, up to  the fourth order.
Three innovated techniques are developed to analyze the expansion terms.
The first is the central limit theorem for the sum of quadratic normalized degrees $\sum_i  (d_i - \E(d_i))^2/v_{ii}$, where $d_i$ is the degree of node $i$.
The second is a small upper bound of  a weighted cubic sum $\sum_i f_i(\widehat{\beta}_i - \beta_i)^3$,
which has an additional vanishing factor $n^{-1/2}$ in contrast to the order of $\sum_i |f_i| |\widehat{\beta}_i - \beta_i|^3$.
The third is the consistency rate of the restricted MLE $\widehat{\bs{\beta}}^0$
and the approximate inverse of the Fisher information matrix under the null space.
In the case of fixed dimensional testing problems, we further establish the error bound between the MLE $\widehat{\bs{\beta}}$ and
the restricted MLE $\widehat{\bs{\beta}}^0$  having an order of $\log n/n$ in terms of the maximum norm,
and derive an upper bound of the absolute entry-wise maximum norm between two approximate inverses of the Fisher information matrices under the full space and the restricted null space.
These technical results are collected in Lemmas \ref{lemma:weighte-degree-al}, \ref{lemma-appro-beta-VS}, \ref{lemma-consi-beta}, \ref{lemma:beta3:err}, \ref{lemma-beta-approx-ho}, \ref{lemma:w2-error} and  \ref{lemma-hat-beta-diff}.

\subsection{Related work}

Hypothesis testing problems in random graph models have been studied  from different perspectives,
including detecting a planted clique in an Erd\"{o}s--R\'{e}nyi graph [\cite{verzelen2015community}],
goodness-of-fit tests in stochastic block models [\cite{Lei2016aos,Hu.2020.1722676}] 
or testing whether there are only one community or multiple communities [\cite{jin2019optimal}],
testing between two inhomogeneous Erd\"{o}s--R\'{e}nyi graphs [\cite{10.1214/19-AOS1884}].
However, the powerful likelihood ratio tests are not investigated in these works.

For an adjusted $\beta$-model in which the
connection probability  in \eqref{model-beta} has a rescaled factor $\lambda/n$ with a known parameter
$\lambda$,
\cite{mukherjee2018}  considered a homogeneous null hypothesis with all $\beta_i$ being equal to $0$ against
an  alternative hypothesis with a subset of $\{\beta_i\}$ strictly greater than $0$.
Such hypotheses imply that the connection probability $\P(a_{ij}=1)$ for any pair $(i,j)$ lies between
$\lambda/(2n)$ and  $\lambda/n$.
\cite{mukherjee2018} proposed three explicitly degree-based test statistics:
 $\sum_i d_i$, $\max_i d_i$ and a criticism test based on
$(d_i-\lambda/2)/(\lambda(1-\lambda/2n))^{1/2}$,
and established their asymptotic properties under some conditions.
Their problem settings are different from ours.
First, their null hypothesis is that all parameters are equal to zero while ours cover a wide range of parameter testing problems
including both fixed and increasing dimensions that are more practical relevant.
Second, their test statistics do not involve the MLEs while ours are likelihood ratio tests that are most powerful in the simple  null
according to the well-known Neyman-Pearson lemma.
Since likelihood ratio statistics depend on unknown MLEs, it needs to bridge the relationship between MLEs and observed random edge
variables and turns out to be more challenging. It requires to develop a central limit theorem for the sum of weighted quadratic degrees
and analyze various remainder terms in the expansion of $\ell(\bs{\widehat{\beta}})$ as mentioned in the proofs of our theorems.

We note that the $\beta$-model and the Bradley--Terry model can be recast into a logistic regression form.
Under the ``large $N$, diverging $p_N$" framework in generalized linear models,
\cite{wang2011} obtained a Wilks type of result for the Wald test  under a simple null when $p_N^3/N \to 0$.
In our case, $p_N^3/N \to \infty$, not $0$, where  the dimension of parameter space
 is $p_N = n$  and the total number of observations is $N=O(n^2)$.
In a different setting, by assuming that a sequence of independent and identical distributed samples 
from a regular exponential family, 
\cite{portnoy1988} showed a high dimensional Wilks type of result for the log-likelihood ratio statistic under the simple null.
For logistic regression models with asymptotic regime $p_N/N \to \kappa \in(0, 1/2)$, \cite{sur2019the} showed that the log-likelihood ratio statistic
for testing a single parameter under the null $\beta_i=0$,
converges to a rescaled Chi-square with an inflated factor greater than one.
In contrast, our results do not have such inflated factors and cover a wider class of hypothesis testing problems.

The rest of the paper is organized as follows.
The Wilks type of theorems for the $\beta$-model  and the Bradley--Terry model are presented in Sections \ref{section-beta-model}
and \ref{section-bt-model}, respectively.
Simulation studies and an application to NBA data are given in Section \ref{section-numerical}.
Some further discussions are given in Section \ref{section:discussion}.
Section \ref{section:proof} presents the proofs of Theorems \ref{theorem-LRT-beta} and \ref{theorem-LRT-beta-fixed}.
All other proofs including the proofs of Theorems \ref{theorem-ratio-bt-3} and \ref{theorem-ratio-bt-fixed}
as well as those of supported lemmas and  are relegated to the Supplemental Material.

\section{Wilks' theorems for the $\beta$-model}
\label{section-beta-model}

We consider an undirected graph $\mathcal{G}_n$ with $n$ nodes labelled as ``$1, \ldots, n$".
Let $A=(a_{ij})_{n\times n}$ be the adjacency matrix of $\mathcal{G}_n$, where
$a_{ij}$ denotes whether node $i$ is connected to node $j$.
That is, $a_{ij}$ is equal to $1$ if there is an edge connecting nodes $i$ and $j$; otherwise, $a_{ij}=0$.
Let $d_i = \sum_{j\neq i} a_{ij}$ be the degree of node $i$ and $\mathbf{d}=(d_1, \ldots, d_n)^\top$ be the degree sequence of $\mathcal{G}_n$.
The $\beta$-model postulates that all $a_{ij}$, $1\le i\neq j \le n$, are mutually independent
Bernoulli random variables with edge probabilities given in \eqref{model-beta}.

The logarithm of the likelihood function under the $\beta$-model in \eqref{model-beta} can be written as
\begin{equation*}
\ell(\boldsymbol{\beta}) = \sum_{1\le i < j \le n} \left\{a_{ij}(\beta_i + \beta_j) - \log(1 + e^{\beta_i + \beta_j})\right\} = \sum_{i=1}^n \beta_i d_i - \sum_{1\le i<j\le n} \log(1 + e^{\beta_i + \beta_j}),
\end{equation*}
where $\boldsymbol{\beta}=(\beta_1, \ldots, \beta_n)$.
As we can see, the $\beta$-model is an undirected exponential random graph model with the degree sequence
as the exclusively natural sufficient statistic.
Setting the derivatives with respect to $\beta_i$ to zero, we obtain the likelihood equations
\begin{equation} \label{eq-likelihood-beta}
d_i = \sum_{j\neq i} \frac{e^{\widehat{\beta}_i + \widehat{\beta}_j}}{1 + e^{\widehat{\beta}_i + \widehat{\beta}_j}},~~i=1,\ldots,n,
\end{equation}
where $\boldsymbol{\widehat{\beta}}=(\widehat{\beta}_1, \ldots, \widehat{\beta}_n)^\top$ is
the MLE of $\boldsymbol{\beta}=(\beta_1, \ldots, \beta_n)^\top$.
The fixed point iterative algorithm in \cite{Chatterjee:Diaconis:Sly:2011} can be used to solve $\boldsymbol{\widehat{\beta}}$.

With some ambiguity of notations, we use $V$ to denote the Hessian matrix of
the negative log-likelihood function under both the $\beta$-model and the Bradley--Terry model.
In the case of the $\beta$-model, the elements of $V$ ($=(v_{ij})_{n\times n}$) are
\begin{equation}\label{definition-v-beta}
v_{ii} = \sum_{j\neq i} \frac{e^{\beta_i+\beta_j}}{(1 + e^{\beta_i+\beta_j})^2}, ~~ v_{ij} =  \frac{e^{\beta_i+\beta_j}}{(1 + e^{\beta_i+\beta_j})^2}, ~~i\neq j; ~i,j=1,\ldots, n.
\end{equation}
Note that $V$ is also the Fisher information matrix of $\bs{\beta}$ and the covariance matrix of $\mathbf{d}$.
We define two notations that play important roles on guaranteeing good properties of  $\boldsymbol{\beta}$:
\begin{equation}\label{definition-bncn}
b_n = \max_{i,j} \frac{(1 + e^{\beta_i+\beta_j})^2}{e^{\beta_i+\beta_j}}, \quad c_n=\min_{i,j}\frac{(1 + e^{\beta_i+\beta_j})^2}{e^{\beta_i+\beta_j}},
\end{equation}
where $b_n^{-1}$ and $c_n^{-1}$ are equal to the minimum and maximum variances of $a_{ij}$ over $i\neq j$, and $c_n \ge 4$.

We first present Wilks' theorems in parameter testing problems with an increasing dimension.
We consider a specified null $H_0: \beta_i=\beta_i^0$ for $i=1,\ldots,r$ with $r\to\infty$, where $\beta_i^0$ for $i=1,\ldots,r$ are known numbers, and a homogeneous null
$H_0: \beta_1=\cdots=\beta_r$.
We assume that the random adjacency matrix $A$ is generated under the model with the parameter $\bs{\beta}$.
When $r=n$, the null $H_0: \beta_i=\beta_i^0$, $i=1,\ldots,r$ becomes the so-called simple null.
Recall that $\boldsymbol{\widehat{\beta}^{0}}$ denotes the restricted MLE of $\bs{\beta}$ under the null parameter space.

\begin{theorem}
\label{theorem-LRT-beta}
\begin{itemize}
\item[(a)]
Under the null
$H_0: \beta_i=\beta_i^0$, $i=1,\ldots,r$ with $r\to\infty$, if $b_n^5/c_n^2 = o( r^{1/2}/(\log n)^2 )$,
the log-likelihood ratio statistic $\ell(\boldsymbol{\widehat{\beta}}) - \ell(\boldsymbol{\widehat{\beta}}^0)$
is asymptotically normally distributed in the sense that
\begin{equation} \label{statistics-beta}
\frac{2\{ \ell(\boldsymbol{\widehat{\beta}}) - \ell(\boldsymbol{\widehat{\beta}}^0)\} - r}{\sqrt{2r}} \stackrel{L}{\rightarrow} N(0,1), ~~\mbox{as}~~ n\to\infty,
\end{equation}
where $\boldsymbol{\widehat{\beta}^{0}}=\arg\max_{\bs{\beta}\in \Theta_0} \ell(\bs{\beta})$ and $\Theta_0=\{
 \bs{\beta}: \bs{\beta}\in \R^n, (\beta_1, \ldots, \beta_r) = (\beta_1^0, \ldots, \beta_r^0) \}$.
\item[(b)]
Under the homogenous null
$H_0: \bs{\beta}\in \Theta_0=\{ \bs{\beta}: \bs{\beta}\in \R^n, \beta_1=\cdots=\beta_r\}$,
if $b_n^{15}/c_n^9 = o\left( r^{1/2}/(\log n)^3 \right)$,
the normalized log-likelihood ratio statistic
in \eqref{statistics-beta} also converges in distribution to the standard normality.
\end{itemize}
\end{theorem}

The condition imposed on $b_n$ in Theorem \ref{theorem-LRT-beta} is used to control the increasing rate of $b_n$.
If all $\beta_i$ are different not too much, $b_n \asymp c_n$ and the condition in Theorem \ref{theorem-LRT-beta} (a) becomes
$\max_{i,j} \exp(\beta_i+\beta_j) = o ( n^{1/6}/( \log n)^{2/3} )$.
Further, the condition in Theorem \ref{theorem-LRT-beta} (b) is stronger than that in Theorem \ref{theorem-LRT-beta} (a).
This is partly due to that use a unified consistency rate in Lemma \ref{lemma-con-beta-b} that holds for any $r$ under the specified null.
We note that consistency of the MLE in \cite{Chatterjee:Diaconis:Sly:2011}
is based on the condition that all parameters are bounded above by a constant while
asymptotic normality of the MLE in \cite{Yan:Xu:2013} needs the condition: $\max_i |\beta_i| = o( \log (\log n))$.
In contrast, the condition here seems weaker.
In addition, some intermediate results in Lemmas \ref{lemma:weighte-degree-al},
\ref{lemma-clt-beta-W}, \ref{lemma-consi-beta} and \ref{lemma-con-beta-b}  are built under weaker conditions.
For instance, the consistency rate of $\bs{\widehat{\beta}}^0$ in Lemma \ref{lemma-consi-beta}
only requires  $b_n^2/c_n=o( n^{1/2}/(\log n)^{1/2})$.

The following corollary gives the smallest $r$ to guarantee Wilks' type of results,
which only requires $r$ far larger than a logarithm factor to the power of $6$.

\begin{corollary} \label{corollary-theorem1}
If $b_n$ is bounded by a constant and $r/(\log n)^6\to\infty$,
the normalized log-likelihood ratio statistic in \eqref{statistics-beta} converges
in distribution to the standard normality under both specified and homogenous null hypotheses.
\end{corollary}

We describe briefly the idea for proving Theorem \ref{theorem-LRT-beta} here.
We apply a fourth-order Taylor expansion to  $\ell(\boldsymbol{\widehat{\beta}})$
and $\ell(\boldsymbol{\widehat{\beta}}^0)$ at point $\bs{\beta}$, respectively.
With the use of the maximum likelihood equations and the asymptotic representations of $\widehat{\bs{\beta}}$ and $\widehat{\bs{\beta}}^0$
(see \eqref{eq-expansion-hatbeta-beta} and \eqref{eq-beta0-exapnsion}),
the first-order and second-order expansion terms in the difference $\ell(\boldsymbol{\widehat{\beta}})-\ell(\boldsymbol{\widehat{\beta}}^0)$
can be expressed as the difference between
$\bs{\bar{d}}^\top V^{-1} \bs{\bar{d}}$ and $\bs{\bar{d}}_2^\top V_{22}^{-1} \bs{\bar{d}}_2$
($\bs{\tilde{d}}^\top \widetilde{V}^{-1}\bs{\tilde{d}}$ under the homogenous null; see \eqref{B10-homo-expression}) and several remainder terms,
where $V_{22}$ is the  bottom right  $(n-r)\times (n-r)$ block of $V$,
$\bs{\bar{d}}= \mathbf{d} - \E \mathbf{d}$ and $\bs{\bar{d}}_2$ is the last $n-r$ elements of $\bs{\bar{d}}$.
The left arguments are to show that the difference is approximately a Chi-square distribution with a large degree $r$ of freedom and various remainder terms tend to zero. The aforementioned technical results in Lemmas \ref{lemma-clt-beta-W}--\ref{lemma-beta-homo-expan} are used to bound remainder terms.

By using a simple matrix $S_{22}=\diag(1/v_{r+1,r+1}, \ldots, 1/v_{nn})$ in \eqref{definition-S} to approximate
$V_{22}^{-1}$, one can find that the main term includes a sum of a sequence of normalized degrees in a weighted quadratic form, i.e.,
$\sum_{i=1}^r \bar{d}_i^{\,2}/v_{ii}$, where $\bar{d}_i=d_i-\E d_i$.
For single $i$, $\bar{d}_i^{\,2}/v_{ii}$ is asymptotically a Chi-square distribution
and $(d_i, d_j)$ for any pair $(i,j)$ is asymptotically independent.
But for all $i$, the terms in the sum are not independent.

Note that $\E \sum_{i=1}^r \bar{d}_i^{\,2}/v_{ii} =r$. By exploiting  the independence of the triangular matrix of $A$, the variance of
$\sum_{i=1}^r \bar{d}_i^{\,2}/v_{ii}$ can be calculated as
\[
\mathrm{Var}\left( \sum_{i=1}^r \frac{\bar{d}_i^{\,2}}{v_{ii}} \right)=\sum_{i=1}^r\frac{1}{v_{ii}^2}(2v_{ii}^2+\sum_{j=1,j\neq i}^n u_{ij}) + 2\sum_{1\le i\neq j\le r} \frac{u_{ij}}{v_{ii}v_{jj}},
\]
where $u_{ij}=\mathrm{Cov}( \bar{a}_{ij}^{\,2}, \bar{a}_{ji}^{\,2})$ and $\bar{a}_{ij}=a_{ij}-\E a_{ij}$.
Because the variance of $\bar{a}_{ij}^2$ is
\[
u_{ij}=p_{ij}^4q_{ij} + q_{ij}^4p_{ij} - p_{ij}^2q_{ij}^2=p_{ij}q_{ij}( p_{ij}^3 + q_{ij}^3 - p_{ij}q_{ij}) \le p_{ij}q_{ij} \le \frac{1}{c_n},
\]
where $p_{ij}$ is the probability of node $i$ connecting node $j$ given in \eqref{model-beta}
and $q_{ij}=1-p_{ij}$, we have
\[
\max_{i,j}\left\{\frac{\sum_{j=1,j\neq i}^nu_{ij}}{v_{ii}^2}+2 \sum_{1\le i\neq j\le r} \frac{u_{ij}}{v_{ii}v_{jj}} \right\} \le \frac{b_n^2}{(n-1)c_n}
+ \frac{ 2r^2 b_n^2 }{ (n-1)^2c_n },
\]
It follows that if $b_n^2/c_n=o(n)$, the limit of the ratio of the variance of $\sum_{i=1}^r \bar{d}_i^{\,2}/v_{ii}$ to $2r$ is $1$.
In view of the weak dependence of $\bar{d}_i$, this sum can be approximated by the Chi-square distribution with a large degree $r$ of freedom, as stated in the following lemma.

\begin{lemma}\label{lemma:weighte-degree-al}
Under the $\beta$-model,
if $b_n^4/c_n^3=o(n)$, then
$\sum_{i=1}^r \bar{d}_i^{\,2}/v_{ii}$ is asymptotically normally distributed with mean $r$ and variance $2r$,
where $\bar{d}_i= d_i - \E d_i$.
\end{lemma}

The above lemma shows that the normalized sum $(\sum_{i=1}^r \bar{d}_i^{\,2}/v_{ii}-r)/(2r)^{1/2}$ converges in distribution to the standard normality for arbitrary $r$ tending to infinity in the case of that $b_n$ is a constant.
The proof of Lemma \ref{lemma:weighte-degree-al} is technical.
The quadratic centered degree sequence $\{\bar{d}_i^{\,2}\}_{i=1}^r$ is not independent and also not the commonly seen mixing sequences such as
$\alpha$-mixing, $\phi$-mixing and so on. As a result, classical central limit theorems for independent random variables or
dependent random variables [e.g., \cite{Peligrad1987,withers1987central}] can not be applied. Further, it is not a natural martingale.
Observe that $\bar{d}_i^{\,2} = \sum_{j,k\neq i} \bar{a}_{ij}\bar{a}_{ik}$ and $\E(\bar{a}_{ij}\bar{a}_{ik}|\bar{a}_{ij})=0$.
This is analogous to the property of vanishing conditional expectations in one-sample $U$-statistics [e.g., \cite{HALL19841}] and the quadratic form
$w(X_i, X_j)$ [e.g., \cite{Jong1987PTRF}], where $\E\{ w(X_i, X_j)|X_i \}=0$ for a sequence of independent random variables $\{X_i\}$ and
Martingale theory are used to derive the central limit theorem of the sum $\sum_{i<j} w(X_i,X_j)$. Since there are three indices in the sum
$\sum_i\sum_{j,k\neq i} \bar{a}_{ij}\bar{a}_{ik}$, the methods of constructing martingale in \cite{HALL19841} and \cite{Jong1987PTRF} can not be used here. For the sake of obtaining its asymptotic distribution, we divide $\sum_{i=1}^r \bar{d}_i^{\,2}/v_{ii}$ into two parts:
\begin{eqnarray}\label{eq-lemma1-a}
\sum_{i=1}^r \frac{ (\bar{d}_i^{\,2} - \E \bar{d}_i^{\,2}) }{ v_{ii} }  & = &  \sum_{i=1}^r \sum_{j=1}^n \frac{ (\bar{a}_{ij}^2 - \E \bar{a}_{ij}^2 ) }{ v_{ii} }
+ \sum_{i=1}^r \sum_{j=1,j\neq i}^n \sum_{k=1, k\neq i,j}^n \frac{ \bar{a}_{ij} \bar{a}_{ik} }{ v_{ii} }.
\end{eqnarray}
The first summation in the right-hand side of the above equation scaled by $r^{1/2}$ varnishes while the second summation can be represented as
a sum of martingale differences with a delicate construction. Then we can use Martingale theory [e.g., \cite{Brown1971}] to obtain its central limit theorem, whose details are given in the supplementary material.

Next, we present Wilks' theorems for fixed dimensional parameter hypothesis testing problems.
We consider the specified null $H_0: \beta_i=\beta_i^0$ for $i=1,\ldots,r$ and the homogenous null
$H_0: \beta_1=\cdots=\beta_r$, where $r$ is a fixed positive integer.

\begin{theorem}
\label{theorem-LRT-beta-fixed}
Assume that $b_n^3/c_n=o(n^{1/6}/(\log n))$ and $r$ is a fixed positive integer.
\begin{itemize}
\item[(a)]
Under the null $H_0: \beta_i=\beta_i^0, i=1, \ldots, r$,
the minus twice log-likelihood ratio statistic $2\{\ell(\widehat{\bs{\beta}}) - \ell( \widehat{\bs{\beta}}^0) \}$
converges in distribution to a Chi-square distribution with $r$ degrees of freedom as $n$ goes to infinity.

\item[(b)]
Under the homogenous null $H_0: \beta_1 = \cdots = \beta_r$,
$2\{\ell_\beta(\widehat{\bs{\beta}}) - \ell_\beta( \widehat{\bs{\beta}}^0) \}$
converges in distribution to a Chi-square distribution with $r-1$ degrees of freedom as $n$ goes to infinity.
\end{itemize}

\end{theorem}

Theorem \ref{theorem-LRT-beta-fixed} says that the log-likelihood ratio enjoys the classical Wilks theorem in the case that the difference
between the full space and the null space of the tests is fixed. As mentioned before, the condition imposed on $b_n$ restricts the increasing rate of $b_n$
and is fully filled when $b_n$ is a constant.
The proof of Theorem \ref{theorem-LRT-beta-fixed} needs additional technical steps, in contrast to  the proof of Theorem \ref{theorem-LRT-beta}.
As mentioned before, it requires to bound $\max_{i=r+1,\ldots,n}| \widehat{{\beta}}_i - \widehat{{\beta}}^0_i |$ in Lemmas \ref{lemma-hat-beta-diff}
and \ref{lemma-hat-beta-diff-2b} and to evaluate the maximum absolute entry-wise difference between two approximate inverse matrices in Lemmas \ref{lemma:w2-error} and \ref{lemma:w2-error-2b}.
Further, it needs to carefully analyze the differences between remainder terms under the
full space and the null space since we do not have a scaled vanishing factor $r^{-1/2}$ as in Theorem \ref{theorem-LRT-beta}.

\section{Wilks' theorems for the Bradley--Terry model}
\label{section-bt-model}

In the above section, we considered an undirected graph. Now we consider a weighted directed graph $\mathcal{G}_n$,
where nodes denote subjects joining in paired comparisons and the element of the adjacency matrix $A$ denotes the number of times that one subject is preferred to another subject.
Let $k_{ij}$ be the number of comparisons between subjects $i$ and $j$.
For easy exposition, similar to \cite{simons-yao1999}, we assume $k_{ij}=K$ for all $i\neq j$, where $K$ is a fixed positive constant.
Then, $a_{ij}$ is the number of times that $i$ wins $j$ out of a total number of $K$ comparisons.

The Bradley--Terry model postulates that $a_{ij}$, $1\le i<j \le n$, are mutually independent
binomial random variables, i.e., $a_{ij} \sim \mbox{Binomial}(K, p_{ij})$, where
\begin{equation}\label{btmodel}
p_{ij} = \frac{ e^{ \beta_i - \beta_j } }{ 1 + e^{ \beta_i - \beta_j } }.
\end{equation}
Here, $\beta_i$ measures the intrinsic strength of subject $i$, and
the win-loss probabilities for any two subjects only depend on the difference of their strength parameters.
The bigger the strength parameter is, the higher the probability of subject $i$ having a win over other subjects is.
Let $d_i=\sum_{j\neq i} a_{ij}$ be the total number of wins for subject $i$.

Because the probability is invariable by adding a common constant to all strength parameters $\beta_i$, $i=1, \ldots, n$,
we need a restriction for the identifiability of model. Following \cite{simons-yao1999}, we set $\beta_1=0$ as a constraint.
Notice that the number of free parameters here is $n-1$, different from the $\beta$-model with $n$ free parameters.
The logarithm of the likelihood function under the Bradley--Terry model is
\begin{equation}\label{likelihood-bt}
\ell_{bt}(\boldsymbol{\beta}) = \sum_{i,j=1;i\neq j}^n a_{ij} \left\{\beta_i - \log(e^{\beta_i}+e^{\beta_j})\right\} = \sum_{i=1}^n \beta_id_i  - K \sum_{1\le i<j\le n} \log(e^{\beta_i}+e^{\beta_j}),
\end{equation}
where $\boldsymbol{\beta} = (\beta_2, \ldots, \beta_n)^\top$ and $\beta_1=0$.
To distinguish the log-likelihood function in the $\beta$-model, we use a subscript $bt$ in this section.
As we can see, it is an exponential family distribution on the directed graph $\mathcal{G}_n$ with the out-degree sequence as its natural sufficient statistic.
Setting the derivatives with respect to $\beta_i$ to zero, we obtain the likelihood equations
\begin{equation} \label{estimated-eq-bt-a}
d_i = \sum_{j=1,j\neq i}^n \frac{K e^{\hat{\beta}_i}}{e^{\hat{\beta}_i} + e^{\hat{\beta}_j}}, ~~i=2,\ldots,n,
\end{equation}
where $\boldsymbol{\widehat{\beta}} = (\widehat{\beta}_2, \ldots, \hat{\beta}_n)$ is the MLE of $\boldsymbol{\beta}$ with $\widehat{\beta}_1=0$.
If the directed graph $\mathcal{G}_n$ is strongly connected,
then the MLE uniquely exists [\cite{Ford1957}].
Note that $d_1$ is not involved in \eqref{estimated-eq-bt-a}; indeed, given $d_2, \ldots, d_n$ and $K$, $d_1$ is determined.

Now, we present the Wilks type of theorems for the Bradley--Terry model.
The corresponding definitions of $b_n$ and $c_n$ are as follows:
\begin{equation*}\label{definition-bncn}
b_n = \max_{i,j} \frac{(1 + e^{\beta_i-\beta_j})^2}{e^{\beta_i-\beta_j}},
\quad c_n=\min_{i,j}\frac{(1 + e^{\beta_i-\beta_j})^2}{e^{\beta_i-\beta_j}}.
\end{equation*}
With some ambiguity, we use the same notations $b_n$ and $c_n$ as in the $\beta$-model, where their expressions are based on $\beta_i+\beta_j$.

\begin{theorem} \label{theorem-ratio-bt-3}
Suppose $b_n^7/c_n^4 = o( r^{1/2}/(\log n)^{2})$.
\begin{itemize}
\item[(a)]
Under the specified null
$H_0: \beta_i=\beta_i^0$, $i=2,\ldots,r$, the log-likelihood ratio statistic $\ell_{bt}(\boldsymbol{\widehat{\beta}}) - \ell_{bt}(\boldsymbol{\widehat{\beta}}^0)$
is asymptotically normally distributed in the sense that
\begin{equation} \label{statistics-bt}
\frac{2\{ \ell_{bt}(\boldsymbol{\widehat{\beta}}) - \ell_{bt}(\boldsymbol{\widehat{\beta}}^0)\} - r}{\sqrt{2r}} \stackrel{L}{\rightarrow} N(0,1), ~~\mbox{as}~~ r\to\infty,
\end{equation}
where $\boldsymbol{\widehat{\beta}^{0}}=\arg\max_{\bs{\beta}\in \Theta_0} \ell_{bt}(\bs{\beta})$ and $\Theta_0=\{
 \bs{\beta}: \bs{\beta}\in \R^{n-1}, (\beta_2, \ldots, \beta_r) = (\beta_2^0, \ldots, \beta_r^0) \}$.
\item[(b)]
Under the homogenous null
$H_0: \bs{\beta}\in \Theta_0=\{ \bs{\beta}: \bs{\beta}\in \R^n, \beta_2=\cdots=\beta_r\}$, the normalized log-likelihood ratio statistic
in \eqref{statistics-bt} also converges in distribution to the standard normality.
\end{itemize}
\end{theorem}

The principled strategy for proving Theorem \ref{theorem-LRT-beta} is extended to prove the above theorem.
However, we emphasize some main differences, including different approximate inverses for the Fisher information matrices under the null space, different asymptotic representations of the MLE and restricted MLE and different methods for obtaining consistency rates. Roughly speaking,
we use a diagonal matrix to approximate the Fisher information matrix in the $\beta$-model while the approximate inverse is
a diagonal matrix plus a commonly exceptive number in the Bradley--Terry model.
Second, the main term in the asymptotic representation of $\widehat{\beta}_i$ is $\bar{d}_i/v_{ii}$ in the $\beta$-model while
it is $\bar{d}_i/v_{ii} - \sum_{i=1}^r \bar{d}_i /\tilde{v}_{11}$ under the specified null or $\bar{d}_i/v_{ii} - \bar{d}_1 /v_{11}$
in the homogenous null in the Bradley--Terry model, where $\tilde{v}_{11}=\sum_{i=1}^r v_{ii}$.
Third, the Newton method is used to obtain consistency rate in the $\beta$-model while we use
the common neighbors between any two of subjects as middleman, who have
ratios being simultaneously close to $\max_i (\widehat{\beta}_i - \beta_i)$ and $\min_i (\widehat{\beta}_i - \beta_i)$ to establish the error bound of $\max_i (\widehat{\beta}_i - \beta_i)- \min_i (\widehat{\beta}_i - \beta_i)$ in the Bradley-Terry model as in \cite{simons-yao1999}.

Note that in order to guarantee the existence of the MLE with high probability,
it is necessary to control the increasing rate of $b_n$ as discussed in \cite{simons-yao1999}.
In the case that some $\beta_i$'s are
very large while others are very small, corresponding to a large value of $b_n$,
the subjects with relatively poor merits will stand very little chance of beating those
with relatively large merits.
Whenever all subjects could be partitioned into two sets,
in which the subjects in one set will win all games against those in the other set,
the MLE will not exist [\cite{Ford1957}].

Note that in the above discussion, we have assumed the $k_{ij}$'s, $i\neq j$ are all equal to a constant $K$.  This is only for the purpose of simplifying notations. Theorem \ref{theorem-ratio-bt-3} can be readily extended to the general case,
where $k_{ij}$'s are not necessarily the same (but with a bound).

Next, we present Wilks' theorem under the homogenous testing problem with a fixed dimension.

\begin{theorem}
\label{theorem-ratio-bt-fixed}
If $b_n^{11}/c_n^6 = o( n^{1/2}/(\log n)^{5/2})$,
under the homogenous null $H_0: \beta_2 = \cdots = \beta_r$ with a fixed $r > 2$,
the twice log-likelihood ratio statistic $2\left\{\ell_{bt}(\widehat{\bs{\beta}}) - \ell_{bt}( \widehat{\bs{\beta}}^0) \right\}$
converges in distribution to a Chi-square distribution with $r-2$ degrees of freedom.
\end{theorem}

Different from Theorem \ref{theorem-LRT-beta-fixed} in the $\beta$-model,
the above theorem does not contain a Wilks type of result under
the fixed dimensional specified null $H_0: \beta_i=\beta_i^0, i=2, \ldots, r$.
Some explanations are as follows.
With the use of $S_{22}=\mathrm{diag}( 1/v_{r+1,r+1}, \ldots, 1/v_{nn}) + 1/\tilde{v}_{11}$ to approximate the Fisher information matrix $V_{22}$
under the specified null and with similar arguments as in the proof of \eqref{eq-ell-difference} and \eqref{eq-theorem2-B10}, we have
\begin{equation*}
2\left\{ \ell_{bt}(\widehat{\bs{\beta}}) - \ell_{bt}( \widehat{\bs{\beta}}^0) \right\}=
\sum_{i=1}^r \frac{ \bar{d}_i^{\,2} }{ v_{ii}}  - \frac{ \{ \sum_{i=1}^r \bar{d}_i \}^2 }{ \tilde{v}_{11} }+ \frac{1}{3}(B_2-B_2^0)+ \frac{1}{12}(B_3-B_3^0),
\end{equation*}
where $\tilde{v}_{11}=\sum_{i=1}^r v_{ii}$,
$(B_2-B_2^0)$ is the difference of the third-order expansion term of the log-likelihood function between the full space and the null space, and $(B_3-B_3^0)$ is the corresponding difference of the fourth-order expansion term.
If $v_{11}=\cdots=v_{rr}$, then $\sum_{i=1}^r \bar{d}_i^{\,2}/v_{ii} - \{ \sum_{i=1}^r \bar{d}_i \}^2/\tilde{v}_{11}$ asymptotically
follows a Chi-square distribution. Even if $v_{11}=\cdots=v_{rr}$, $2\left\{ \ell_{bt}(\widehat{\bs{\beta}}) - \ell_{bt}( \widehat{\bs{\beta}}^0) \right\}$ is not approximately a Chi-square distribution.
This is because $B_2-B_2^0$ does not goes to zero whereas it vanishes in the $\beta$-model.
In the case of fixed $r$, a key quantity to measure $B_2-B_2^0$ is
$\max_{i=r+1, \ldots, n}|\widehat{\beta}_i - \widehat{\beta}_i^0|$.
It has the order of $\log n/n$ in the $\beta$-model whereas in the Bradley-Terry model the difference have the following representation:
\[
\widehat{\beta}_i - \widehat{\beta}_i^0 = \frac{ \bar{d}_1 }{ v_{11} } - \frac{ \sum_{i=1}^r \bar{d}_i  }{ \tilde{v}_{11} } + O_p\left( \frac{ b_n^2 \log n }{ n} \right),~~ i=r+1, \ldots, n,
\]
under the specified null.
The difference of two distributions of $\bar{d}_1/v_{11}$ and $(\sum_{i=1}^r \bar{d}_i )^2/\tilde{v}_{11}$ is much larger than the order of $\log n/n$. Under the homogenous null, the approximate inverse is $\mathrm{diag}(1/\tilde{v}_{22}, v_{r+1,r+1}, \ldots, v_{nn}) + 1/v_{11}$,
where the off-diagonal elements are the same as the approximate inverse for approximating $V^{-1}$ in the full parameter space.
This makes that $\widehat{\beta}_i - \widehat{\beta}_i^0$ does not contain the difference of the above two terms.
It leads to that $B_2-B_2^0$ vanishes in the homogenous null while it does not vanish in the specified null.
Therefore, the Wilks type of result does not hold in the fixed dimensional specified null in the Bradley--Terry model.

To give some intuition on the distribution of $2\left\{ \ell_{bt}(\widehat{\bs{\beta}}) - \ell_{bt}( \widehat{\bs{\beta}}^0) \right\}$ under the specified null, we draw its density curve against that of the Chi-square distribution. We consider several specified null $H_0: (\beta_2, \beta_3)=(-c, c)$ with $c=0, 0.5, 1$ and other parameters are $\beta_i= 0.2(i-1)\log n/(n-1)$, where $\beta_0=0$.
The plots are shown in Figure \ref{fig:bt-a}, where the simulation is repeated $5,000$ times.
As we can see, three findings are: (1) the distribution is far away from chi-square distributions with degree $2$ nor $3$;
(2) even if $\beta_2=\beta_3=0$, the density curve is very different from that of a Chi-square distribution;
(3) the density curve depends  crucially on $n$ and seems not sensitive to the choices of parameters.

\begin{figure}[htbp]
\centering
\caption{Comparisons of density curves of LRT and Chi-square distributions with degrees 2 or 3 freedom.}
\subfigure[Performance of minus twice log-likelihood ratio statistic under the specified null (n=200)]{\includegraphics[width=0.9\textwidth]{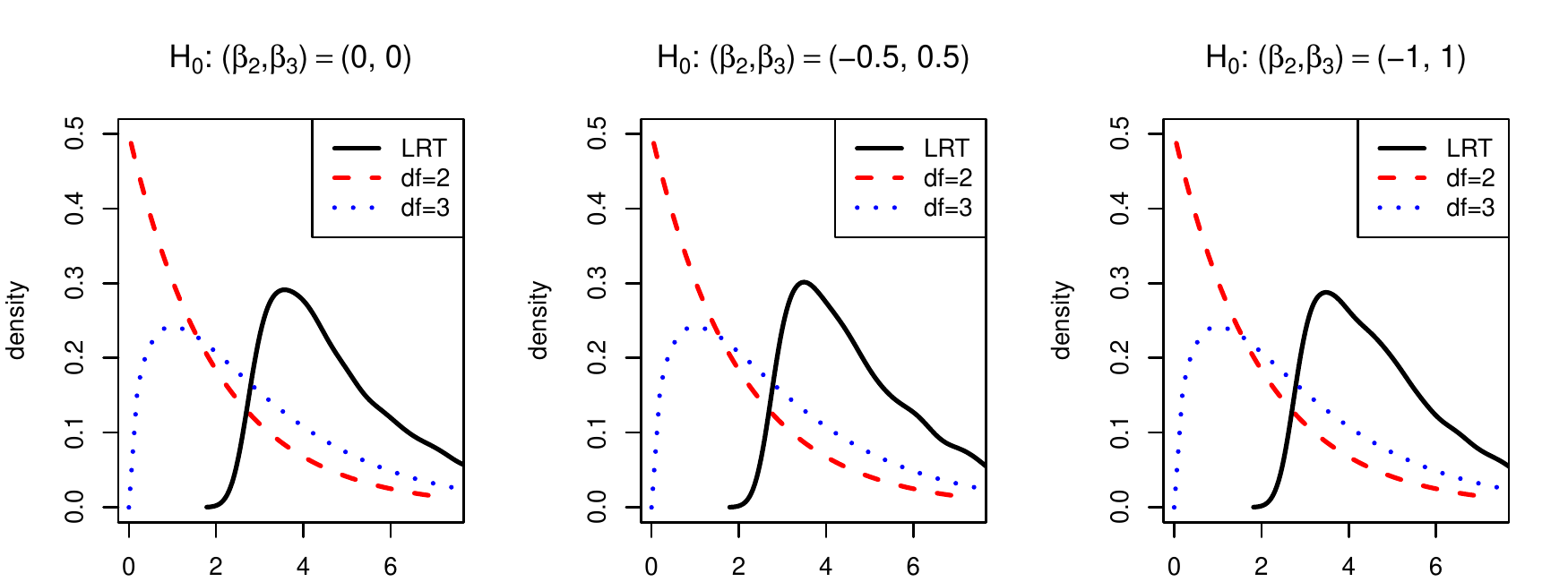}}
\subfigure[n=100]{\includegraphics[width=0.9\textwidth]{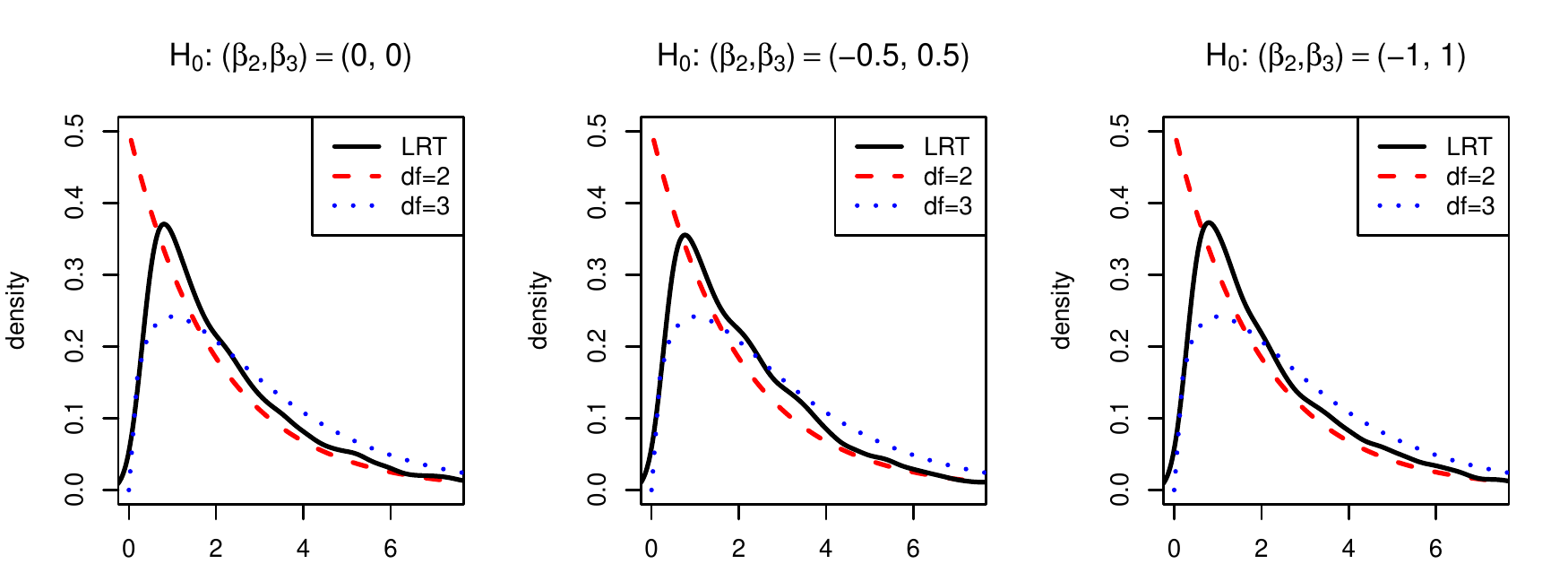}}
\label{fig:bt-a}
\end{figure}

\section{Numerical Results}
\label{section-numerical}
In this section, we illustrate the theoretical results via numerical studies.

\subsection{Simulation studies}

We carry out simulations to evaluate the performance of the log-likelihood ratio statistics for finite number of nodes.
We considered the four null hypotheses: (1) $H_{01}$: $\beta_i=(i-1)L_n/(n-1)$, $i=1,\ldots,n$;
(2) $H_{02}$: $\beta_1=\cdots=\beta_r$, $r=n/2$;
(3) $H_{03}$: $\beta_i= (i-1)r/5$;
(4) $H_{04}$: $\beta_1=\cdots=\beta_r$ with a fixed $r$, where $L_n$ is set to evaluate different asymptotic regimes.
$H_{01}$ corresponds to the so-called simple null while $H_{03}$ aims to test
whether a fixed number of parameters are equal to specified values.
$H_{02}$ and $H_{04}$ aim to test whether a given set of parameters with increasing or fixed dimensions are equal.
Under $H_{02}$, $H_{03}$ and $H_{04}$, we set the left $n-r$ parameters as: $\beta_i = (i-1)L_n/(n-1)$ for $i=r+1,\ldots,n$.
For homogenous null $H_{02}$ and $H_{04}$, we set $\beta_1 = \cdots=\beta_r=0$.
Remark that in the Bradley--Terry model, $\beta_1$ $(=0)$ is a reference parameter and is excluded in the above null.
Four values for $L_n$ were chosen, i.e., $L_n= 0$, $0.2\log n$, $0.4\log n$ and $0.5\log n$.

We evaluate the Type I errors and powers of the log-likelihood ratio statistics, and draw their QQ plots. 
For the increasing dimensional null hypotheses, we use the Chi-square approximation instead of the normal approximation due to that
the former performs better than the latter in finite sample sizes.
Two values for $n$ were considered: $n=100$ and $n=200$.
For the Bradley--Terry model, we assumed that each pair has one comparison, i.e., $K=1$. Further,  motivated by the schedules of the NBA regular season that is briefly described in next section,
we considered additionally a relatively small size $n=30$ and let the number of paired comparisons $k_{ij}$ equal to 3 for all $1\le i\neq j\le n$.
Each simulation was repeated $5,000$ times.

Due to the limited space, we only show plots of  quantiles of Chi-square distributions vs sample quantiles in the case of $n=200$ under the $\beta$-model
and other QQ plots are similar.
From figure \ref{fig:beta}, we can see that the sample quantiles agree well with theoretical quantiles when $L=0, 0.2\log n$.
On the other hand, when $L=0.4\log n$, there are a little derivation from the reference line $y=x$ under the null $H_{01}$.
When $L=0.5\log n$, the MLE failed to exist with a positive frequency (see Table \ref{table-type-I}).

\begin{figure}[htbp]
\centering
\caption{QQ plots for the $\beta$-model under the null.  The horizontal and vertical axes in each QQ-plot are the respective theoretical (based on the Chi-square distribution) and empirical quantiles.
The straight lines correspond to $y=x$. The first, second, third column correspond to $L_n=0, 0.2\log n, 0.4\log n$, respectively.}
\subfigure[QQ-plot for normalized log-likelihood ratio statistic in \eqref{statistics-beta} under $H_{01}$]{\includegraphics[width=0.9\textwidth]{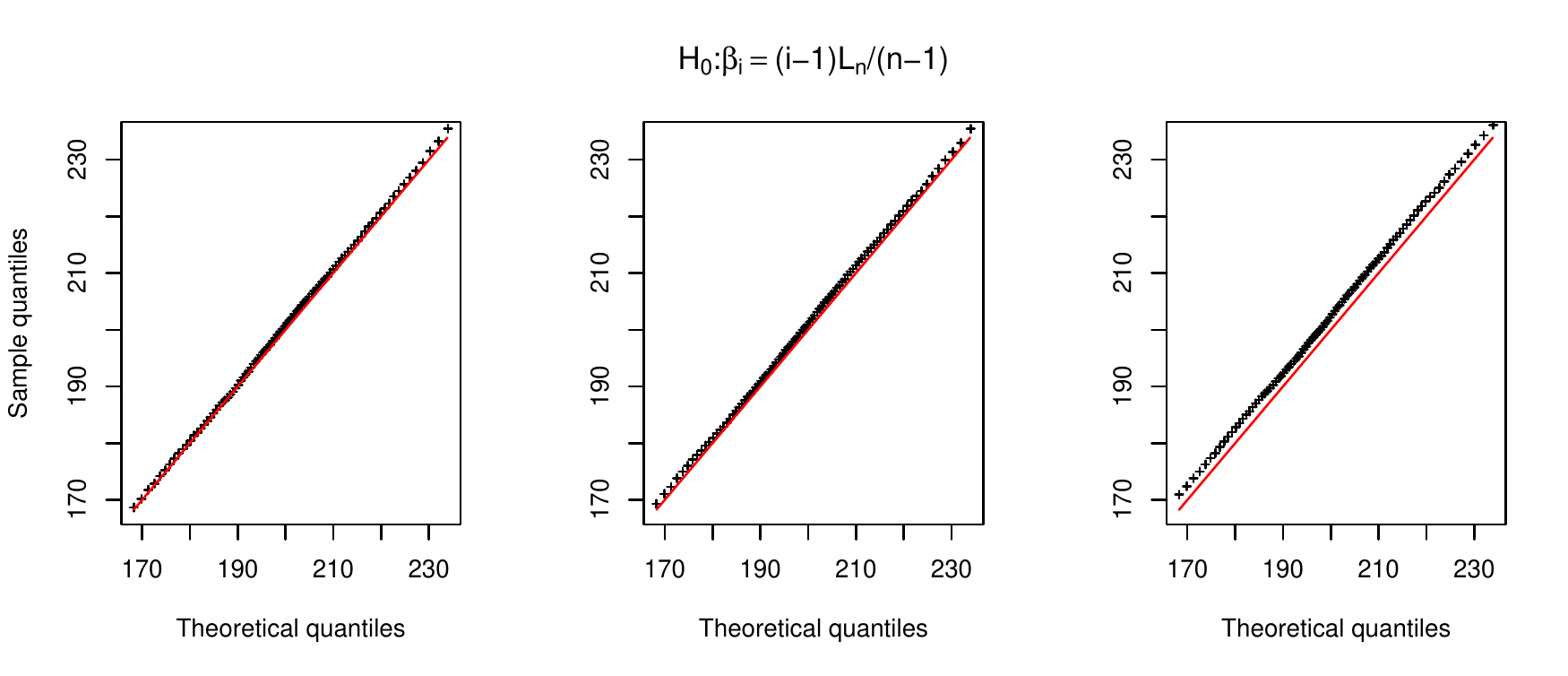}}
\subfigure[QQ-plot for normalized log-likelihood ratio statistic in \eqref{statistics-beta} under $H_{02}$]{\includegraphics[width=0.9\textwidth]{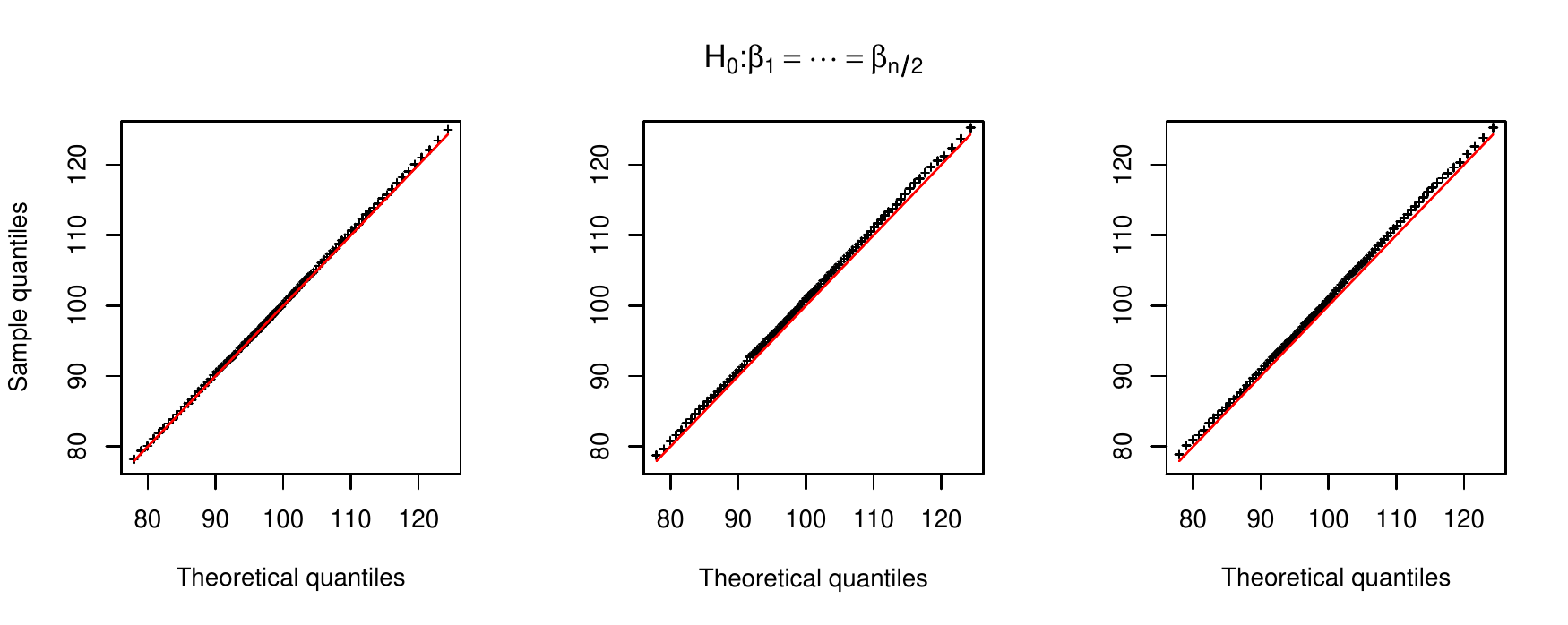}}
\subfigure[QQ-plot for log-likelihood ratio statistic under $H_{03}$]{\includegraphics[width=0.9\textwidth]{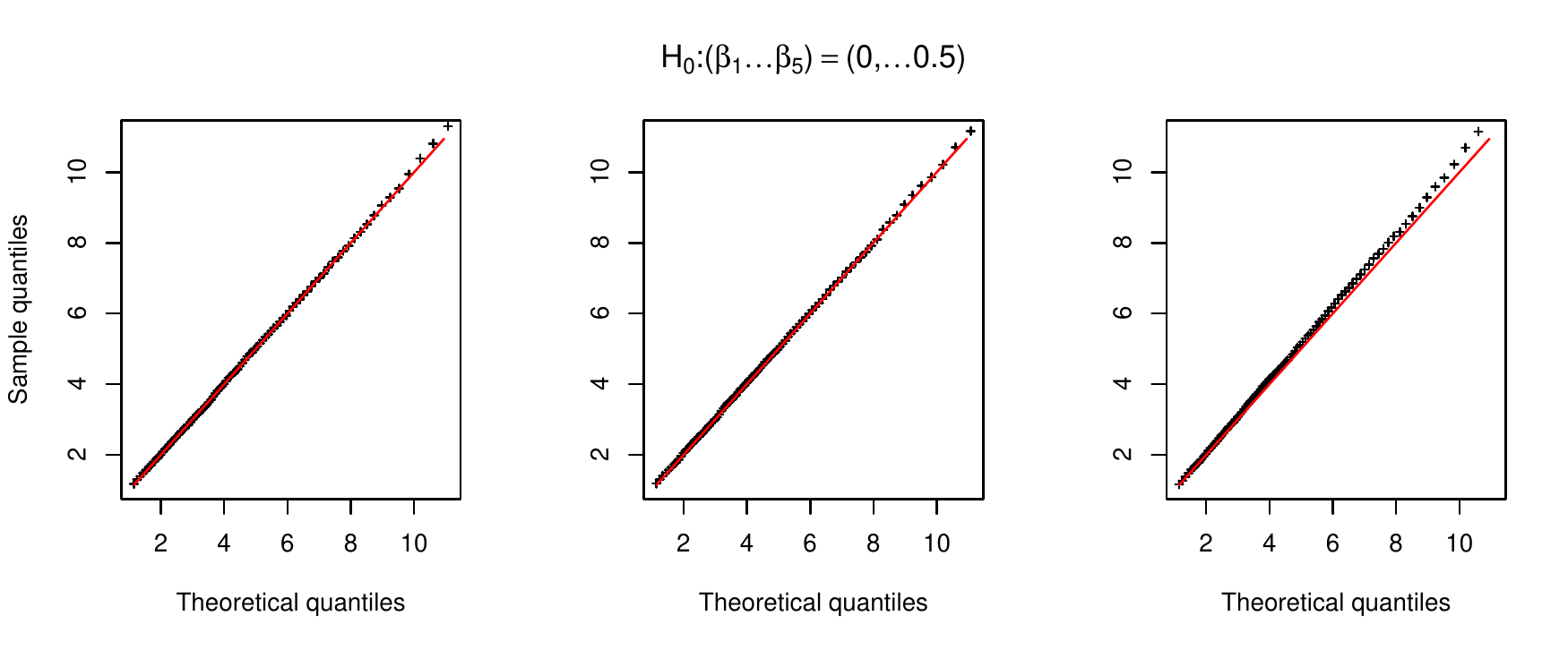}}
\label{fig:beta}     
\end{figure}

\begin{table}[h!]
\centering
\caption{Type I errors $(\times 100)$ of the likelihood ratio tests under the nominal levels $\alpha=0.05, 0.1$~/frequencies $(\times 100)$ that the MLE does not exist ($L_n = c\log n$).}
\label{table-type-I}
\vskip5pt
\small
\begin{tabular}{ccc lll l}
\hline
\multicolumn{7}{c}{ Type I errors under the $\beta$-model}\\
\hline
NULL      &   $n$      & $\alpha$    & $c =0$  & $c = 0.2$  & $c = 0.4 $  & $c=0.5 $  \\
\hline
$H_{01}$  &   $100$    & $0.05$           &$ 5.28 / 0 $&$ 5.88 / 0 $&$ 5.65 / 0.8 $&$ 5.26 / 19.08 $  \\
          &            & $0.10$           & $ 10.42 / 0 $&$ 11.04 / 0 $&$ 11.03 / 0.8 $&$ 10.48 / 19.08 $   \\
          &   $200$    & $0.05$           &$ 5.34 / 0 $&$ 4.88 / 0 $&$ 5 / 0.06 $&$ 4.75 / 10.4 $    \\
          &            & $0.10$           &$ 10.28 / 0 $&$ 10.06 / 0 $&$ 9.79 / 0.06 $&$ 9.75 / 10.4 $ \\
$H_{02}$  &  $100$     & $0.05$           &$ 4.44 / 0 $&$ 5.26 / 0 $&$ 4.82 / 0 $&$ 5.46 / 2.48 $ \\
          &            & $0.10$           &$ 9.46 / 0 $&$ 10.22 / 0 $&$ 9.36 / 0 $&$ 10.64 / 2.48 $ \\
          &  $200$     & $0.05$           &$ 4.94 / 0 $&$ 5.18 / 0 $&$ 5.14 / 0 $&$ 5.26 / 0.32 $\\
          &            & $0.10$            &$ 10.18 / 0 $&$ 10.14 / 0 $&$ 9.92 / 0 $&$ 10.15 / 0.32 $\\
$H_{03}$  &  $100$     & $0.05$           &$ 5.14 / 0 $&$ 5.14 / 0 $&$ 4.87 / 0.6 $&$ 5.65 / 21.08 $ \\
          &            & $0.10$           &$ 10.52 / 0 $&$ 10.78 / 0 $&$ 10.58 / 0.6 $&$ 10.95 / 21.08 $ \\
          &  $200$     & $0.05$           &$ 5.12 / 0 $&$ 4.96 / 0 $&$ 5.74 / 0.06 $&$ 4.35 / 10.84 $ \\
          &            & $0.10$           &$ 10.4 / 0 $&$ 9.86 / 0 $&$ 11.11 / 0.06 $&$ 9.83 / 10.84 $ \\
$H_{04}$  &  $100$     & $0.05$           &$ 5.44 / 0 $&$ 5.54 / 0 $&$ 4.94 / 0.74 $&$ 5.14 / 20.64 $ \\
          &            & $0.10$           &$ 11.00 / 0 $&$ 10.80 / 0 $&$ 10.20 / 0.74 $&$ 10.16 / 20.64 $\\
          &  $200$     & $0.05$           &$ 5.32 / 0 $&$ 5.06 / 0 $&$ 5.18 / 0.06 $&$ 5.18 / 10.34 $ \\
          &            & $0.10$           &$ 9.94 / 0 $&$ 9.7 / 0 $&$ 10.21 / 0.06 $&$ 9.9 / 10.34 $\\
\hline
\multicolumn{7}{c}{ Type I errors under the Bradley--Terry model}\\
\hline
$H_{01}$  &   $100$    & $0.05$    &$ 5.02 / 0 $&$ 5.6 / 0 $&$ 5.72 / 0 $&$ 5.82 / 0 $           \\
          &            & $0.10$    &$ 10.04 / 0 $&$ 10.94 / 0 $&$ 10.84 / 0 $&$ 11.12 / 0 $          \\
          &   $200$    & $0.05$    &$ 4.9 / 0 $&$ 5.1 / 0 $&$ 5 / 0 $&$ 4.74 / 0 $           \\
          &            & $0.10$    &$ 9.92 / 0 $&$ 10.2 / 0 $&$ 10.08 / 0 $&$ 9.24 / 0 $           \\
$H_{02}$  &  $100$     & $0.05$    &$ 4.7 / 0 $&$ 5.64 / 0 $&$ 4.88 / 0 $&$ 5.56 / 0 $           \\
          &            & $0.10$    &$ 9.76 / 0 $&$ 10.82 / 0 $&$ 10.3 / 0 $&$ 10.32 / 0 $          \\
          &  $200$     & $0.05$    &$ 5.22 / 0 $&$ 4.92 / 0 $&$ 5.28 / 0 $&$ 5.06 / 0 $\\
          &            & $0.10$    &$ 10.14 / 0 $&$ 9.7 / 0 $&$ 10.1 / 0 $&$ 9.62 / 0 $ \\
$H_{04}$  &  $100$     & $0.05$    &$ 5.08 / 0 $&$ 5.02 / 0 $&$ 4.62 / 0 $&$ 5.32 / 0 $  \\
          &            & $0.10$    &$ 8.96 / 0 $&$ 9.66 / 0 $&$ 10.36 / 0 $&$ 10.48 / 0 $  \\
          &  $200$     & $0.05$    &$ 4.74 / 0 $&$ 4.62 / 0 $&$ 5.18 / 0 $&$ 4.9 / 0 $ \\
          &            & $0.10$    &$ 9.7 / 0 $&$ 9.24 / 0 $&$ 10.98 / 0 $&$ 9.78 / 0 $ \\
\hline
\end{tabular}
\end{table}

The simulated Type I errors are reported in Table \ref{table-type-I}.
From this table, we can see that the MLE failed with positive frequencies in the $\beta$-model when $L_n= 0.5\log n$
while the frequencies of MLE nonexistence in other cases is very small, less than $0.1$.
Most of simulated type I errors are close to the target nominal level and the difference between simulated values
and nominal levels are relatively smaller when $n=200$ in contrast with those when $n=100$.

Next, we investigate the powers of the log-likelihood ratio tests.
We consider a homogenous hypothesis testing problem: $H_0$: $\beta_1=\cdots=\beta_r$ in the $\beta$-model and
$\beta_2=\cdots=\beta_r$ in the Bradley-Terry model ($\beta_1=0$ is the reference parameter).
The true model was set to be $\beta_i=ic/r$, $i=1,\ldots,r$.
The other parameters were set as $\beta_i=0.2(i-r)\log n/n$ for $i=r+1,\ldots,n$.
In the Bradley--Terry model, each pair has only one comparison.
The results are shown in Table \ref{powers-a}.
We can see that when $c=0$, all simulated type I errors agree reasonably well with the nominal level $0.05$.
Further, when $n$ and $r$ are fixed, as $c$ increases, the power tends to increase and is close to $100\%$ when $c=1.6$.
Similar phenomenon can be observed when $r$ increases while $n$ and $c$ are fixed, or when $n$ increases while $c$ and $r$ are fixed.
Further, we did additional simulations in the Bradley--Terry model under the situation that imitates the schedule  of the NBA regular season.
The number of nodes is $n=30$ and each pair of nodes has $3$ comparisons. Other parameters are the same as before.
The results are shown in Table \ref{table-bt-powers}. From this table, we can see that the Type I errors are well controlled and
powers are visibly high when $c=1.2$. This shows that the asymptotic approximation is good even in the case $n$ is small, as long as the number of comparisons
in each pair is over $3$.

\begin{table}[h!]
\centering
\caption{Powers of the proposed likelihood ratio tests}
\label{powers-a}
\vskip5pt
\small
\begin{tabular}{ccc ccc c}
\hline
\\
\multicolumn{7}{c}{Powers in the $\beta$-model} \\
\hline
\\
  $n$            & $r$  &$c=0$       & $c=0.4$ & $c=0.8$ & $c=1.2$ & $c=1.6$ \\
\hline
 $100$          & $5$    &$ 4.68 $ &$ 13.58 $&$ 50.30 $&$ 88.14 $&$ 98.56$              \\
                & $10$   &$ 5.08 $ &$ 14.24 $&$ 61.62 $&$ 95.28 $&$ 99.82$               \\
 $200$          & $5$    &$ 5.24 $ &$ 16.28 $&$ 60.14 $&$ 92.52 $&$ 99.34 $                \\
                & $10$   &$ 5.46 $ &$ 29.60 $&$ 89.34 $&$ 99.86 $&$ 100 $                \\
\hline
\multicolumn{7}{c}{Powers in the Bradley--Terry model}\\
\hline
$100$           & $6$   &$ 4.88  $&$ 10.00  $&$ 21.00  $&$ 79.06  $&$ 97.00 $                           \\
                & $11$  & $ 5.02  $&$ 12.62 $&$ 32.20  $&$ 96.74  $&$ 99.96 $                           \\
$200$           & $6$   &$ 5.16 $&$ 16.76 $&$ 43.00 $&$ 98.82 $   & $100$                         \\
                & $11$   &$ 4.98  $&$ 26.48 $&$ 66.88  $&$ 100  $&$ 100  $                          \\
     \hline
\end{tabular}
\end{table}

\begin{table}[h!]
\centering
\caption{Powers in the Bradley-Terry model with small $n$ $(=30)$.}
\label{table-bt-powers}
\vskip2pt
\small
\begin{tabular}{ccc ccc }
\hline
 $r$    &$c=0$       & $c=0.4$ & $c=0.8$ & $c=1.2$ & $c=1.6$ \\
\hline
 $3$  &$ 5.58  $&$ 9.10  $&$ 14.60  $&$ 55.08  $&$ 81.04  $      \\
 $4$  &$ 4.36  $&$ 8.54  $&$ 16.92  $&$ 65.78  $&$ 90.30  $       \\
 $5$  &$ 4.82  $&$ 9.98  $&$ 18.72  $&$ 73.66  $&$ 94.96  $       \\
 $6$  &$ 5.36  $&$ 9.62  $&$ 20.10  $&$ 80.98  $&$ 97.54  $       \\
 $8$  &$ 4.84  $&$ 11.44  $&$ 26.48  $&$ 92.62  $&$ 99.82  $ \\
 $10$ &$ 5.24  $&$ 12.34  $&$ 29.98  $&$ 95.88  $&$ 100  $  \\
     \hline
\end{tabular}
\end{table}

\subsection{An application to the NBA data}
National Basketball Association (NBA) is one of the most successful men's professional basketball league in the world.
The current league organization divides its total thirty teams into two conferences: the western conference and the eastern conference.
In the regular season, every team plays with every other team  three or four times.
It would be of interest to test whether some teams have the same merits.
Here we use the recent 2020-21 NBA season data as an illustrative example.

The fitted merits in the Bradley--Terry model are presented in Table \ref{table-merits}, in which Houston Rockets is the reference team.
As we can see, the ranking based on the won-loss percentage and that based on the fitted merits are similar.
As shown in the simulations, the asymptotic Chi-square distribution for the likelihood ratio statistic provides
good approximation even when $n=30$.
We use the log-likelihood ratio statistic to test whether there are significant differences among top 3 teams and 6 teams in respective
conferences.

Since the first three teams in the eastern conference--``Philadephia 76ers", ``Brooklyn Nets", ``Milwaukee Bucks"
have similar won-loss percentages, we may want to test their equality.
By using a Chi-square approximation, we get a value $0.551$ for the log-likelihood ratio with a p-value $0.759$.
For testing the equality of ``Philadephia 76ers" and ``Boston Celtics", it yields
a p-value $0.031$, showing there exists significant difference between these two teams.
For testing equality among top 4 teams in the western conference according the ranking of the won-loss percentage,
we get a value $1.892$ for the log-likelihood ratio with a p-value $0.595$, showing that the differences of the 4 teams do not exhibit statistical significance.

\begin{table}[h!]
\centering
\caption{Fitted merits based on the 2020-21 NBA season data. The column $\hat{\sigma}$ corresponds to standard errors.}
\label{table-merits}
\vskip2pt
\small
\begin{tabular}{l llll | llll }
\hline
& \multicolumn{4}{c|}{  Eastern Conference } & \multicolumn{4}{c}{  Western Conference  }\\
&    Team      &  W-L    & $\widehat{\beta}_i$   & $\hat{\sigma}$       &     Team      & W-L       & $\widehat{\beta}_i$   & $\hat{\sigma}$  \\
\hline
1 & Philadephia        & 49-23  &$ 1.863 $&$ 0.377 $& Utah J.      & 52-20 &$ 2.164 $&$ 0.383 $ \\
2 & Brooklyn N.        & 48-24  &$ 1.804 $&$ 0.375 $& Phoenix S.   & 51-21 &$ 1.836 $&$ 0.374 $\\
3 & Milwaukee B.       & 46-26  &$ 1.617 $&$ 0.365 $& Denver N.    & 47-25 &$ 1.806 $&$ 0.369 $\\
4 & New Y. K.          & 41-31  &$ 1.383 $&$ 0.364 $& LA C.        & 47-25 &$ 1.709 $&$ 0.367 $\\
5 & Miami H.           & 40-32  &$ 1.363 $&$ 0.363 $& Los A. L.    & 42-30 &$ 1.647 $&$ 0.364 $\\
6 & Atlanta H.         & 41-31  &$ 1.355 $&$ 0.363 $& Portland T. B. &42-30 &$ 1.380 $&$ 0.364 $\\
7 & Boston C.          & 36-36  &$ 1.123 $&$ 0.365 $& Memphis G.     &38-34 &$ 1.348 $&$ 0.360 $\\
8 & Washinngton W.     & 34-38  &$ 1.006 $&$ 0.359 $& Golden S. W.   &39-33 &$ 1.294 $&$ 0.360 $\\
9 & Indiana P.         & 34-38  &$ 0.891 $&$ 0.361 $& Dallas M.      &42-30 &$ 1.233 $&$ 0.356 $\\
10 & Chicago B.        & 31-41  &$ 0.827 $&$ 0.363 $& New O. P.      &31-41 &$ 0.941 $&$ 0.364 $\\
11 & Charlotte H.      & 33-39  &$ 0.817 $&$ 0.362 $& San A. S.      &33-39 &$ 0.845 $&$ 0.363 $\\
12 & Toronto Raptors   & 27-45  &$ 0.616 $&$ 0.367 $& Sacramento K.  &31-41 &$ 0.758 $&$ 0.368 $\\
13 & Orlando M.        & 21-51  &$ 0.328 $&$ 0.370 $& Minnesota T.   &23-49 &$ 0.355 $&$ 0.374 $\\
14 & Cleveland C.      & 22-50  &$ 0.305 $&$ 0.373 $& Oklahoma C.    &22-50 &$ 0.317 $&$ 0.373 $\\
15 & Detroit P.        & 20-52  &$ 0.274 $&$ 0.371 $& Houston R.     &17-55 &$ 0 $ & \\
\hline
\end{tabular}

\end{table}

\section{Discussion}
\label{section:discussion}

We have established the Wilks type of results for fixed and increasing dimensional parameter hypothesis testing problems
under the $\beta$-model and the Bradley--Terry model.
It is worth noting that the conditions imposed on $b_{n}$ and $c_n$ may not be best possible.
The simulation results indicate that there are still good asymptotic approximations when these conditions are violated.
Note that the asymptotic behaviors of likelihood ratio statistics depend not only on $b_{n}$ (or $c_n$),
but also on the configuration of all parameters.
Moreover, both models assume a logistic distribution on observed edges $a_{ij}$.
It would be of interest to investigate whether these conditions could be relaxed and whether the results continue to hold
in some generalized models.

We only consider dense paired comparisons, in which all pairs have comparisons.
However, this is not an unrealistic assumption for
many situations. For example, the Major League Baseball schedule in
the United States and Canada arranges that all teams play each other in a regular season.
In some other applications, not all possible comparisons are available. For examples, some
games might be cancelled due to bad weather. If only a small proportion of comparisons are not available,
then it has little impact on the results developed in this paper. An interesting scenario is
that paired comparisons are sparse, in which a large number of subjects do not have direct comparisons.
The errors for the MLEs depend crucially on the sparse condition [\cite{yan2012sparse}].
This has impact on the remainder terms in the asymptotic expansion of the log-likelihood function.
Extension to sparse paired comparisons seems not trivial.
We will investigate this problem in future work.

\section{Appendix}
\label{section:proof}

In this section, we present proofs for Theorems \ref{theorem-LRT-beta} and \ref{theorem-LRT-beta-fixed}.
The proofs of Theorems \ref{theorem-ratio-bt-3} and \ref{theorem-ratio-bt-fixed} are presented in the Supplementary Material A.

We introduce some notations.
For a vector $\mathbf{x}=(x_1, \ldots, x_n)^\top\in \R^n$, denote by $\|\mathbf{x}\|$ for a general norm on vectors with the special cases
$\|\mathbf{x}\|_\infty = \max_{1\le i\le n} |x_i|$ and $\|\mathbf{x}\|_1=\sum_i |x_i|$ for the $\ell_\infty$- and $\ell_1$-norm of $\mathbf{x}$ respectively.
For an $n\times n$ matrix $J=(J_{ij})$, let $\|J\|_\infty$ denote the matrix norm induced by the $\ell_\infty$-norm on vectors in $\R^n$, i.e.,
\[
\|J\|_\infty = \max_{\mathbf{x}\neq 0} \frac{ \|J\mathbf{x}\|_\infty }{\|\mathbf{x}\|_\infty}
=\max_{1\le i\le n}\sum_{j=1}^n |J_{ij}|,
\]
and $\|J\|$ be a general matrix norm. $\|J\|_{\max}$ denotes the maximum absolute entry-wise norm, i.e., $\|J\|_{\max}=\max_{i,j} |J_{ij}|$.
The  notation $f(n)=O\left(g(n)\right)$ or
$f(n)\lesssim g(n)$ means  there is a constant $c>0$ such
that $\left|f(n)\right|\leq c|g(n)|$. $f(n) \asymp g(n)$ means that $f(n)\lesssim g(n)$ and $g(n)\lesssim f(n)$.
$f(n)=o(g(n))$ means  $\lim_{n\rightarrow\infty}f(n)/g(n)=0$.
The notation $\sum_{j<i}$  is a shorthand for $\sum_{i=1}^n \sum_{j=1}^{i-1}$.

We define a matrix class $\mathcal{L}_n(m, M)$ with two positive numbers $m$ and $M$.
We say an $n\times n$ matrix $V=(v_{ij})$ belongs to the matrix class
$\mathcal{L}_n( m, M)$ if
\[
\begin{array}{cl}
v_{ii}=\sum_{j\neq i} v_{ij}, & i=1, \ldots, n \\
m \le v_{ij} \le M, & i,j=1, \ldots, n; i\neq j.
\end{array}
\]
Define two diagonal matrices:
\begin{equation}\label{definition-S}
S=\mathrm{diag}(1/v_{11}, \ldots, 1/v_{nn}),\quad  S_{22}=\mathrm{diag}(1/v_{r+1,r+1}, \ldots, 1/v_{nn}),
\end{equation}
where $S_{22}$ is the bottom right $(n-r)\times (n-r)$ block of $S$ for $r\in\{ 1, \ldots, n-1\}$.
\cite{Yan:Xu:2013} proposed to use the diagonal matrix $S$ to approximate $V^{-1}$.
\begin{lemma}\label{lemma-appro-beta-VS}
For $V\in \mathcal{L}_n(1/b_n, 1/c_n)$ with $n\ge 3$ and its bottom right $(n-r)\times (n-r)$ block $V_{22}$ with $r\in\{1,\ldots, n-1\}$, we have
\begin{equation}\label{ineq-V-S-appro-upper-b}
\max\{ \|V^{-1} - S \|_{\max}, \| V_{22}^{-1} - S_{22} \|_{\max} \} \le \frac{2b_n^2 }{ c_n (n-1)^2 } \left( \frac{nb_n}{2(n-2)c_n} + \frac{1}{2} \right).
\end{equation}
\end{lemma}
The proof of Lemma \ref{lemma-appro-beta-VS} is an extension of that of Proposition 1 in \cite{Yan:Xu:2013} and presented in the Supplementary Material B.
In Theorem 6.1 of \cite{hillar2012inverses}, they obtained a tight upper bound of $\|J\|_\infty$ for symmetric diagonally dominant $m\times m$ dimensional  matrices $J$ satisfying $J \ge J(\alpha, \ell)= \alpha I_m + \ell \mathbf{1}_m \mathbf{1}_m^\top$:
\begin{equation*} 
\|J^{-1}\|_\infty  \le \|[J(\alpha,\ell)]^{-1}\|_\infty \le \frac{ \alpha + 2 \ell (m-1) }{ \alpha (\alpha + \ell m ) },
\end{equation*}
where $A \ge B$ means $A - B$ is a nonnegative matrix, $\alpha\ge (m-2)\ell$,  $I_m$ denotes the $m\times m$ identity matrix, and
$\mathbf{1}_m$ denotes the $m$-dimensional column vector consisting of all ones.
As applied here, we have that for $V\in \mathcal{L}_n(1/b_n, 1/c_n)$ with $n\ge 3$ and its bottom right $(n-r)\times (n-r)$ block $V_{22}$ with $r\in\{1, \ldots, n-r\}$,
\begin{equation}\label{ineq-tight-V}
 \|V^{-1}\|_\infty \le \frac{3b_n}{2n-1}, \quad  \|V_{22}^{-1} \|_\infty \le \frac{ b_n}{ n-1 }\left(1+ \frac{n-r-2}{2n-r-1} \right)\le \frac{3b_n}{2(n-1)}.
\end{equation}
It is noteworthy that the upper bounds in \eqref{ineq-V-S-appro-upper-b} and \eqref{ineq-tight-V} are independent of $r$.
This property implies some remainder terms in the proofs of Theorems \ref{theorem-LRT-beta}  are in regardless of $r$. 

We define a function $\mu(x) =  e^x/(1 + e^x)$ and a notation $\pi_{ij}=\beta_i+\beta_j$ for easy of exposition.
A direct calculation gives that the derivative of $\mu(x)$ up to the third order are
\begin{eqnarray}\label{eq-derivative-mu-various}
\mu^\prime(x) = \frac{e^x}{ (1+e^x)^2 },~~  \mu^{\prime\prime}(x) = \frac{e^x(1-e^x)}{ (1+e^x)^3 },~~ \mu^{\prime\prime\prime}(x) =  \frac{ e^x [ (1-e^x)^2 - 2e^x] }{ (1 + e^x)^4 }.
\end{eqnarray}
According to the definition of $c_n$ in \eqref{definition-bncn}, we have the following inequalities:
\begin{equation}\label{ineq-mu-deriv-bound}
|\mu^\prime(\pi_{ij})| \le \frac{1}{c_n}, ~~ |\mu^{\prime\prime}(\pi_{ij})| \le \frac{1}{c_n},~~ |\mu^{\prime\prime\prime}(\pi_{ij})| \le \frac{1}{c_n}.
\end{equation}
The above inequalities will be used in the proofs repeatedly.
Recall that $\bar{a}_{ij} = a_{ij} - \E(a_{ij})$ denotes
the centered random variable of $a_{ij}$ and define $\bar{a}_{ii}=0$ for all $i=1, \ldots, n$.
Correspondingly, denote $\bar{d}_i = d_i - \E(d_i)$ and $\bs{\bar{d}}=(\bar{d}_1, \ldots, \bar{d}_n)^\top$.

\subsection{Proofs for Theorem \ref{theorem-LRT-beta} (a)}
\label{section:theorem1-a}

To prove Theorem \ref{theorem-LRT-beta} (a), we need three lemmas below.

\begin{lemma} \label{lemma-clt-beta-W}
Recall that $V_{22}$ is the bottom right $(n-r)\times (n-r)$ block of $V$.
Let $W=V^{-1}-S$,
$\widetilde{W}_{22}=V_{22}^{-1}-S_{22}$ and $\bs{\bar{d}}_2=(\bar{d}_{r+1}, \ldots, \bar{d}_n)^\top$. For any given $r\in \{0, \ldots, n-1\}$, we have
\[
\bs{\bar{d}}_2^\top \widetilde{W}_{22} \bs{\bar{d}}_2 = O_p\left( \frac{b_n^3}{c_n^3}(1-\frac{r}{n})^3 \right),
\]
where $r=0$ implies $\bs{\bar{d}}_2=\bs{\bar{d}}$, $V_{22}=V$ and $\widetilde{W}_{22}=W$.
\end{lemma}

Lemma \ref{lemma-clt-beta-W}
states that the remainder terms $\bs{\bar{d}}_2^\top \widetilde{W}_{22} \bs{\bar{d}}_2$ and $\bs{\bar{d}}^\top W \bs{\bar{d}}$ in \eqref{eq-theorem2-B10}
 is in the order of $O_p\left( (b_n^3/c_n^3)(1-r/n)^3 \right)$ for any given $r\ge 0$.

\begin{lemma}\label{lemma-consi-beta}
Under the null $H_0: (\beta_1, \ldots, \beta_r)=(\beta_1^0, \ldots, \beta_r^0)$ for any given $r\in\{0, \ldots, n-1\}$,
if
\begin{equation}\label{con-lemma3-beta-c}
\frac{ b_n^2 }{ c_n } = o\left( \frac{n}{(n-r)} \times \sqrt{ \frac{ n }{\log n} } \right),
\end{equation}
 then with probability at least $1-2/n$, the restricted MLE $\widehat{\bs{\beta}}^0$ exists and satisfies
\begin{equation*}\label{ineq-En-beta}
\| \widehat{\bs{\beta}}^0 - \bs{\beta} \|_\infty  \le  \frac{3nb_n}{(2n-1)}\sqrt{ \frac{\log n}{n} },
\end{equation*}
where $r=0$ means there is no any restriction on $\bs{\beta}$ and implies $\widehat{\bs{\beta}}^0=\widehat{\bs{\beta}}$. 
Further, if the restricted MLE exists, it must be unique.
\end{lemma}


From Lemma \ref{lemma-consi-beta}, we can see that the consistency rate for the restricted MLE $\bs{\widehat{\beta}}^0$ in terms of the $L_\infty$-norm
is independent of $r$ while the condition depends on $r$. The larger $r$ is, the weaker the condition is. When $r=0$, the lemma gives the error bound for the MLE $\bs{\widehat{\beta}}$.
When $b_n$ is a constant, this corresponds to the assumption in \cite{Chatterjee:Diaconis:Sly:2011} and the $L_\infty$-norm error bound of the MLE reduces to their error bound.

\begin{lemma}
\label{lemma:beta3:err}
If \eqref{con-lemma3-beta-c} holds,  then for an arbitrarily given $r\in\{0, \ldots, n-1\}$, 
\begin{eqnarray*}
\sum_{i=r+1}^n  (\widehat{\beta}_i-\beta_i)^3 \sum_{j=1,j\neq i}^n \mu^{\prime\prime}( \beta_i+\beta_j ) & = & O_p\left( \frac{b_n^4\log n}{ c_n^2} \left(\frac{n-r}{n}\right)^{1/2} \right), \\
\sum_{i,j=r+1, j\neq i}^n  (\widehat{\beta}_i-\beta_i)^2(\widehat{\beta}_j-\beta_j)\mu^{\prime\prime}( \beta_i+\beta_j ) & = & O_p
\left( \frac{ (n-r)b_n^5 (\log n)^2 }{ nc_n^2}  \right).
\end{eqnarray*}
If $\widehat{\beta}_i$ is replaced with $\widehat{\beta}_i^0$ for $i=r+1,\ldots,n$, then the above upper bound still holds.
\end{lemma}

If we directly use the error bound for $\| \bs{\widehat{\beta}} - \bs{\beta}\|_\infty$ in \eqref{ineq-En-beta} to bound the summation in the above lemma, it will
produce the following bound:
\begin{equation*}\label{ineq-z-st}
 \sum_{i=r+1}^n\sum_{j=1,j\neq i}^n  (\widehat{\beta}_i-\beta_i)^3  \mu^{\prime\prime}( \beta_i+\beta_j )
\lesssim b_n^3 \left( \frac{\log n}{n} \right)^{3/2} \sum_{i\neq j}\left|\frac{e^{\beta_i+\beta_j} -1}{ 1 + e^{\beta_i+\beta_j}}  \right|.
\end{equation*}
If all $\beta_i$s are positive constant, the term in the above right-hand side scaled by $r^{1/2}$ does not go to zero while Lemma \ref{lemma:beta3:err} shows that
it does go to zero.
We explain briefly reasons here. The above process neglects the integrity for $\sum_i (\widehat{\beta}_i - \beta_i)^3 \sum_j \mu^{\prime\prime}(\pi_{ij})$, which does have a much smaller error bound than
that for $\sum_i |\widehat{\beta}_i - \beta_i|^3 \sum_j |\mu^{\prime\prime}(\pi_{ij})|$.
The proof of Lemma \ref{lemma:beta3:err} uses the asymptotic representation of $\bs{\widehat{\beta}}$ in \eqref{eq-expansion-hatbeta-beta},
which leads to that the summarization is involved with a main term having the form of the weighted cubic sum $\sum_i \bar{d}_i^{\,3}/v_{ii}^3$. The variance of $\bar{d}_i^{\,3}$ is in order of $n^3$,
although $\E \bar{d}_i^{\,6}$ contains $n^6$ mixed items for $\bar{a}_{ij}$.
Lemma \ref{lemma:beta3:err} plays an important role in \eqref{eq-ell-difference} for proving $B_2/r^{1/2}\to 0$.

We are now ready to prove the first part of Theorem \ref{theorem-LRT-beta}.

\begin{proof}[Proof of Theorem \ref{theorem-LRT-beta} (a)]
Under the null $H_0: (\beta_1, \ldots, \beta_r)=(\beta_1^0, \ldots, \beta_r^0)$, the data generating parameter $\bs{\beta}$ is equal to $(\beta_1^0, \ldots, \beta_r^0, \beta_{r+1}, \ldots, \beta_n)^\top$.
For convenience, we suppress the superscript $0$ in $\beta_i^0, i=1,\ldots,r$ when causing no confusion.
The following calculations are based on the event $E_n$ that $\bs{\widehat{\beta}}$ and $\bs{\widehat{\beta}}^0$ simultaneously exist and satisfy
\begin{equation}\label{ineq-beta-beta0-upp}
\max\left\{ \| \widehat{\bs{\beta}} - \bs{\beta} \|_\infty, \| \widehat{\bs{\beta}}^0 - \bs{\beta} \|_\infty \right\}
\le \frac{3nb_n}{(2n-1)}\sqrt{ \frac{\log n}{n} }.
\end{equation}
By Lemma \ref{lemma-consi-beta}, $\P(E_n) \ge 1 - O(n^{-1})$ if $b_n^2/c_n=o\left\{ (n/\log n)^{1/2} (1-r/n)^{-1}\right\}$.

Applying a fourth-order Taylor expansion to $\ell(\widehat{\bs{\beta}} ) $ at point $ \bs{\beta}$, it yields
\begin{eqnarray*}
\ell(\widehat{\bs{\beta}} ) - \ell( \bs{\beta} ) & = &
\underbrace{
\frac{\partial \ell( \bs{\beta} ) }{ \partial \bs{\beta}^\top } ( \widehat{\bs{\beta}} - \bs{\beta} ) +
\frac{1}{2} ( \widehat{\bs{\beta}} - \bs{\beta} )^\top \frac{ \partial^2 \ell( \bs{\beta} ) }{ \partial \bs{\beta} \bs{\beta}^\top } ( \widehat{\bs{\beta}} - \bs{\beta} ) }_{B_1} \\
& & +  \frac{1}{6} \underbrace{\sum_{i=1}^n \sum_{j=1}^n \sum_{k=1}^n \frac{ \partial^3 \ell(\bs{\beta})}{ \partial \beta_i \partial \beta_j \partial \beta_k }
( \widehat{\beta}_i - \beta_i)( \widehat{\beta}_j - \beta_j)( \widehat{\beta}_j - \beta_j) }_{B_2} \\
&& +  \frac{1}{4!} \underbrace{ \sum_{t=1}^n \sum_{i=1}^n \sum_{j=1}^n \sum_{k=1}^n \frac{ \partial^4 \ell(\bs{\tilde{\beta}})}{ \partial \beta_t \partial \beta_i \partial \beta_j \partial \beta_k } ( \widehat{\beta}_t - \beta_t)
( \widehat{\beta}_i - \beta_i)( \widehat{\beta}_j - \beta_j)( \widehat{\beta}_j - \beta_j) }_{B_3},
\end{eqnarray*}
where $\bs{\tilde{\beta}} = t \bs{\beta} + (1-t ) \bs{\widehat{\beta}}$ for some $t\in(0,1)$.
Correspondingly, $\ell(\widehat{\bs{\beta}}^0 )$ has the following expansion:
\begin{equation*}
\ell(\widehat{\bs{\beta}}^0 ) - \ell( \bs{\beta} ) = B_1^0 + \frac{1}{6}B_2^0 + \frac{1}{4!}B_3^0,
\end{equation*}
where $B_i^0$ is the version of $B_i$ with $\widehat{\bs{\beta}}$ replaced by $\widehat{\bs{\beta}}^0$.
Therefore,
\begin{equation}\label{eq-ell-difference}
2\{ \ell(\widehat{\bs{\beta}} ) - \ell(\widehat{\bs{\beta}}^0 ) \}
= 2( B_1 - B_1^0)  + \frac{1}{3}(B_2 - B_2^0) + \frac{1}{12}(B_3 - B_3^0).
\end{equation}
Recall that
\begin{eqnarray*}
\frac{\partial \ell( \bs{\beta} ) }{\partial \bs{\beta}^\top }  =  \bs{d} - \E \bs{d},
~~
V= - \frac{\partial^2 \ell( \bs{\beta} )
}{\partial \bs{\beta} \partial \bs{\beta}^\top }.
\end{eqnarray*}
Therefore, $B_1$ can be written as
\begin{equation}
\label{lrt-a-beta-B1}
B_1  =   ( \bs{\widehat{\beta} } - \bs{\beta})^\top \bs{\bar{d}}
- \frac{1}{2} ( \bs{\widehat{\beta} } - \bs{\beta})^\top V( \bs{\widehat{\beta} } - \bs{\beta}).
\end{equation}
For the third-order expansion terms in $B_2$, observe that if three distinct indices $i, j, k$ are distinct, then
\begin{equation}\label{ell-third}
 \frac{\partial^3 \ell( \bs{\beta} )
}{\partial \beta_i\beta_j\beta_k } = 0,
~~
\frac{\partial^3 \ell(\boldsymbol{\beta})}{\partial \beta_i^3 }=-\sum_{j\neq i} \mu^{\prime\prime}( \pi_{ij} ),~~
\frac{\partial^3 \ell(\boldsymbol{\beta})}{\partial \beta_i^2\partial \beta_j }= -\mu^{\prime\prime}( \pi_{ij} ),
\end{equation}
and if there are at least three different values among the four indices $i, j, k, t$, then
\begin{equation}\label{ell-fourth}
 \frac{\partial^4 \ell( \bs{\beta} )
}{\partial \beta_i \partial \beta_j \partial \beta_k \partial \beta_t } = 0.
\end{equation}
Therefore, $B_2$ and $B_3$ have the following expressions:
\begin{eqnarray}
\label{lrt-a-beta-B2}
-B_2  & = &  \sum_{i=1}^n  (\widehat{\beta}_i-\beta_i)^3 \sum_{j=1,j\neq i}^n \mu^{\prime\prime}( \pi_{ij} )
+ 3 \sum_{i,j=1, j\neq i}^n  (\widehat{\beta}_i-\beta_i)^2(\widehat{\beta}_j-\beta_j)\mu^{\prime\prime}( \pi_{ij} ), \\
\nonumber
-B_3 & = &  \sum\limits_{i=1}^n  (\widehat{\beta}_i -\beta_i)^4 \sum\limits_{j=1,j\neq i}^n \mu^{\prime\prime\prime}( \bar{\pi}_{ij} )
+4\sum\limits_{i=1}^n \sum\limits_{j=1,j\neq i}^n \mu^{\prime\prime\prime}( \bar{\pi}_{ij} )  (\widehat{\beta}_i -\beta_i)^3(\widehat{\beta}_j -\beta_j) \\
\label{lrt-a-beta-B3}
&& + 3 \sum\limits_{i=1}^n \sum\limits_{j=1, j\neq i}^n  \mu^{\prime\prime\prime}( \bar{\pi}_{ij} )(\widehat{\beta}_i-\beta_i)^2(\widehat{\beta}_j-\beta_j)^2,
\end{eqnarray}
where $\bar{\pi}_{ij}$ lies between $\widehat{\pi}_{ij}$ and $\pi_{ij}$.

It is sufficient to demonstrate: (1) $\{2( B_1 - B_1^0)-r\}/(2r)^{1/2}$ converges in distribution to the standard normal distribution as $r\to\infty$;
(2) $(B_2 - B_2^0)/r^{1/2}=o_p(1)$; (3) $(B_3-B_3^0)/{r^{1/2}}=o_p(1)$.
The second claim is a direct result of Lemma \ref{lemma:beta3:err}.
Note that $\widehat{\beta}_i^0 = \beta_i$, $i=1, \ldots, r$. So $B_3^0$ has less terms than $B_3$.
In view of \eqref{ineq-mu-deriv-bound} and \eqref{ineq-beta-beta0-upp}, if $b_n^4/c_n=o( r^{1/2}/(\log n)^2 )$, then
\begin{eqnarray}
\label{ineq-B3-upper}
\frac{ |B_3| }{ r^{1/2} } & \lesssim & \frac{1}{r^{1/2}} \cdot
\frac{n^2}{c_n} \cdot \| \bs{\widehat{\beta}} - \bs{\beta} \|_\infty^4 \lesssim \frac{ b_n^4(\log n)^2 }{ r^{1/2} c_n } = o(1), \\
\label{ineq-B30-upper}
\frac{ |B_3^0| }{ r^{1/2} } & \lesssim & \frac{1}{r^{1/2}} \cdot
\frac{r(n-r)}{c_n} \cdot \| \bs{\widehat{\beta}}^0 - \bs{\beta} \|_\infty^4 \lesssim \frac{ b_n^4(\log n)^2 }{ r^{1/2} c_n } = o(1),
\end{eqnarray}
which shows the third claim.
Therefore, the remainder of the proof is verify claim (1). This contains three steps.
Step 1 is about explicit expressions of $\widehat{\bs{\beta}}$ and $\widehat{\bs{\beta}}^0$.
Step 2 is about the explicit expression of $B_1-B_1^0$. Step 3 is about showing that the main term involved with $B_1-B_1^0$
asymptotically follows a normal distribution and the remainder terms goes to zero.

Step 1. We characterize the asymptotic representations of $\widehat{\bs{\beta}}$ and $\widehat{\bs{\beta}}^0$.
Recall that $\pi_{ij}=\beta_i+\beta_j$.
To simplify notations, define $\widehat{\pi}_{ij} = \widehat{\beta}_i + \widehat{\beta}_j$.
A second-order Taylor expansion gives that
\begin{eqnarray*}\label{eq-expansion-beta-a}
\mu( \widehat{\pi}_{ij} )
&=& \mu( \pi_{ij} ) + \mu^\prime(\pi_{ij}) (\widehat{\pi}_{ij} - \pi_{ij}) +
\frac{1}{2} \mu^{\prime\prime}( \tilde{\pi}_{ij} ) (\widehat{\pi}_{ij} - \pi_{ij})^2,
\end{eqnarray*}
where $\tilde{\pi}_{ij}$ lies between $\widehat{\pi}_{ij}$ and $\pi_{ij}$.
Let
\begin{equation}\label{eq:definition:h}
h_{ij}= \frac{1}{2}\mu^{\prime\prime}( \tilde{\pi}_{ij} )(\widehat{\pi}_{ij} - \pi_{ij})^2,~~
h_i=\sum_{j\neq i}h_{ij}, ~~\bs{h}=(h_1, \ldots, h_n)^\top.
\end{equation}
In view of \eqref{ineq-mu-deriv-bound} and \eqref{ineq-En-beta}, we have
\begin{equation}\label{ineq-beta-h}
\| \bs{h} \|_\infty \le \frac{1}{2}(n-1) \max_{i,j} |h_{ij} | \lesssim \frac{n}{c_n}\| \bs{\widehat{\beta}} - \bs{\beta} \|_\infty^2
\lesssim \frac{b_n^2 \log n}{c_n}.
\end{equation}
By \eqref{eq-likelihood-beta} and \eqref{eq-expansion-beta-a}, we have
\begin{equation*}
d_i-\E(d_i)=\sum_{j=1,j\neq i}^n v_{ij}\{ (\widehat{\beta}_i-\beta_i)+(\widehat{\beta}_j-\beta_j)\} + h_i, ~~~i=1, \ldots, n.
\end{equation*}
Writing the above equations into the matrix form, we have
\begin{equation*}
\bs{d} - \E( \bs{d} ) = V ( \widehat{\boldsymbol{\beta}} - \boldsymbol{\beta} ) + \bs{h}.
\end{equation*}
It yields that
\begin{equation}\label{eq-expansion-hatbeta-beta}
\boldsymbol{\widehat{\beta}} - \boldsymbol{\beta} = V^{-1} \bs{\bar{d}} - V^{-1}\mathbf{h},
\end{equation}
where, by \eqref{ineq-tight-V} and \eqref{ineq-beta-h},
\begin{equation}\label{ineq-beta-h-b}
\| V^{-1}\mathbf{h} \|_\infty \le \| V^{-1} \|_\infty \| \mathbf{h} \|_\infty \lesssim \frac{ b_n^3 \log n }{ n c_n }.
\end{equation}
Recall  $\bs{\bar{d}}_2=(d_{r+1}, \ldots, d_n)^\top$. Let $\boldsymbol{\widehat{\beta}}_2^0= (\widehat{\beta}_{r+1}^0, \ldots, \widehat{\beta}_n^0)^\top$ and
$ \boldsymbol{\beta}_2 = (\beta_{r+1}, \ldots, \beta_n)$.
Similar to \eqref{ineq-beta-h} and \eqref{eq-expansion-hatbeta-beta}, we have
\begin{equation}\label{eq-beta0-exapnsion}
\boldsymbol{\widehat{\beta}}^0_2 - \boldsymbol{\beta}_2 = V_{22}^{-1}\bs{\bar{d}}_2  - V_{22}^{-1}\mathbf{\widetilde{h}}_2,
\end{equation}
where $\mathbf{\widetilde{h}}_2 = (\tilde{h}_{r+1}, \ldots, \tilde{h}_n)^\top$, $\mathbf{\widetilde{h}} = (\tilde{h}_{1}, \ldots, \tilde{h}_n)^\top$ and
\begin{equation}\label{defintion-tilde-h}
\tilde{h}_i  =   \sum_{j=1,j\neq i}^n \mu^{\prime\prime}( \tilde{\pi}_{ij}^0 ) ( \widehat{\pi}_{ij}^0 - \pi_{ij} )^2,~~
|\tilde{h}_i| \lesssim \frac{ b_n^2 \log n}{c_n},
 ~i=1, \ldots, n
\end{equation}
In the above equation, $\tilde{\pi}_{ij}^0$ lies between $\pi_{ij}$ and $\widehat{\pi}_{ij}^0=\widehat{\beta}_i^0 + \widehat{\beta}_j^0$ for all $i,j=1, \ldots, n$.

Step 2. We derive the explicit expression of $B_1-B_1^0$.
Substituting \eqref{eq-expansion-hatbeta-beta} and \eqref{eq-beta0-exapnsion} into the expressions of $B_1$ in \eqref{lrt-a-beta-B1} and $B_1^0$ respectively, it yields
\begin{eqnarray*}\label{B1-expression}
2B_1 & = & \bs{\bar{d}}^\top V^{-1} \bs{\bar{d}} - \bs{h}^\top V^{-1} \bs{h},\\
\label{likelihood-beta-composite}
 2B_1^0 & = & \bs{\bar{d}}_2^\top V_{22}^{-1} \bs{\bar{d}}_2 -  \bs{\widetilde{h}}_2^\top V_{22}^{-1} \bs{\widetilde{h}}_2.
\end{eqnarray*}
By setting $V^{-1}=S+W$ and $V_{22}^{-1}=S_{22}+ \widetilde{W}_{22}$, we have
\begin{equation}\label{eq-theorem2-B10}
2(B_1 - B_1^0) = \sum_{i=1}^r \frac{ \bar{d}_i^{\,2} }{v_{ii}}  + \bs{\bar{d}}^\top W \bs{\bar{d}} - \bs{\bar{d}}_2^\top \widetilde{W}_{22} \bs{\bar{d}}_2
+ \bs{\bar{h}}^\top V^{-1} \bs{\bar{h}} - \bs{\bar{h}}_2^\top V_{22}^{-1} \bs{\bar{h}}_2.
\end{equation}

Step 3.  
We show three claims: (i) $(\sum_{i=1}^r  \bar{d}_i^{\,2}/v_{ii}-r)/(2r)^{1/2}\stackrel{L}{\to} N(0,1)$;
(ii) $\bs{\bar{d}}^\top W \bs{\bar{d}}/r^{1/2}=o_p(1)$ and $\bs{\bar{d}}_2^\top \widetilde{W}_{22} \bs{\bar{d}}_2/r^{1/2}=o_p(1)$; (iii) $\bs{\bar{h}}^\top V^{-1} \bs{\bar{h}}/r^{1/2}=o_p(1)$ and $\bs{\bar{h}}_2^\top V_{22}^{-1} \bs{\bar{h}}_2/r^{1/2}=o_p(1)$.
The first and second claims directly follows from Lemma \ref{lemma:weighte-degree-al} and Lemma \ref{lemma-clt-beta-W}, respectively.
By \eqref{ineq-beta-h} and \eqref{ineq-beta-h-b}, we have
\[
|\mathbf{h}^\top V^{-1} \mathbf{h}| \le n \| \mathbf{h} \|_\infty \| V^{-1} \mathbf{h} \|_\infty
\lesssim
 n \cdot \frac{b_n^2\log n}{c_n} \cdot \frac{ b_n^3 \log n}{nc_n} \lesssim \frac{b_n^5 (\log n)^2}{c_n^2}.
\]
If $b_n^5/c_n^2 =o( r^{1/2}/(\log n)^{2})$, then
\begin{equation}\label{eq-simi-aVh}
\frac{1}{r^{1/2}} |\mathbf{h}^\top V^{-1} \mathbf{h}| \lesssim \frac{ b_n^5(\log n)^2 }{n^{1/2}}  = o(1).
\end{equation}
In view of \eqref{ineq-tight-V} and \eqref{defintion-tilde-h}, with the same arguments as in the proof of the above inequality, we have
\begin{equation}\label{eq-simi-a}
\frac{1}{r^{1/2}} |\bs{\tilde{h}}_2^\top V_{22}^{-1} \bs{\tilde{h}}_2| \lesssim \frac{ (n-r)b_n^5(\log n)^2 }{n r^{1/2}c_n^2}  = o(1).
\end{equation}
This demonstrates claim (iii).
It completes the proof.
\end{proof}

\subsection{Proofs for Theorem 1 (b)}
\label{subsection-proof-th1b}
Let  $\bs{\widetilde{d}}=(\sum_{i=1}^r d_i, d_{r+1},\ldots, d_n)$ and
 $\widetilde{V}$ denote the Fisher information matrix of $\widetilde{\bs{\beta}}=(\beta_1, \beta_{r+1}, \ldots, \beta_n)^\top$
under the null $H_0: \beta_1 = \cdots= \beta_r$, where
\begin{equation}\label{definition-tilde-V}
\widetilde{V}=\begin{pmatrix} \tilde{v}_{11} & \bs{\tilde{v}}_{12}^\top \\ \bs{\tilde{v}}_{12} & V_{22} \end{pmatrix},
\end{equation}
where $V_{22}$ is the lower right $(n-r)\times (n-r)$ block of $V$, $\bs{\tilde{v}}_{12} =
(\tilde{v}_{1,r+1}, \ldots, \bar{v}_{1, n})^\top$, and
\[
\tilde{v}_{11}= 2r(r-1)\cdot \frac{ e^{2\beta_1} }{ ( 1 + e^{2\beta_1})^2 } + r\sum_{j=r+1}^n \tilde{v}_{1j}, ~~
\tilde{v}_{1j} =  \frac{ r e^{\beta_1 + \beta_j } }{ ( 1 + e^{\beta_1 + \beta_j})^2 },~j=r+1, \ldots, n.
\]
Note that $\widetilde{V}$ is also the covariance matrix of $\bs{\widetilde{d}}$.
Similar to approximate $V^{-1}$ by $S$, we use $\widetilde{S}=\mathrm{diag}(1/\tilde{v}_{11}, 1/v_{r+1, r+1}, \ldots, 1/v_{nn})$ to approximate $\widetilde{V}^{-1}$.
The approximation error is stated in the following lemma.
\begin{lemma}\label{lemma-beta-approx-ho}
For any $r\in\{0, \ldots, n-1\}$ and $n\ge 3$, we have
\begin{equation}\label{approxi-inv2-beta-ho}
\|\widetilde{W}:= \widetilde{V}^{-1}-\widetilde{S} \|_{\max} \le \frac{ b_n }{ (n-1)^2 c_n^2 }\left( \frac{b_n n}{2(n-2)c_n} + \frac{1}{2} \right).
\end{equation}
\end{lemma}
As we can see, the order of the above approximation error is the same as that in \eqref{ineq-V-S-appro-upper-b} and is independent of $r$.
The proof of Lemma \ref{lemma-beta-approx-ho} is given in the Supplementary Material B.
Similar to \eqref{ineq-tight-V}, by Theorem 6.1 of \cite{hillar2012inverses}, we  have
the $L_\infty$-norm bound of $\widetilde{V}^{-1}$:
\begin{equation*}\label{ineq-tilde-V-inf}
\| \widetilde{V}^{-1} \|_\infty \le  \frac{ b_n(3n-2r-1)}{ (n-1)(2n-r-1) }.
\end{equation*}
However, we use \eqref{approxi-inv2-beta-ho} to analyze $\widetilde{V}^{-1}$ here due to that the first diagonal element of $\widetilde{V}^{-1}$
is much smaller than other diagonal elements, up to a scaled factor $1/r$, and the upper bound of $\| \widetilde{V}^{-1} \|_\infty$ neglects
the difference between the first diagonal element and others.

Recall that $\bs{\widehat{\beta}}^0$ denotes
the restricted MLE of $\bs{\beta} =(\beta_1,  \ldots, \beta_n)^\top$. Under the null $H_0: \beta_1=\cdots = \beta_r$,
we have $\widehat{\beta}_1^0 = \cdots = \widehat{\beta}_r^0$.
Similar to the proof of Lemma \ref{lemma-consi-beta}, we have the following consistency result.

\begin{lemma}\label{lemma-con-beta-b}
Under the null $H_0: \beta_1=\cdots = \beta_r$, if
\[
\left( \frac{b_n}{c_n} + \frac{b_n^3}{c_n^3}\cdot \frac{r(n-r)}{n^2}  \right)\left(  b_n + \frac{ b_n^3 }{ c_n^2} ( \frac{r^{1/2}(n-r)^{1/2}}{n^{3/2}} + \frac{n-r}{n} ) \right) = o\left( \sqrt{\frac{n}{\log n}} \right),
\]
then with probability at least $1-2(n-r+1)/n^2$, $\bs{\widehat{\beta}}^0$ exists and satisfies
\begin{equation*}
\| \bs{\widehat{\beta}}^0 - \bs{\beta} \|_\infty \lesssim \left(  b_n + \frac{ b_n^3 }{ c_n^2} ( \frac{r^{1/2}(n-r)^{1/2}}{n^{3/2}} + \frac{n-r}{n} )  \right)\sqrt{ \frac{\log n}{n} } .
\end{equation*}
Further, if  $\bs{\widehat{\beta}}^0$ exists, it must be unique.
\end{lemma}

From the above lemma, we can see that the condition to guarantee consistency and the error bound depends on
$r$. Larger $r$ means a weaker condition and a smaller error bound.
A rough condition in regardless of $r$ to guarantee consistency is $b_n^6/c_n^5=o( (n/\log n)^{1/2})$ that
will be used in the proof of Theorem \ref{theorem-LRT-beta} (b).
It implies an error bound $(b_n^3/c_n^2)(\log n/n)^{1/2}$ that is generally larger than that in Lemma \ref{lemma-consi-beta}.
Similar to Lemma \ref{lemma-clt-beta-W}, we have the following bound, which is independent of $r$.

\begin{lemma}\label{lemma-tilde-W}
For any given $r\in\{1, \ldots, n-1\}$, we have
\[
(\bs{\widetilde{d}} - \E\bs{\widetilde{d}} )^\top \widetilde{W}  (\bs{\widetilde{d}} - \E\bs{\widetilde{d}} ) = O_p\left( \frac{ b_n^3 }{ c_n^3 } \right).
\]
\end{lemma}

The asymptotic representation of $\bs{\widehat{\beta}}^0$ is given below.
\begin{lemma}
\label{lemma-beta-homo-expan}
Under the null $H_0: \beta_1=\cdots=\beta_r$,
if $b_n^6/c_n^5 = o( (n/\log n)^{1/2})$, then for any given $r\in\{0, \ldots, n-1\}$, we have
\begin{equation*}
\begin{array}{rcl}
\widehat{\beta}_1^0 - \beta_1 & = & \frac{ \sum_{i=1}^r \bar{d}_i }{ \tilde{v}_{11} }+g_1, \\
\widehat{\beta}_i^0 - \beta_i & = & \frac{ \bar{d}_i }{ v_{ii} } + g_i,~~ i=r+1, \ldots, n,
\end{array}
\end{equation*}
where $g_1, g_{r+1}, \ldots, g_n$ with probability at least $1- O(n^{-1})$ satisfy
\[
g_i = (\widetilde{V}^{-1}\bs{\tilde{h}})_i + [\widetilde{W} (\bs{\tilde{d}}- \E \bs{\tilde{d}})]_i = O\left( \frac{b_n^9 \log n}{ nc_n^7 } \right),
\]
 uniformly, and  $\bs{\tilde{h}}=(\tilde{h}_1, \tilde{h}_{r+1}, \ldots, \tilde{h}_n)^\top$ satisfies
\begin{equation}\label{eq-homo-tildeh}
\begin{array}{rcl}
|\tilde{h}_1| & \lesssim & \frac{rb_n^6\log n}{c_n^5}, \\
\max_{i=r+1, \ldots, n} |\tilde{h}_i| & \lesssim & \frac{ b_n^6 \log n}{c_n^5}.
\end{array}
\end{equation}
\end{lemma}

With some ambiguity of notations, we still use the notation $\bs{\tilde{h}}$ here that is a little different from $\bs{\tilde{h}}$ defined in \eqref{defintion-tilde-h} in Section \ref{section:theorem1-a}.
Specifically, the first element of $\bs{\tilde{h}}$ can be viewed as the sum of $\tilde{h}_i$, $i=1,\ldots, r$ in  \eqref{defintion-tilde-h}.
This difference leads to that the remainder term here is larger than that in \eqref{eq-expansion-hatbeta-beta}.

Now, we are ready to prove  Theorem \ref{theorem-LRT-beta} (b).

\begin{proof}[Proof of Theorem \ref{theorem-LRT-beta} (b)]
Under the null $H_0: \beta_1=\cdots=\beta_r$, the data generating parameter $\bs{\beta}$ is equal to $(\underbrace{\beta_1, \ldots, \beta_1}_r, \beta_{r+1}, \ldots, \beta_n)^\top$.
The following calculations are based on the event $E_n$ that $\bs{\widehat{\beta}}$ and $\bs{\widehat{\beta}}^0$ simultaneously exist and satisfy
\begin{equation}\label{ineq-beta-beta0-upp}
 \| \widehat{\bs{\beta}} - \bs{\beta} \|_\infty\le \frac{3nb_n}{(2n-1)}\sqrt{ \frac{\log n}{n} }, \mbox{~~and~~}
  \| \widehat{\bs{\beta}}^0 - \bs{\beta} \|_\infty\lesssim \frac{b_n^3}{c_n^2}\sqrt{ \frac{\log n}{n} }
\end{equation}
By Lemmas \ref{lemma-consi-beta} and \ref{lemma-con-beta-b}, $\P(E_n) \ge 1 - O(n^{-1})$ if $b_n^3/c_n^2 = o( (n/\log n)^{1/2})$.

Similar to the proof of Theorem \ref{theorem-LRT-beta} (a),
it is sufficient to demonstrate: (i) $\{2( B_1 - B_1^0)-r\}/(2r)^{1/2}$ converges in distribution to the standard normal distribution;
(ii) $(B_2 - B_2^0)/r^{1/2}=o_p(1)$; (iii) $(B_3-B_3^0)/r^{1/2}=o_p(1)$.
The fourth-order Taylor expansion for $\ell(\bs{\widehat{\beta}}^0)$ here is with regard to the vector $(\beta_1, \beta_{r+1}, \ldots, \beta_n)^\top$
because $\beta_1, \ldots, \beta_r$ are the same under the null here. As we shall see, the expressions of $B_1^0$ and $B_2^0$ are a little different from $B_1$ and $B_2$
except from the difference $\bs{\widehat{\beta}}$ and $\bs{\widehat{\beta}}^0$.

With the same arguments as in \eqref{ineq-B3-upper} and \eqref{ineq-B30-upper}, we have claim (iii) under the condition $b_n^{12}/c_n^9 = o\left( r^{1/2}/(\log n) \right)$.
In Lemma \ref{lemma:beta3:err}, we show $B_2/r^{1/2} = o_p(1)$. For claim (ii), it is sufficient to show $B_2^0/r^{1/2}=o_p(1)$.
Under the null $H_0:  \beta_1=\ldots=\beta_r$,  $B_2^0$ can be written as
\begin{eqnarray*}
B_2^0  & = & \{ 4r(r-1)\mu^{\prime\prime}(\pi_{11})+r\sum_{j=r+1}^n \mu^{\prime\prime}(\pi_{ij}) \} (\widehat{\beta}_1^0 - \beta_1)^3
+3r \sum_{i=r+1}^n \mu^{\prime\prime}( \pi_{1i}) ( \widehat{\beta}_1^0  - \beta_1)^2 ( \widehat{\beta}_i^0 - \beta_i)  \\
&&+3r \sum_{i=r+1}^n \mu^{\prime\prime}(\pi_{1i}) ( \widehat{\beta}_1^0  - \beta_1) ( \widehat{\beta}_i^0 - \beta_i)^2
+ 3\sum_{i,j=r+1, i\neq j}^n \mu^{\prime\prime}(\pi_{ij}) ( \widehat{\beta}_i^0  - \beta_i) ( \widehat{\beta}_j^0 - \beta_j)^2.
\end{eqnarray*}
With the use of the asymptotic representation of $\bs{\widehat{\beta}}^0$ in Lemma \ref{lemma-beta-homo-expan}, if
\[
\frac{ b_n^{9} }{ c_n^7 } = o\left( \frac{ (rn)^{1/3} }{ \log n } \right)
 \mbox{~~and~~} \frac{ b_n^{5}}{ c_n^{2} } = o\left( \frac{r^{1/2}}{ (\log n)^2 } \right),
\]
then we have
\begin{equation}\label{eq-thereom1b-z43}
\frac{B_2^0}{r^{1/2}} = o_p( 1 ),
\end{equation}
whose detailed calculations are given in the supplementary material.

Next, we show claim (i). Recall $\bs{\tilde{d}}=(\sum_{i=1}^r d_i, d_{r+1},\ldots, d_n)$. The expression of $B_1^0$ is
\begin{equation}\label{B10-homo-expression}
B_1^0  = \frac{1}{2}( \bs{\widetilde{d}} - \E \bs{\widetilde{d}} )^\top \widetilde{V}^{-1} ( \bs{\widetilde{d}} - \E \bs{\widetilde{d}} )
-\frac{1}{2}\bs{\widetilde{h}}^\top \widetilde{V}^{-1} \bs{\widetilde{h}},
\end{equation}
where  $\mathbf{\widetilde{h}} = (\tilde{h}_1, \tilde{h}_{r+1}, \ldots, \tilde{h}_n)^\top$  satisfies
\eqref{eq-homo-tildeh}.
In view of \eqref{approxi-inv2-beta-ho} and \eqref{eq-homo-tildeh}, setting $\widetilde{V}^{-1}=\widetilde{S}+\widetilde{W}$ yields
\begin{eqnarray}
\nonumber
\bs{\tilde{h}}^\top \widetilde{V}^{-1} \bs{\tilde{h}} & \le &
\underbrace{\frac{\tilde{h}_1^2}{\tilde{v}_{11}} + \sum_{i=r+1}^n \frac{ h_i^2 }{ v_{ii} }} + \underbrace{ |\tilde{w}_{11}| \tilde{h}_1^2 + \| \widetilde{W} \|_{\max} \left(
2 |\tilde{h}_1| \sum_{i=r+1}^n |h_i| + \sum_{i,j=r+1}^n |h_i| |h_j| \right)}\\
\nonumber
& \lesssim &  b_n \left( \frac{b_n^6 \log n}{ c_n^5}  \right)^2 + \frac{ b_n^3 }{ n^2 c_n^2 } \left\{ r^2 + 2r(n-r)+(n-r)^2 \right\} \left( \frac{b_n^6 \log n}{ c_n^5}  \right)^2 \\
\label{eq-thereom1b-a}
& \lesssim & \frac{ b_n^{15} (\log n)^2 }{ c_n^{12}}.
\end{eqnarray}
This shows that if $b_n^{15}/c_n^{12} = o\left( r^{1/2}/(\log n)^2 \right)$,
\begin{equation*}\label{eq-thereom1b-a}
\frac{ |\bs{\tilde{h}}^\top \bs{\widetilde{V}}^{-1} \bs{\tilde{h}} | }{\sqrt{r}} = o_p(1).
\end{equation*}

Now, we evaluate the difference between $( \mathbf{\widetilde{d}} - \E \mathbf{\widetilde{d}} )^\top \widetilde{V}^{-1} ( \mathbf{\widetilde{d}} - \E \mathbf{\widetilde{d}} )$
and $\bs{\bar{d}}^\top V^{-1} \bs{\bar{d}}$.
By using $\widetilde{S}$ and $S$ to approximate $\widetilde{V}^{-1}$ and $V^{-1}$ respectively,
we have
\begin{eqnarray}
\nonumber
&&\bs{\bar{d}}^\top V^{-1} \bs{\bar{d}} -
( \mathbf{\widetilde{d}} - \E \mathbf{\widetilde{d}} )^\top \widetilde{V}^{-1} ( \mathbf{\widetilde{d}} - \E \mathbf{\widetilde{d}} ) \\
\label{eq-theorem1b-b}
& = & \sum_{i=1}^r \frac{ \bar{d}_i^{\,2} }{ v_{ii} } - \frac{ ( \tilde{d}_1- \E \tilde{d}_1) ^2 }{ \tilde{v}_{11} }
+ \bs{\bar{d}}^\top W \bs{\bar{d}} - ( \mathbf{\widetilde{d}} - \E \mathbf{\widetilde{d}} )^\top \widetilde{W} ( \mathbf{\widetilde{d}} - \E \mathbf{\widetilde{d}} ).
\end{eqnarray}
By Lemmas \ref{lemma-clt-beta-W} and \ref{lemma-tilde-W}, if $b_n^3/c_n^3=o(r^{1/2})$, then
\begin{equation}\label{eq-theorem1b-c}
\frac{1}{r^{1/2}} \max\{ \bs{\bar{d}} W \bs{\bar{d}}, ( \mathbf{\widetilde{d}} - \E \mathbf{\widetilde{d}} )^\top \widetilde{W} ( \mathbf{\widetilde{d}} - \E \mathbf{\widetilde{d}} ) \} = o_p(1).
\end{equation}
Since $\sum_{i=1}^r d_i = 2 \sum_{1\le i<j \le r} a_{ij} + \sum_{i=1}^r \sum_{j=r+1}^n a_{ij}$,
by the central limit theorem for the bounded case (\citet{Loeve:1977}, page 289),
$\tilde{v}_{11}^{-1/2}\sum_{i=1}^r (d_i - \E d_i)$ converges in distribution to the standard normal distribution if $\tilde{v}_{11}\to\infty$.
Therefore, as $r\to\infty$,
\[
\frac{ [\sum_{i=1}^r \{ d_i-\E(d_i) \}]^2/\tilde{v}_{11} }{ r } = o_p(1).
\]
By combining \eqref{B1-expression}, \eqref{B10-homo-expression}, \eqref{eq-simi-aVh}, \eqref{eq-theorem1b-b} and \eqref{eq-theorem1b-c}, it yields
\[
\frac{2(B_1-B_1^0)}{\sqrt{2r}}
= \frac{1}{\sqrt{2r}} \sum_{i=1}^r \frac{ ( d_i- \E d_i )^2 }{ v_{ii} }  + o_p(1).
\]
Therefore, claim (i) immediately follows from Lemma \ref{lemma:weighte-degree-al}.
This completes the proof.
\end{proof}

\subsection{Proofs for Theorem \ref{theorem-LRT-beta-fixed} (a)}

Let $\bar{\dd}_1 = ( \bar{d}_1, \ldots, \bar{d}_r)^\top$, $\bar{\dd}_2 = ( \bar{d}_{r+1}, \ldots, \bar{d}_n)^\top$ and
\begin{equation}\label{VW-divide}
V =
\begin{pmatrix} V_{11}  & V_{12} \\
V_{21} & V_{22}
\end{pmatrix}, ~~
W = \begin{pmatrix} W_{11} & W_{12} \\
W_{21} & W_{22}
\end{pmatrix},
\end{equation}
where $V_{11}$ and $W_{11}$ are respective $r\times r$ dimensional sub-matrices of $V$ and $W$, and $W=V^{-1}-S$.
Recall that $V_{22}$ denotes the Fisher information matrix of $\bs{\beta}_2=(\beta_{r+1}, \ldots, \beta_n)^\top$
under the null $H_0: (\beta_1, \ldots, \beta_r)=(\beta_1^0, \ldots, \beta_r^0)$ and $S_{22}=\mathrm{diag}( 1/v_{r+1, r+1}, \ldots, 1/v_{nn})$.
It is worthy to mention that $r$ is a fixed constant in this section.

To prove Theorem \ref{theorem-LRT-beta-fixed} (a), we need the following three lemmas.
The lemma below gives an upper bound of $\| W_{22} - \widetilde{W}_{22} \|_{\max}$,
whose magnitudes are $b_n^6/(n^3c_n^5)$. It is much smaller than that the error bounds of $W_{22}$  and $\widetilde{W}_{22}$
themselves in \eqref{ineq-V-S-appro-upper-b} by a vanishing factor $n^{-1}$.

\begin{lemma}\label{lemma:w2-error}
For a fixed constant $r$, the error between $W_{22}$ and $\widetilde{W}_{22}$ in terms of the maximum absolute entry-wise norm has the following bound:
\begin{equation}\label{ineq-W-diff-upper}
\| W_{22} - \widetilde{W}_{22} \|_{\max} \lesssim \frac{ b_n^6 }{ n^3c_n^5 }.
\end{equation}
\end{lemma}

The following lemma gives the upper bounds of three remainder terms in \eqref{eq-theorem2a-B1022} that tend to zero.

\begin{lemma}\label{lemma-W-widetilde-d}
Suppose $r$ is a fixed constant. \\
(a)If $b_n^3/c_n^2=o( n^{3/2}/(\log n)^{1/2})$, then $\bar{\dd}_1^\top W_{11} \bar{\dd}_1 = o_p(1)$. \\
(b)If $b_n^3/c_n^3=o( n^{1/2} )$, then $\bar{\dd}_1^\top W_{12} \bar{\dd}_2 = o_p(1)$. \\
(c)If $b_n^3/c_n^3=o( n^{3/4} )$, then
\[
\bar{\dd}_2^\top ( W_{22} - \widetilde{W}_{22}) \bar{\dd}_2 = o_p(1).
\]
\end{lemma}

The lemma below establishes the upper bound of $\max_{i=r+1, \ldots, n} | \widehat{\beta}_i - \widehat{\beta}_i^0 |$.

\begin{lemma}\label{lemma-hat-beta-diff}
Under the null $H_0: \beta_i=\beta_i^0$, $i=1,\ldots, r$ with a fixed $r$,
if $b_n^3/c_n = o( n/\log n)$, then with probability at least $1-O(n^{-1})$,
\[
\max_{i=r+1, \ldots, n} | \widehat{\beta}_i - \widehat{\beta}_i^0 | \lesssim \frac{b_n^3 \log n}{nc_n}.
\]
\end{lemma}

The above error bound is in the magnitude of $n^{-1}$, up to a factor $b_n^3\log n$, which makes the remainder terms in \eqref{eq-theorem2-whh} be asymptotically neglected.
Note that this error bound is much smaller than those for $\| \bs{\widehat{\beta}}^0 - \bs{\beta}^0 \|_\infty$
and $\| \bs{\widehat{\beta}} - \bs{\beta}^0 \|_\infty$ by a vanishing factor $n^{-1/2}$, whose magnitudes are $b_n(\log n/n)^{1/2}$.

Now, we are ready to prove Theorem \ref{theorem-LRT-beta-fixed} (a).

\begin{proof}[Proof of Theorem \ref{theorem-LRT-beta-fixed} (a)]

The following calculations are based on the event $E_n$ that $\bs{\widehat{\beta}}$ and $\bs{\widehat{\beta}}^0$ simultaneously exist and satisfy
\begin{equation}\label{ineq-beta-beta0-upp-f}
\max\left\{ \| \widehat{\bs{\beta}} - \bs{\beta} \|_\infty, \| \widehat{\bs{\beta}}^0 - \bs{\beta} \|_\infty \right\}
\le \frac{3nb_n}{(2n-1)}\sqrt{ \frac{\log n}{n} },~~\max_{i=r+1, \ldots, n} | \widehat{\beta}_i - \widehat{\beta}_i^0 | \lesssim \frac{b_n^3 \log n}{nc_n}.
\end{equation}
By Lemmas \ref{lemma-consi-beta} and \ref{lemma-hat-beta-diff}, $\P(E_n) \ge 1 - O(n^{-1})$ if $b_n^3/c_n=o\left( n/\log n \right)$.

Similar to the proof of Theorem \ref{theorem-LRT-beta} (a),
it is sufficient to demonstrate: (1) $2( B_1 - B_1^0)$ converges in distribution to the Chi-square distribution with $r$ degrees of freedom;
(2) $B_2 - B_2^0$ and $B_3-B_3^0$ are asymptotically neglected remainder terms,
where $B_1-B_1^0$ is given in \eqref{eq-theorem2-B10}, $B_2$ and $B_3$ are given in \eqref{lrt-a-beta-B2} and \eqref{lrt-a-beta-B3}, $B_2^0$ and $B_3^0$ are respective versions of $B_2$ and $B_3$ by replacing
$\bs{\widehat{\beta}}$ with $\bs{\widehat{\beta}}^0$.
Claims (1) and (2) are shown in three steps in turns.

Step 1. We show $2( B_1 - B_1^0)\stackrel{L}{\to} \chi^2_r$. Using the matrix form in \eqref{VW-divide}, $ B_1 - B_1^0$ in \eqref{eq-theorem2-B10} can be written as
\begin{eqnarray}
\nonumber
2(B_1 - B_1^0) & = & \sum_{i=1}^r \frac{ \bar{d}_i^{\,2} }{v_{ii}}  + \underbrace{\mathbf{\bar{d}}_1^\top W_{11} \mathbf{\bar{d}}_1 +
 2\mathbf{\bar{d}}_1^\top W_{11} \mathbf{\bar{d}}_2 + \mathbf{\bar{d}}_2^\top ( W_{22} - \widetilde{W}_{22}) \mathbf{\bar{d}}_2}_{Z_1}\\
 \label{eq-theorem2a-B1022}
&&+ \underbrace{\mathbf{\widetilde{h}}_2^\top V_{22}^{-1} \mathbf{\widetilde{h}}_2 - \mathbf{h}^\top V^{-1} \mathbf{h}}_{Z_2}.
\end{eqnarray}
It is sufficient to demonstrate: (i) $\sum_{i=1}^r  \bar{d}_i^{\,2}/v_{ii}$ converges in distribution to a Chi-square distribution with $r$ degrees of freedom;
(ii) $Z_1 = o_p(1)$; (iii) $Z_2=o_p(1)$.
Claim (ii) directly follows from Lemma  \ref{lemma-W-widetilde-d}.
Because $\tilde{d}_i = \sum_{j=r+1}^n a_{ij}$ is independent over $i=1,\ldots, r$ and $r$ is a fixed constant,
the classical central limit theorem for the bounded case (\cite{Loeve:1977}, p. 289) gives that
the vector $(\bar{d}_1/v_{11}^{1/2}, \ldots, \bar{d}_r/v_{rr}^{1/2})$
follows a $r$-dimensional standard normal distribution. This verifies claim (i).
Now, we show $Z_2=o_p(1)$.
Recall the definition of $\mathbf{h}$ in \eqref{eq:definition:h}.
By setting $V^{-1}=S+W$ and $V_{22}^{-1}=S_{22}+W_{22}$, we have
\begin{eqnarray}
\label{eq-th2a-hVh}
\mathbf{h}^\top V^{-1} \mathbf{h} &  = & \sum_{i=1}^n \frac{ h_i^2 }{ v_{ii} } + \hh_1^\top W_{11} \hh_1
+ 2 \hh_1^\top W_{12} \hh_2 + \hh_2^\top W_{22} \hh_2, \\
\label{eq-th2a-h2V22h2}
\widetilde{\hh}_2^\top V_{22}^{-1} \widetilde{\hh}_2 & = & \sum_{i=r+1}^n \frac{ \tilde{h}_i^2 }{ v_{ii} }
+ \widetilde{\hh}_2^\top \widetilde{W}_{22} \widetilde{\hh}_2,
\end{eqnarray}
where $\mathbf{h}_1=(h_1, \ldots, h_r)^\top$, $\mathbf{h}_2=(h_{r+1}, \ldots, h_n)^\top$, and $h_i$ and  $\tilde{h}_i$ are given in \eqref{eq:definition:h} and \eqref{defintion-tilde-h},
respectively.
Since $\max\{\|\mathbf{h}\|_\infty, \| \bs{\widetilde{h}} \|_\infty\} \lesssim b_n^2\log n /c_n$ (see \eqref{ineq-beta-h} and
\eqref{defintion-tilde-h}), we have
\[
\sum_{i=1}^n \frac{ h_i^2 }{ v_{ii} } - \sum_{i=r+1}^n \frac{ \tilde{h}_i^2 }{ v_{ii} } = \sum_{i=r+1}^n \frac{ h_i^2 - \tilde{h}_i^2  }{ v_{ii} } + O_p( \frac{ b_n^5(\log n)^2 }{nc_n^2}).
\]
The difference $h_i-\tilde{h}_i$ is bounded as follows:
\begin{eqnarray}
\nonumber
\max_{i=r+1,\ldots, n} |h_i-\tilde{h}_i| & \le &   \sum_{j\neq i} \left\{ |\mu^{\prime\prime}(\widetilde{\pi}_{ij})( \widehat{\pi}_{ij} - \pi_{ij})^2 - \mu^{\prime\prime}(\tilde{\pi}_{ij}^0)
( \widehat{\pi}_{ij} - \pi_{ij})^2| \right. \\
\nonumber
&&
\left.+ |\mu^{\prime\prime}(\tilde{\pi}_{ij}^0) ( \widehat{\pi}_{ij} - \pi_{ij})^2 - \mu^{\prime\prime}(\tilde{\pi}_{ij}^0)(( \widehat{\pi}_{ij} - \pi_{ij})^2)^2]| \right\} \\
\nonumber
& \lesssim & \frac{n}{c_n} \left\{ | \tilde{\pi}_{ij}-\tilde{\pi}_{ij}^0| \cdot ( \widehat{\pi}_{ij} - \pi_{ij})^2 +
 | \widehat{\pi}_{ij}-\widehat{\pi}_{ij}^0| \cdot ( |\widehat{\pi}_{ij} - \pi_{ij}|  +  |\widehat{\pi}_{ij}^0- \pi_{ij}|) \right\} \\
\nonumber
&\lesssim & \frac{n}{c_n} \cdot \left( b_n\sqrt{\frac{\log n}{n}}\right)^3 + \frac{n}{c_n} \cdot \frac{ b_n^3 \log n}{nc_n} \cdot \left( b_n\sqrt{\frac{\log n}{n}}\right)^2\\
\label{ineq-h-h-r1}
& \lesssim & \frac{ b_n^3 (\log n)^2 }{n^{1/2} c_n},
\end{eqnarray}
where the second inequality is due to \eqref{ineq-mu-deriv-bound} and the mean value theorem and the third inequality follows from \eqref{ineq-beta-beta0-upp-f}.
Therefore, if $ b_n^5/c_n^2 = o( n/(\log n)^2 )$, then
\begin{equation}\label{eq-theorem2-hhd}
| \sum_{i=1}^n \frac{ h_i^2 }{ v_{ii} } - \sum_{i=r+1}^n \frac{ \tilde{h}_i^2 }{ v_{ii} } | = O_p(  \frac{ b_n^5 (\log n)^2 }{nc_n^2} ) + O(\frac{ b_n^4 (\log n)^2 }{n^{3/2} c_n})  = o_p(1).
\end{equation}
By \eqref{ineq-V-S-appro-upper-b} and \eqref{ineq-beta-h}, we have
\begin{eqnarray}
| \hh_1^\top W_{11} \hh_1 | & \le & r^2 \| W_{11} \|_{\max} \| \hh_1 \|_\infty \lesssim r^2 \cdot b_n^2 \log n \cdot \frac{ b_n^3 }{ n^2c_n^2 }\lesssim  \frac{ b_n^5 \log n}{ nc_n^2},
\\
| \hh_1^\top W_{12} \hh_2 | & \lesssim & r(n-r) \cdot b_n^2\log n \cdot \frac{ b_n^3 }{ n^2c_n^2 }
\lesssim \frac{ b_n^5 \log n}{ nc_n^2 }.
\end{eqnarray}
To evaluate the bound of $\hh_2^\top W_{22} \hh_2 - \widetilde{\hh}_2^\top \widetilde{W}_{22}^{-1} \widetilde{\hh}_2$,
we divide it into three terms:
\begin{eqnarray}
\nonumber
& & \hh_2^\top W_{22} \hh_2 - \widetilde{\hh}_2^\top \widetilde{W}_{22}^{-1} \widetilde{\hh}_2 \\
\label{eq-theorem2-whh}
& = & \underbrace{\hh_2^\top W_{22} \hh_2 - \hh_2^\top \widetilde{W}_{22} \hh_2}_{C_1} +
\underbrace{\hh_2^\top \widetilde{W}_{22} \hh_2
- \widetilde{ \hh}_2 \widetilde{W}_{22} \hh_2}_{C_2}
+
\underbrace{\widetilde{\hh}_2^\top W_{22} \hh_2 - \widetilde{\hh}_2 \widetilde{W}_{22} \widetilde{\hh}_2}_{C_3}.
\end{eqnarray}
The first term $C_1$ is bounded as follows.
By \eqref{ineq-W-diff-upper} and \eqref{ineq-beta-h}, we have
\begin{eqnarray}
\nonumber
| \hh_2^\top ( W_{22} - \widetilde{W}_{22} ) \hh_2 | & \le & (n-r)^2 \| W_{22} - \widetilde{W}_{22} \|_{\max} \| \hh \|_\infty \\
\label{ineq-upper-B1}
& \lesssim & n^2 \cdot \frac{ b_n^6 }{ n^3c_n^5 } \cdot b_n^2 \log n  \lesssim \frac{ b_n^8}{ n c_n^5}.
\end{eqnarray}
In view of \eqref{ineq-V-S-appro-upper-b} and \eqref{ineq-h-h-r1}, the upper bounds of $C_2$ and $C_3$ are derived as follows:
\begin{eqnarray}
\nonumber
|C_2|& = 
&   |( \bs{h}_2- \widetilde{\bs{h} }_2) ^\top \widetilde{W}_{22} \bs{h}_2 |
 \le  (n-r)^2 \| \widetilde{W}_{22} \|_{\max} \| \bs{h}_2- \widetilde{\bs{h}}_2 \|_\infty \| \bs{h}_2 \|_\infty \\
\label{ineq-upper-B2}
& \lesssim & n^2 \cdot \frac{b_n^3}{n^2c_n^2} \cdot \frac{ b_n^3 (\log n)^{2} }{ n^{1/2}c_n } \cdot b_n^2 \log n
 \lesssim \frac{ b_n^8 (\log n)^{3} }{ n^{1/2}c_n^3 }.
\end{eqnarray}
and
\begin{eqnarray}
\nonumber
|C_3|& = & |\widetilde{\bs{h}}_2^\top \widetilde{W}_{22} (\bs{h}_2 -  \widetilde{\bs{h}}_2) |
 \le  (n-r)^2 \|\widetilde{\bs{h}}_2\|_\infty  \| \widetilde{W}_{22} \|_{\max} \| \bs{h}_2 -  \widetilde{\bs{h}}_2 \|_\infty \\
\nonumber
& \le & (n-r)^2 \cdot \| W_{22} \|_{\max} \cdot \| \widetilde{\hh}_2 \|_\infty \cdot \| \widetilde{\hh}_2 - \hh_2 \|_\infty
+ (n-r)^2 \cdot \| W_{22} - \widetilde{W}_{22}  \|_{\max} \cdot \| \widetilde{\hh}_2 \|_\infty^2 \\
\label{ineq-upper-B3}
& \lesssim & \frac{ b_n^8 (\log n)^{3} }{ n^{1/2}c_n^3 }.
\end{eqnarray}
By combining \eqref{eq-th2a-hVh}--\eqref{ineq-upper-B3}, it yields
\begin{equation}
|\mathbf{h}^\top V^{-1} \mathbf{h} -
\widetilde{\hh}_2^\top V_{22}^{-1} \widetilde{\hh}_2|
\lesssim \frac{ b_n^8 (\log n)^{5/2} }{ n^{1/2}c_n } .
\end{equation}
This completes the proof of the first step.

Step 2. We bound $B_2 - B_2^0$.
For a cubic term in $B_2-B_2^0$, a direct scaling method gives that
\begin{eqnarray}
\nonumber
&&| \sum_{i=1}^n [ (\widehat{\beta}_i - \beta_i^0)^3 - ( \widehat{\beta}_i^0 - \beta_i^0)^3 ] [\sum_{j\neq i} \mu^{\prime\prime}( \pi_{ij} ) ]|\\
\nonumber
& \le & \frac{n(n-1)}{c_n} \cdot \| \bs{\widehat{\beta}} - \bs{\widehat{\beta}}^0 \|_\infty \cdot 2[\sum_i (\widehat{\beta}_i - \beta_i^0)^2 +   \sum_i (\widehat{\beta}_i^0 - \beta_i^0)^2] \\
\nonumber
& \lesssim & \frac{ b_n^9 (\log n)^3 }{ c_n^3 }.
\end{eqnarray}
The term in the above right hand does not tend to zero. Because $r$ is fixed, the approach for showing $(B_2-B_2^0)/r^{1/2}=o_p(1)$ in the proof of Theorem \ref{theorem-LRT-beta} does not work yet.
 To prove that this term does go to zero, we did a careful analysis on the difference
$B_2-B_2^0$ by using asymptotic representations of $\widehat{\beta}_i - \beta_i$ and $\widehat{\beta}_i^0 - \beta_i$.
With the use of Lemmas \ref{lemma:beta3:err} and \ref{lemma-hat-beta-diff}, we have
\begin{equation}\label{ineq-B2-B20}
|B_2 - B_2^0| \lesssim \frac{ b_n^9 (\log n)^3 }{ n^{1/2} c_n^3 },
\end{equation}
whose detailed proofs are given in Section 4.4 in Supplementary Material A.

Step 3. We bound $B_3 - B_3^0$. With the same reason as in Step 2, we could not yet use the method of proving $(B_3-B_3^0)/r^{1/2}=o_p(1)$
in the proof of Theorem \ref{theorem-LRT-beta}.
With the use of asymptotic representations of $\widehat{\beta}_i - \beta_i$ and $\widehat{\beta}_i^0 - \beta_i$ and Lemma \ref{lemma-hat-beta-diff}, we can show
\begin{equation}\label{eq-th2a-B3B30}
|B_3 - B_3^0| \lesssim \frac{b_n^6 (\log n)^{3} }{ n^{1/2}c_n},
\end{equation}
whose detailed proofs are given in Section 4.5 in Supplementary Material A.
This completes the proof.
\end{proof}

\subsection{Proofs for Theorem \ref{theorem-LRT-beta-fixed} (b)}

Recall that $\bs{\bar{d}}_1 = ( \bar{d}_1, \ldots, \bar{d}_r)^\top$, $\bs{\bar{d}}_2 = ( \bar{d}_{r+1}, \ldots, \bar{d}_n)^\top$, $\bs{\widetilde{d}}=(\sum_{i=1}^r d_i, d_{r+1},\ldots, d_n)$ and $\widetilde{V}$ is given in \eqref{approxi-inv2-beta-ho}.
$\widetilde{V}$ is the Fisher information matrix of $\widetilde{\bs{\beta}}=(\beta_1, \beta_{r+1}, \ldots, \beta_n)^\top$
under the null $H_0: \beta_1 = \cdots= \beta_r$.
Remark that $r$ is a fixed constant in this section.
Partition $\widetilde{W}$ into four blocks
\begin{equation}\label{VW-divide}
\widetilde{W} = \begin{pmatrix} \tilde{w}_{11} & \bs{\tilde{w}}_{12} \\
\bs{\tilde{w}}_{21} & \widetilde{W}_{22}
\end{pmatrix},
\end{equation}
where $\tilde{w}_{11}$ is a scalar and the dimension of $\widetilde{W}_{22}$ is $(n-r)\times (n-r)$.
It should be noted $\widetilde{W}_{22}$ is different from $\widetilde{W}_{22}$
in the proof of Theorem \ref{theorem-LRT-beta} (a),
where $\widetilde{W}_{22}=V_{22}^{-1}-S_{22}$.
With some ambiguity of notation, we use the same notation here. However, both share very similar properties.

To prove Theorem \ref{theorem-LRT-beta-fixed} (b), we need the following three lemmas,
whose proofs are respectively similar to those of Lemmas \ref{lemma:w2-error}, \ref{lemma-W-widetilde-d} and \ref{lemma-hat-beta-diff} and omitted.

Recall that $W_{22}$ is the bottom right $(n-r)\times (n-r)$ block of $W=V^{-1}-S$.
The lemma below gives an upper bound of $\| W_{22} - \widetilde{W}_{22} \|_{\max}$,
whose magnitudes are $b_n^3/(n^3c_n^5)$.

\begin{lemma}\label{lemma:w2-error-2b}
For a fixed constant $r$, the error between $W_{22}$ and $\widetilde{W}_{22}$ in terms of the maximum absolute entry-wise norm has the following bound:
\begin{equation}\label{ineq-WW-diff-2b}
\| W_{22} - \widetilde{W}_{22} \|_{\max} \lesssim \frac{ b_n^6 }{ n^3c_n^5 }.
\end{equation}
\end{lemma}

It is remark that the absolute entry-wise error between $W_{22}$ and $\widetilde{W}_{22}$ in the Bradley-Terry model is not in the order of $O(n^{-3})$ that holds in the $\beta$-model (i.e., \eqref{ineq-WW-diff-2b}),
but $O(n^{-3})$ adding a special matrix whose order is $O(n^{-2})$ (see (289) on page 20 in Supplementary Material B).
The following lemma gives the upper bounds of three remainder terms in \eqref{eq-theorem2b-B1022} that tend to zero,
whose proof is similar to the proof of Lemma \ref{lemma-W-widetilde-d}  and is omitted.

\begin{lemma}\label{lemma-W-widetilde-d-2b}
Suppose $r$ is a fixed constant. \\
(a)If $b_n^3/c_n^2=o( n^{3/2}/(\log n)^{1/2})$, then $(\sum_{i=1}^r\bar{d}_i) \tilde{w}_{11} (\sum_{i=1}^r\bar{d}_i) = o_p(1)$. \\
(b)If $b_n^3/c_n^3=o( n^{1/2} )$, then $(\sum_{i=1}^r\bar{d}_i) \bs{\tilde{w}}_{12}^\top \bar{\dd}_2 = o_p(1)$. \\
(c)If $b_n^3/c_n^3=o( n^{3/4} )$, then
\[
\bar{\dd}_2^\top ( W_{22} - \widetilde{W}_{22}) \bar{\dd}_2 = o_p(1).
\]
\end{lemma}

The lemma below establishes the upper bound of $ \widehat{\beta}_i - \widehat{\beta}_i^0 $.

\begin{lemma}\label{lemma-hat-beta-diff-2b}
Under the null $H_0: \beta_1=\cdots = \beta_r$ with a fixed $r$,
if $b_n^3/c_n = o( n/\log n)$, then with probability at least $1-O(n^{-1})$,
\[
\max_{i=r+1, \ldots, n} | \widehat{\beta}_i - \widehat{\beta}_i^0 | \lesssim \frac{b_n^3 \log n}{nc_n}.
\]
\end{lemma}

The above error bound is in the magnitude of $n^{-1}$, up to a factor $b_n^3\log n$,
which makes the remainder terms in \eqref{eq-theorem2b-B1022} be asymptotically neglected.
Note that this error bound is the same as that in Lemma \ref{lemma-hat-beta-diff}.

Now, we are ready to prove Theorem \ref{theorem-LRT-beta-fixed} (b).

\begin{proof}[Proof of Theorem \ref{theorem-LRT-beta-fixed} (b)]
Note that $\bs{\widehat{\beta}}^0$ denotes the restricted MLE under the null space $\Theta_0 = \{ \bs{\beta}\in \R^n: \beta_1= \cdots = \beta_r \}$.
The following calculations are based on the event $E_n$ that is defined in \eqref{ineq-beta-beta0-upp}.
By Lemmas \ref{lemma-consi-beta} and \ref{lemma-hat-beta-diff-2b}, $\P(E_n) \ge 1 - O(n^{-1})$ if $b_n^3/c_n^2=o( n/\log n)$.

Similar to the proof of Theorem \ref{theorem-LRT-beta-fixed} (a),
it is sufficient to demonstrate: (1) $2( B_1 - B_1^0)$ converges in distribution to the Chi-square distribution with $r$ degrees of freedom;
(2)
\[
B_2-B_2^0 = O_p\left( \frac{b_n^9(\log n)^3 }{ n^{1/2} c_n^3 }\right), ~~ B_3 - B_3^0 = O_p\left( \frac{ b_n^6 (\log n)^3 }{ n^{1/2} c_n } \right).
\]
The proof of claim (2) is similar to those of \eqref{ineq-B2-B20} and \eqref{eq-th2a-B3B30} and omitted.
We only present the proof of claim (1) here. 

We show $2( B_1 - B_1^0)\stackrel{L}{\to} \chi^2_r$.
Using the matrix form in \eqref{VW-divide}, $ B_1 - B_1^0$ in \eqref{eq-ell-difference} can be written as
\begin{eqnarray}
\nonumber
& &2(B_1 - B_1^0) \\
\nonumber
&  = & \underbrace{\sum_{i=1}^r \frac{ \bar{d}_i^{\,2} }{v_{ii}}  - \frac{ ( \sum_{i=1}^r \bar{d}_i )^2 }{ \tilde{v}_{11} }}_{Z_1}
 + \mathbf{\bar{d}}_1^\top W_{11} \mathbf{\bar{d}}_1 + 2\mathbf{\bar{d}}_1^\top W_{12} \mathbf{\bar{d}}_2 + (\sum_{i=1}^r \bar{d}_i )^2 \tilde{w}_{11} \\
 \label{eq-theorem2b-B1022}
&&  + 2 ( \sum_{i=1}^r \bar{d}_i) \bs{\tilde{w}}_{12}^\top  \mathbf{\bar{d}}_2
+  \mathbf{\bar{d}}_2^\top ( W_{22} - \widetilde{W}_{22}) \mathbf{\bar{d}}_2
+ \underbrace{\mathbf{h}^\top V^{-1} \mathbf{h} - \mathbf{\widetilde{h}}^\top \widetilde{V}^{-1} \mathbf{\widetilde{h}}}_{Z_2},
\end{eqnarray}
where $\bs{h}$ is defined in \eqref{eq:definition:h} and $\mathbf{\widetilde{h}}$ is in \eqref{eq-homo-tildeh}.
In view of Lemmas \ref{lemma-W-widetilde-d} and \ref{lemma-W-widetilde-d-2b}, it is sufficient to demonstrate: (i) $Z_1$ converges in distribution to a Chi-square distribution with $r-1$ degrees of freedom;
(ii) $Z_2=o_p(1)$.
Because $\tilde{d}_i = \sum_{j=r+1}^n a_{ij}$ is independent over $i=1,\ldots, r$ and $r$ is a fixed constant,
the classical central limit theorem for the bounded case (\cite{Loeve:1977}, p. 289) gives that
the vector $(\bar{d}_1/v_{11}^{1/2}, \ldots, \bar{d}_r/v_{rr}^{1/2})$
follows a $r$-dimensional standard normal distribution.
Because $v_{11}=\cdots=v_{rr}$ under the null $H_0:\beta_1=\cdots=\beta_r$ and $\tilde{v}_{11}=rv_{11}$, we have
\[
\sum_{i=1}^r \frac{ \bar{d}_i^{\,2} }{v_{ii}}  - \frac{ ( \sum_{i=1}^r \bar{d}_i )^2 }{ \tilde{v}_{11} }
= \left( \frac{\bar{d}_1}{v_{11}^{1/2}}, \ldots, \frac{\bar{d}_r}{v_{11}^{1/2}}\right)\left( I_r - \frac{1}{r} \mathbf{1}_r \mathbf{1}_r^\top \right)
\left( \frac{\bar{d}_1}{v_{11}^{1/2}}, \ldots, \frac{\bar{d}_r}{v_{11}^{1/2}}\right)^\top.
\]
Because $\mathrm{rank}( I_r - \mathbf{1}_r \mathbf{1}_r^\top/r )=r-1$, it follows that we have claim (i).
Now, we show $Z_2=o_p(1)$.
By setting $V^{-1}=S+W$ and $V_{22}^{-1}=S_{22}+W_{22}$, we have
\begin{eqnarray*}
\mathbf{h}^\top V^{-1} \mathbf{h} &  = & \sum_{i=1}^n \frac{ h_i^2 }{ v_{ii} } + \hh_1^\top W_{11} \hh_1
+ 2 \hh_1^\top W_{12} \hh_2 + \hh_2^\top W_{22} \hh_2, \\
\widetilde{\hh}_2^\top V_{22}^{-1} \widetilde{\hh}_2 & = & \frac{ \tilde{h}_1^2 }{ \tilde{v}_{11}}
+ \sum_{i=r+1}^n \frac{ \tilde{h}_i^2 }{ v_{ii} }
+ \tilde{w}_{11} \tilde{h}_1^2 + 2 \tilde{h}_1 \tilde{w}_{12} \bs{\tilde{h}}_2 + \bs{\tilde{h}}_2^\top \widetilde{W}_{22} \bs{\tilde{h}}_2,
\end{eqnarray*}
where $\mathbf{h}_1=(h_1, \ldots, h_r)^\top$, $\mathbf{h}_2=(h_{r+1}, \ldots, h_n)^\top$, and $h_i$ and  $\tilde{h}_i$ are given in \eqref{eq:definition:h} and \eqref{eq-homo-tildeh},
respectively.
In view of \eqref{ineq-beta-h} and \eqref{eq-homo-tildeh}, we have
\[
\sum_{i=1}^n \frac{ h_i^2 }{ v_{ii} } - \sum_{i=r+1}^n \frac{ \tilde{h}_i^2 }{ v_{ii} } = \sum_{i=r+1}^n \frac{ h_i^2 - \tilde{h}_i^2  }{ v_{ii} } + O( \frac{ b_n^5(\log n)^2 }{nc_n^2}).
\]
With the same arguments as in the proof of \eqref{ineq-h-h-r1}, we have
\begin{equation*}
\label{ineq-h-h-r1-2b}
\max_{i=r+1,\ldots, n} |h_i-\tilde{h}_i|   \lesssim  \frac{ b_n^3 (\log n)^2 }{n^{1/2} c_n}.
\end{equation*}
Therefore, if $ b_n^6/c_n = o( n/(\log n)^2 )$, then
\begin{equation}\label{eq-theorem2-hhd}
| \sum_{i=1}^n \frac{ h_i^2 }{ v_{ii} } - \sum_{i=r+1}^n \frac{ \tilde{h}_i^2 }{ v_{ii} } | = O_p(  \frac{ b_n^5 (\log n)^2 }{nc_n^2} ) + O(\frac{ b_n^4 (\log n)^2 }{n^{3/2} c_n})  = o(1).
\end{equation}
By \eqref{ineq-V-S-appro-upper-b}, \eqref{ineq-beta-h} and \eqref{approxi-inv2-beta-ho}, we have
\begin{eqnarray*}
| \hh_1^\top W_{11} \hh_1 | & \le & r^2 \| W_{11} \|_{\max} \| \hh_1 \|_\infty \lesssim r^2 \cdot b_n^2 \log n \cdot \frac{ b_n^3 }{ n^2c_n^2 }\lesssim  \frac{ b_n^5 \log n}{ nc_n^2},
\\
| \hh_1^\top W_{12} \hh_2 | & \lesssim & r(n-r) \cdot b_n^2\log n \cdot \frac{ b_n^3 }{ n^2c_n^2 }
\lesssim \frac{ b_n^5 \log n}{ nc_n^2 },  \\
\frac{ \tilde{h}_1^2 }{ \tilde{v}_{11}} & \lesssim & \frac{ r^2 ( b_n^2\log n/c_n)^2 }{ rn/b_n} \lesssim \frac{ b_n^3 (\log n)^2 }{ nc_n^2}, \\
|\tilde{w}_{11} \tilde{h}_1^2| & \lesssim &  ( b_n^2\log n/c_n)^2 \cdot \frac{ b_n^3 }{ n^2 c_n^2 } \lesssim \frac{ b_n^5 (\log n)^2 }{ n^2 c_n^4}, \\
 |\tilde{h}_1 \tilde{w}_{12} \bs{\tilde{h}}_2| & \lesssim & r(n-r) \cdot ( b_n^2\log n/c_n)^2 \cdot \frac{ b_n^3 }{ n^2 c_n^2 }\lesssim \frac{ b_n^5 (\log n)^2 }{ n c_n^4}.
\end{eqnarray*}
As in the proofs of \eqref{eq-theorem2-whh}--\eqref{ineq-upper-B3}, we have
\[
|\bs{h}_2 W_{22} \bs{h}_2 - \widetilde{\hh}_2^\top \widetilde{W}_{22} \widetilde{\hh}_2| \lesssim  \frac{ b_n^8 (\log n)^3 }{ n^{1/2} c_n^3 }.
\]
Combining the above inequalities and \eqref{eq-theorem2-hhd}, it yields
\begin{equation*}
|\mathbf{h}^\top V^{-1} \mathbf{h} -
\widetilde{\hh}^\top \widetilde{V}^{-1} \widetilde{\hh}|
\lesssim \frac{ b_n^8 (\log n)^3 }{ n^{1/2}c_n }.
\end{equation*}
which shows claim (ii).

\end{proof}

\noindent {\bf Acknowledgment}
The views expressed are those of the authors and should not be construed to represent the positions of the Department of the Army or Department of Defense. Yan is partially supported by the National Natural Science
Foundation of China (no. 12171188) and the Fundamental Research Funds
for the Central Universities. Xu is partially supported by the General Research Fund of Hong Kong (17308820).
 Zhu is partially
supported by the National Science Foundation (DMS 1407698 and DMS
1821243).

\setlength{\itemsep}{-1.5pt}
\setlength{\bibsep}{0ex}
\bibliography{reference3}
\bibliographystyle{apa}

\newpage
\begin{center}
{\Large Supplementary material A for ``Wilks' theorems in the $\beta$-model"\footnote{Supplementary Materials B and C are available by sending emails to
\texttt{Email:} tingyanty@mail.ccnu.edu.cn}} \\
\medskip
Ting Yan, Yuanzhang Li, Jinfeng Xu, Yaning Yang and Ji Zhu
\end{center}

Supplementary Material A contains the proofs of supported lemmas in the proofs of Theorems 1 and 2 as well as inequalities (39), (59) and (60) in the main text.
This supplementary material is organized as follows.
Section \ref{section-variance} gives the variances of the weighted quadratic sum $\sum_i f_i \bar{d}_i^{\,2}$,
the weighted cubic sum $\sum_i f_i \bar{d}_i^{\,3}$ and an upper bound of a mixed sum $\sum_{i,j} f_{ij} \bar{d}_i^{\,2} \bar{d}_j$, which will be used in the proofs
of supported lemmas repeatedly.

Section \ref{section:beta-th1a} contains the proofs of supported lemmas in the proof of Theorem 1 (a).
This section is organized as follows.
Sections \ref{subsection-proof-lemma2} and \ref{subsection-proof-lemma3}
present the proofs of Lemmas 3 and 4.
Section \ref{subsection-L2norm} gives an additional result about the upper bound for $\bs{\widehat{\beta}}$ in terms of the $L_2$-norm.
Section \ref{subsec-asy-widehatbeta} gives an asymptotically explicit expression for $\bs{\widehat{\beta}}^0$ that will be used in the proof of Lemma 5.
Section \ref{subsec-prooflemma4} presents the proof of Lemma 5.
The proof of Lemma 2 about approximation error of using $S_{22}$ to approximate $V_{22}^{-1}$ is present in Supplementary Material C.
We defer the proof of Lemma 1 to Section \ref{section-lemma1} since it contains many long calculations.

Section \ref{section-th1b} presents proofs of supported lemmas in the proof of Theorem 1 (b)
This section is organized as follows.
Sections \ref{section-lemma5}, \ref{subsection-proof-lemma6} and \ref{section-lemma7}
present the proofs of Lemmas 7, 8 and 9, respectively.
Section \ref{section-proof-39-B20} presents the proof of (39) in the main text.

Section \ref{section-theorem2a} presents proofs of supported Lemmas in the proof of Theorem 2 (a)
as well as  proof of (59) and (60) in the main text.
This section is organized as follows.
Sections \ref{section-proof-lemma1010},  \ref{section-proof-lemma11} and \ref{section-proof-lemma12}
present the proofs of Lemmas 10, 11 and 12, respectively.
Sections \ref{subsection:B2B20} and \ref{subsection-B3B30} presents the proofs of orders of two remainder terms
$B_2-B_2^0$  in (59) and $B_3-B_3^0$ in (60) in the main text, respectively.

Section \ref{section-lemma1} presents the proof of Lemma 1.
Section \ref{section-bernstein} reproduces Bernstein's inequality and a Martingale central limit theorem for easy readability.

All notation is as defined in the main text unless explicitly noted otherwise. Equation and lemma numbering continues
in sequence with those established in the main text.

We first recall useful inequalities on the derivatives of $\mu(x)$, which will be used in the proofs repeatedly.
Recall that
\[
\mu(x) = \frac{ e^x }{ 1 + e^x}.
\]
A direct calculation gives that the derivative of $\mu(x)$ up to the third order are
\begin{eqnarray}\label{eq-derivative-mu-various}
\mu^\prime(x) = \frac{e^x}{ (1+e^x)^2 },~~  \mu^{\prime\prime}(x) = \frac{e^x(1-e^x)}{ (1+e^x)^3 },~~ \mu^{\prime\prime\prime}(x) =  \frac{ e^x [ (1-e^x)^2 - 2e^x] }{ (1 + e^x)^4 }.
\end{eqnarray}
Note that $\bs{\beta}=(\beta_1, \ldots, \beta_n)$ denotes the data generating parameter, under which the data are generated.
Recall that
\[
\pi_{ij} = \beta_i + \beta_j, ~~\widehat{\pi}_{ij}=\widehat{\beta}_i + \widehat{\beta}_j,~~
\widehat{\pi}_{ij}^0 = \widehat{\beta}_i^0 + \widehat{\beta}_j^0.
\]
According to the definition of $c_n$, we have
\begin{equation}\label{ineq-mu-deriv-bound}
|\mu^\prime(\pi_{ij})| \le \frac{1}{c_n}, ~~
|\mu^{\prime\prime}(\pi_{ij} )| \le \frac{1}{c_n},~~ |\mu^{\prime\prime\prime}(\pi_{ij})| \le \frac{1}{c_n}.
\end{equation}
For a $\bs{\widetilde{\beta} }$ satisfying $\| \bs{\widetilde{\beta} } - \bs{\beta} \|_\infty =o(1)$, we also have
\begin{equation}\label{ineq-mu-tilde}
|\mu^\prime(\tilde{\pi}_{ij} )| \lesssim \frac{1}{c_n}, ~~ |\mu^{\prime\prime}(\tilde{\pi}_{ij})| \lesssim \frac{1}{c_n},~~ |\mu^{\prime\prime\prime}(\tilde{\pi}_{ij})| \lesssim \frac{1}{c_n}.
\end{equation}
These facts will be used in the proofs repeatedly.
Recall that
\[
\bar{a}_{ij}=a_{ij}-\E(a_{ij})
\]
is the centered random variable of $a_{ij}$ and $\bar{a}_{ii}=0$ for all $i=1, \ldots, n$.
Correspondingly, $\bar{d}_i = d_i - \E(d_i)$ and $\bs{\bar{d}}=(\bar{d}_1, \ldots, \bar{d}_n)^\top$.

\section{Variances of weighted sums for $\{\bar{d}_{i}^{\,2}\}_{i=1}^n$ and $\{ \bar{d}_i^{\,3} \}_{i=1}^n$}
\label{section-variance}

This section presents the expressions of the variances of the weighted quadratic sum $\sum_i f_i \bar{d}_i^{\,2}$
and  the weighted cubic sum $\sum_i f_i \bar{d}_i^{\,3}$, as well as the upper bound of a mixed sum $\sum_{i,j} f_{ij} \bar{d}_i^{\,2} \bar{d}_j$. They are stated in Lemmas \ref{lemma:var:quadra}, \ref{lemma:var:cuibic} and \ref{lemma:var:cuibic2}, respectively.
Recall that $\bar{a}_{ij}=a_{ij}-\E(a_{ij})$ for $i\neq j$ and $\bar{a}_{ii}=0$ for all $i=1, \ldots, n$, and
\[
\bar{d}_i = d_i - \E (d_i) = \sum_j \bar{a}_{ij}.
\]
For a given sequence $\{f_i\}_{i=1}^n$,
the variance of the weighted quadratic sum  $\sum_i f_i\bar{d}_i^{\,2}$ is given below.

\begin{lemma}\label{lemma:var:quadra}
Let $u_{ij}= \mathrm{Cov}(\bar{a}_{ij}^{\,2}, \bar{a}_{ji}^{\,2} )$ and $v_{ii}=\sum_j\mathrm{Var}(\bar{a}_{ij})$.
For a given sequence $\{f_i\}_{i=1}^r$, we have
\begin{equation}\label{eq:expression:variance}
\mathrm{Var}( \sum_{i=1}^r f_i \bar{d}_i^{\,2} )=\sum_{i=1}^r f_i^2 (2v_{ii}^2 + \sum_{j=1,j\neq i}^n u_{ij} )
+ 2\sum_{ 1\le i<j \le r } f_i f_j u_{ij}.
\end{equation}
\end{lemma}

\begin{proof}
The calculation of the variance of $\sum_i f_i \bar{d}_i^{\,2}$ can be divided into two parts:
\begin{equation}\label{eq:variance:weight}
\mathrm{Var}( \sum_{i=1}^r f_i \bar{d}_i^{\,2} )=
\underbrace{\sum_{i=1}^r f_i^2 \mathrm{Var}( \bar{d}_i^{\,2} )}_{\mbox{part 1}} + \underbrace{2 \sum_{1\le i<j\le r} f_i f_j \mathrm{Cov}( \bar{d}_i^{\,2}, \bar{d}_j^{\,2} )}_{\mbox{part 2}}.
\end{equation}
The first part can be calculated as follows:
\begin{eqnarray*}
\mathrm{Var}( \bar{d}_i^{\,2}) = \mathrm{Cov} ( (\sum_{\alpha=1 }^n \bar{a}_{i\alpha})^{\,2}, (\sum_{h=1 }^n \bar{a}_{ih} )^{\,2} )
= \mathrm{Cov}( \sum_{\alpha=1}^n \sum_{\beta=1}^n \bar{a}_{i\alpha} \bar{a}_{i\beta}, \sum_{h=1}^n \sum_{g=1}^n \bar{a}_{ih} \bar{a}_{ig}).
\end{eqnarray*}
Note that the random variables $\bar{a}_{ij}$ for $1\le i < j \le n$ are mutually independent.
There are only two cases in terms of $(\alpha, \beta, h, g)$ for which
$\mathrm{Cov}( \bar{a}_{i\alpha} \bar{a}_{i\beta}, \bar{a}_{ih} \bar{a}_{ig} )$ is not equal to zero:
(Case A) $\alpha=\beta=h=g \neq i$;
(Case B) $\alpha=h, \beta=g$ or $\alpha=g, \beta=h$.
By respectively considering Case A and Case B,  a direct calculation gives that
\begin{equation}\label{eq:variance:2}
\mathrm{Var}( \bar{d}_i^{\,2} )= 2 v_{ii}^2 + \sum_{j=1,j\neq i}^n u_{ij}.
\end{equation}

The second part of \eqref{eq:variance:weight} can be calculated as follows:
\begin{equation*}
\mathrm{Cov}( \bar{d}_i^{\,2}, \bar{d}_j^{\,2} ) = \mathrm{Cov}( (\sum_{\alpha=1}^n \bar{a}_{i\alpha})^{\,2}, ( \sum_{h=1}^n \bar{a}_{jh})^{\,2} )
=\mathrm{Cov}( \sum_{\alpha=1}^n \sum_{\beta=1}^n \bar{a}_{i\alpha} \bar{a}_{i\beta}, \sum_{h=1}^n \sum_{g=1}^n \bar{a}_{jh} \bar{a}_{jg}).
\end{equation*}
In the above, the only case for $\mathrm{Cov}( \bar{a}_{i\alpha} \bar{a}_{i\beta}, \bar{a}_{jh} \bar{a}_{jg})$ not being equal to $0$,
is $\alpha=\beta=j$ and $h=g=i$.
According to the definition of $u_{ij}$, we have
\begin{equation}\label{eq:vairance:3}
\mathrm{Cov}( \bar{d}_i^{\,2}, \bar{d}_j^{\,2} ) = \E(\bar{a}_{ij}^{\,2} \bar{a}_{ji}^{\,2} )- \E(\bar{a}_{ij}^{\,2})\E(\bar{a}_{ji}^{\,2})=u_{ij}.
\end{equation}
By combing \eqref{eq:variance:2}, \eqref{eq:vairance:3} and \eqref{eq:variance:weight}, it yields equation \eqref{eq:expression:variance}.
\end{proof}

Now, we present the variance of the cubic weighted sum.

\begin{lemma}\label{lemma:var:cuibic}
For a given sequence $\{f_i\}_{i=1}^r$,
the variance of $\sum_{i=1}^r f_i \bar{d}_i^{\,3}$ has the following expression:
\begin{eqnarray} 
\nonumber
\mathrm{Var}( \sum_{i=1}^r f_i \bar{d}_i^{\,3} )  =  \sum_{i=1}^r f_i^2 \left\{ \sum_{t=1}^n (\E \bar{a}_{i t}^6 - (\E \bar{a}_{i t}^3)^2) +3 \sum_{1\le h\neq g \le n} \E \bar{a}_{i h}^4 \E \bar{a}_{ig}^2 + \right.~~~~~~~~~~~~~~~~~~ \\
\label{eq:expression:variance3}
~~~~~~~~~~~~~~~~~ \left.6 \sum_{1\le h\neq g \neq t \le n } \E \bar{a}_{i g}^2 \E \bar{a}_{ih}^2 \E \bar{a}_{it}^2 \right\} + 2\sum_{ 1\le i<j \le r } f_i f_j \{ \E(\bar{a}_{ij}^{\,3} \bar{a}_{ji}^{\,3} )- \E(\bar{a}_{ij}^{\,3})\E(\bar{a}_{ji}^{\,3})\} \\
\nonumber
~~~~~~~~~~~~~~~~~ +8\sum_{ 1\le i<j \le r } \left(f_{ij}\E \bar{a}_{ij}\bar{a}_{ji} \right) \left( \sum_{g=1,g\neq i,j}^n \E \bar{a}_{i g}^{\,2} \right) \left( \sum_{h=1,h\neq i, j}^n \E \bar{a}_{j h}^{\,2}\right).~~~~~~~~~~~~~~~~~~~
\end{eqnarray}
\end{lemma}

\begin{proof}
Similar to the proof of Lemma \ref{lemma:var:quadra},
the calculation of the variance of $\sum_i f_i \bar{d}_i^{\,3}$ can also be divided into two parts:
\begin{equation}\label{eq:variance:weight3}
\mathrm{Var}( \sum_i f_i \bar{d}_i^{\,3} )=
\underbrace{\sum_{i=1}^n f_i^2 \mathrm{Var}( \bar{d}_i^{\,3} )}_{\mbox{part 1}} + \underbrace{2 \sum_{1\le i<j\le n} f_i c_j \mathrm{Cov}( \bar{d}_i^{\,3}, \bar{d}_j^{\,3} )}_{\mbox{part 2}}.
\end{equation}
The first part can be expressed as
\begin{equation}\label{eq-cal3-a}
\mathrm{Var}( \bar{d}_i^{\,3})
= \sum_{\alpha=1}^n \sum_{\beta=1}^n \sum_{\gamma=1}^n \sum_{h=1}^n \sum_{g=1}^n \sum_{t=1}^n \mathrm{Cov}(  \bar{a}_{i\alpha} \bar{a}_{i\beta}\bar{a}_{i\gamma},  \bar{a}_{ih} \bar{a}_{ig} \bar{a}_{it} ).
\end{equation}
Note that the random variables $\bar{a}_{ij}$ for $1\le i < j \le n$ are mutually independent
and $\bar{a}_{ii}=0$ when $i=j$.
The first part can be calculated as follows.
There are six cases to consider according to the number of distinct values of six indices: $\alpha, \beta, \gamma, h, g, t$.\\
(Case A) All six indices, $\alpha, \beta, \gamma, h, g, t$, are equal. In this case, the summation in \eqref{eq-cal3-a} becomes
\[
\sum_{t=1}^n (\E \bar{a}_{i t}^6 - (\E \bar{a}_{i t}^3)^2).
\]
(Case B) All six indices, $\alpha, \beta, \gamma, h, g, t$, have exactly two distinct values. By considering all possible pairs, e.g.,
$(\bar{a}_{i g}^3, \bar{a}_{ig}\bar{a}_{ih}^2), (\bar{a}_{ig}^3, \bar{a}_{ig}^2 \bar{a}_{ih}), (\bar{a}_{ig}^3, \bar{a}_{ih}^3), \ldots $,  pairs where the covariance is not zero are
those like $(\bar{a}_{i g}^3, \bar{a}_{ig}\bar{a}_{ih}^2)$ for $g\neq h$. In this case, the summation in \eqref{eq-cal3-a} becomes
\[
3\sum_{g=1}^n \sum_{h=1,h\neq g}^n \E \bar{a}_{i g}^4 \E \bar{a}_{ih}^2.
\]
(Case C) All six indices, $\alpha, \beta, \gamma, h, g, t$, have exactly three distinct values. By considering all possible pairs, e.g.,
$(\bar{a}_{i g}^3, \bar{a}_{ig}\bar{a}_{ih}\bar{a}_{it}), (\bar{a}_{ig}^2\bar{a}_{ih}, \bar{a}_{ig}^2 \bar{a}_{it}), (\bar{a}_{ig}^2\bar{a}_{ih}, \bar{a}_{ih}^2\bar{a}_{it}), \ldots $,  pairs where the covariance is not zero are
those like $(\bar{a}_{i g}\bar{a}_{ih}\bar{a}_{it}, \bar{a}_{ig}\bar{a}_{ih}\bar{a}_{it})$ for distinct $g, h, t$. In this case, the summation in \eqref{eq-cal3-a} becomes
\[
6 \sum_{1\le h\neq g \neq t \le n } \E \bar{a}_{i g}^2 \E \bar{a}_{ih}^2 \E \bar{a}_{it}^2.
\]
(Case D) All six indices, $\alpha, \beta, \gamma, h, g, t$, have exactly four, five, or six distinct values. In all these cases,
$\E\bar{a}_{i\alpha} \bar{a}_{i\beta}\bar{a}_{i\gamma}\bar{a}_{ih} \bar{a}_{ig} \bar{a}_{it}$ and $\E\bar{a}_{i\alpha} \bar{a}_{i\beta}\bar{a}_{i\gamma}\E \bar{a}_{ih} \bar{a}_{ig} \bar{a}_{it}$
are equal zero because at least such one $\bar{a}_{i\beta}$ is independent of others.
By combining the above cases, it yields,
\begin{equation}\label{eq:variance:22}
\mathrm{Var}( \bar{d}_i^{\,3} )= \sum_{t=1}^n (\E \bar{a}_{i t}^6 - (\E \bar{a}_{i t}^3)^2) +3 \sum_{1\le h\neq g \le n} \E \bar{a}_{i h}^4 \E \bar{a}_{ig}^2 + 6 \sum_{1\le h\neq g \neq t \le n } \E \bar{a}_{i g}^2 \E \bar{a}_{ih}^2 \E \bar{a}_{it}^2.
\end{equation}

With the similar arguments as in the calculation of part $1$, part $2$ in \eqref{eq:variance:weight3} has the following expression:
\begin{equation}\label{eq:vairance:33}
\mathrm{Cov}( \bar{d}_i^{\,3}, \bar{d}_j^{\,3} ) = \E(\bar{a}_{ij}^{\,3} \bar{a}_{ji}^{\,3} )- \E(\bar{a}_{ij}^{\,3})\E(\bar{a}_{ji}^{\,3})
+ 4\E \bar{a}_{ij}\bar{a}_{ji} \left( \sum_{g\neq i,j} \E \bar{a}_{i g}^{\,2} \right) \left( \sum_{h\neq i, j} \E \bar{a}_{j h}^{\,2}\right).
\end{equation}
By combing \eqref{eq:variance:22}, \eqref{eq:vairance:33} and \eqref{eq:variance:weight}, it yields equation \eqref{eq:expression:variance3}.
\end{proof}

Now, we present an upper bound of the variance of a mixed weighted sum.

\begin{lemma}\label{lemma:var:cuibic2}
For a fixed array $\{ f_{ij} \}_{i,j=1}^n$,
an upper bound of the variance of $\sum_{i\neq j} f_{ij} \bar{d}_i^{\,2}\bar{d}_j$ is below:
\begin{eqnarray*}
\mathrm{Var}( \sum_{i\neq j} f_{ij}\bar{d}_i^{\,2}\bar{d}_j )  \lesssim \frac{n^6}{c_n^3} \max_{i,j} |f_{ij}|^2.
\end{eqnarray*}
\end{lemma}

\begin{proof}
Note that
\[
\mathrm{Var}( \sum_{i\neq j} f_{ij}\bar{d}_i^{\,2}\bar{d}_j ) \le \max_{i,j} |f_{ij}|^2 \sum_{i\neq j} \sum_{\alpha\neq \gamma}
\mathrm{Cov}( \bar{d}_i^{\,2} \bar{d}_j, \bar{d}_{\alpha}^{\,2} \bar{d}_\gamma ).
\]
For $i\neq j$ and $\alpha\neq \gamma$,
the calculation of the covariance between $\bar{d}_i^{\,2}\bar{d}_j$ and $\bar{d}_{\alpha}^{\,2} \bar{d}_{\gamma}$ can also be divided into eight cases:
(Case 1) $i=\alpha, j=\gamma$; (Case 2) $i=\alpha, j\neq \gamma$; (Case 3) $i\neq \alpha, j=\gamma$; (Case 4) $i\neq \alpha, j\neq \gamma$;
(Case 5) $i=\gamma, j=\alpha$; (Case 6) $i=\gamma, j\neq \alpha$; (Case 7) $i\neq \gamma$, $j= \alpha$; (Case 8) $i\neq \gamma$, $j\neq \alpha$.
By writing the covariance between $\bar{d}_i^{\,2}\bar{d}_j$ and $\bar{d}_{\alpha}^{\,2} \bar{d}_{\gamma}$ into the following form
\begin{equation*}\label{eq-cal3-aa}
\mathrm{Cov}( \bar{d}_i^{\,2}\bar{d}_j, \bar{d}_{\alpha}^{\,2} \bar{d}_\gamma)
= \sum_{k=1}^n \sum_{s=1}^n \sum_{t=1}^n \sum_{\zeta=1}^n \sum_{\eta=1}^n \sum_{\xi=1}^n \left( \E  \bar{a}_{ik} \bar{a}_{is}\bar{a}_{jt}  \bar{a}_{\alpha \zeta} \bar{a}_{\alpha \eta} \bar{a}_{\gamma\xi}
-\E  \bar{a}_{ik} \bar{a}_{is}\bar{a}_{jt} \E \bar{a}_{\alpha \zeta} \bar{a}_{\alpha \eta} \bar{a}_{\gamma\xi} \right),
\end{equation*}
and using the similar arguments as in the proof of Lemma \ref{lemma:var:cuibic}, we have
\begin{eqnarray}
|\mathrm{Cov}( \bar{d}_i^{\,2}\bar{d}_j, \bar{d}_{\alpha}^{\,2} \bar{d}_\gamma)|
\begin{cases}
 = \E \bar{a}_{ij}^6 - (\E \bar{a}_{ij}^3)^2, &  \mbox{Case 1}, \\
= 0, & \mbox{Case 2}, \\
= \sum_k \sum_{\zeta} \sum_{\eta} \E \bar{a}_{ik}^2 \E \bar{a}_{j\zeta}^2 \E \bar{a}_{\alpha \eta}^2,  & \mbox{Case 3}, \\
= \sum_k \sum_{\zeta} \E \bar{a}_{ik}^2 \E \bar{a}_{\alpha\zeta}^2 \E \bar{a}_{j\beta}\bar{a}_{\beta j}, & \mbox{Case 4}, \\
= \sum_k \sum_{t} \E \bar{a}_{ik}^3 \E \bar{a}_{jt}^3, & \mbox{Case 5}, \\
= \sum_k \sum_{t} \E \bar{a}_{ik}^2 \E\bar{a}_{\alpha t}^2 \E \bar{a}_{ji}\bar{a}_{ij} + \sum_t \E \bar{a}_{it}^3 \E\bar{a}_{j\alpha}\bar{a}_{\alpha j}, & \mbox{Case 6}, \\
= \sum_{t} \E \bar{a}_{i\gamma}^2\bar{a}_{\gamma i} \E \bar{a}_{jt}^3,  & \mbox{Case 7}, \\
= \left(\sum_{k\neq i, j} \E \bar{a}_{ik}^{\,2} \right) \left(\sum_{\zeta \neq \alpha, \gamma} \E \bar{a}_{\alpha \zeta}^{\,2} \right)
\E \bar{a}_{j\gamma} \bar{a}_{\gamma j} & \mbox{Case 8}.
\end{cases}
\end{eqnarray}
Let
\[
p_{ij} = \frac{e^{\beta_i+\beta_j}}{ ( 1 + e^{\beta_i+\beta_j}) },~~q_{ij}=1-p_{ij}
\]
In view of that
\begin{eqnarray*}
\E \bar{a}_{ij}^{\,6} - (\E \bar{a}_{ij}^{\,3})^2 & = &  p_{ij}q_{ji}( p_{ij}^5+q_{ji}^5 - p_{ij}q_{ji}(p_{ij}^2-q_{ji}^2)^2), \\
\E \bar{a}_{ij}^{\,3} & = & q_{ij}^3p_{ij} - p_{ij}^3 q_{ij}, \\
\max_{i,j} p_{ij}q_{ij} & = & \max_{i,j} \frac{e^{\beta_i+\beta_j}}{ ( 1 + e^{\beta_i+\beta_j})^2 }  \le \frac{1}{c_n},
\end{eqnarray*}
by combining the above cases, it completes the proof.
\end{proof}

\section{Proofs of supported lemmas in the proof of Theorem 1 (a)} 
\label{section:beta-th1a}

We first reproduce some basic results here in the main text.
Recall that an $n\times n$ matrix $V=(v_{ij})$ belongs to the matrix class
$\mathcal{L}_n(m, M)$ if
\[
\begin{array}{cl}
v_{ii}=\sum_{j\neq i} v_{ij}, & i=1, \ldots, n \\
m \le v_{ij} \le M, & i,j=1, \ldots, n; i\neq j.
\end{array}
\]
We use the diagonal matrix
\begin{equation}\label{definition-S}
S=\mathrm{diag}(1/v_{11}, \ldots, 1/v_{nn}),
\end{equation}
to approximate $V^{-1}$. For $V\in \mathcal{L}_n(1/b_n, 1/c_n)$, \cite{Yan:Xu:2013} proved
\begin{equation}\label{ineq-V-S-appro-upper-b}
\|W:=V^{-1} - S \|_{\max} \lesssim \frac{b_n^3}{n^2c_n}.
\end{equation}
Further, for its bottom right $(n-r)\times (n-r)$ block $V_{22}$ of $V\in \mathcal{L}_n(1/b_n, 1/c_n)$ and $n\ge 3$, we have
\begin{equation}\label{ineq-V22-S22-app}
\| \widetilde{W}_{22}:= V_{22}^{-1} - S_{22} \|_{\max} \le \frac{ b_n^2 }{ (n-1)^2 c_n } \left( 1 + \frac{ nb_n }{ (n-2)c_n } \right)\lesssim \frac{ b_n^3}{ n^2 c_n^2 }, ~~r=0, \ldots, n-1,
\end{equation}
where
\begin{equation}\label{definition-S22}
S_{22}=\mathrm{diag}(1/v_{r+1, r+1}, \ldots, 1/v_{nn}).
\end{equation}
From \eqref{ineq-V22-S22-app}, we can see that the error bound by using $S_{22}$ to approximate $V_{22}^{-1}$ is independent of $r$ and depends only on $b_n$, $c_n$ and n.
Moreover, by Theorem 6.1 of \cite{hillar2012inverses},
we have that for $V\in \mathcal{L}_n(1/b_n, 1/c_n)$ and its bottom right $(n-r)\times (n-r)$ block $V_{22}$,
\begin{equation}\label{ineq-tight-V}
\frac{c_n}{2(n-1)} \le \|V^{-1}\|_\infty \le \frac{3b_n}{2n-1}, \quad \|V_{22}^{-1} \|_\infty \le \frac{ 3b_n}{ 2n-1 }.
\end{equation}

Recall that $V = -\partial^2 \ell( \bs{\beta} )/\partial \bs{\beta} \partial \bs{\beta}^\top$, where row $i$ column $j$ element $v_{ij}$ of $V$ is
\begin{equation*}
v_{ii} = \sum\nolimits_{j\neq i} v_{ij}, ~~ v_{ij} = \frac{e^{\beta_i + \beta_j}}{(1 + e^{\beta_i + \beta_j})^2}=\mu^\prime(\pi_{ij}), ~~i \neq j; i,j=1,\ldots, n,
\end{equation*}
which is also the covariance matrix of $\mathbf{d}$.

This section is organized as follows.
Sections \ref{subsection-proof-lemma2} and \ref{subsection-proof-lemma3} presents
the proofs of Lemmas 3 and 4, respectively.
Section \ref{subsection-L2norm} contains an additional result about an $L_2$-norm error bound for $\bs{\widehat{\beta}}$. 
Section \ref{subsec-asy-widehatbeta} presents an asymptotically explicit expression for $\bs{\widehat{\beta}}$ that is used in the proof of Lemma 4. 
Section \ref{subsec-prooflemma4} presents the proof of Lemma 5.

\subsection{Proof of Lemma 3}
\label{subsection-proof-lemma2}

\begin{proof}[Proof of Lemma 3]
Recall that  $\bs{d}_2=(d_{r+1},\ldots, d_n)^\top$, $\bs{\bar{d}}_2=( \bar{d}_{r+1}, \ldots, \bar{d}_n)^\top$, and
$\widetilde{W}_{22}=V_{22}^{-1} - S_{22}$. Note that when $r=0$, $\bs{\bar{d}}_2=\bs{\bar{d}}$ and $\widetilde{W}_{22}=W$.
The aim is to prove
\begin{equation}\label{Wd-op1}
 \bs{\bar{d}}_2^\top \widetilde{W}_{22} \bs{\bar{d}}_2 = O_p\left( \frac{b_n^3( 1- r/n)^{3/2} }{ c_n^3 } \right).
\end{equation}
We first have
\begin{equation}\label{wd-expectation}
\E[  \bs{\bar{d}}_2^\top \widetilde{W}_{22}  \bs{\bar{d}}_2 ] =0,
\end{equation}
which is due to that
\[
\E[  \mathbf{\bar{d}_2}^\top \widetilde{W}_{22}  \mathbf{\bar{d}_2} ]
= \mathrm{tr} (\E[  \mathbf{\bar{d}_2}^\top \mathbf{\bar{d}_2}  ] \widetilde{W}_{22} )
= \mathrm{tr} (V_{22}\widetilde{W}_{22}) = \mathrm{tr} ( I_{n-r} - V_{22}S_{22} ) = 0.
\]

Let $\widetilde{W}_{22}=(\tilde{w}_{ij})_{(n-r)\times (n-r) }$.
Next, we bound the variance of $\sum_{i,j=r+1}^n \bar{d}_i \tilde{w}_{(i-r)(j-r)} \bar{d}_j$.
Recall that $v_{ij}=\mathrm{Var}(\bar{a}_{ij})=\mu^\prime(\beta_i + \beta_j)$.
There are four cases for calculating the covariance
\[
g_{ij\zeta\eta}=\mathrm{Cov}\big( \bar{d}_i  \tilde{w}_{(i-r)(j-r)} \bar{d}_j, \bar{d}_\zeta w_{(\zeta-r)(\eta-r)} \bar{d}_\eta ).
\]
Case 1: $i=j=\zeta=\eta$. In view of \eqref{eq:variance:2}, we have
\begin{equation}\label{equ:varii}
\mathrm{var}( \bar{d}_i^{\,2} ) = \sum_{j}\mathrm{Var}( \bar{a}_{ij}^4 ) + 2 v_{ii}^2.
\end{equation}
Let $p_{ij}=\mu(\beta_i+\beta_j)$ and $q_{ij}=1-p_{ij}$. By \eqref{ineq-mu-deriv-bound},
\[
\max_{i,j} p_{ij}(1-p_{ij}) \le \frac{1}{c_n}.
\]
Note that the function $x^4 + (1-x)^4$ with $x\in[0,1]$ attains its maximum value at points $0$ or $1$.
It follows that
\[
\mathrm{Var}( \bar{a}_{ij}^{\,4} ) = p_{ij}q_{ij}\{ p_{ij}^7 + q_{ij}^7 - p_{ij}q_{ij}(p_{ij}^3+q_{ij}^3)^2 \}
= p_{ij}q_{ij}( p_{ij}^8 + q_{ij}^8 - 2p_{ij}^4q_{ij}^4 )
\le \frac{1}{c_n}.
\]
Thus, we have
\begin{equation*}
|g_{iiii}|\le \tilde{w}_{(i-r)(i-r)}^2\cdot \left( \frac{2(n-1)^2}{c_n^2} + \frac{n-1}{c_n} \right).
\end{equation*}
Case 2: Three indices among the four indices are the same. Without loss of generality, we assume that
$j=\zeta=\eta$ and $i\neq j$.
Observe that
\[
\mathrm{Cov}( \bar{d}_i \bar{d}_j, \bar{d}_j^{\,2}) = \sum_{k,h,\alpha,\gamma} ( \E \bar{a}_{ik}\bar{a}_{jh} \bar{a}_{j\alpha}
\bar{a}_{j\gamma} - \E \bar{a}_{ik}\bar{a}_{jh} \E \bar{a}_{j\alpha} \bar{a}_{j\gamma} )
\]
and, for distinct $k,h,\alpha,\gamma$,
\begin{eqnarray*}
\E \bar{a}_{ik}\bar{a}_{jh} \bar{a}_{j\alpha}
\bar{a}_{j\gamma} & = & 0, \\
\E \bar{a}_{ij}\bar{a}_{ji} \bar{a}_{j\alpha}
\bar{a}_{j\gamma} & = & 0, \\
\E \bar{a}_{ih}\bar{a}_{jh} \bar{a}_{j\alpha}
\bar{a}_{j\gamma} & = & 0.
\end{eqnarray*}
It follows that
\[
\mathrm{Cov}( \bar{d}_i \bar{d}_j, \bar{d}_j^{\,2}) = \E \bar{a}_{ij}\bar{a}_{ji}^3 - \E \bar{a}_{ij}\bar{a}_{ji}\E \bar{a}_{ji}^2
+ 2\sum_{h\neq i,j} \E \bar{a}_{ij}\bar{a}_{ji} \E \bar{a}_{jh}^{\,2}.
\]
Therefore,  by \eqref{ineq-mu-deriv-bound},
\begin{eqnarray*}
|g_{ijjj}|&\le & |\tilde{w}_{(i-r)(j-r)}\tilde{w}_{(j-r)(j-r)}|\cdot \frac{n}{c_n^2}.
\end{eqnarray*}
Similarly, we have the upper bounds in other cases.\\
Case 3. Two indices among the four are the same (e.g. $i=j$ or $j=\zeta$):
\begin{eqnarray*}
|g_{ii\eta\zeta}|&=&|\tilde{w}_{(i-r)(i-r)}\tilde{w}_{(\zeta-r)(\eta-r)}(2v_{i\zeta}v_{i\eta}+v_{ii}v_{\zeta\eta})|
\le |\tilde{w}_{(i-r)(i-r)}\tilde{w}_{(\zeta-r)(\eta-r)}|\cdot \frac{n}{c_n^2};\\
|g_{ijj\eta}|&=&|\tilde{w}_{(i-r)(i-r)}\tilde{w}_{(j-r)(\eta-r)}(2v_{ji}v_{j\eta}+v_{ij}v_{j\eta})|
\le 3|\tilde{w}_{(i-r)(i-r)}\tilde{w}_{(j-r)(\eta-r)}|\cdot \frac{1}{c_n^2}.
\end{eqnarray*}
Case 4: All four indices are different
\begin{eqnarray*}
|g_{ij\zeta\eta}|&=& |\tilde{w}_{(i-r)(j-r)}\tilde{w}_{(\zeta-r)(\eta-r)}(v_{i\zeta}v_{j\eta}+v_{i\eta}v_{j\zeta})|\le 2|\tilde{w}_{(i-r)(j-r)}\tilde{w}_{(\zeta-r)(\eta-r)}|\frac{1}{c_n^2}.
\end{eqnarray*}
Consequently, by \eqref{ineq-V-S-appro-upper-b}, we have
\begin{eqnarray*}
&&\mathrm{Var} ( \bs{\bar{d}}_2^{\top} \widetilde{W}_{22} \bs{\bar{d}}_2 )  \\
& = & \sum_{i,j,\zeta,\eta={r+1}}^n \mathrm{Cov}\big( \bar{d}_i  \tilde{w}_{(i-r)(j-r)} \bar{d}_j,
\bar{d}_\zeta \tilde{w}_{(\zeta-r)(\eta-r)} \bar{d}_\eta ) \\
& \lesssim &  \left( \frac{b_n^3}{n^2 c_n^2 } \right)^2 \times \left( (n-r)\cdot \frac{ n^2 }{c_n^2 }
+ (n-r)^2 \cdot \frac{n}{c_n^2} + (n-r)^3 \cdot \frac{n}{c_n^2} + (n-r)^4 \cdot \frac{1}{c_n^2 } \right) \\
& \lesssim & \frac{b_n^6(1-r/n)^3}{c_n^6}.
\end{eqnarray*}
It follows that from Chebyshev's inequality and \eqref{wd-expectation}, we have
\begin{align*}
 & \P( \left( |\bs{\bar{d}}_2^\top \widetilde{W}_{22}  \bs{\bar{d}}_2  |
\ge  \rho_n\frac{ b_n^3(1-r/n)^{3/2} }{c_n^6} \right) \\
\le & \frac{ c_n^6 }{ b_n^3(1-r/n)^3\rho_n^2 }\times \mathrm{Var} ( \mathbf{\bar{d}_2}^{\top} W_{22} \mathbf{\bar{d}_2} ) \\
\lesssim & \frac{1}{\rho_n^2} \to 0,
\end{align*}
where $\{\rho_n\}_{n=1}^\infty$ is any positive sequence tending to infinity.
This completes the proof.
\end{proof}

\subsection{Proof of Lemma 4}
\label{subsection-proof-lemma3}

Before beginning the proof of Lemma 4, we introduce one useful lemma.
Let $F(\mathbf{x}): \R^n \to \R^n$ be a function vector on $\mathbf{x}\in\R^n$.
We say that a Jacobian matrix $F^\prime(\mathbf{x})$ with $\mathbf{x}\in \R^n$ is Lipschitz continuous on a convex set $D\subset\R^n$ if
for any $\mathbf{x},\mathbf{y}\in D$, there exists a constant $\lambda>0$ such that
for any vector $\mathbf{v}\in \R^n$ the inequality
\begin{equation*}
\| [F^\prime (\mathbf{x})] \mathbf{v} - [F^\prime (\mathbf{y})] \mathbf{v} \|_\infty \le \lambda \| \mathbf{x} - \mathbf{y} \|_\infty \|\mathbf{v}\|_\infty
\end{equation*}
holds.
We will use the Newton iterative sequence to establish the existence and consistency of the MLE.
\cite{Gragg:Tapia:1974} gave the optimal error bound for the Newton method under the Kantovorich conditions
[\cite{Kantorovich1948Functional}]. We only show partial results here that are enough for our applications.

\begin{lemma}[\cite{Gragg:Tapia:1974}]\label{lemma:Newton:Kantovorich}
Let $D$ be an open convex set of $\R^n$ and $F:D \to \R^n$ be Fr\'{e}chet differentiable on $D$
with a Jacobian $F^\prime(\mathbf{x})$ that is Lipschitz continuous on $D$ with Lipschitz coefficient $\lambda$.
Assume that $\mathbf{x}_0 \in D$ is such that $[ F^\prime (\mathbf{x}_0) ]^{-1} $ exists,
\[
\| [ F^\prime (\mathbf{x}_0 ) ]^{-1} \|  \le \aleph,~~ \| [ F^\prime (\mathbf{x}_0) ]^{-1} F(\mathbf{x}_0) \| \le \delta, ~~ h= 2 \aleph \lambda \delta \le 1,
\]
and
\[
B(\mathbf{x}_0, t^*) \subset D, ~~ t^* = \frac{2}{h} ( 1 - \sqrt{1-h} ) \delta = \frac{ 2\delta }{ 1 + \sqrt{1-h} }.
\]
Then: (1) The Newton iterations $\mathbf{x}_{k+1} = \mathbf{x}_k - [ F^\prime (x_\mathbf{k}) ]^{-1} F(\mathbf{x}_k)$ exist and $\mathbf{x}_k \in B(\mathbf{x}_0, t^*) \subset D$ for $k \ge 0$. (2)
$\mathbf{x}^* = \lim \mathbf{x}_k$ exists, $\mathbf{x}^* \in \overline{ B(\mathbf{x}_0, t^*) } \subset D$ and $F(\mathbf{x}^*)=0$.
\end{lemma}

\begin{proof}[Proof of Lemma 4]
Under the null $H_0: (\beta_1, \ldots, \beta_r)^\top=(\beta_1^0, \ldots, \beta_r^0)^\top$, $\beta_1, \ldots, \beta_r$ are known and $\beta_{r+1}, \ldots, \beta_n$ are unknown.
Recall that $\bs{\widehat{\beta}}^0$ denotes the restricted MLE under the null space,
where $\widehat{\beta}_i = \beta_i^0$, $i=1,\ldots,r$.
For convenience, we will use $\bs{\beta}$ and $\bs{\widehat{\beta}}^0$ to denote
the vectors $(\beta_{r+1}, \ldots, \beta_n)^\top$ and $(\widehat{\beta}_{r+1}^0, \ldots, \widehat{\beta}_n^0)^\top$ in this proof, respectively.
Note that $\bs{\widehat{\beta}}^0=\bs{\widehat{\beta}}$ when $r=0$.

Define a system of score functions based on likelihood equations:
\begin{equation}\label{eqn:def:F}
 F_i(\bs{\beta})=  \sum\limits_{j=1, j\neq i}^n \mu(\beta_i + \beta_j) - d_i ,~~i=1, \ldots, n,
\end{equation}
and $F(\bs{\beta})=(F_{r+1}(\bs{\beta}), \ldots, F_n(\bs{\beta}))^\top$.

Let $B(\bs{\beta}, 1/(2b_n))=\{\bs{\gamma}=(\gamma_{r+1}, \ldots, \gamma_n)\in \R^{n-r}: \| \bs{\beta} - \bs{\gamma}\|_\infty \le 1/(2b_n) \}$ be a convex set containing $\bs{\beta}$.
We will derive the error bound between $\bs{\widehat{\beta}}^0$ and $\bs{\beta}$ through
obtaining the convergence rate of the Newton iterative sequence $\bs{\beta}^{(n+1)}= \bs{\beta}^{(n)} - [F^\prime (\bs{\beta}^{(n)})]^{-1}
F (\bs{\beta}^{(n)})$,
where we choose the true parameter $\bs{\beta}$ as the starting point $\bs{\beta}^{(0)}:=\bs{\beta}$.
To this end, it is sufficient to demonstrate the Kantovorich conditions in Lemma \ref{lemma:Newton:Kantovorich}, where we set $D=B(\bs{\beta}, 1/(2c_n))$.
The Kantororich conditions require the Lipschitz continuous of $F^\prime(\bs{\beta})$ and the upper bounds of
$F(\bs{\beta}^*)$.
The proof proceeds three steps. Step 1 is about the Lipschitz continuous property of the Jacobian matrix $F'( \bs{\beta} )$.
Step 2 is about the tail probability of $F(\bs{\beta})$. Step 3 is a combining step.

Step 1. We claim that
the Jacobian matrix $F'( \bs{\gamma} )$ of $F(\bs{\gamma})$ on $\bs{\gamma}$ is Lipschitz continuous on $B(\bs{\beta}, 1/(2c_n))$ with the Lipschitz coefficient  $3(4n-4-3r)/2c_n$.
This is verified as follows.
Let $(\gamma_1, \ldots, \gamma_r)=(\beta_1^0, \ldots, \beta_r^0)$.
The Jacobian matrix $F^\prime(\bs{\gamma})$ of $F(\bs{\gamma})$ can be calculated as follows.
By finding the partial derivative of $F_i$ with respect to $\bs{\gamma}$, for $i\neq j\in \{r+1, \ldots, n\}$, we have
\[
\frac{\partial F_i(\bs{\gamma}) }{ \partial \gamma_j} =  \mu^\prime (\gamma_i+\gamma_j), ~~
\frac{ \partial F_i(\bs{\gamma})}{ \partial \gamma_i} =  \sum_{j\neq i} \mu^\prime (\gamma_i+\gamma_j),
\]
\[
\frac{\partial^2 F_i( \bs{\gamma}) }{ \partial \gamma_i \partial \gamma_j} =  \mu^{\prime\prime} (\gamma_i+\gamma_j),~~
\frac{ \partial^2 F_i(\bs{\gamma})}{\partial \gamma_i^2} =  \sum_{j\neq i} \mu^{\prime\prime} (\gamma_i+\gamma_j),
\]
\begin{equation}\label{eq-le-con-a}
\frac{\partial^2 F_i(\bs{\gamma}) }{ \partial \beta_k \partial \beta_j} = 0,~~k\in \{\ell: \ell \neq i,j; \ell =1, \ldots, n\}.
\end{equation}
By the mean value theorem and \eqref{eq-derivative-mu-various}, we have
\[
|\mu^{\prime\prime} (\beta_i+\beta_j) - \mu^{\prime\prime} (\gamma_i+\gamma_j)|\le \frac{1}{4} \| \bs{\gamma}-\bs{\beta}\|_\infty \le \frac{1}{2c_n}.
\]
For $\bs{\gamma}\in D$, this shows
\[
\max_{i,j} |\mu^{\prime\prime} (\gamma_i+\gamma_j)| \le \frac{3}{2c_n}.
\]
It follows that
\begin{equation}
\label{inequ:second:deri}
|\frac{\partial^2 F_i( \bs{\gamma} ) }{\partial \gamma_i^2} |\le \frac{ 3(n-1)}{2c_n},~~
|\frac{\partial^2 F_i( \bs{\gamma} ) }{\partial \gamma_j\partial \gamma_i}| \le \frac{3}{2c_n},~~i\neq j.
\end{equation}
Let
\[
\mathbf{g}_{ij}(\bs{\gamma})=(\frac{\partial^2 F_i(\bs{\gamma}) }{ \partial \gamma_{r+1} \partial \gamma_j}, \ldots,
\frac{\partial^2 F_i(\bs{\gamma}) }{ \partial \gamma_n \partial \gamma_j})^\top.
\]
It leads to that
$\|\mathbf{g}_{ii}(\bs{\gamma})\|_1 \le 3(2n-r-2)/(2c_n)$. 
Note that when $i\neq j$ and $k\neq i, j$,
\[
\frac{\partial^2 F_i(\bs{\gamma}) }{ \partial \gamma_k \partial \gamma_j} =0.
\]
Therefore, we have
$\|\mathbf{g}_{ij}(\bs{\gamma})\|_1 \le 3/(2c_n)$, for $j\neq i$. Consequently, for vectors $\mathbf{x}, \mathbf{y}, \mathbf{v}\subset D$, we have
\begin{eqnarray*}
& & \| [F^\prime(\mathbf{x})]\mathbf{v} - [F^\prime(\mathbf{y})] \mathbf{v} \|_\infty \\
& \le & \max_{i=r+1, \ldots, n} \{\sum_{j=r+1}^n [ \frac{\partial F_i}{\partial \beta_j }(\mathbf{x}) - \frac{\partial F_i}{\partial \beta_j }(\mathbf{y})] v_j \} \\
& \le & \|\mathbf{v}\|_\infty \max_{i=r+1, \ldots,n} \sum_{j=r+1}^n |\frac{\partial F_i}{\partial \beta_j }
(\mathbf{x}) - \frac{\partial F_i}{\partial \beta_j }(\mathbf{y}) |  \\
& = & \|\mathbf{v}\|_\infty \max_{i=r+1,\ldots,n} \sum_{j=r+1}^n |\int_0^1 [\mathbf{g}_{ij}(t\mathbf{x}+(1-t)\mathbf{y})]^\top (\mathbf{x}-\mathbf{y})dt | \\
& \le & \frac{3(2n-2-r+2(n-r-1))}{2c_n} \|\mathbf{v}\|_\infty\|\mathbf{x}-\mathbf{y}\|_\infty \\
& = & \frac{3(4n-4-3r)}{2c_n} \|\mathbf{v}\|_\infty\|\mathbf{x}-\mathbf{y}\|_\infty,
\end{eqnarray*}
where $t\in(0,1)$ is some real number.

Step 2. We give the tail probability of $\|F(\bs{\beta}) \|_\infty$ satisfying
\begin{equation}\label{ineq-union-d}
\P\Bigg( \max_{i=r+1, \ldots, n} |F_i(\bs{\beta})| \le \sqrt{n\log n} \Bigg)  \ge  1 - \frac{ 2 }{ n }.
\end{equation}
This is verified as follows.
Recall that $a_{ij}$, $1\le i<j \le n$, are independent Bernoulli random variables
and $F_i(\bs{\beta}) = \sum_{j\neq i} (\E a_{ij} - a_{ij})$.
By \citeauthor{Hoeffding:1963}'s \citeyearpar{Hoeffding:1963} inequality, we have
\begin{equation*}
\P\left( |F_i(\bs{\beta}) | \ge \sqrt{n\log n}  \right) \le 2\exp (- 2\frac{n\log n}{n} ) \le  \frac{2}{n^2},~~i=1, \ldots, n.
\end{equation*}
By the union bound, we have
\begin{eqnarray}
\label{ineq-d-upper}
\P\Bigg( \max_{i=1, \ldots, n} |F_i(\bs{\beta})| \ge \sqrt{n\log n} \Bigg)
\le  \sum_{i=1}^n \P\left(|F_i(\bs{\beta})| \geq \sqrt{n\log n} \right)
\le  \frac{2n}{n^2 },
\end{eqnarray}
such that
\[
\P\Bigg( \max_{i=r+1, \ldots, n} |F_i(\bs{\beta})| \le \sqrt{n\log n} \Bigg) \ge \P\Bigg( \max_{i=1, \ldots, n} |F_i(\bs{\beta})| \le \sqrt{n\log n} \Bigg) \ge  1 - \frac{ 2 }{ n }.
\]

Step 3. This step is one combining step.
The following calculations are based on the event $E_n$:
\[
E_n = \{ \max_{i=1,\ldots, n} |F_i(\bs{\beta})| \le   (n\log n)^{1/2}  \}.
\]
Recall that $V_{22}=(v_{ij})=  F^\prime(\bs{\beta})$.
By \eqref{ineq-tight-V}, we have
$\aleph =\|V_{22}^{-1}\|_\infty \le  3b_n/(2n-1)$.
By the event $E_n$, we have
\begin{eqnarray*}
\| F(\bs{\beta}) \|_\infty  \le (n\log n)^{1/2}.
\end{eqnarray*}
Repeatedly utilizing  \eqref{ineq-tight-V}, we have
\begin{eqnarray*}
\delta=\| [F'(\bs{\beta})]^{-1}F(\bs{\beta}) \|_\infty \le \| [F'(\bs{\beta})]^{-1}\|_\infty \|F(\bs{\beta}) \|_\infty
\le   \frac{3nb_n}{2n-1} \sqrt{\frac{\log n}{n}}.
\end{eqnarray*}
In Step 1, we show that $F^\prime(\bs{\beta})$ is Lipschitz continuous with Lipschitz coefficient $\lambda=3(4n-4-3r)/2c_n$.
Note that for any $r\in [0,n-1]$, $4n-4-3r\le n-4$.
Therefore, if $ b_n^2/c_n=o( (n/\log n)^{1/2} )$, then
\begin{eqnarray*}
h =2\aleph \lambda \delta & \le & \frac{3b_n}{2n-1} \times  \frac{3(4n-4-3r)}{2c_n}
\times    \frac{3nb_n}{2n-1} \sqrt{\frac{\log n}{n}}  \\
& = &    \frac{27n(4n-4-3r)b_n^2}{2(2n-1)^2c_n} \sqrt{ \frac{\log n}{n} } =o(1).
\end{eqnarray*}
The above arguments verify the Kantovorich conditions.
By Lemma \ref{lemma:Newton:Kantovorich}, it yields that
\begin{equation}
\label{eq-hatbeta-upper}
\| \bs{\widehat{\beta}}^0 - \bs{\beta} \|_\infty \le   \frac{3nb_n}{2n-1} \sqrt{\frac{\log n}{n}}.
\end{equation}
Step 2 implies $\P (E_n^c) \le 1 - 2/n$. This completes the proof.
\end{proof}

\subsection{The upper bound of $\| {\widehat{\beta}} - {\beta} \|_2$}
\label{subsection-L2norm}

We derive the upper bound for $\bs{\widehat{\beta}}$ in terms of the $L_2$-norm.

\begin{lemma}\label{lemma-L2-con}
If $b_n^2/c_n = o( ( n/\log n)^{1/2} )$, with probability at least $1-2/n$, we have
\[
\| \bs{\widehat{\beta}} - \bs{\beta} \|_2  \lesssim  b_n (\log n)^{1/2}.
\]
\end{lemma}

\begin{proof}
By \eqref{eq-hatbeta-upper}, if $b_n^2/c_n = o( ( n/\log n)^{1/2} )$, with probability at least $1-2/n$, $\bs{\widehat{\beta}}$ exists.
Because $\bs{\widehat{\beta}}$ minimizes $-\ell(\bs{\beta})$, by the mean value theorem, we have
\[
-\ell( \bs{\beta} ) \ge -\ell( \bs{\widehat{\beta}} ) = -\ell( \bs{\beta} )
- \frac{ \partial \ell( \bs{\beta} ) }{ \partial \bs{\beta}^\top } ( \bs{\widehat{\beta}}-\bs{\beta} )
- \frac{1}{2} ( \bs{\widehat{\beta}} - \bs{\beta} )^\top \frac{ \partial \ell( \bs{\tilde{\beta}} ) }{ \partial \bs{\beta} \partial \bs{\beta}^\top }( \bs{\widehat{\beta}}-\bs{\beta} ),
\]
where
\[
\| \bs{\tilde{\beta}}-\bs{\beta}\|_\infty \le \| \bs{\widehat{\beta}} - \bs{\beta} \|_\infty = O\left(  b_n \sqrt{\frac{\log n}{n}} \right).
\]
It follows from the Cauchy-Schwarz inequality that
\[
\frac{1}{2} ( \bs{\widehat{\beta}} - \bs{\beta} )^\top V( \bs{\tilde{\beta}} )( \bs{\widehat{\beta}} - \bs{\beta} ) \le -( \bs{d} - \E \bs{d} )^\top ( \bs{\widehat{\beta}} - \bs{\beta} ) \le \| \bs{d} - \E \bs{d}\|_2 \cdot \| \bs{\widehat{\beta}}- \bs{\beta} \|_2.
\]
This shows that
\[
\| \bs{\widehat{\beta}} - \bs{\beta} \|_2 \le \frac{ 2\| \bs{d} - \E \bs{d}\|_2 }{ \lambda_{\min}(V(\bs{\tilde{\beta}})) },
\]
where $\lambda_{\min}(V(\bs{\tilde{\beta}}))$ denotes the smallest eigenvalue of $V(\bs{\tilde{\beta}})$.
Because for any vector $\bs{x}=(x_1, \ldots, x_n)^\top \in \R^n$,
\[
\bs{x}^\top V \bs{x} = \sum_i v_{ii} x_i^2 + 2\sum_{i<j} x_i v_{ij} x_j = \sum_{i<j} v_{ij} (x_i + x_j)^2,
\]
we have
\[
\lambda_{\min}(V(\bs{\tilde{\beta}})) \ge (n-1) \min_{i,j} \frac{ e^{\tilde{\beta}_i + \tilde{\beta}_j} }{ (1 + e^{\tilde{\beta}_i + \tilde{\beta}_j}) }  \gtrsim  \frac{n}{b_n}.
\]
By \eqref{ineq-union-d}, with probability at least $1-2/n$, we have
\[
\sum_i ( d_i - \E d_i )^2 \le n^2 \log n.
\]
Consequently,
\[
\| \bs{\widehat{\beta}}- \bs{\beta} \|_2  \lesssim  b_n (\log n)^{1/2}.
\]
\end{proof}

\subsection{Asymptotic expression for ${\widehat{\beta}}$}
\label{subsec-asy-widehatbeta}

The following lemma gives an asymptotically explicit expression for $\bs{\widehat{\beta}}^0$ that will be repeatedly used in the proof.

\begin{lemma}\label{lemma-hatbeta-exp}
Suppose that $\beta_1, \ldots, \beta_r$ with $r\in\{0,\ldots,n-1\}$ are known.
If $b_n^2/c_n=o( n/(n-r)\cdot (n/\log n)^{1/2} )$, then with probability at least $1 - 6/n$, the following holds uniformly:
\[
\widehat{\beta}_i^0 - \beta_i = \frac{ \bar{d}_i }{ v_{ii} } + g_i,~~ i=r+1, \ldots, n,
\]
where
\begin{equation}\label{ineq-g}
g_i = O\left( \frac{ b_n^3\log n }{ nc_n } \right).
\end{equation}
\end{lemma}

\begin{proof}[Proof of Lemma \ref{lemma-hatbeta-exp}]
Since $\beta_1, \ldots, \beta_r$ with $r\in\{0,\ldots,n-1\}$ are known, with some ambiguity of notations, here we use $\bs{\widehat{\beta}}^0$ and $\bs{\beta}$
to denote vectors $(\widehat{\beta}_{r+1}^0, \ldots, \widehat{\beta}_n^0)^\top$ and $(\beta_{r+1}, \ldots, \beta_n)^\top$, respectively.
By \eqref{eq-hatbeta-upper}, if $b_n^2/c_n=o( (n/\log n)^{1/2} )$, then $\P(E_n) \ge 1 -2/n$, where
\[
E_n : = \left\{ \| \bs{\widehat{\beta}}^0 - \bs{\beta}  \|_\infty \lesssim  b_n \sqrt{ \frac{\log n}{n} } \right\}.
\]
The following calculations are based on the event $E_n$.

Write $\mu_{ij}( \bs{\beta} )=\mu( \beta_i + \beta_j )$.
Let $F_i( \bs{\beta} ) = \sum_{j\neq i} \mu_{ij}( \bs{\beta} ) - d_i$, $i=1, \ldots, n$ and $F( \bs{\beta} )=(F_{r+1}( \bs{\beta} ), \ldots, F_n( \bs{\beta}))^\top$.
By applying a second order Taylor expansion to $F( \bs{\widehat{\beta}} )$, we have
\begin{equation}
\label{equ-lemma-gamma-b}
F( \bs{\widehat{\beta}}^0 )  = F( \bs{\beta}^0 ) +  \frac{\partial F(\bs{\beta}^0)}{\partial \bs{\beta}^\top } ( \bs{\widehat{\beta}}^0 - \bs{\beta})
+ \frac{1}{2} \left[\sum_{k=r+1}^{n} (\widehat{\beta}_k^0 - \beta_k) \frac{\partial^2F(\bs{\tilde{\beta}})}{\partial \beta_k \partial \bs{\beta}^\top} \right]\times ( \bs{\widehat{\beta}}^0 - \bs{\beta} ),
\end{equation}
where $\bs{\tilde{\beta}}$ lies between $\bs{\widehat{\beta}}$ and $\bs{\beta}$.
We evaluate the last term in the above equation row by row.
Its $k$th row is
\begin{equation}\label{definition-R}
R_k := \frac{1}{2} ( \bs{\widehat{\beta}}^0 - \bs{\beta} )^\top  \frac{\partial^2 F_k( \bs{\tilde{\beta}})}{\partial \bs{\beta} \partial \bs{\beta}^\top} ( \bs{\widehat{\beta}}^0 - \bs{\beta}),~~k=1, \ldots, n.
\end{equation}
A directed calculation gives that
\[
\frac{\partial^2 F_k( \bs{\tilde{\beta}} )}{\partial \beta_i \partial \beta_j} =
\begin{cases}
 \sum_{t\neq k} \mu^{\prime\prime}(\tilde{\beta}_k + \tilde{\beta}_t ),  &   i=j=k  \\
\mu^{\prime\prime}(\tilde{\beta}_k + \tilde{\beta}_j ), & i=k, i\neq j; j=k, i\neq j \\
0, &  i \neq j \neq k.
\end{cases}
\]
It follows that
\[
R_k = \frac{1}{2} \sum_{j=r+1,j\neq k}^n \mu^{\prime\prime}(\tilde{\beta}_k + \tilde{\beta}_j) |( \widehat{\beta}_k - \beta_k^* )^2
+ \sum_{j,k=r+1;j\neq k}^n | \mu^{\prime\prime}(\tilde{\beta}_k + \tilde{\beta}_j)| |( \widehat{\beta}_k - \beta_k^* )( \widehat{\beta}_j - \beta_j^*).
\]
By \eqref{ineq-mu-tilde} and event $E_n$, we have
\begin{equation}
\label{eq:rk}
\max_{k=1, \ldots, n} |R_k| \lesssim n \cdot \frac{1}{c_n} \cdot \| \bs{\widehat{\beta}} - \bs{\beta}\|_\infty \lesssim \frac{b_n^2\log n}{c_n}.
\end{equation}
Let $R=(R_{r+1}, \ldots, R_{n})^\top$ and $V=(v_{ij})=\partial F( \bs{\beta})/\partial \bs{\beta}^\top$.
Because $F(\widehat{\beta})=0$, by \eqref{equ-lemma-gamma-b}, we have
\begin{equation}\label{eq-expression-beta-star}
\bs{\widehat{\beta}} - \bs{\beta} =  V_{22}^{-1} \bs{\bar{d}}_2 + V_{22}^{-1} R.
\end{equation}
Note that $V\in \mathcal{L}_n(1/b_n, 1/c_n)$.
By \eqref{eq:rk} and \eqref{ineq-tight-V}, we have
\begin{eqnarray}\label{eq-hatbeta-cc}
\|V_{22}^{-1}R \|_\infty & \le & \|V_{22}^{-1} \|_\infty\|R \|_\infty \lesssim \frac{ b_n^3\log n }{ nc_n }.
\end{eqnarray}

Now, we bound the error term $\| (V_{22}^{-1} - S_{22}) \bs{\bar{d}}_2 \|_\infty$, where $S_{22}=\mathrm{diag}(1/v_{r+1,r+1}, \ldots, 1/v_{nn})$.
Note that
\begin{equation*}
\| (V^{-1} - S) \bs{\bar{d}} \|_\infty = \| V^{-1}( V - S^{-1} ) S \bs{\bar{d}} \|_\infty \le \| V^{-1} \|_\infty \| (V-S^{-1})S \bs{\bar{d}} \|_\infty,
\end{equation*}
and
\[
[(V-S^{-1})S\bs{\bar{d}}]_i = \sum_{j=r+1,j\neq i}^n \frac{v_{ij} }{ v_{jj} } \bar{d}_j = \sum_{j=r+1,j\neq i}^n \sum_{k=1, k\neq j}^n \frac{v_{ij}}{v_{jj}} \bar{a}_{jk}.
\]
The summation of the above right hand can be viewed as the sum of $r\times (n-r) + (n-r)(n-r-1)/2 $ independent random variables by noting that it is equal to
\[
\sum_{j=r+1,j\neq i}^n \sum_{k=r+1,k\neq j}^n ( \frac{v_{ij}}{v_{jj}} \bar{a}_{jk} + \frac{v_{ik}}{v_{kk}} \bar{a}_{kj} )
+ \sum_{j=r+1,j\neq i}^n \sum_{k=1}^r\frac{v_{ij}}{v_{jj}} \bar{a}_{jk}.
\]
For any $i\neq j$ and for any $k$, we have
\[
\frac{v_{ij}}{v_{kk} } \le \frac{ b_n }{ (n-1)c_n }.
\]
It follows that
\begin{equation}\label{eq-vv-a}
\sum_{j=r+1,j\neq i}^n \sum_{k=1, k\neq j}^n (\frac{v_{ij}}{v_{jj}})^2 \E a_{jk}^2 \le (n-1)(n-r-1) \cdot \left( \frac{ b_n }{ (n-1)c_n } \right)^2 \cdot \frac{1}{c_n}\lesssim \frac{b_n^2 }{c_n^3}.
\end{equation}
By Bernstein's inequality in Lemma \ref{lemma:bernstein} and \eqref{eq-vv-a}, with probability $1- 4/[n(n-1)]$, for large $n$, we have that
\[
| [(V-S^{-1})S\bar{d}]_i | \le \sqrt{ 2\log n \left( \sum_{j\neq i, k\neq j} \frac{v_{ij}}{v_{jj}} \E a_{jk}^2  \right)} + \frac{2}{3} \cdot  \frac{ b_n }{ (n-1)c_n } \cdot \log n
< 2.1 \frac{b_n}{c_n^{3/2}} \sqrt{\log n}.
\]
Therefore, with probability at least $1-4/n$, the following holds:
\[
\| (V_{22}-S_{22}^{-1})S_{22} \bs{\bar{d}}_2 \|_\infty < \frac{ 2.1b_n }{ c_n^{3/2} } \sqrt{\log n}.
\]
It yields that
\begin{equation}\label{eq-hatbeta-bb}
\| (V^{-1}_{22} - S_{22}) \bs{\bar{d}}_2 \|_\infty \le  \| V_{22}^{-1} \|_\infty \| (V_{22}-S_{22}^{-1})S \bs{\bar{d}}_2 \|_\infty \le \frac{ 2.1b_n^2  \sqrt{\log n} }{ nc_n^{3/2} }.
\end{equation}

Let
\[
g_i = \{(V_{22}^{-1} - S_{22}) \bs{\bar{d}}_2 \}_i + (V_{22}^{-1}R)_i, ~~i=r+1,\ldots, n.
\]
By combining \eqref{eq-expression-beta-star}, \eqref{eq-hatbeta-cc} and \eqref{eq-hatbeta-bb}, with probability at least $1-6 /n$, the following holds:
\[
\widehat{\beta}_i^0 - \beta_i = \frac{ \bar{d}_i }{ v_{ii} } + g_i, i=r+1, \ldots, n.
\]
where $g_i$ satisfies \eqref{ineq-g}.
\end{proof}

\subsection{Proof of Lemma 5}
\label{subsec-prooflemma4}
In this section, we give the proof of Lemma 5.

\begin{proof}[Proof of Lemma 5]
The bounds of $Q_{r1}$ and $Q_{r2}$ in Lemma 5 are reproduced here:
\begin{eqnarray}
\label{eq-S1-bound}
Q_{r1}:=\sum_{i=r+1}^n  (\widehat{\beta}_i-\beta_i)^3 \sum_{j=1,j\neq i}^n \mu^{\prime\prime}( \pi_{ij} ) & = &
O_p\left( \frac{ b_n^4 \log n ( 1- r/n)^{1/2} }{ c_n^2 } \right), \\
\label{eq-S2-bound}
Q_{r2}:=\sum_{i,j=r+1, j\neq i}^n  \widehat{\beta}_i-\beta_i)^2(\widehat{\beta}_j-\beta_j)\mu^{\prime\prime}( \pi_{ij} ) & = &
O_p\left( \frac{b_n^5(\log n)^2(n-r)}{ nc_n^2}  \right).
\end{eqnarray}

Note that $\beta_1, \ldots, \beta_r$ are known and $\beta_{r+1}, \ldots, \beta_n$ are unknown.
Let $E_{n1}$ be the event that
\begin{equation}\label{eq-a-zcc}
E_{n1} = \{ \widehat{\beta}_i - \beta_i = \frac{ \bar{d}_i }{ v_{ii} } + g_i, ~~i=r+1, \ldots, n, \},
\end{equation}
where
\begin{equation}\label{eq-a-zcg}
|g_i| \lesssim \frac{b_n^3 \log n }{ n c_n }.
\end{equation}
Let $E_{n2}$ be the event
\begin{equation}\label{eq-upp-bard}
E_{n2} =\{ \| \bs{\bar{d} } \|_\infty \le (n\log n)^{1/2}\}
\end{equation}
By Lemma \ref{lemma-hatbeta-exp} and \eqref{ineq-union-d},
$E_{n1}\bigcap E_{n2}$ holds with probability at least $1-8/n$.
The following calculations are based on $E_{n1}\bigcap E_{n2}$.

Let
\begin{equation}
f_i := \sum_{j=1,j\neq i}^n \mu^{\prime\prime}( \pi_{ij} ),\quad  f_{ij} :=  \mu^{\prime\prime}( \pi_{ij} )
\end{equation}
In view of \eqref{ineq-mu-deriv-bound}, we have
\begin{equation}\label{ineq-fi-fij}
\max_{i=r+1,\ldots,n} |f_i| \le \frac{ n-r-1}{ c_n }, ~~ \max_{i,j} |f_{ij} | \le \frac{ 1}{c_n}.
\end{equation}
By substituting the expression of $\widehat{\beta}_i - \beta_i$ in \eqref{eq-a-zcc} into  $f_i(\widehat{\beta}_i-\beta_i)^3$, we get
\begin{equation}\label{eq-a-za}
\sum_{i=r+1}^n  f_i(\widehat{\beta}_i-\beta_i)^3   =   \sum_{i=r+1}^n f_i \cdot \frac{ \bar{d}_i^3 }{ v_{ii}^3 }
+ 3 \sum_{i=r+1}^n f_i \cdot \frac{ \bar{d}_i^2 g_i }{ v_{ii}^2 } + 3 \sum_{i=r+1}^n f_i \cdot \frac{ \bar{d}_i g_i^2 }{v_{ii}}
+ \sum_{i=r+1}^n f_i g_i^3.
\end{equation}
We bound the four terms in the above right hand in an inverse order.
The fourth term can be bounded as follows:
\begin{eqnarray}
\nonumber
|\sum_{i=r+1}^n  f_i g_i^3 | & \le & (n-r) \max_{i=r+1,\ldots,n} |f_i| \max_i |g_i^3|, \\
\nonumber
& \lesssim & (n-r) \cdot \frac{ n }{c_n } \cdot (\frac{ b_n^3 \log n }{ nc_n})^3, \\
\label{eq-a-four}
& \lesssim &  \frac{ (n-r)b_n^9 (\log n)^3 }{ n^2 c_n^4 } .
\end{eqnarray}
In view of \eqref{eq-a-zcg}, \eqref{eq-upp-bard} and \eqref{ineq-fi-fij},  the upper bound of the third term is
\begin{eqnarray}
\nonumber
|\sum_{i=r+1}^n \frac{\bar{d}_i}{v_{ii}} f_i g_i^2 | & \le & (n-r) \cdot \max_i \frac{1}{v_{ii}} \cdot \| \bs{\bar{d}} \|_\infty \cdot \max_{i=r+1,\ldots,n} |f_i| \cdot \max_i |g_i^2| \\
\nonumber
& \lesssim & (n-r)\cdot \frac{b_n}{n} \cdot (n\log n)^{1/2} \cdot \frac{ n }{c_n } \cdot (\frac{ b_n^3 \log n }{ nc_n})^2 \\
\label{eq-a-three}
& \lesssim &  \frac{ (n-r)^2 b_n^7 (\log n)^{5/2} }{ n^{5/2} c_n^3 } .
\end{eqnarray}
By Corollary 1 in the main text, we have that
\[
\sum_{i=r+1}^n \frac{ \bar{d}_i^2 }{v_{ii}} = O_p( n-r).
\]
In view of \eqref{eq-a-zcg} and \eqref{ineq-fi-fij}, the second term can be bounded as follows:
\begin{eqnarray}
\nonumber
|\sum_{i=r+1}^n \frac{\bar{d}_i^2}{v_{ii}^2} f_i g_i | & \le & \max_i \frac{ 1 }{v_{ii}} \cdot \max_{i=r+1,\ldots,n} |f_i| \cdot \max_i |g_i|  \cdot \sum_{i=r+1}^n \frac{ \bar{d}_i^2 }{v_{ii} } \\
\nonumber
& \lesssim & \frac{b_n}{n} \cdot \frac{ n }{c_n } \cdot \frac{ b_n^3 \log n }{ nc_n} \cdot O_p(n-r) \\
\label{eq-a-second}
& = & O_p( \frac{ (n-r) b_n^4\log n }{ n c_n^2 } ).
\end{eqnarray}
Now, we bound the first term.
By Lemma \ref{lemma:var:cuibic}, we have that
\begin{eqnarray*}
&&\mathrm{Var}( \sum_{i=r+1}^n \frac{f_i}{v_{ii}^3} \bar{d}_i^{\,3} ) \\
& < &
\max_i \frac{ |f_i|^2 }{v_{ii}^6} \left\{ n(n-r) \cdot \max_{i,j} \mathrm{Var}( \bar{a}_{ij}^3 ) + 3n^2(n-r) \cdot \max_{i,j} \E \bar{a}_{ij}^4 \cdot \max_{i,j} \E \bar{a}_{ij}^2 \right\} \\
&&+ 6n^3(n-r) (\max_{i,j} \E \bar{a}_{ij}^2 )^3 + 2n(n-r) \max_{i,j} \mathrm{Var}(\bar{a}_{ij}^3) + 8n^3(n-r) ( \max_{i,j} \E \bar{a}_{ij} )^3.
\end{eqnarray*}
Because
\begin{eqnarray*}
\mathrm{Var}(\bar{a}_{ij}^3 ) & = & p_{ij} q_{ij} ( p_{ij}^5q_{ij} + q_{ij}^5p_{ij} - 2p_{ij}^3q_{ij}^3) \le  p_{ij} q_{ij} \le \frac{1}{c_n},\\
\E \bar{a}_{ij}^2 & = & p_{ij}q_{ij} \le \frac{1}{c_n}, \\
\E \bar{a}_{ij}^4 & = & p_{ij}^4 q_{ij} + q_{ij}^4 p_{ij} \le p_{ij}q_{ij} \le \frac{1}{c_n},
\end{eqnarray*}
an upper bound of $\mathrm{Var}( \sum_{i=r+1}^n f_i\bar{d}_i^3/v_{ii}^3 )$ is
\begin{equation}\label{ineq-d-firsta}
\mathrm{Var}( \sum_{i=r+1}^n \frac{f_i\bar{d}_i^3}{v_{ii}^3} ) 
\lesssim \frac{ n^2}{c_n^2 } \cdot \frac{ b_n^6 }{ n^6 } \cdot \frac{ n^3(n-r) }{ c_n^3} \lesssim \frac{ (n-r)b_n^6 }{ n c_n^5 }.
\end{equation}
Because
\[
\E \bar{d}_{i}^{\,3} = \sum_{\alpha, \gamma, \zeta} \E \bar{a}_{i\alpha}\bar{a}_{i\gamma}\bar{a}_{i\zeta} = \sum_{\alpha\neq i} \E \bar{a}_{i\alpha}^{\,3},
\]
we have
\[
|\E \bar{a}_{ij}^{\,3}|  = | p_{ij}q_{ij}( p_{ij}^2 - q_{ij}^2)|\le\frac{1}{c_n},
\]
such that
\begin{equation}\label{eq-ed3}
| \sum_{i=r+1}^n  \left (  \frac{\E \bar{d}_i^3 }{ v_{ii}^3 } \right)f_i | \lesssim (n-r) \cdot \frac{n}{c_n} \cdot \frac{b_n^3}{n^3} \cdot \frac{n-r}{c_n} \lesssim \frac{(n-r)^2b_n^3}{n^2c_n^2}.
\end{equation}
In view of \eqref{ineq-d-firsta} and \eqref{eq-ed3}, we have
\begin{equation}\label{ineq-d-first}
|\sum_{i=r+1}^n  \left (  \frac{ \bar{d}_i^3 }{ v_{ii}^3 } \right)f_i | = O_p\left( \frac{ b_n^3 }{c_n^{5/2}} \left( \frac{ n-r }{ n} \right)^{1/2} \right).
\end{equation}
By combining the upper bounds of the above four terms in \eqref{eq-a-four}, \eqref{eq-a-three}, \eqref{eq-a-second} and \eqref{ineq-d-first}, it yields that
\[
Q_{r1} = O_p\left(  \frac{ (n-r)b_n^9 (\log n)^3 }{ n^2 c_n^4 }
+ \frac{ (n-r)b_n^7 (\log n)^{3/2} }{ n^{3/2} c_n^3 } +
\frac{ (n-r) b_n^4 \log n}{ n c_n^2 } + \frac{ b_n^3 }{c_n^{5/2} } \left( \frac{n-r}{n} \right)^{1/2} \right).
\]
This leads to \eqref{eq-S1-bound}.

Now we bound the following terms in \eqref{eq-S2-bound}:
\[
Q_{r2} = \sum_{i,j=r+1, j\neq i}^n  (\widehat{\beta}_i-\beta_i)^2(\widehat{\beta}_j-\beta_j)\mu^{\prime\prime}( \pi_{ij} ).
\]
By substituting \eqref{eq-a-zcc} into the above expression, we get
\begin{eqnarray*}\label{eq-s2-expan}
Q_{r2} & = &  \sum_{r+1 \le  i\neq j \le n} f_{ij} \cdot \frac{\bar{d}_i^2\bar{d}_j}{v_{ii}^2v_{jj}} +
\sum_{r+1 \le i\neq j \le n} f_{ij} \cdot \frac{ \bar{d}_i^2 }{ v_{ii}^2 } \cdot g_i
 + 2\sum_{r+1 \le i\neq j \le n } \frac{ \bar{d}_i \bar{d}_j}{v_{ii}v_{jj}} \cdot g_i \cdot f_{ij} \\
 & &
+ 2 \sum_{r+1 \le i\neq j \le n} \frac{ \bar{d}_i }{ v_{ii} } \cdot g_i g_j f_{ij}
+ \sum_{r+1 \le i\neq j \le n}  \frac{ \bar{d}_j }{ v_{jj} } \cdot g_i^2 f_{ij}
+ \sum_{r+1\le i\neq j \le n} g_i^2 g_j f_{ij}.
\end{eqnarray*}
In view of \eqref{ineq-d-upper}, \eqref{ineq-fi-fij} and \eqref{eq-a-zcg}, we have the following bounds
\begin{eqnarray}
\label{ineq-s1-a}
|\sum_{r+1\le i\neq j \le n} g_i^2 g_j f_{ij}| & \lesssim & \frac{ b_n^9 (\log n)^3 (n-r)^2 }{ n^3c_n^3 }, \\
\label{ineq-s1-b}
\max\{|\sum_{r+1\le i\neq j \le n} \frac{ \bar{d}_i }{ v_{ii} } \cdot g_i g_j f_{ij}|, |\sum_{r+1\le i\neq j\le n}  \frac{ \bar{d}_j }{ v_{jj} } \cdot g_i^2 f_{ij} |  & \lesssim & \frac{ b_n^7 (\log n)^{1/2} (n-r)^2 }{ c_n^2 n^{5/2} }, \\
\label{ineq-s1-c}
~~\max\{ | \sum_{r+1\le i\neq j\le n} \frac{ \bar{d}_i \bar{d}_j}{v_{ii}v_{jj}} \cdot g_i f_{ij}|,
|\sum_{r+1\le i\neq j\le n} \frac{ \bar{d}_i^2 }{ v_{ii}^2 } \cdot g_i  f_{ij} | \} & \lesssim &
\frac{ b_n^5 (\log n)^2(n-r)^2 }{ n^2c_n^2}.
\end{eqnarray}
The left argument is to bound the first term in \eqref{eq-s2-expan}. Because
\[
|\E \bar{d}_i^2 \bar{d}_j |= |\E \bar{a}_{ij}^2 \bar{a}_{ji} | \le p_{ij}q_{ij} \le \frac{1}{c_n},
\]
by \eqref{ineq-fi-fij}, we have
\[
|\sum_{r+1 \le i\neq j \le n} f_{ij} \cdot \frac{\E \bar{d}_i^2\bar{d}_j}{v_{ii}^2v_{jj}}| \le (n-r)^2 \frac{n}{c_n} \cdot \frac{ b_n^3}{n^3} \lesssim \frac{ b_n^3(n-r)^2 }{ n^2c_n^2}.
\]
By Lemma \ref{lemma:var:cuibic2}, we have
\[
\mathrm{Var}\left( \sum_{r+1\le i\neq j\le n} f_{ij} \cdot \frac{ \bar{d}_i^2\bar{d}_j}{v_{ii}^2v_{jj}} \right) \lesssim \frac{ (n-r)^2b_n^6 }{n^2c_n^2}.
\]
Similar to \eqref{ineq-d-firsta},  Chebyshev's inequality gives that
\begin{equation}\label{ineq-s1-d}
\sum_{i\neq j} f_{ij} \cdot \frac{\bar{d}_i^2\bar{d}_j}{v_{ii}^2v_{jj}} = O_p( \frac{ b_n^3 \log n}{ c_n}  ).
\end{equation}
By combining \eqref{eq-s2-expan} and \eqref{ineq-s1-a}--\eqref{ineq-s1-d},  we have
\[
Q_{r2}=O_p\left(
\frac{ b_n^9 (\log n)^3 (n-r)^2 }{n^3 c_n^3 } + \frac{b_n^7 (\log n)^{1/2} (n-r)^2 }{ n^{5/2} c_n }
+ \frac{ b_n^5 (\log n)^2 (n-r)^2 }{ n^2 c_n^2 } +
\frac{ (n-r)b_n^3 }{ nc_n }\right)
\]
If
\[
\frac{ b_n^5 (\log n)^2 (n-r) }{ nc_n^2} = o( r^{1/2} ),
\]
it yields \eqref{eq-S2-bound}.
\end{proof}

\section{Proofs of supported lemmas in the proof of Theorem 1 (b)}
\label{section-th1b}

This section is organized as follows.
Sections \ref{section-lemma5}, \ref{subsection-proof-lemma6} and \ref{section-lemma7}
present the proofs of Lemmas 7, 8 and 9, respectively.
Section \ref{section-proof-39-B20} presents the proof of (39) in the main text.

We reproduce some notations and some useful results in Section 6.2 here.
Recall  $\bs{\widetilde{d}}=(\sum_{i=1}^r d_i, d_{r+1},\ldots, d_n)$ and
 $\widetilde{V}$ denote the Fisher information matrix of $\widetilde{\bs{\beta}}=(\beta_1, \beta_{r+1}, \ldots, \beta_n)^\top$
under the null $H_0: \beta_1 = \cdots= \beta_r$, where
\begin{equation}\label{definition-tilde-V}
\widetilde{V}=\begin{pmatrix} \tilde{v}_{11} & \bs{\tilde{v}}_{12}^\top \\ \bs{\tilde{v}}_{12} & V_{22} \end{pmatrix},
\end{equation}
where $V_{22}$ is the lower right $(n-r)\times (n-r)$ block of $V$, $\bs{\tilde{v}}_{12} =
(\tilde{v}_{1,r+1}, \ldots, \bar{v}_{1, n})^\top$, and
\[
\tilde{v}_{11}= 2r(r-1)\cdot \frac{ e^{2\beta_1} }{ ( 1 + e^{2\beta_1})^2 } + r\sum_{j=r+1}^n \tilde{v}_{1j}, ~~
\tilde{v}_{1j} =  \frac{ r e^{\beta_1 + \beta_j } }{ ( 1 + e^{\beta_1 + \beta_j})^2 },~j=r+1, \ldots, n.
\]
We use $\widetilde{S}=\mathrm{diag}(1/\tilde{v}_{11}, 1/v_{r+1, r+1}, \ldots, 1/v_{nn})$ to approximate $\widetilde{V}^{-1}$ and have the following approximation error
\begin{equation}\label{approxi-inv2-beta-ho}
\|\widetilde{W}:= \widetilde{V}^{-1}-\widetilde{S} \|_{\max} \lesssim \frac{b_n^3}{n^2c_n^2}.
\end{equation}

\subsection{Proof of Lemma 7}
\label{section-lemma5}

In this section, we present the proof of Lemma 7.
We introduce
an error bound in the Newton method by \cite{Kantorovich-Akilov1964}
under the Kantorovich conditions [\cite{Kantorovich1948Functional}].

\begin{lemma}[Theorem 6 in \cite{Kantorovich-Akilov1964}]\label{lemma:Newton:Kantovorich-b}
Let $D$ be an open convex subset of $\R^n$ and
$F:D \to \R^n$ be Fr\'{e}chet differntiable.
Assume that, at some $\bs{x}_0 \in D$, $F^\prime(\bs{x}_0)$ is invertible and that
\begin{eqnarray}
\label{eq-kantororich-a}
\| F^\prime(\bs{x}_0)^{-1} ( F^\prime(\bs{x}) - F^\prime(\bs{y}))\| \le K\|\bs{x}-\bs{y}\|,~~ \bs{x}, \bs{y}\in D, \\
\label{eq-kantororich-b}
\| F^\prime(\bs{x}_0)^{-1} F(\bs{x}_0) \| \le \eta, h=K\eta \le 1/2, \\
\nonumber
\bar{S}(\bs{x}_0, t^*) \subseteq D, t^*=2\eta/( 1+ \sqrt{ 1-2h}).
\end{eqnarray}
Then:
(1) The Newton iterates $\bs{x}_{n+1} = \bs{x}_n - F^\prime (\bs{x}_n)^{-1} F(\bs{x}_n)$, $n\ge0$ are well-defined,
lie in $\bar{S}(\bs{x}_0, t^*)$ and converge to a solution $\bs{x}^*$ of $F(\bs{x})=0$. \\
(2) The solution $\bs{x}^*$ is unique in $S(\bs{x}_0, t^{**})\cap D$, $t^{**}=(1 + \sqrt{1-2h})/K$ if $2h<1$
and in $\bar{S}(\bs{x}_0, t^{**})$ if $2h=1$. \\
(3) $\| \bs{x}^* - \bs{x}_n \| \le t^*$ if $n=0$ and $\| \bs{x}^* - \bs{x}_n \| \le 2^{1-n} (2h)^{ 2^n -1 } \eta $ if $n\ge 1$.
\end{lemma}

Now, we are ready to prove Lemma 7.

\begin{proof}[Proof of Lemma 7]
Recall that $\bs{\widehat{\beta}}^0$ denotes the restricted MLE under the null $H_0: \beta_1=\cdots =\beta_r$.
In what follows, $\bs{\widehat{\beta}}^0$ and $\bs{\beta}$ denote
respective vectors $(\widehat{\beta}_1^0, \widehat{\beta}_{r+1}^0, \ldots,
\widehat{\beta}_n^0)^\top$ and $(\beta_1, \beta_{r+1}, \ldots, \beta_n)^\top$ with some ambiguity of notations.
Define a system of score functions based on likelihood equations:
\begin{equation}\label{eqn:def:F}
\begin{array}{rcl}
F_1(\bs{\beta}) & = &  \sum\limits_{i=1}^r \sum\limits_{j=1, j\neq i}^n \mu(\beta_i + \beta_j) - \sum\limits_{i=1}^r d_i, \\
 F_i(\bs{\beta}) & = &  \sum\limits_{j=1, j\neq i}^n \mu(\beta_i + \beta_j) - d_i ,~~i=r+1, \ldots, n,
\end{array}
\end{equation}
and $F(\bs{\beta})=(F_1(\bs{\beta}), F_{r+1}(\bs{\beta}), \ldots, F_n(\bs{\beta}))^\top$, where
$\beta_1=\ldots=\beta_r$.

Let $B(\bs{\beta}, 1/(2c_n))=\{\bs{\gamma}=(\gamma_1, \gamma_{r+1}, \ldots, \gamma_n)\in \R^{n-r+1}: \| \bs{\beta} - \bs{\gamma}\|_\infty \le 1/(2c_n) \}$ be a neighbouring set containing $\bs{\beta}$.
We will derive the error bound between $\bs{\widehat{\beta}}$ and $\bs{\beta}$ through
obtaining the convergence rate of the Newton iterative sequence $\bs{\beta}^{(n+1)}= \bs{\beta}^{(n)} - [F^\prime (\bs{\beta}^{(n)})]^{-1}
F (\bs{\beta}^{(n)})$,
where we choose the true parameter $\bs{\beta}$ as the starting point $\bs{\beta}^{(0)}:=\bs{\beta}$.
To this end, it is sufficient to demonstrate the Kantovorich conditions in Lemma \ref{lemma:Newton:Kantovorich}, where we set $D=B(\bs{\beta}, 1/(2c_n))$.
The proof proceeds three steps. Step 1 is about verifying condition \eqref{eq-kantororich-a}.
Step 2 is about verifying \eqref{eq-kantororich-b}. Step 3 is a combining step.

Step 1. We claim that for any $\bs{x}, \bs{y} \in B(\bs{\beta}, 1/(2c_n))$,
\begin{equation}\label{eq-con-ho-a}
\| [F^\prime(\bs{\beta})]^{-1} \{ F^\prime(\bs{x}) - F^\prime(\bs{y}) \}\| \lesssim \left( b_n + \frac{(n-r)b_n^3}{nc_n^3} \right) \|\bs{x} - \bs{y}\|.
\end{equation}
This is verified as follows.
Let $\pi_{ij}=\gamma_i + \gamma_j$ and $\mu(\pi_{ij}) = e^{\pi_{ij}}/( 1 + e^{\pi_{ij}})$.
The Jacobian matrix $F^\prime(\bs{\gamma})$ of $F(\bs{\gamma})$ can be calculated as follows.
By finding the partial derivative of $F_i$ with respect to $\gamma_j$, we have
\[
\frac{ \partial F_i (\bs{\gamma}) }{ \partial \gamma_j } =
\begin{cases}
r(r-1)\mu^\prime(\pi_{11}) + r \sum_{k=2}^{n-r+1} \mu^\prime(\pi_{1k}), & i=1,j=1, \\
r \mu^\prime (\pi_{1j}), & i=1, j=2,\ldots, n-r+1, \\
r\mu^\prime (\pi_{i1}), & i=2,\ldots, n-r+1, j=1, \\
r\mu^\prime (\pi_{i1}) + \sum_{k=2}^{n-r+1} \mu^\prime( \pi_{ik} ), &i=2,\ldots, n-r+1, j=i, \\
 \mu^\prime( \pi_{ij} ), &i,j=2,\ldots, n-r+1, j\neq i,
\end{cases}
\]
and
\[
\frac{ \partial^2 F_i (\bs{\gamma}) }{ \partial \gamma_j\partial \gamma_k } =
\begin{cases}
r(r-1)\mu^{\prime\prime}(\pi_{11}) + r \sum\limits_{t=2}^{n-r+1} \mu^{\prime\prime}(\pi_{1t}), & i=1,j=1,k=1, \\
r \mu^\prime (\pi_{1j}), & i=1, k=j=2,\ldots, n-r+1, \\
0, & i=1; k,j=2,\ldots, n-r+1; k\neq j, \\
r\mu^{\prime\prime} (\pi_{i1}), & i=2,\ldots, n-r+1, j=k=1, \\
r\mu^{\prime\prime} (\pi_{i1}) + \sum\limits_{t=2, t\neq i}^{n-r+1} \mu^{\prime\prime}( \pi_{it} ), &i=j=k=2,\ldots, n-r+1, \\
\mu^{\prime\prime} (\pi_{ij}), &i,j,k=2,\ldots, n-r+1, i\neq j, j=k, \\
0, &i, j,k=2,\ldots, n-r+1, k\neq j, j\neq i, i\neq k.
\end{cases}.
\]
By the mean value theorem and \eqref{eq-derivative-mu-various}, we have
\[
|\mu^{\prime\prime} (\beta_i+\beta_j) - \mu^{\prime\prime} (\gamma_i+\gamma_j)|\le \frac{1}{4} \| \bs{\gamma}-\bs{\beta}\|_\infty \le \frac{1}{2c_n}.
\]
This shows
\[
\max_{i,j} |\mu^{\prime\prime} (\gamma_i+\gamma_j)| \le \frac{3}{2c_n}.
\]
It follows that
\begin{equation}
\label{inequ:second:deri}
\left| \frac{ \partial^2 F_i (\bs{\gamma}) }{ \partial \gamma_j\partial \gamma_k } \right|=
\begin{cases}
\frac{3r(n-1)}{2c_n}, & i=1,j=1,k=1, \\
\frac{3r}{2c_n}, & i=1, j= k=2,\ldots, n-r+1, \\
0, & i=1; k,j=2,\ldots, n-r+1; k\neq j, \\
\frac{3r}{2c_n}, & i=2,\ldots, n-r+1, j=k=1, \\
\frac{3(n-1)}{2c_n}, &i=j=k=2,\ldots, n-r+1, \\
\frac{3}{2c_n}, &i,j,k=2,\ldots, n-r+1, i\neq j, j=k, \\
0, &i, j,k=2,\ldots, n-r+1, k\neq j, j\neq i, i\neq k.
\end{cases}
\end{equation}
For any $i,j\in \{1, \ldots, n-r+1\}$, define
\[
\mathbf{g}_{ij}(\bs{\gamma})=\left(\frac{\partial^2 F_i(\bs{\gamma}) }{ \partial \gamma_{1} \partial \gamma_j}, \ldots,
\frac{\partial^2 F_i(\bs{\gamma}) }{ \partial \gamma_{n-r+1} \partial \gamma_j} \right)^\top.
\]
By \eqref{inequ:second:deri}, we have
\begin{equation}
\|\mathbf{g}_{ij}(\bs{\gamma})\|_1 \le
\begin{cases}
\frac{ 6r(n-1) }{ 2c_n }, & i=1, j=1, \\
\frac{ 6r }{ 2c_n }, & i=1, j=2, \ldots, n-r+1, \\
\frac{ 3(n-1) }{ 2c_n }, & i=j=2, \ldots, n-r+1, \\
\frac{ 6r }{ 2c_n }, & i,j=2,\ldots, n-r+1; i\neq j.
\end{cases}
\end{equation}
Consequently,  for any vectors $\bs{x}, \bs{y} \subset D$, we have
\begin{align*}
 & | [F^\prime(\bs{x})]_{ij} - [F^\prime(\bs{y})]_{ij} | \\
 = &   |\int_0^1 [\mathbf{g}_{ij}(t\bs{x}+(1-t)\bs{y})]^\top (\bs{x}-\bs{y})dt | \\
 \le &
 \begin{cases}
 \frac{ 6r(n-1) }{ 2c_n } \| \bs{x}-\bs{y} \|_\infty, & i=1, j=1, \\
\frac{ 6r }{ 2c_n }\| \bs{x}-\bs{y} \|_\infty, & i=1, j=2, \ldots, n-r+1, \\
\frac{ 6r }{ 2c_n }\| \bs{x}-\bs{y} \|_\infty, & i=2, \ldots, n-r+1, j=1, \\
\frac{ 3(n-1) }{ 2c_n }\| \bs{x}-\bs{y} \|_\infty, & i=j=2, \ldots, n-r+1, \\
\frac{ 6 }{ 2c_n }\| \bs{x}-\bs{y} \|_\infty, & i,j=2,\ldots, n-r+1; i\neq j.
 \end{cases}
\end{align*}
It follows that
\[
\sum_{j=1}^{n-r+1} | [F^\prime(\bs{x})]_{ij} - [F^\prime(\bs{y})]_{ij} |
\le \begin{cases}
\frac{ 6r(2n-r-1) }{ 2c_n } \| \bs{x}-\bs{y} \|_\infty, & i=1, \\
\frac{ 3n + 6r + 6(n-r-1) }{ 2c_n } \| \bs{x}-\bs{y} \|_\infty, & i=2, \ldots, n-r+1.
\end{cases}
\]
This gives that
\begin{equation}\label{eq-wf-a}
\sum_{k=1}^{n-r+1}\widetilde{S}_{ii} | [F^\prime(\bs{x})]_{ik} - [F^\prime(\bs{y})]_{ik}|
\le \begin{cases}
\frac{ 6b_n(2n-r-1) }{ 2(n-1)c_n }, & i=1, \\
\frac{ (6n-3)b_n }{ 2(n-1)c_n } \| \bs{x}-\bs{y} \|_\infty, & i=2, \ldots, n-r+1.
\end{cases}
\end{equation}
and, by \eqref{approxi-inv2-beta-ho},
\begin{equation}\label{eq-wf-b}
\sum_{k=1}^{n-r+1}\left|\widetilde{W}_{ik} \right| | [F^\prime(\bs{x})]_{kj} - [F^\prime(\bs{y})]_{kj}|
\lesssim
\frac{ b_n^3 }{ n^2c_n^2 } \times \frac{ 6r(2n-r-1)}{ c_n } \lesssim \frac{ (n-r) b_n^3 }{ nc_n^3 }.
\end{equation}
Note that $\widetilde{W}=\widetilde{V}^{-1} - \widetilde{S}$
and $\widetilde{S}=\mathrm{diag}(1/\tilde{v}_{11}, 1/v_{r+1,r+1}, \ldots, 1/v_{nn})$.
By combining \eqref{eq-wf-a} and \eqref{eq-wf-b}, we have \eqref{eq-con-ho-a}.

Step 2. We claim that with probability at least $1-2(n-r+1)/n^2$, we have
\begin{equation}\label{ineq-union-da}
\|\widetilde{V}^{-1}F^\prime(\bs{\beta})\|_\infty \lesssim
\left\{ b_n + \frac{ b_n^3 }{ c_n^2} \left( \frac{ r^{1/2} }{ n } + \frac{n-r}{n} \right) \right\} \sqrt{\frac{\log n}{n}}.
\end{equation}
Recall that $a_{ij}$, $1\le i < j \le n$, are independent Bernoulli random variables
and $\bar{d}_i = \sum_{j\neq i} \bar{a}_{ij}$.
By \citeauthor{Hoeffding:1963}'s \citeyearpar{Hoeffding:1963} inequality, we have
\begin{equation*}
\P\left( |\bar{d}_i | \ge \sqrt{n\log n}  \right) \le 2\exp \left(- 2\frac{n\log n}{n} \right) \le  \frac{2}{n^2},~~i=1, \ldots, n.
\end{equation*}
By the union bound, we have
\begin{eqnarray}
\label{ineq-d-upper}
\P\Bigg( \max_{i=r+1, \ldots, n} | \bar{d}_i | \ge \sqrt{ n\log n} \Bigg)
\le  \sum_{i=r+1}^n \P\left(|\bar{d}_i| \geq \sqrt{n\log n} \right)
\le  \frac{2(n-r)}{n^2 },
\end{eqnarray}
such that
\[
\P\Bigg( \max_{i=r+1, \ldots, n} |\bar{d}_i| \le \sqrt{n\log n} \Bigg) \ge \P\Bigg( \max_{i=1, \ldots, n} |\bar{d}_i| \le \sqrt{n\log n} \Bigg) \ge  1 - \frac{ 2(n-r) }{ n^2 }.
\]
Note that
\[
\sum_{i=1}^r \bar{d}_i = \sum_{1\le i\neq j \le r} 2\bar{a}_{ij} + \sum_{i=1}^r \sum_{j=r+1}^n \bar{a}_{ij},
\]
and the terms in the above summation are independent.
\citeauthor{Hoeffding:1963}'s \citeyearpar{Hoeffding:1963} inequality gives that
\begin{align*}
 & \P\left( |\sum_{i=1}^r \bar{d}_i | \ge \sqrt{2\left\{\frac{r(r-1)}{2} + r(n-r-1) \right\} \log n}  \right)  \\
\le & 2\exp \left(- 2\frac{4\left\{\frac{r(r-1)}{2} + r(n-r-1)\right\}\log n}{4\left\{\frac{r(r-1)}{2} + r(n-r-1)\right\}} \right)
\le  \frac{2}{n^2}.
\end{align*}
The above arguments imply that with probability at least $1-2(n-r+1)/n^2$,
\[
\widetilde{S}_{ii} |F_i(\bs{\beta})| \le \begin{cases}
\frac{b_n}{rn} \times \sqrt{ r(2n-r-3)\log n} \le \frac{b_n(2n-r-3)^{1/2}}{(rn)^{1/2}} \sqrt{\frac{\log n}{n}}, & i=1, \\
b_n \sqrt{\frac{\log n}{n}}, & i=2,\ldots, n-r+1.
\end{cases}
\]
and, by \eqref{approxi-inv2-beta-ho},
\begin{eqnarray*}
\sum_{j=1}^{n-r+1}|\widetilde{W}_{ij}||F_j(\bs{\beta})| & \lesssim &
\frac{b_n^3}{n^2c_n^2} \times (\sqrt{ r(2n-r-3)\log n} + (n-r)\sqrt{n\log n} )\\
& \lesssim &
\frac{ b_n^3 }{ c_n^2} \left( \frac{ r^{1/2} }{ n } + \frac{n-r}{n} \right) \sqrt{\frac{\log n}{n}}.
\end{eqnarray*}

Step 3. This step is one combining step. By \eqref{eq-con-ho-a}, we can set
\[
K=O\left( \frac{b_n}{c_n} + \frac{b_n^3}{c_n^3}\cdot \frac{r}{n}  \right),
\]
and
\[
\eta = O\left(  \left\{ b_n + \frac{ b_n^3 }{ c_n^2} \left( \frac{ r^{1/2} }{ n } + \frac{n-r}{n} \right) \right\} \sqrt{\frac{\log n}{n}} \right).
\]
 in Lemma \ref{lemma:Newton:Kantovorich-b}.
If $b_n^6 /c_n^5=o( (n/\log n)^{1/2})$,
then
\[
h=K\eta \lesssim  \left( \frac{b_n^2}{c_n} + \frac{ b_n^4 }{ c_n^3 } \cdot \left( \frac{ r^{1/2} }{ n } + \frac{n-r}{n} \right) + \frac{ b_n^4}{c_n^3 } \cdot \frac{r}{ n }
+ \frac{ b_n^6 }{ c_n^5 } \left( \frac{ r^{3/2} }{ n^2 } + \frac{r(n-r)}{n^2} \right) \right)\sqrt{\frac{\log n}{n}}\to 0.
\]
This completes the proof.
\end{proof}

\subsection{Proof of Lemma 8}
\label{subsection-proof-lemma6}

In this section, we present the proof of Lemma 8.

\begin{proof}[Proof of Lemma 8]
Recall that  $\bs{\tilde{d}}=(\sum_{i=1}^r d_i, d_{r+1},\ldots, d_n)^\top$ and
$\widetilde{W} = \widetilde{V} - \widetilde{S}$.
It is sufficient to demonstrate:
\begin{equation}\label{wd-expectation2}
\E[  (\bs{\tilde{d}} - \E\bs{\tilde{d}} )^\top \widetilde{W}  (\bs{\tilde{d}} - \E\bs{\tilde{d}} ) ] =0,
\end{equation}
and
\begin{equation}\label{Wd-op2}
\mathrm{Var}( (\bs{\tilde{d}} - \E\bs{\tilde{d}} )^\top \widetilde{W}  (\bs{\tilde{d}} - \E\bs{\tilde{d}} ) )=
O\left( \frac{ b_n^3 }{ c_n^3 } \right).
\end{equation}
The claim of \eqref{wd-expectation2} is due to that
\begin{eqnarray*}
\E[  (\bs{\tilde{d}} - \E\bs{\tilde{d}} )^\top \widetilde{W} (\bs{\tilde{d}} - \E\bs{\tilde{d}} ) ] & = &
\mathrm{tr} (\E[  (\bs{\tilde{d}} - \E\bs{\tilde{d}} )^\top (\bs{\tilde{d}} - \E\bs{\tilde{d}} )  ] \widetilde{W} ) \\
& = & \mathrm{tr} ( \widetilde{V} \widetilde{W} ) = \mathrm{tr} ( I_{n-r+1} - \widetilde{V} \widetilde{S} ) = 0.
\end{eqnarray*}
Let
\[
R = \begin{pmatrix} \overline{W}_{11} &  \overline{W}_{12} \\
\overline{W}_{21} & \widetilde{W}_{22}
\end{pmatrix}
\]
where $\widetilde{W}_{22}$ is the bottom right $(n-r)\times (n-r)$ block of $\widetilde{W}$,
$\overline{W}_{11}$ is the $r\times r$ matrix with all its elements being equal to $\tilde{w}_{11}$,
and $\overline{W}_{12}$ is the $r \times (n-r)$ matrix with all its row being equal to
the vector $(\tilde{w}_{12}, \ldots, \tilde{w}_{1,n-r+1})$, and $\overline{W}_{21}$ is the transpose of $\overline{W}_{12}$.
Therefore, we have
\[
(\bs{\tilde{d}} - \E\bs{\tilde{d}} )^\top \widetilde{W}  (\bs{\tilde{d}} - \E\bs{\tilde{d}} ) = \bs{\bar{d}}^\top R \bs{\bar{d}}.
\]
Because
\[
\| R \|_{\max}=\| \widetilde{W} \|_{\max}  \lesssim \frac{ b_n^3 }{ n^2c_n^2 },
\]
with the same arguments as in the proof of Lemma 3, we have \eqref{Wd-op2}.
This completes the proof.
\end{proof}

\subsection{Proof of Lemma 9}
\label{section-lemma7}
In this section, we present the proof of Lemma 9.

\begin{proof}[Proof of Lemma 9]

Since $\beta_1=\cdots=\beta_r$ and $\widehat{\beta}_1^0 = \cdots = \widehat{\beta}_r^0$ with $r\in\{1,\ldots,n-1\}$ under the null,
with some ambiguity of notations,
we still use $\bs{\widehat{\beta}}^0$ and $\bs{\beta}$
to denote vectors $(\widehat{\beta}^0_1, \widehat{\beta}_{r+1}^0, \ldots, \widehat{\beta}_n^0)^\top$ and $(\beta_1, \beta_{r+1}, \ldots, \beta_n)^\top$, respectively.
By Lemma 5, if $b_n^6/c_n^5=o( (n/\log n)^{1/2} )$, then $\P(E_n) \ge 1 -2/n$, where
\[
E_n : = \left\{ \| \bs{\widehat{\beta}}^0 - \bs{\beta}  \|_\infty \lesssim  \frac{b_n^3}{c_n^2} \sqrt{ \frac{\log n}{n} } \right\}.
\]
The following calculations are based on the event $E_n$.

A second order Taylor expansion gives that
\begin{eqnarray*}
\mu( 2 \widehat{\beta}_1^0 ) & = & \mu( 2\beta_1) + 2\mu^{\prime}(2\beta_1)( \widehat{\beta}_1 - \beta_1)
+ \frac{1}{2}\cdot 4 \mu^{\prime\prime}( 2\tilde{\beta}_{1} ) ( \widehat{\beta}_1^0 - \beta_1)^2,
\\
\mu( \widehat{\beta}_{ij}^0 ) & = & \mu( \pi_{ij} ) + \mu^{\prime}(\pi_{ij})( \widehat{\pi}_{ij}^0 - \pi_{ij})
+ \frac{1}{2}\cdot \mu^{\prime\prime}( \tilde{\pi}_{ij} ) ( \widehat{\pi}_{ij}^0 - \pi_{ij} )^2,
\end{eqnarray*}
where $\tilde{\pi}_{ij}$ lies between $\pi_{ij}$ and $\widehat{\pi}_{ij}^0$, and, for any $i,j$,
\[
\pi_{ij}=\beta_i+\beta_j, ~~\widehat{\pi}_{ij}^0 = \widehat{\beta}_i^0 + \widehat{\beta}_j^0, ~~
\tilde{\pi}_{ij}=\tilde{\beta}_i + \tilde{\beta}_j.
\]
It follows that
\begin{equation}
\label{eq-homo-dia}
\sum_{i=1}^r d_i - \sum_{i=1}^r \E d_i  =   2r(r-1)\mu^{\prime}(2\beta_1)( \widehat{\beta}_1^0 - \beta_1) + r\sum_{j=r+1}^n \mu^{\prime}(\pi_{1j})( \widehat{\pi}_{1j}^0 - \pi_{1j})
+  h_1,
\end{equation}
and, for $i=r+1, \ldots, n$,
\begin{equation}
\label{eq-homo-dib}
d_i - \E d_i  =   r\mu^{\prime}( \pi_{i1})(\widehat{\pi}_{i1} - \pi_{i1} ) +
\sum_{j=r+1, j\neq i}^n \mu^{\prime}( \pi_{ij})(\widehat{\pi}_{ij} - \pi_{ij} ) + h_i,
\end{equation}
where
\begin{eqnarray*}
h_1 = 2r(r-1) \mu^{\prime\prime}( 2\tilde{\pi}_{11} ) ( \widehat{\beta}_1^0 - \beta_1)^2
+r \sum_{j=r+1}^n \frac{1}{2} \mu^{\prime\prime}( \tilde{\pi}_{1j} ) ( \widehat{\beta}^0_1 - \beta_{1} )^2
\\
h_i = r \mu^{\prime\prime}( \tilde{\pi}_{i1} ) ( \widehat{\pi}_{i1}^0 - \pi_{i1})^2
+ \sum_{j=r+1}^n \frac{1}{2} \mu^{\prime\prime}( \tilde{\pi}_{ij} ) ( \widehat{\pi}_{ij} - \pi_{ij} )^2,~~i=r+1, \ldots, n.
\end{eqnarray*}
Writing \eqref{eq-homo-dia} and \eqref{eq-homo-dia} into a matrix form, we have
\begin{equation}\label{eq-expression-beta-star}
\bs{\widetilde{d}} - \E \bs{\widetilde{d}} = \widetilde{V}( \bs{\widehat{\beta}}^0 - \bs{\beta} ) + \bs{\tilde{h}},
\end{equation}
where $\bs{\tilde{h}}=(h_1, h_{r+1}, \ldots, h_n)^\top$.  It is equivalent to
\[
\bs{\widehat{\beta}}^0 - \bs{\beta} = \widetilde{V}^{-1} ( \bs{\widetilde{d}} - \E \bs{\widetilde{d}} ) - \widetilde{V}^{-1}\bs{\tilde{h}}.
\]
In view of that $\max_{ij}|\mu^{\prime\prime}(\pi_{ij})| \le 1/c_n$ and the event $E_n$, we have
\begin{equation}\label{ineq-home-be-h1}
|h_1|  \lesssim  \frac{rn}{c_n} \|  \bs{\widehat{\beta}}^0 - \bs{\beta} \|_\infty^2
 \lesssim  \frac{ rb_n^6\log n }{c_n^5 }.
\end{equation}
and, for $k=r+1, \ldots, n$,
\begin{equation}\label{ineq-home-be-hk}
|h_k| \lesssim  \frac{n}{c_n} \|  \bs{\widehat{\beta}}^0 - \bs{\beta} \|_\infty^2 \lesssim  \frac{ b_n^6\log n }{ c_n^5 }.
\end{equation}

By letting $\widetilde{V} = \widetilde{S} + \widetilde{W}$, in view of (35), \eqref{ineq-home-be-h1} and \eqref{ineq-home-be-hk}, we have
\begin{eqnarray}
\nonumber
\| \widetilde{V}^{-1}\bs{h} \|_\infty & \lesssim & \frac{ |h_1| }{ \tilde{v}_{11} } + \max_{i=r+1, \ldots, n} \frac{ |h_i| }{ v_{ii} }
+ \| \widetilde{W} \|_{\max} \left(\sum_{i=r+1}^n |h_i| + |h_1| \right) \\
\nonumber
& \lesssim & \frac{ b_n^6 \log n}{c_n^5 } \cdot \frac{ b_n }{n } + \frac{ b_n^3 }{ n^2 c_n^2 } \cdot \left\{ \frac{rb_n^6\log n }{ c_n^5}
+ (n-r)  \frac{ b_n^6 \log n}{c_n^5 } \right\} \\
\label{ineq-home-be-vh}
& \lesssim & \frac{ b_n^9 \log n}{ nc_n^7 }.
\end{eqnarray}
Now, we bound the error term $\| \widetilde{W} (\bs{\widetilde{d}}- \E \bs{\widetilde{d}}) \|_\infty$.
Note that
\begin{align*}
 &[\widetilde{W} (\bs{\widetilde{d}}- \E \bs{\widetilde{d}})]_i \\
 = & \tilde{w}_{i1} \sum_{j=1}^r \bar{d}_j + \sum_{j=r+1}^n \tilde{w}_{ij}\bar{d}_j, \\
 = & \tilde{w}_{i1} \sum_{1\le k<j \le r} (\bar{a}_{kj}+\bar{a}_{jk}) + \sum_{k=1}^r \sum_{j=r+1}^n (\tilde{w}_{i1}\bar{a}_{kj} +\tilde{w}_{jk}\bar{a}_{jk})
 \\
& +  \sum_{r+1 \le k < j \le n} ( \tilde{w}_{kj} \bar{a}_{kj} + \tilde{w}_{jk} \bar{a}_{jk} ).
\end{align*}
The summation of the above right hand can be viewed as the sum of $n(n-1)/2$ independent random variables.
Because $\E \bar{a}_{ij}^2 \le 1/c_n$, we have
\begin{align*}
& \E \{[\widetilde{W} (\bs{\tilde{d}}- \E \bs{\tilde{d}})]_i \}^2 \\
\le & \left\{\frac{r(r-1)}{2c_n} + \frac{4r(n-r)}{c_n} + \frac{4(n-r)(n-r+1)}{c_n} \right\}\|\widetilde{W} \|_{\max}^2 \\
 \lesssim & \frac{ n^2 }{c_n} \|\widetilde{W} \|_{\max}^2.
\end{align*}
It follows from Bernstein's inequality in Lemma \ref{lemma:bernstein} and inequality (35), with probability $1- N^{-2}$,  we have that
\begin{eqnarray*}
| [\widetilde{W} (\bs{\tilde{d}}- \E \bs{\tilde{d}})]_i | & \le & \sqrt{ 2\log N \left( \E \{[\widetilde{W} (\bs{\tilde{d}}- \E \bs{\tilde{d}})]_i \}^2   \right)} + \frac{2}{3} \cdot  \frac{ b_n }{ (n-1)c_n } \cdot \log n \\
& \lesssim & \frac{ b_n^3}{ n^2 c_n^2 } \cdot \frac{n (\log n)^{1/2}}{ c_n^{1.2}} \\
& \lesssim & \frac{ b_n^3(\log n)^{1/2} }{ n c_n^{5/2} },
\end{eqnarray*}
where $N=n(n-1)/2$.
By the uniform bound,  with probability at leas $1- 4/(n-1)^3$, we have
\begin{equation}\label{eq-hatbeta-bb}
\| \widetilde{W} (\bs{\tilde{d}}- \E \bs{\tilde{d}}) \|_\infty \lesssim \frac{ b_n^3(\log n)^{1/2} }{ n c_n^{5/2} }.
\end{equation}

By combining \eqref{eq-expression-beta-star}, \eqref{ineq-home-be-vh} and  \eqref{eq-hatbeta-bb}, with probability at least $1- O(n^{-1})$, we have
\begin{eqnarray}
\label{expan-hatbeta-hoa}
\widehat{\beta}_1^0 - \beta_1 & = & \frac{ \sum_{i=1}^r \bar{d}_i }{ \tilde{v}_{11} }+g_1, \\
\label{expan-hatbeta-hob}
\widehat{\beta}_i^0 - \beta_i & = & \frac{ \bar{d}_i }{ v_{ii} } + g_i, i=r+1, \ldots, n,
\end{eqnarray}
where $g_1, g_{r+1}, \ldots, g_n$  simultaneously satisfy
\[
g_i = (\widetilde{V}^{-1}\bs{h})_i + [\widetilde{W} (\bs{\tilde{d}}- \E \bs{\tilde{d}})]_i = O\left( \frac{b_n^9 \log n}{ nc_n^7 } \right).
\]
\end{proof}

\subsection{Proof of (39)}
\label{section-proof-39-B20}
Let $\tilde{\ell}( \bs{{\beta}} )=-\ell( \bs{{\beta}} )$.
The expression of $B_2^0$ can be written as
\begin{eqnarray*}
-B_2^0 & = & \underbrace{\frac{ \partial^3 \tilde{\ell}( \bs{{\beta}} ) }{ \partial \beta_1^3 } (\widehat{\beta}_1^0 - \beta_1)^3}_{Q_1}
+ 3\underbrace{\sum_{i=r+1}^n \frac{  \partial^3 \tilde{\ell}( \bs{{\beta}} ) }{ \partial \beta_1^2 \partial \beta_i }
( \widehat{\beta}_1 - \beta_1)^2( \widehat{\beta}_i - \beta_i )}_{Q_2} \\
&&+ 3 \underbrace{\sum_{i,j=r+1}^n \frac{  \partial^3 \tilde{\ell}( \bs{{\beta}} ) }{ \partial \beta_1 \partial \beta_i \partial \beta_j}
( \widehat{\beta}_1 - \beta_1)( \widehat{\beta}_i - \beta_i ) ( \widehat{\beta}_j - \beta_j )}_{Q_3} \\
&&+ \underbrace{\sum_{i,j,k=r+1}^n \frac{  \partial^3 \tilde{\ell}( \bs{{\beta}} ) }{ \partial \beta_i \partial \beta_j \partial \beta_k}
( \widehat{\beta}_i - \beta_i)( \widehat{\beta}_j - \beta_j ) ( \widehat{\beta}_k - \beta_k )}_{Q_4},
\end{eqnarray*}
where
\begin{eqnarray*}
Q_1 & = & \{ 4r(r-1)\mu^{\prime\prime}(\pi_{11})+r\sum_{j=r+1}^n \mu^{\prime\prime}(\pi_{ij}) \} (\widehat{\beta}_1^0 - \beta_1)^3, \\
Q_2&=& r \sum_{i=r+1}^n \mu^{\prime\prime}( \pi_{1i}) ( \widehat{\beta}_1^0  - \beta_1)^2 ( \widehat{\beta}_i^0 - \beta_i),  \\
Q_3&=& r \sum_{i=r+1}^n \mu^{\prime\prime}(\pi_{1i}) ( \widehat{\beta}_1^0  - \beta_1) ( \widehat{\beta}_i^0 - \beta_i)^2, \\
Q_4& = & \sum_{i,j=r+1, i\neq j}^n \mu^{\prime\prime}(\pi_{ij}) ( \widehat{\beta}_i^0  - \beta_i) ( \widehat{\beta}_j^0 - \beta_j)^2.
\end{eqnarray*}

We shall in turn bound each term in the above summation.
To simplify notations, let
\begin{eqnarray*}
f_1 & = & \frac{ \partial^3 \tilde{\ell}( \bs{\beta} ) }{ \partial \beta_1^3 }=4r(r-1)\mu^{\prime\prime}(\pi_{11})
+ r\sum_{j=r+1}^n \mu^{\prime\prime}(\pi_{1j}),  \\
f_{1j} & = & \frac{  \partial^3 \tilde{\ell}( \bs{\beta} ) }{ \partial \beta_1 \partial \beta_j^2 }=
r\mu^{\prime\prime}(\pi_{1j}),~~j=r+1,\ldots, n, \\
f_{ij} & = & \frac{  \partial^3 \tilde{\ell}( \bs{\beta} ) }{ \partial \beta_i^2 \partial \beta_j },~~i,j=r+1,\ldots, n.
\end{eqnarray*}
In view of \eqref{ineq-mu-deriv-bound}, we have
\begin{equation}\label{ineq-f1-fj-fij}
|f_1| \lesssim \frac{rn}{c_n}, \quad |f_{1j}| \lesssim \frac{r}{c_n}, \quad |f_{ij}| \lesssim \frac{1}{c_n}.
\end{equation}
Because $\sum_{i=1}^r \bar{d}_i$ can be expressed as the sum of $r(r-1)/2 + r(n-r)$ independent and bounded random variables,
\[
\sum_{i=1}^r \bar{d}_i = 2\sum_{1\le i<j\le r} \bar{a}_{ij} + \sum_{i=1}^r \sum_{j=r+1}^n \bar{a}_{ij}
\]
and
\[
\E (2\sum_{1\le i<j\le r} \bar{a}_{ij} + \sum_{i=1}^r \sum_{j=r+1}^n \bar{a}_{ij})^2
\le \frac{2r(r-1)}{c_n} + \frac{r(n-r)}{c_n},
\]
by Bernstern's inequality, with probability at least $ 1- 2(rn)^{-2}$, we have
\begin{equation}\label{eq-sum-hom-d}
\left| \sum_{i=1}^r \bar{d}_i \right| \lesssim \sqrt{ 2\log (rn) \times \frac{2r(r-1)+r(n-r)}{c_n} } + \frac{12}{3}\log n \lesssim
\sqrt{\frac{rn\log n}{c_n}}.
\end{equation}
By \eqref{expan-hatbeta-hoa} and \eqref{eq-sum-hom-d}, we have
\begin{eqnarray}
\nonumber
| f_1(\widehat{\beta}_1^0 - \beta_1)^3 | & = & \frac{rn}{c_n}  \left|\left( \frac{ \sum_{i=1}^r \bar{d}_i }{ \tilde{v}_{11} } + g_1 \right)^3 \right|\\
\nonumber
& = &  \frac{rn}{c_n}  \left|
 ( \frac{ \sum_{i=1}^r \bar{d}_i }{ \tilde{v}_{11} })^3 + 3 ( \frac{ \sum_{i=1}^r \bar{d}_i }{ \tilde{v}_{11} })^2 g_1
 + 3 ( \frac{ \sum_{i=1}^r \bar{d}_i }{ \tilde{v}_{11} }) g_1^2 + g_1^3
\right| \\
\nonumber
& \lesssim &  \frac{rn}{c_n} \left\{
\frac{b_n^3}{ (rn)^3}  \cdot (\frac{rn\log n}{c_n})^{3/2} +
\frac{b_n^2}{ (rn)^2} \cdot \frac{rn\log n}{c_n} \cdot \frac{ b_n^9 \log n}{ nc_n^7} \right. \\
\nonumber
&& \left. +
\frac{b_n}{ (rn)} \cdot \sqrt{\frac{rn\log n}{c_n}} \cdot (\frac{ b_n^9 \log n}{ nc_n^7})^2
+ (\frac{ b_n^9 \log n}{ nc_n^7})^3
\right\} \\
\nonumber
& \lesssim &  \frac{ b_n^3 }{ c_n^{5/2} } \cdot \frac{ (\log n)^{1/2} }{ (rn)^{1/2} } + \frac{ b_n^{11} }{ c_n^9 } \cdot \frac{ (\log n)^2 }{ n} \\
\label{ineq-Q1-ho-a}
&&+ \frac{ b_n^{19}}{c_n^{15/2}} \cdot \frac{ (rn)^{1/2}(\log n)^{5/2}}{ n^2 }
+ \frac{ b_n^{27} (\log n)^3r }{ n^2}.
\end{eqnarray}
Therefore, if $b_n^9/c_n^7=o( n^{1/3}/(\log n) )$, then
\begin{equation}\label{ineq-1b-Q1}
Q_1 = O\left( \frac{ b_n^{27} (\log n)^3 }{ n} \right) = o(1).
\end{equation}

We now bound $Q_2$.  By \eqref{expan-hatbeta-hoa} and \eqref{expan-hatbeta-hob}, we have
\begin{eqnarray*}
&&r \left| \sum_{i=r+1}^n \mu^{\prime\prime}(\pi_{1i}) ( \widehat{\beta}_1^0 - \beta_1)^2(\widehat{\beta}_i - \beta_i) \right| \\
& = & r \left| \sum_{i=r+1}^n \mu^{\prime\prime}(\pi_{1i}) ( \frac{ \tilde{d}_1 }{ \tilde{v}_{11}} + g_1)^2
( \frac{\bar{d}_i}{ v_{ii} } + g_i ) \right| \\
& = & r \left| \sum_{i=r+1}^n \mu^{\prime\prime}(\pi_{1i}) ( (\frac{ \tilde{d}_1 }{ \tilde{v}_{11}})^2  +2 \frac{ \tilde{d}_1 }{ \tilde{v}_{11}} g_1 + g_1^2)
( \frac{\bar{d}_i}{ v_{ii} } + g_i ) \right| \\
& \lesssim &  r \sum_{i=r+1}^n |\mu^{\prime\prime}(\pi_{1i})| \left( \left|
(\frac{ \tilde{d}_1 }{ \tilde{v}_{11}})^2\cdot \frac{ \bar{d}_i }{ v_{ii} } \right|
+ \left|(\frac{ \tilde{d}_1 }{ \tilde{v}_{11}})^2 \cdot g_i \right| 
+ \left| \frac{ \bar{d}_i }{ v_{ii}} \cdot g_1^2 \right| + \left| g_1^2 g_i \right|
\right)\\
& \lesssim &
\frac{ r(n-r)}{ c_n } \left[
\left( \frac{ b_n \sqrt{rn\log n/c_n} }{ rn } \right)^2 \frac{ b_n \sqrt{n\log n} }{ n}
+ \left( \frac{ b_n \sqrt{rn\log n/c_n} }{ rn } \right)^2 \cdot \frac{ b_n^9 \log n}{ nc_n^7 }
\right. \\
&&
\left.
+ \frac{ b_n \sqrt{n\log n}}{ n} \cdot \left( \frac{ b_n^9 \log n}{ nc_n^7 } \right)^2
+ \left( \frac{ b_n^9 \log n}{ nc_n^7 } \right)^3
\right] \\
& \lesssim &  \frac{ b_n^3 (n-r)(\log n)^{3/2} }{ n^{3/2} c_n^2} + \frac{ b_n^{11} (\log n)^2(n-r) }{ n^2c_n^9}\\
&&+ \frac{ b_n^{19} (\log n)^{5/2} r(n-r) }{ n^{5/2} c_n^{15}} + \frac{ b_n^{27} (\log n)^{3}r(n-r)}{ n^3c_n^{21}}.
\end{eqnarray*}
Therefore, if $b_n^9/c_n^7 = o( n^{1/3}/(\log n) )$, then
\begin{equation}\label{eq-hom-Q2}
Q_2 \lesssim \frac{ b_n^{27} (\log n)^{3}}{ n c_n^{21}}  =o(1).
\end{equation}

We now consider $Q_3$.
\begin{eqnarray*}
Q_3&=& r \sum_{i=r+1}^n \mu^{\prime\prime}(\pi_{1i}) ( \widehat{\beta}_1^0  - \beta_1) ( \widehat{\beta}_i^0 - \beta_i)^2 \\
& = & r \sum_{i=r+1}^n \mu^{\prime\prime}(\pi_{1i})
( \frac{ \sum_{k=1}^r \bar{d}_k }{ \tilde{v}_{11} } + g_1 )( \frac{ \bar{d}_i }{ v_{ii} } + g_i )^2 \\
& \lesssim & \frac{ r(n-r)}{ c_n} \left( \frac{ |\sum_{k=1}^r \bar{d}_k| }{ \tilde{v}_{11} } + |g_1| \right)( \frac{ \bar{d}_i }{ v_{ii} } + g_i )^2 \\
& \lesssim & \frac{ r(n-r)}{ c_n} \left( \frac{b_n}{ rn} \cdot \sqrt{ \frac{rn\log n}{ c_n } }  + \frac{ b_n^9 \log n}{ nc_n^7 } \right)( \frac{ b_n^2 n\log n }{ n^2 }  + ( \frac{b_n^9\log n}{ nc_n^7} )^2 ) \\
& \lesssim & \frac{ b_n^3(\log n)^{3/2} r^{1/2} (n-r) }{n^{3/2}c_n} + \frac{ b_n^{19} (\log n)^2 r^{1/2}(n-r) }{ n^{5/2}c_n^{14}}  \\
&&+
\frac{ b_n^{11} (\log n)^2 r(n-r)}{ n^2 c_n^7 }
+ \frac{ (n-r)b_n^{27} (\log n)^3 }{ n^2 c_n^{21}}.
\end{eqnarray*}
If $b_n^9/c_n^7 = o( n^{1/3}/(\log n) )$, then
\[
Q_3 \lesssim \frac{ b_n^3(\log n)^{3/2} }{ c_n} + \frac{ (n-r)b_n^{27} (\log n)^3 }{ n^2 c_n^{21}} = O(\frac{ b_n^3(\log n)^{3/2} }{ c_n})+o(1).
\]

Finally, we bound $Q_4$. It can be written as
\begin{eqnarray*}
Q_4& = & 3\sum_{i,j=r+1;i\neq j}^n \mu^{\prime\prime}( \pi_{ij})
( \widehat{\beta}_i - \beta_i)^2( \widehat{\beta}_j - \beta_j ) \\
&& + 3\sum_{i,j=r+1;i\neq j}^n \mu^{\prime\prime}( \pi_{ij})
( \widehat{\beta}_i - \beta_i)^2( \widehat{\beta}_j - \beta_j )
\end{eqnarray*}
With similar arguments as in the proof of Lemma 5, we have
\begin{eqnarray*}
&& \left| \sum_{i,j=r+1;i\neq j}^n \mu^{\prime\prime}( \pi_{ij})
( \widehat{\beta}_i - \beta_i)^2( \widehat{\beta}_j - \beta_j ) \right | \\
& \lesssim & \frac{ b_n^{27} (\log n)^3 }{ nc_n^{22}}( 1- \frac{r}{n})^2
+ \frac{ b_n^{19} }{ c_n^{15}} \cdot \frac{ (\log n)^{5/2} }{ n^{1/2}} \cdot ( 1- \frac{r}{n})^2 \\
&&+ \frac{ b_n^{11}(\log n)^{2}}{ c_n^8 } ( 1- \frac{r}{n})^2 + \frac{ b_n^3 \log n}{ c_n}( 1- \frac{r}{n}).
\end{eqnarray*}
This gives that
\begin{eqnarray*}
\frac{Q_4}{r^{1/2}} & \lesssim & \frac{1}{r^{1/2}} \left( \frac{ b_n^{27} (\log n)^3 }{ nc_n^{22}}( 1- \frac{r}{n})^2
+ \frac{ b_n^{19} }{ c_n^{15}} \cdot \frac{ (\log n)^{5/2} }{ n^{1/2}} \cdot ( 1- \frac{r}{n})^2 \right. \\
&& \left.+ \frac{ b_n^{11}(\log n)^{2}}{ c_n^8 } ( 1- \frac{r}{n})^2 + \frac{ b_n^3 \log n}{ c_n}( 1- \frac{r}{n}) + \frac{ b_n^{12} (\log n)^2 (n-r)^2 }{ n^2c_n^9} \right)
\end{eqnarray*}

If
\[
\frac{ b_n^{12} }{ c_n^9 } = o\left( \frac{ r^{1/2} }{ (\log n)^3 } \right), \mbox{~~and~~} \frac{ b_n^{19}}{ c_n^{14} } = o\left( \frac{n^{1/2}}{ (\log n)^3 } \right),
\]
then
\[
\max\left\{ \frac{|Q_1|}{ r^{1/2} }, \frac{|Q_2|}{ r^{1/2} }, \frac{|Q_3|}{ r^{1/2} }, \frac{|Q_4|}{ r^{1/2} } \right\} = o_p(1),
\]
which shows (39) in the main text.

\section{Proofs of supported Lemmas in the proof of Theorem 2 (a)}
\label{section-theorem2a}

This section presents the proofs of supported Lemmas in the proof of Theorem 2 (a) and two vanishing remainder terms.
This section is organized as follows.
Sections \ref{section-proof-lemma1010},  \ref{section-proof-lemma11} and \ref{section-proof-lemma12}
present the proofs of Lemmas 10, 11 and 12, respectively.
Sections \ref{subsection:B2B20} and \ref{subsection-B3B30} presents the proofs of orders of two remainder terms
$B_2-B_2^0$  in (59) and $B_3-B_3^0$ in (60) in the main text, respectively.

\subsection{Proof of Lemma 10}
\label{section-proof-lemma1010}

This section presents the proof of Lemma 10.

\begin{proof}[Proof of Lemma 10]
Note that
\[
\left[
\begin{pmatrix} S_{11} & \mathbf{0} \\
\mathbf{0} & S_{22}
\end{pmatrix}
+
\begin{pmatrix} W_{11} & W_{12} \\
W_{21} & W_{22}
\end{pmatrix}
\right]
\begin{pmatrix}
V_{11} & V_{12} \\
V_{21} & V_{22}
\end{pmatrix}
=I_{n\times n},
\]
where $V_{11}$ is the upper left $r\times r$ sub-matrix of $V$.
Because
\[
W_{21} V_{12} + W_{22}V_{22} + S_{22}V_{22} = I_{(n-r)\times (n-r) },
\]
and
\[
(\widetilde{W}_{22}+S_{22})V_{22} = I_{(n-r)\times (n-r) },
\]
we have
\[
W_{21}V_{12} + W_{22}V_{22} = \widetilde{W}_{22}V_{22} \Longrightarrow W_{22} - \widetilde{W}_{22}= - V_{22}^{-1}W_{21}V_{12},
\]
where
\[
\widetilde{W}_{22} = V_{22}^{-1} - S_{22}.
\]
With the similar arguments as in the proof of \eqref{ineq-V-S-appro-upper-b}, we have
\[
\| \widetilde{W}_{22} \|_{\max} \lesssim \frac{b_n^3}{n^2c_n^2}.
\]
Note that $r$ is a fixed positive integer.
A direct calculation gives that
\begin{eqnarray*}
|(S_{22} W_{21} V_{12})_{ij}| &  = & |\sum_{k=1}^{n-r}\sum_{h=1}^r (S_{22})_{ik} (W_{21})_{k h} (V_{12})_{hj}| \\
& = & | \sum_{h=1}^r \frac{ 1}{v_{i+r,i+r}} (W_{21})_{i h} (V_{12})_{hj} |\\
& \lesssim &  r \cdot \frac{ b_n }{n-1} \cdot \frac{ b_n^3 }{ n^2c_n^2 } \cdot \frac{1}{c_n} \lesssim \frac{b_n^4}{ n^3c_n^3},
\end{eqnarray*}
and
\begin{eqnarray*}
|(W_{22} W_{21} V_{12})_{ij}| & = & |\sum_{k=1}^{n-r}\sum_{h=1}^r (W_{22})_{ik} (W_{21})_{k h} (V_{12})_{hj}| \\
& \le & (n-r)r \cdot \| W \|_{\max}^2 \cdot \frac{1}{c_n} \lesssim \frac{ b_n^6 }{ n^3c_n^5 } .
\end{eqnarray*}
This shows that
\begin{equation}\label{ineq-W-diff-upper}
\| W_{22} - \widetilde{W}_{22} \|_{\max} \lesssim \frac{ b_n^6 }{ n^3c_n^5 },
\end{equation}
which has a much smaller error in contract to $\| W_{22} \|_{\max}$ and $\| \widetilde{W}_{22} \|_{\max}$ whose magnitudes are $b_n^3/(n^2c_n^2)$.
\end{proof}

\subsection{Proof of Lemma 11}
\label{section-proof-lemma11}

In this section, we present the proof of Lemma 11.
\begin{proof}[Proof of Lemma 11]
Note that $r$ is a fixed constant, $\bs{\bar{d}}_1= (\bar{d}_1, \ldots, \bar{d}_r)^\top$, and
$\bs{\bar{d}}_2 = (\bar{d}_{r+1}, \ldots, \bar{d}_n)^\top$.

(a) We  bound $\bs{\bar{d}}_1^\top W_{11} \bs{\bar{d}}_1$.
By \eqref{ineq-d-upper}, with probability at least $1-2/n$, we have
\[
\|\bs{\bar{d}}\|_\infty \le \sqrt{n\log n}.
\]
It follows that, by \eqref{ineq-V-S-appro-upper-b},
\begin{eqnarray*}
\bs{\bar{d}}_1^\top W_{11} \bs{\bar{d}}_1 & \lesssim  & \frac{ b_n^3 }{ n^2 c_n^2} \cdot (n\log n)^{1/2} \lesssim \frac{ b_n^3 (\log n)^{1/2} }{ n^{3/2}c_n^2}
\end{eqnarray*}

(b) We bound $\bs{\bar{d}}_1^\top W_{12} \bs{\bar{d}}_2$.
Note that
\[
\E \bs{\bar{d}}^\top (V^{-1}-S) \bs{\bar{d}} = \E \mathrm{tr}( V^{-1}-S) \bs{\bar{d}} \bs{\bar{d}}^\top )
= \mathrm{tr}( I - VS) = 0.
\]
It follows that
\[
\E \bs{\bar{d}}_1^\top W_{12} \bs{\bar{d}}_2 =0.
\]
Now, we calculate
\[
\mathrm{Var}( \bs{\bar{d}}_1^\top W_{12} \bs{\bar{d}}_2 ) = \sum_{i=1}^r \sum_{j=r+1}^n \sum_{\alpha=1}^r \sum_{\gamma=r+1}^n
\mathrm{Cov}( \bar{d}_i w_{ij}\bar{d}_j, \bar{d}_\alpha w_{\alpha\gamma} \bar{d}_\gamma).
\]
Note that
\[
|\mathrm{Cov}( \bar{d}_i w_{ij}\bar{d}_j, \bar{d}_\alpha w_{\alpha\gamma} \bar{d}_\gamma)|
\le \| W \|_{\max}^2 |\mathrm{Cov}( \bar{d}_i \bar{d}_j, \bar{d}_\alpha  \bar{d}_\gamma)|
\]
We evaluate $\mathrm{Cov}( \bar{d}_i \bar{d}_j, \bar{d}_\alpha  \bar{d}_\gamma)$ according to four cases:
(Case A) $i=\alpha\in\{1,\ldots, r\}$, $j=\gamma \in \{r+1,\ldots, n\}$;
(Case B) $i=\alpha\in\{1,\ldots, r\}$, $j \neq \gamma \in \{r+1,\ldots, n\}$;
(Case C) $i\neq \alpha\in\{1,\ldots, r\}$, $j=\gamma \in \{r+1,\ldots, n\}$;
(Case D) $i\neq \alpha\in\{1,\ldots, r\}$, $j\neq \gamma \in \{r+1,\ldots, n\}$.\\
Case A: the expression of $\mathrm{Cov}( \bar{d}_i \bar{d}_j, \bar{d}_i \bar{d}_j)$ is
\begin{eqnarray*}
\mathrm{Cov}( \bar{d}_i \bar{d}_j, \bar{d}_i \bar{d}_j) & = & \sum_s \sum_t \sum_\eta \sum_\zeta
( \E \bar{a}_{is} \bar{a}_{jt} \bar{a}_{i\eta} \bar{a}_{j\zeta} -
\E \bar{a}_{is} \bar{a}_{jt} \E \bar{a}_{i\eta} \bar{a}_{j\zeta} ) \\
& = & \sum_s \sum_t ( \E \bar{a}_{is} \bar{a}_{jt} \bar{a}_{is} \bar{a}_{jt} -
\E \bar{a}_{is} \bar{a}_{jt} \E \bar{a}_{is} \bar{a}_{jt} ).
\end{eqnarray*}
Case B: the expression of $\mathrm{Cov}( \bar{d}_i \bar{d}_j, \bar{d}_i \bar{d}_\gamma)$ is
\begin{eqnarray*}
\mathrm{Cov}( \bar{d}_i \bar{d}_j, \bar{d}_i \bar{d}_\gamma) & = & \sum_s \sum_t \sum_\eta \sum_\zeta
( \E \bar{a}_{is} \bar{a}_{jt} \bar{a}_{i\eta} \bar{a}_{\gamma \zeta} -
\E \bar{a}_{is} \bar{a}_{jt} \E \bar{a}_{i\eta} \bar{a}_{\gamma\zeta} ) \\
& = & \sum_s ( \E \bar{a}_{is} \bar{a}_{j\gamma} \bar{a}_{is} \bar{a}_{\gamma j} -
\E \bar{a}_{is} \bar{a}_{j\gamma} \E \bar{a}_{is} \bar{a}_{j\gamma} )
+ \E \bar{a}_{i\gamma}\bar{a}_{\gamma i} \bar{a}_{ij}\bar{a}_{ji}.
\end{eqnarray*}
Case C: the expression of $\mathrm{Cov}( \bar{d}_i \bar{d}_j, \bar{d}_\alpha \bar{d}_j)$ is
\begin{eqnarray*}
\mathrm{Cov}( \bar{d}_i \bar{d}_j, \bar{d}_\alpha \bar{d}_j) & = & \sum_s \sum_t \sum_\eta \sum_\zeta
( \E \bar{a}_{is} \bar{a}_{jt} \bar{a}_{\alpha\eta} \bar{a}_{j \zeta} -
\E \bar{a}_{is} \bar{a}_{jt} \E \bar{a}_{\alpha\eta} \bar{a}_{j \zeta} ) \\
& = & \sum_t ( \E \bar{a}_{i\alpha} \bar{a}_{j t} \bar{a}_{\alpha i} \bar{a}_{j t} -
\E \bar{a}_{i\alpha} \bar{a}_{j t} \E \bar{a}_{\alpha i} \bar{a}_{j t})
+ \E \bar{a}_{i j}\bar{a}_{j i} \bar{a}_{\alpha j}\bar{a}_{j\alpha}.
\end{eqnarray*}
Case D: the expression of $\mathrm{Cov}( \bar{d}_i \bar{d}_j, \bar{d}_\alpha \bar{d}_\gamma)$ is
\begin{eqnarray*}
\mathrm{Cov}( \bar{d}_i \bar{d}_j, \bar{d}_\alpha \bar{d}_\gamma) & = & \sum_s \sum_t \sum_\eta \sum_\zeta
( \E \bar{a}_{is} \bar{a}_{jt} \bar{a}_{\alpha\eta} \bar{a}_{\gamma \zeta}) -
\E \bar{a}_{is} \bar{a}_{jt} \E \bar{a}_{\alpha\eta} \bar{a}_{\gamma \zeta} ) \\
& = & 0.
\end{eqnarray*}
By combining the above four cases, it yields
\[
\mathrm{Var}( \bs{\bar{d}}_1^\top W_{12} \bs{\bar{d}}_2 ) \le r(n-r)\cdot \|W\|_{\max}^2 \cdot n^2 \max_{i,j}(\E \bar{a}_{ij}^2)^2
\lesssim n\cdot \left( \frac{b_n^3 }{ n^2 c_n^2} \right)^2 \cdot \frac{n^2}{c_n^2} \lesssim \frac{b_n^6}{c_n^2n}.
\]
By Chebyshev's inequality, we have
\[
\bs{\bar{d}}_1^\top W_{12} \bs{\bar{d}}_2 = O_p\left( \sqrt{\frac{b_n^6}{nc_n^6} } \right).
\]

(c) We bound $\bs{\bar{d}}_2^\top (W_{22}-\widetilde{W}_{22}) \bs{\bar{d}}_2$.
By Lemma \ref{lemma:var:quadra}, we have that
\[
\mathrm{ Var}( \sum_i \bar{d}_i^2 ) \lesssim n \sum_i v_{ii}^2 \lesssim \frac{ n^3 }{c_n^2}.
\]
Chebychev's inequality gives that
\[
| \sum_i \bar{d}_i^2 - \E \sum_i \bar{d}_i^2 | = O_p\left( \sqrt{ \frac{ n^3 }{c_n^2} } \right).
\]
It follows that
\[
| \bs{\bar{d}}_2^\top (W_{22}-\widetilde{W}_{22}) \bs{\bar{d}}_2 |
\lesssim \| W_{22}-\widetilde{W}_{22} \|_{\max} \cdot \bs{\bar{d}}_2^\top \bs{\bar{d}}_2
\lesssim \frac{ b_n^6 }{ c_n^5 n^{3/2} }\cdot \sqrt{ \frac{ n^3 }{c_n^2} } \lesssim \frac{ b_n^6 }{ n^{3/2}c_n^6 }.
\]
\end{proof}

\subsection{Proof of Lemma 12}
\label{section-proof-lemma12}

This section presents the proof of Lemma 12.

\begin{proof}[Proof of Lemma 12]

In (30) and (33), we have shown
\begin{eqnarray}
\label{ineq-2a-hatbeta}
\bs{\widehat{\beta}} - \bs{\beta} & = & V^{-1} \bs{\bar{d}} + V^{-1} \bs{h} \\
\label{ineq-2a-hatbeta-zero}
\bs{\widehat{\beta}}^0_2 - \bs{\beta}_2 & = & V_{22}^{-1} \bs{\bar{d}}_2 + V_{22}^{-1} \bs{\tilde{h}}_2,
\end{eqnarray}
where
\[
\| V^{-1} \bs{h} \|_\infty \lesssim \frac{ b_n^3\log n }{ nc_n},~~ \| V_{22}^{-1} \bs{\tilde{h}}_2 \|_\infty \lesssim \frac{ b_n^3\log n }{ nc_n}.
\]
Subtracting both sides in \eqref{ineq-2a-hatbeta-zero} from \eqref{ineq-2a-hatbeta} over $i=r+1, \ldots, n$ yields
\begin{eqnarray*}
\widehat{\beta}_i - \widehat{\beta}_i^0 & =  & (V^{-1} \bs{\bar{d}})_i - (V_{22}^{-1} \bs{\bar{d}}_2)_i + O\left( \frac{ b_n^3 }{ nc_n} \right) \\
& = & \sum_{j=1}^n W_{ij} \bar{d}_j - \sum_{j=r+1}^n (\widetilde{W}_{22})_{(i-r)(j-r)} \bar{d}_j + O\left( \frac{ b_n^3 }{ nc_n} \right),
\end{eqnarray*}
where the second equality is due to $V^{-1}=S+W$ in \eqref{ineq-V-S-appro-upper-b} and $V_{22}^{-1} = S_{22} + \widetilde{W}_{22}$ in
\eqref{ineq-V22-S22-app}.
By \eqref{eq-hatbeta-bb}, we have
\[
\| W \bs{\bar{d}} \|_\infty \lesssim O_p( \frac{ 2.1b_n^2  \sqrt{\log n} }{ nc_n^{3/2} } ),~~
\| \widetilde{W}_{22} \bs{\bar{d}}_2\|_\infty \lesssim O_p( \frac{ 2.1b_n^2  \sqrt{\log n} }{ nc_n^{3/2} } ).
\]
It completes the proof.
\end{proof}

\subsection{The proof of (59) for bounding $B_2-B_2^0$}
\label{subsection:B2B20}

In this section, we prove inequality (59), reproduced below:
\begin{equation}\label{ineq-2a-62}
B_2 - B_2^0 \lesssim \frac{ b_n^9 (\log n)^3 }{ n^{1/2} c_n^3 }.
\end{equation}

Before proving \eqref{ineq-2a-62}, we show one lemma.
Note that in (30), we show
\begin{equation}\label{betabeta}
\widehat{\beta}_i - \beta_i = (V^{-1} \bs{\bar{d}})_i + (V^{-1} \bs{h} )_i,
\end{equation}
where
\[
h_i = \sum_{j\neq i} h_{ij} =  \sum_{j\neq i} \frac{1}{2} \mu^{\prime\prime}( \tilde{\pi}_{ij} ) ( \widehat{\pi}_{ij} - \pi_{ij})^2,
\]
and $\widetilde{\pi}_{ij}$ lies between $\pi_{ij}=\beta_i+\beta_j$ and $\widehat{\pi}_{ij}=\widehat{\beta}_i+\widehat{\beta}_j$.
In (33), we show
\begin{equation}\label{betabeta0}
\widehat{\beta}_i^0 - \beta_i = (V_{22}^{-1} \bs{\bar{d}}_2)_{i-r} + (V_{22}^{-1} \bs{h}^{0})_{i-r}, ~~~i=r+1, \ldots, n,
\end{equation}
where $\bs{h}^0 : = ( h_{r+1}^0, \ldots, h_n^0)^\top$, and
\[
h_i^0 = \sum_{j\neq i} h_{ij}^0 = \sum_{j\neq i}  \frac{1}{2} \mu^{\prime\prime}( \tilde{\pi}_{ij}^0 ) ( \widehat{\pi}_{ij}^0 - \pi_{ij} )^2,~~i=1,\ldots, n.
\]
and $\tilde{\pi}_{ij}^0$ lies between $\widehat{\pi}_{ij}^0 = \widehat{\beta}_i^0 + \widehat{\beta}_j^0$ and $\pi_{ij}=\beta_i+\beta_j$.

\begin{lemma}\label{lemma-gg0}
If $b_n^3/c_n^2=o( (n/\log n)^{1/2})$, then $g_i - g_i^0$, $i=r+1, \ldots, n$ are bounded by
\begin{equation}\label{ineq-ggg0}
\max_{i=r+1, \ldots, n} | g_i - g_i^0 | \lesssim \frac{ b_n^7 (\log n)^{2} }{ n^{3/2} c_n^2 },
\end{equation}
where
\begin{eqnarray}
\nonumber
g_i & = & \widehat{\beta}_i - \beta_i - \frac{ \bar{d}_i }{v_{ii}} = (W\bs{\bar{d}})_i + (V^{-1}\bs{h})_i, \\
\nonumber
g_i^0 & = & \widehat{\beta}_i^0- \beta_i - \frac{\bar{d}_i }{v_{ii}} = (\widetilde{W}_{22}\bs{\bar{d}}_2)_{i-r} + (V_{22}^{-1} \bs{h}^{0} )_{i-r}.
\end{eqnarray}
\end{lemma}

\begin{proof}
If $b_n^3/c_n^2=o( (n/\log n)^{1/2})$, then with probability at least $1-O(n^{-1})$, we have
\eqref{betabeta} and \eqref{betabeta0}. Note that
\begin{equation}
\label{eq-gg0}
g_i - g_i^0  =  (W\bs{\bar{d}})_i - (\widetilde{W}_{22}\bs{\bar{d}}_2)_{i-r} +
(V^{-1} \bs{h} )_i - (V_{22}^{-1} \bs{h}^{0} )_{i-r}.
\end{equation}
We first bound $(W\bs{\bar{d}})_i - (\widetilde{W}_{22}\bs{\bar{d}}_2)_{i-r}$ over $i=r+1, \ldots, n$.
This term can be represented as
\begin{equation}\label{eq-WW22}
(W\bs{\bar{d}})_i - (\widetilde{W}_{22}\bs{\bar{d}}_2)_{i-r} = \sum_{j=1}^r w_{ij} \bar{d}_j  + \sum_{j=1}^{n-r} (W_{22} - \widetilde{W}_{22})_{i-r, j} \bar{d}_{j+r}.
\end{equation}
Because $r$ is a fixed constant and
\[
\sum_{j=1}^r w_{ij} \bar{d}_j = \sum_{1\le i<j \le r} ( w_{ij}\bar{a}_{ij}+w_{ji}\bar{a}_{ji} )
+ \sum_{i=1}^r \sum_{j=r+1}^n w_{ij}\bar{a}_{ij},
\]
by Bernstein's inequality, with probability at least $1-O(n^{-2})$, the first term in the above right-hand side satisfies
\begin{equation}\label{ineq-ww-aa1}
|\sum_{j=1}^r w_{ij} \bar{d}_j| \lesssim  \frac{ b_n^3 }{ n^2 c_n^2 } \sqrt{n\log n} \lesssim \frac{ b_n^3 (\log n)^{1/2} }{ n^{3/2} c_n^2 }.
\end{equation}
For the second term of the right hand side in \eqref{eq-WW22}, we use Benstern's inequality to bound it. Note that
\begin{align*}
   & \sum_{j=1}^{n-r} (W_{22} - \widetilde{W}_{22})_{i-r, j} \bar{d}_{j+r} \\
  = &  2\sum_{1\le k < j \le n-r} \{(W_{22} - \widetilde{W}_{22})_{ik} \}\bar{a}_{k+r,j+r}
 + \sum_{j=1}^{n-r} \sum_{k=1}^r (W_{22} - \widetilde{W}_{22})_{ij} \bar{a}_{j+r,k}
\end{align*}
Because the terms involved in the sum are independent and bounded,
Benstern's inequality in Lemma \ref{lemma:bernstein} gives that with probability at least $1- O(n^{-2})$, we have
\begin{eqnarray}\nonumber
|[(W_{22} - \widetilde{W}_{22}) \bs{\bar{d}}_2]_i | & \lesssim & q_n \sqrt{ 4 \log n \cdot \frac{n(n-r)}{c_n}}
+ \frac{4}{3} q_n \log n \\
\nonumber
& \lesssim & \frac{b_n^6}{n^3 c_n^5} \cdot \frac{n\log n}{ c_n^{1/2} } \\
\label{ew-WWd}
& \lesssim & \frac{ b_n^6 (\log n)^{1/2} }{ n^2 c_n^{11/2} },
\end{eqnarray}
where
\[
\E \left( \sum_{j=1}^{n-r} (W_{22} - \widetilde{W}_{22})_{i-r, j} \bar{d}_{j+r}  \right)^2 \lesssim \frac{ (n-r)(n-r-1)+r(n-r) q_n^2}{ c_n},
\]
and
\[
q_n =\| W_{22} - \widetilde{W}_{22} \|_{\max} \lesssim \frac{ b_n^6 }{ n^2 c_n}.
\]
By combining \eqref{eq-WW22}, \eqref{ineq-ww-aa1} and \eqref{ew-WWd}, with probability at least $1-O(n^{-2})$, we have
\begin{equation}\label{ineq-W22-2}
\max_{i=r+1,\ldots, n} |(W\bs{\bar{d}})_i - (\widetilde{W}_{22}\bs{\bar{d}}_2)_{i-r} | \lesssim \frac{ b_n^3 (\log n)^{1/2} }{ n^{3/2} c_n^2 } + \frac{ b_n^6 (\log n)^{1/2} }{ n^2 c_n^{11/2} }.
\end{equation}

Now, we bound the second term $(V^{-1}\bs{h})_{i} - (V_{22}^{-1} \bs{h}^{0} )_{i-r}$ in \eqref{eq-gg0}. For $i=r+1, \ldots, n$, observe that
\begin{eqnarray*}
h_i - h_{i}^0 &= & \frac{1}{2} \sum_{j\neq i} [\mu^{\prime\prime} (\tilde{\pi}_{ij})(\widehat{\pi}_{ij} - \pi_{ij} )^2
- \mu^{\prime\prime} (\tilde{\pi}_{ij}^0)(\widehat{\pi}_{ij}^0 - \pi_{ij} )^2 ].
\end{eqnarray*}
With the use of the mean value theorem and Lemma 10, we have
\begin{eqnarray*}
& &| \mu^{\prime\prime} (\tilde{\pi}_{ij})(\widehat{\pi}_{ij} - \pi_{ij} )^2
- \mu^{\prime\prime} (\tilde{\pi}_{ij}^0)(\widehat{\pi}_{ij}^0 - \pi_{ij} )^2 | \\
&\le & | \mu^{\prime\prime} (\tilde{\pi}_{ij}) - \mu^{\prime\prime} (\tilde{\pi}_{ij}^0) |(\widehat{\pi}_{ij} - \pi_{ij} )^2
+ |\mu^{\prime\prime} (\tilde{\pi}_{ij}^0)[(\widehat{\pi}_{ij} - \pi_{ij} )^2-(\widehat{\pi}_{ij}^0 - \pi_{ij} )^2] |
\\
& \le & \left( 4| \mu^{\prime\prime\prime}( \dot{\pi}_{ij})| \| \bs{\widehat{\beta}} - \bs{\beta}\|_\infty^2 +  |\mu^{\prime\prime} (\tilde{\pi}_{ij}^0)|\widehat{\pi}_{ij} - \widehat{\pi}_{ij}^0 | \right)
\times ( | \widehat{\pi}_{ij}^0 - \pi_{ij}| +  |\widehat{\pi}_{ij} - \pi_{ij}| ) \\
&\lesssim & \frac{1}{c_n} \left( b_n \sqrt{\frac{\log n}{n}} \right)^3 + \frac{1}{c_n} \cdot b_n \sqrt{\frac{\log n}{n}} \cdot \frac{ b_n^3 \log n}{ nc_n }\\
& \lesssim & \frac{ b_n^4 (\log n)^{3/2} }{ n^{3/2} c_n^2 }.
\end{eqnarray*}
This gives
\begin{equation}\label{eq-RR0}
|h_i - h_i^0 | \lesssim \frac{ b_n^4 (\log n)^{3/2} }{ n^{1/2} c_n^2 }.
\end{equation}
It follows that
\begin{equation}\label{eq-srr0}
(S\bs{h})_{i} - (S_{22} \bs{h}^0)_{i-r} \le \frac{ | h_i - h_i^0 |}{v_{ii}} \lesssim \frac{ b_n^5 (\log n)^{2}}{n^{3/2}c_n^2}.
\end{equation}
By Lemma 5 and \eqref{eq-RR0}, we have
\begin{eqnarray*}
(W \bs{h} )_{i} - (\tilde{W}_{22} \bs{h}^0)_i & = & \sum_{j=1}^r w_{ij} h_j + \sum_{j=r+1}^n ( w_{ij}h_j - (\widetilde{W}_{22})_{i-r,j-r} h_j^0 ) \\
& = & \sum_{j=1}^r w_{ij} h_j + \sum_{j=r+1}^n ( w_{ij}h_j - w_{ij}h_j^0 + w_{ij}h_j^0 - (\widetilde{W}_{22})_{i-r,j-r} h_j^0 ).
\end{eqnarray*}
Because $r$ is a fixed constant, we have
\[
| \sum_{j=1}^r w_{ij} h_j | \lesssim \frac{ b_n^3 }{ n^2 c_n^2 } \cdot b_n^2 \log n,
\]
\[
|\sum_{j=r+1}^n ( w_{ij}h_j - w_{ij}h_j^0| \lesssim n \cdot \frac{ b_n^3 }{ n^2 c_n^2 } \cdot \frac{ b_n^4 (\log n)^{3/2} }{ n^{1/2} c_n^2 }
\lesssim \frac{ b_n^7 (\log n)^{3/2} }{ n^{3/2} c_n^4 },
\]
\[
\sum_{j=r+1}^n ( w_{ij}h_j^0 - (\widetilde{W}_{22})_{i-r,j-r} h_j^0 )
\lesssim n \cdot \frac{ b_n^6 }{ n^3 c_n^5 } \cdot b_n^2 \log n \lesssim \frac{ b_n^6 \log n}{ n^2 c_n^5 },
\]
and
\begin{eqnarray*}
|(W \bs{h} )_{i} - (\tilde{W}_{22} \bs{h}^0)_i | & \lesssim & \| W_{22}- \widetilde{W}_{22} \|_{\max} \sum_i |R_i - R_i^0 | \\
&\lesssim & \frac{ b_n^6 }{ n^3 c_n} \cdot \frac{ b_n^5 (\log n)^2 }{ n c_n } \cdot n \\
& \lesssim & \frac{ b_n^{11} (\log n)^{2} }{ n^{3}c_n^2}.
\end{eqnarray*}
Consequently,
\[
|(W \bs{h} )_{i} - (\tilde{W}_{22} \bs{h}^0)_i | \lesssim \frac{ b_n^7 (\log n)^{3/2} }{ n^{3/2} c_n^4 }.
\]
By combining \eqref{eq-srr0} and the above inequality, it yields
\begin{equation}\label{ineq-vrr}
|(V^{-1}\bs{h})_i - (V_{22}^{-1} \bs{h}^{0} )_i| \lesssim \frac{ b_n^5 (\log n)^{2}}{n^{3/2}c_n^2} + \frac{ b_n^7 (\log n)^{3/2} }{ n^{3/2} c_n^4 }.
\end{equation}

Because
\begin{eqnarray*}
(V^{-1} \bs{h} )_i - (V_{22}^{-1} \bs{h}^0)_{i-r} & = & (S \bs{h})_i + (W\bs{h})_i - (S_{22} \bs{h}^0 )_{i-r} - ( \widetilde{W}_{22} \bs{h}^0 )_{i-r}, 
\end{eqnarray*}
in view of \eqref{eq-gg0}, \eqref{ineq-vrr} and \eqref{ineq-W22-2}, it yields
\begin{equation}\label{ineq-ggg0}
\max_{i=r+1, \ldots, n} | g_i - g_i^0 | \lesssim \frac{ b_n^7 (\log n)^{2} }{ n^{3/2} c_n^2 }.
\end{equation}
\end{proof}

Now, we are ready to prove \eqref{ineq-2a-62}.

\begin{proof}[Proof of \eqref{ineq-2a-62}]
$B_2 - B_2^0$ can be written as
\begin{eqnarray}
\nonumber
B_2 - B_2^0 & = & \underbrace{ \sum_{i=1}^r  (\widehat{\beta}_i-\beta_i)^3 \sum_{j=1,j\neq i}^n \mu^{\prime\prime}( \pi_{ij}^0 ) }_{T_1}
+ \underbrace{ \sum_{i,j=1, j\neq i}^r  (\widehat{\beta}_i-\beta_i)^2(\widehat{\beta}_j-\beta_j)\mu^{\prime\prime}( \pi_{ij}^0 ) }_{T_2} \\
\nonumber
&& + \underbrace{\left( \sum_{i=1}^r \sum_{j=1, j\neq i}^n + \sum_{j=1}^r \sum_{i=1,i\neq j}^n \right)  (\widehat{\beta}_i-\beta_i)^2(\widehat{\beta}_j-\beta_j)\mu^{\prime\prime}( \pi_{ij}^0 )}_{T_3} \\
\nonumber
&& + \underbrace{ \sum_{i=r}^n  \{ (\widehat{\beta}_i-\beta_i)^3 - (\widehat{\beta}_i^0-\beta_i)^3 \} \sum_{j\neq i} \mu^{\prime\prime}( \beta_i^0 + \beta_j^0 ) }_{T_4} \\
\label{ineq-B2-B20}
&& + \underbrace{ \sum_{i,j=r, j\neq i}^n  \left\{(\widehat{\beta}_i-\beta_i)^2(\widehat{\beta}_j-\beta_j)-(\widehat{\beta}_i^0-\beta_i)^2
(\widehat{\beta}_j^0-\beta_j)\right\} \mu^{\prime\prime}( \pi_{ij}^0 ) }_{T_5}.
\end{eqnarray}
Because $r$ is a fixed constant, the first three terms in the expression of $B_2 - B_2^0$ can be easily bounded by
\[
|T_1+T_2+T_3|\le \frac{1}{c_n} \| \bs{\widehat{\beta}} - \bs{\beta} \|_\infty \left( r(n-1) + r(r-1) + 2r(n-1) \right)
\lesssim \frac{ b_n^3(\log n)^{3/2} }{ n^{1/2} c_n }.
\]
We now bound $T_4$. In view of \eqref{ineq-ggg0}, we have
\begin{eqnarray*}
& & | (\widehat{\beta}_i-\beta_i)^3 - (\widehat{\beta}_i^0 -\beta_i)^3 | \\
&=& | ( \frac{\bar{d}_i}{v_{ii}} + g_i )^3 - ( \frac{\bar{d}_i}{v_{ii}} + g_i^0 )^3 | \\
& \le & 3 \frac{\bar{d}_i^2 }{ v_{ii}^2 }|g_i-g_i^0| + 3 \frac{ |\bar{d}_i| }{ v_{ii} }|g_i^2-(g_i^0)^2| +  |g_i^3 - (g_i^0)^3| \\
& \le &  \frac{ b_n^7 (\log n)^2 }{ n^{3/2} c_n^2 } \times \left(
\frac{ b_n^2 \log n}{ n} + \frac{ b_n^4 (\log n)^{3/2} }{ n^{3/2} } + \frac{ b_n^6 (\log n)^2 }{ n^2 c_n^2 } \right),~~i=r+1, \ldots, n,
\end{eqnarray*}
and
\begin{eqnarray*}
&&(\widehat{\beta}_i-\beta_i^0)^2(\widehat{\beta}_j-\beta_j^0) -  (\widehat{\beta}_i^0-\beta_i^0)^2(\widehat{\beta}_j^0-\beta_j^0) \\
& \le & \frac{ \bar{d}_i^2 }{ v_{ii}^2 } | g_j - g_j^0 |
+ 2 \frac{ \bar{d}_i \bar{d}_j }{ v_{ii} v_{jj} } |g_i - g_i^0 |
+ 2\frac{ \bar{d}_i }{ v_{ii} } | g_i g_j - g_i^0 g_j^0 |
+ \frac{ \bar{d}_j }{ v_{jj}} |g_i^2- (g_i^0)^2 | + |g_i^2 g_j - (g_i^0)^2 g_j^0 |
\\
& \lesssim & \frac{ b_n^7(\log n)^2 }{ n^{3/2} c_n^2 } \left\{
\frac{ b_n^2 (n\log n) }{ n^2 } + \frac{ b_n^3\log n}{nc_n} \cdot \frac{ b_n(n\log n)^{1/2} }{ n }
+  (\frac{ b_n^3\log n}{nc_n})^2
\right\},
\end{eqnarray*}
where the last inequality for terms $g_ig_j-g_i^0g_j^0$  and $g_i^2g_j - (g_i^0)^2 g_j^0$ are due to that
\[
|g_ig_j-g_i^0g_j^0| \le |g_i||g_j- g_j^0| + |g_j^0||g_i-g_i^0| \lesssim \frac{ b_n^3 \log n}{ nc_n} \max_{i=r+1,\ldots, n} | g_i - g_i^0 |,
\]
and
\[
| g_i^2 g_j - (g_i^0)^2 g_j^0| \le  g_i^2 | g_j-g_j^0| + |(g_i^2 - (g_i^0)^2)| |g_j^0|
\le (g_i^2+ |g_ig_j^0|+ | g_i^0 g_j^0|) \max_{i=r+1,\ldots,n} |g_i - g_i^0|.
\]
Therefore, $T_4$ and $T_5$ can be bounded by
\[
T_4 \lesssim \frac{ (n-r)n }{ c_n} \frac{ b_n^7 (\log n)^2 }{ n^{3/2} c_n^2 } \cdot b_n^2 \left( \frac{\log n}{n} \right)
\lesssim \frac{ b_n^9 (\log n)^3 }{ n^{1/2} c_n^3 }
\]
and
\[
T_5 \lesssim \frac{ (n-r)^2 }{ c_n } \cdot \frac{ b_n^7 (\log n)^2 }{ n^{3/2} c_n^2 } \cdot \frac{b_n^4(\log n)^2 }{ n^{3/2}c_n}
\lesssim  \frac{ b_n^{11} (\log n)^4 }{ nc_n^4 }.
\]
By combining inequalities for $T_i, i=1,\ldots, 5$, it yields \eqref{ineq-2a-62}.
\end{proof}

\subsection{The upper bound of $B_3 - B_3^0$ in the proof of Theorem 2 (a)}
\label{subsection-B3B30}

In this section, we present the proof of the error bound of $B_3-B_3^0$ in (60) in the main text,
reproduced below:
\begin{equation}\label{ineq-2a-B330}
|B_3 - B_3^0| \lesssim \frac{b_n^6 (\log n)^{3} }{ n^{1/2}c_n}.
\end{equation}

\begin{proof}[Proof of \eqref{ineq-2a-B330}]
Because
\[
\frac{ \partial^4 \ell( \bs{\beta} )}{ \partial \beta_t \partial \beta_i \partial \beta_j \partial \beta_k } = 0 \mbox{~if there are at least three different indices among $i,j,t,k$},
\]
we have
\begin{equation}
\begin{array}{rcl}
B_3^0 & = &  \sum\limits_{i=1}^n  (\widehat{\beta}_i^0 -\beta_i)^3 \sum\limits_{j\neq i} \left[ \mu^{\prime\prime\prime}( \tilde{\pi}_{ij}^0 )
\left\{(\widehat{\beta}_i^0 -\beta_i)+4(\widehat{\beta}_j^0 -\beta_j)\right\} \right] \\
&& + 6 \sum\limits_{i=1}^n \sum\limits_{j=1, j\neq i}^n  (\widehat{\beta}_i^0-\beta_i)^2(\widehat{\beta}_j^0-\beta_j)^2
\mu^{\prime\prime\prime}( \tilde{\pi}_{ij}^0),
\end{array}
\end{equation}
and
\begin{equation}
\begin{array}{rcl}
B_3 & = &  \sum\limits_{i=1}^n  (\widehat{\beta}_i -\beta_i)^3 \sum\limits_{j\neq i} \left[ \mu^{\prime\prime\prime}( \tilde{\pi}_{ij} )
\left\{(\widehat{\beta}_i -\beta_i)+4(\widehat{\beta}_j -\beta_j)\right\} \right] \\
&& + 6 \sum\limits_{i=1}^n \sum\limits_{j=1, j\neq i}^n  (\widehat{\beta}_i - \beta_i)^2(\widehat{\beta}_j-\beta_j)^2
\mu^{\prime\prime\prime}( \tilde{\pi}_{ij}),
\end{array}
\end{equation}
where $\tilde{\pi}_{ij}=\tilde{\beta}_i + \tilde{\beta}_j$ and $\tilde{\pi}^0_{ij} = \tilde{\beta}_i^0 + \tilde{\beta}_j^0$.
Note that $\widehat{\beta}_i^0 -\beta_i=0$ over $i=1,\ldots, r$.
The difference between $B_3$ and $B_3^0$ can be expressed as the sum of the following four terms:
\begin{eqnarray*}
&&B_3 - B_3^0 \\
& = & \underbrace{\left( \sum\limits_{i=1}^r \sum\limits_{j=1,j\neq i}^n + \sum_{i=r+1}^n \sum_{j=1}^r \right) (\widehat{\beta}_i -\beta_i)^3  \left[ \mu^{\prime\prime\prime}( \tilde{\pi}_{ij} )
\left\{(\widehat{\beta}_i -\beta_i)+4(\widehat{\beta}_j -\beta_j)\right\} \right]}_{C_1} \\
&& + \underbrace{\left( 6 \sum\limits_{i=1}^r \sum\limits_{j=1, j\neq i}^n + 6\sum\limits_{i=r+1}^n \sum\limits_{j=1}^r \right)  (\widehat{\beta}_i - \beta_i)^2(\widehat{\beta}_j-\beta_j)^2
\mu^{\prime\prime\prime}( \tilde{\pi}_{ij})}_{C_2} \\
&& + \underbrace{\sum\limits_{i=r+1}^n \left\{ (\widehat{\beta}_i -\beta_i)^4 \sum\limits_{j=r+1,j\neq i}^n \mu^{\prime\prime\prime}( \tilde{\pi}_{ij} )
- (\widehat{\beta}_i^0 -\beta_i)^4 \sum\limits_{j=r+1,j\neq i}^n \mu^{\prime\prime\prime}( \tilde{\pi}_{ij}^0 ) \right\}}_{C_3}
\\
&&+ \underbrace{4\sum\limits_{i=r+1}^n \sum\limits_{j=r+1,j\neq i}^n \left\{ (\widehat{\beta}_i -\beta_i)^3  \mu^{\prime\prime\prime}( \tilde{\pi}_{ij} )(\widehat{\beta}_j -\beta_j)
- (\widehat{\beta}_i^0 - \beta_i)^3  \mu^{\prime\prime\prime}( \tilde{\pi}_{ij}^0 )(\widehat{\beta}_j^0 -\beta_j)\right\}}_{C_4}
 \\
&&+ \underbrace{6 \sum\limits_{i=r+1}^n \sum\limits_{j=r+1, j\neq i}^n \left\{ (\widehat{\beta}_i - \beta_i)^2(\widehat{\beta}_j-\beta_j)^2
\mu^{\prime\prime\prime}( \tilde{\pi}_{ij})- (\widehat{\beta}_i^0 - \beta_i)^2(\widehat{\beta}_j^0-\beta_j)^2
\mu^{\prime\prime\prime}( \tilde{\pi}_{ij}^0 ) \right\}}_{C_5}.
\end{eqnarray*}
We shall evaluate the above four terms in turn. Notice that $r$ is a fixed constant.
By Lemma , with probability at least, the upper bounds of $C_1$ and $C_2$ satisfies
\begin{equation}\label{remain-C1}
|C_1| \lesssim \frac{rn}{c_n} \| \bs{\widehat{\beta}} - \bs{\beta}\|_\infty^4 \lesssim  \frac{n}{c_n} \left( b_n \sqrt{\frac{\log n}{n}} \right)^4 \lesssim \frac{b_n^4(\log n)^2 }{ nc_n},
\end{equation}
and
\begin{equation}\label{remain-C2}
|C_2| \lesssim \frac{rn}{c_n} \| \bs{\widehat{\beta}} - \bs{\beta}\|_\infty^4 \lesssim  \frac{n}{c_n} \left( b_n \sqrt{\frac{\log n}{n}} \right)^4 \lesssim \frac{b_n^4(\log n)^2 }{ nc_n}.
\end{equation}

Before bounding $C_3$, $C_4$ and $C_5$,
we drive one useful inequality. By finding the fourth derivative of $\mu(x)$ with respect to $x$, we have
\begin{eqnarray*}
\mu^{\prime\prime\prime\prime}(x) & = & \frac{ e^x( 1- 8 e^x + 3e^{2x} ) }{ ( 1 + e^x )^4 } - \frac{ 4 e^{2x} ( 1 - 4 e^x + 4e^{2x} )}{ ( 1+ e^x)^5 } \\
& = & \frac{ e^x ( 1 -11e^x + 11 e^{2x} - e^{3x})  }{ ( 1 + e^x )^5 } \\
& = & \frac{ e^x }{ ( 1+e^x)^2 } \cdot \frac{ 1 -11e^x + 11 e^{2x} - e^{3x} }{ (1+e^x)^3 }.
\end{eqnarray*}
It is easy to see that
\[
\frac{ 11 }{ 3} (1+e^x)^3 \ge 1  + 11e^x + 11e^{2x} + e^{3x} \ge | 1 -11e^x + 11 e^{2x} - e^{3x} |.
\]
It follows that
\[
|\mu^{\prime\prime\prime\prime}(x)| \le \frac{ 11e^x }{ 3( 1+e^x)^2 }.
\]
Therefore, for any $\dot{\pi}_{ij}$ satisfying $|\dot{\pi}_{ij} - \pi_{ij}|\to 0$, we have
\begin{equation}\label{eq-fourth-mu}
| \mu^{\prime\prime\prime\prime}(\dot{\pi}_{ij} )| \le \frac{ 11\mu^{\prime}( \dot{\pi}_{ij} ) }{ 3 } \lesssim \mu^{\prime}( \pi_{ij} ) \lesssim \frac{1}{c_n}
\end{equation}
It follows from the mean value theorem that for any $\dot{\pi}_{ij}$ satisfying $|\dot{\pi}_{ij} - \pi_{ij}|\to 0$,
\begin{equation}\label{ineq-mu-thr-diff}
| \mu^{\prime\prime\prime}( \dot{\pi}_{ij} ) - \mu^{\prime\prime\prime}( \pi_{ij} ) | \lesssim \frac{1}{c_n} |\dot{\pi}_{ij}- \pi_{ij} |.
\end{equation}

By Lemmas 3 and 10, for $i=r+1, \ldots, n$,  we have
\begin{eqnarray}
\nonumber
&&\left|(\widehat{\beta}_i -\beta_i)^4\mu^{\prime\prime\prime}( \tilde{\pi}_{ij} )-
(\widehat{\beta}_i^0 -\beta_i)^4\mu^{\prime\prime\prime}( \tilde{\pi}_{ij}^0) \right| \\
\nonumber
& \le & \left| (\widehat{\beta}_i -\beta_i)^4 \left\{\mu^{\prime\prime\prime}( \tilde{\pi}_{ij} )
-  \mu^{\prime\prime\prime}( \tilde{\pi}_{ij}^0 )\right\} \right| + \left| \left\{(\widehat{\beta}_i -\beta_i)^4
- (\widehat{\beta}_i^0 -\beta_i)^4 \right\} \mu^{\prime\prime\prime}( \tilde{\pi}_{ij}^0 ) \right| \\
\nonumber
& \lesssim & \frac{1}{c_n} \left( | \widehat{\beta}_i - \beta_i |^4 \cdot ( | \widehat{\beta}_i - \beta_i | + |\widehat{\beta}_j - \beta_j | ) \right.\\
\nonumber
&& \left.+ | \widehat{\beta}_i - \widehat{\beta}_i^0|^2  \cdot ( | \widehat{\beta}_i - \beta_i |^2 + | \widehat{\beta}_i^0 - \beta_i|^2 ) \right)\\
\nonumber
& \lesssim & \frac{ 1 }{ c_n }\cdot \left( b_n \sqrt{\frac{\log n}{n}} \right)^5 + \frac{1}{c_n}\left(\frac{ b_n^3 \log n }{ n c_n }\right)^2 \cdot \left( b_n \sqrt{\frac{\log n}{n}} \right)^2 \\
\label{ineq-fourth-C1}
& \lesssim &  \frac{b_n^5(\log n)^{5/2} }{ n^{5/2}c_n} + \frac{ b_n^8 (\log n)^3 }{ n^3 c_n^3},
\end{eqnarray}
where the second inequality is due to \eqref{ineq-mu-thr-diff}.
Similarly, for $i,j=r+1, \ldots, n$, $i\neq j$, we have
\begin{eqnarray}
\nonumber
&&\left| (\widehat{\beta}_i -\beta_i)^3(\widehat{\beta}_j -\beta_j)\mu^{\prime\prime\prime}( \tilde{\pi}_{ij} )
 -(\widehat{\beta}_i^0 -\beta_i)^3(\widehat{\beta}_j^0 -\beta_j)\mu^{\prime\prime\prime}( \tilde{\pi}_{ij}^0 ) \right| \\
\nonumber
& \le & \underbrace{\left| (\widehat{\beta}_i -\beta_i)^3(\widehat{\beta}_j -\beta_j)\mu^{\prime\prime\prime}( \tilde{\pi}_{ij} )
-  (\widehat{\beta}_i^0 -\beta_i)^3(\widehat{\beta}_j^0 -\beta_j)\mu^{\prime\prime\prime}( \tilde{\pi}_{ij} ) \right|}_{ E_1} \\
\nonumber
&&+  \left|(\widehat{\beta}_i^0 -\beta_i)^3(\widehat{\beta}_j^0 -\beta_j)\mu^{\prime\prime\prime}( \tilde{\pi}_{ij} )
 -(\widehat{\beta}_i^0 -\beta_i)^3(\widehat{\beta}_j^0 -\beta_j)\mu^{\prime\prime\prime}( \tilde{\pi}_{ij}^0 ) \right|\\
\label{ineq-fourth-C2}
 & \lesssim & \frac{ 1 }{ c_n }\cdot \left( b_n \sqrt{\frac{\log n}{n}} \right)^5 + \frac{1}{c_n}\left( b_n \sqrt{\frac{\log n}{n}} \right)^3 \cdot \left(\frac{ b_n^3 \log n }{ n c_n }\right),
\end{eqnarray}
where the second inequality for $E_1$ follows from
\begin{eqnarray*}
&&|(\widehat{\beta}_i -\beta_i)^3(\widehat{\beta}_j -\beta_j) - (\widehat{\beta}_i^0 -\beta_i)^3(\widehat{\beta}_j^0 -\beta_j)| \\
& \le & |(\widehat{\beta}_i -\beta_i)^3(\widehat{\beta}_j -\beta_j) - (\widehat{\beta}_i^0 -\beta_i)^3(\widehat{\beta}_j -\beta_j)| \\
&& + |(\widehat{\beta}_i^0 -\beta_i)^3(\widehat{\beta}_j -\beta_j) - (\widehat{\beta}_i^0 -\beta_i)^3(\widehat{\beta}_j^0 -\beta_j)| \\
& \lesssim & | \widehat{\beta}_i - \widehat{\beta}_i^0 |\{ (\widehat{\beta}_i -\beta_i)^2 + (\widehat{\beta}_i^0 -\beta_i)^2 \} | \widehat{\beta}_j -\beta_j | \\
&& + |(\widehat{\beta}_i^0 -\beta_i)^3|\cdot |\widehat{\beta}_j - \widehat{\beta}_j^0| \\
& \lesssim & \left( b_n \sqrt{\frac{\log n}{n}} \right)^3 \cdot \left(\frac{ b_n^3 \log n }{ n c_n }\right).
\end{eqnarray*}
Again, for $i\neq j, i,j=r+1,\ldots, n$, we have
\begin{eqnarray}
\nonumber
&&  (\widehat{\beta}_i^0-\beta_i)^2(\widehat{\beta}_j^0-\beta_j)^2\mu^{\prime\prime\prime}( \tilde{\pi}_{ij}^0)
-(\widehat{\beta}_i - \beta_i)^2(\widehat{\beta}_j-\beta_j)^2\mu^{\prime\prime\prime}( \tilde{\pi}_{ij})  \\
\nonumber
& \le & \underbrace{| (\widehat{\beta}_i^0-\beta_i)^2(\widehat{\beta}_j^0-\beta_j)^2\mu^{\prime\prime\prime}( \tilde{\pi}_{ij}^0)
-(\widehat{\beta}_i - \beta_i)^2(\widehat{\beta}_j-\beta_j)^2\mu^{\prime\prime\prime}( \tilde{\pi}_{ij}^0) |}_{E_2} \\
\nonumber
&& + |(\widehat{\beta}_i-\beta_i)^2(\widehat{\beta}_j-\beta_j)^2\mu^{\prime\prime\prime}( \tilde{\pi}_{ij}^0)
-(\widehat{\beta}_i - \beta_i)^2(\widehat{\beta}_j-\beta_j)^2\mu^{\prime\prime\prime}( \tilde{\pi}_{ij}) \\
\label{ineq-fourth-C3}
& \lesssim & \frac{ 1 }{ c_n }\cdot \left( b_n \sqrt{\frac{\log n}{n}} \right)^5 + \frac{1}{c_n}\left( b_n \sqrt{\frac{\log n}{n}} \right)^3 \cdot \left(\frac{ b_n^3 \log n }{ n c_n }\right),
\end{eqnarray}
where the inequality for $E_2$ follows from
\begin{eqnarray*}
 & & | (\widehat{\beta}_i^0-\beta_i)^2(\widehat{\beta}_j^0-\beta_j)^2 - (\widehat{\beta}_i - \beta_i)^2(\widehat{\beta}_j-\beta_j)^2 | \\
& \le & | (\widehat{\beta}_i^0-\beta_i)^2(\widehat{\beta}_j^0-\beta_j)^2 - (\widehat{\beta}_i - \beta_i)^2(\widehat{\beta}_j^0-\beta_j)^2 | \\
& & + |(\widehat{\beta}_i-\beta_i)^2(\widehat{\beta}_j^0-\beta_j)^2 - (\widehat{\beta}_i - \beta_i)^2(\widehat{\beta}_j-\beta_j)^2 | \\
& \lesssim & | \widehat{\beta}_j^0-\beta_j |^2 | \widehat{\beta}_i -\widehat{\beta}_i^0 | ( |\widehat{\beta}_i^0-\beta_i| + | \widehat{\beta}_i-\beta_i | ) \\
&& + (\widehat{\beta}_i-\beta_i)^2 |\widehat{\beta}_j^0 - \widehat{\beta}_j| (|\widehat{\beta}_j^0-\beta_j|+ |\widehat{\beta}_j-\beta_j | ) \\
& \lesssim & \left( b_n \sqrt{\frac{\log n}{n}} \right)^3 \cdot \left(\frac{ b_n^3 \log n }{ n c_n }\right).
\end{eqnarray*}

By \eqref{ineq-fourth-C1}, we have
\[
|C_3| \lesssim (n-r)^2\cdot \left( \frac{b_n^5(\log n)^{5/2} }{ n^{5/2}c_n} + \frac{ b_n^8 (\log n)^3 }{ n^3 c_n^3} \right) \lesssim \frac{b_n^5(\log n)^{3} }{ n^{1/2}c_n}.
\]
By \eqref{ineq-fourth-C2}, we have
\[
|C_4| \lesssim (n-r)^2 \cdot \left( \frac{b_n^5(\log n)^{5/2} }{ n^{5/2}c_n} +  \frac{ b_n^6 (\log n)^{5/2} }{ n^{5/2} c_n^2 } \right) \lesssim \frac{b_n^6 (\log n)^{5/2} }{ n^{1/2}c_n}.
\]
By \eqref{ineq-fourth-C2}, we have
\[
|C_5| \lesssim (n-r)^2 \cdot \left( \frac{b_n^5(\log n)^{5/2} }{ n^{5/2}c_n} +  \frac{ b_n^6 (\log n)^{5/2} }{ n^{5/2} c_n^2 } \right) \lesssim \frac{b_n^6 (\log n)^{5/2} }{ n^{1/2}c_n}.
\]
By combining the above three inequalities with \eqref{remain-C1} and \eqref{remain-C2}, it yields
\[
|B_3 - B_3^0| \lesssim \frac{b_n^6 (\log n)^{3} }{ n^{1/2}c_n}.
\]
This completes the proof.
\end{proof}

\section{Proof of Lemma 1}
\label{section-lemma1}
This section presents the proof of Lemma 1.

\begin{proof}[Proof of Lemma 1]
The following inequalities will be repeatedly used in the proofs: for any $i\neq j$,
\[
\frac{1}{b_n} \le \E \bar{a}_{ij}^2 \le \frac{1}{c_n},
\]
\[
\frac{1}{b_n} \le |\E \bar{a}_{ij}^3| = p_{ij}(1-p_{ij})|\{ (1-p_{ij})^2 - p_{ij}^2 \}| \le \frac{1}{c_n},
\]
\[
\frac{1}{b_n} \le \E \bar{a}_{ij}^4 = p_{ij}(1-p_{ij})\{ (1-p_{ij})^3 + p_{ij}^3 \} \le \frac{1}{c_n},
\]
where $p_{ij}=\E a_{ij}$. We shall not cite them explicitly.

Note that
\begin{eqnarray}\label{eq-lemma1-a}
\sum_{i=1}^r \frac{ (\bar{d}_i^{\,2} - \E \bar{d}_i^{\,2}) }{ v_{ii} }  & = &  \sum_{i=1}^r \sum_{j=1}^n \frac{ (\bar{a}_{ij}^2 - \E \bar{a}_{ij}^2 ) }{ v_{ii} }
+ \sum_{i=1}^r \sum_{j=1,j\neq i}^n \sum_{k=1, k\neq i,j}^n \frac{ \bar{a}_{ij} \bar{a}_{ik} }{ v_{ii} }.
\end{eqnarray}
By Lemma \ref{lemma:var:quadra}, we have
\begin{equation}\label{eq-lemma1-b}
\lim_{r\to\infty} \frac{1}{2r} \mathrm{Var} \left( \sum_{i=1}^r \frac{ (\bar{d}_i^{\,2} - \E \bar{d}_i^{\,2}) }{ v_{ii} } \right)=1.
\end{equation}
Because $\bar{a}_{ij}$, $i=1, \ldots, r$, $j=i+1,\ldots,n$, are independent, we have
\begin{eqnarray}
\nonumber
\mathrm{Var}\left(\sum_{i=1}^r \sum_{j=1}^n \frac{ \bar{a}_{ij}^2 }{ v_{ii} } \right)
& = & 4\mathrm{Var}\left(\sum_{i=1}^{r-1} \sum_{j=i+1}^{i} \frac{ \bar{a}_{ij}^2 }{ v_{ii} } \right)
+ \sum_{i=1}^r \sum_{j=r+1}^n \mathrm{Var}\left( \frac{ \bar{a}_{ij}^2 }{ v_{ii} } \right) \\
\nonumber
& \le & 4\times \frac{(r-1)r }{2c_n} \times \frac{ b_n^2 }{ (n-1)^2} + r(n-r) \frac{ b_n^2 }{ (n-1)^2c_n} \\
\label{eq-lemma1-c}
& \lesssim & \frac{ r b_n^2 }{ nc_n }.
\end{eqnarray}
It follows that if $b_n^2/c_n = o(n)$, then
\[
\frac{1}{r^{1/2}} \left( \sum_{i=1}^r \sum_{j=1}^n \frac{ \bar{a}_{ij}^2 }{ v_{ii} }- r \right)=o_p(1).
\]
Therefore, it is sufficient to demonstrate
\begin{equation}\label{eq-martingale-clt}
\frac{1}{ (2r)^{1/2} } \sum_{i=1}^r \sum_{j=1,j\neq i}^n \sum_{k=1, k\neq i,j}^n \frac{ \bar{a}_{ij} \bar{a}_{ik} }{ v_{ii} }
~~\stackrel{\mathcal{L}}{\longrightarrow}~~ N(0,1),
\end{equation}
as $r$ goes to infinity.

We shall apply Martingale theory to derive the central limit theorem in \eqref{eq-martingale-clt}.
The martingale sequence is constructed as follows.
Define $\sigma$-fields $\mathcal{F}_i, i=1,\ldots, r$ as follows:
\begin{eqnarray*}
\mathcal{F}_1  &  = & \sigma\left\{ \bar{a}_{12}, \bar{a}_{13}, \ldots, \bar{a}_{1n} \right\}, \\
\mathcal{F}_2  &  = & \sigma \left\{ \underbrace{\bar{a}_{12}, \ldots, \bar{a}_{1n}} , \underbrace{\bar{a}_{23}, \ldots, \bar{a}_{2n}} \right\}, \\
  & \vdots & \\
\mathcal{F}_r  & = & \sigma \left\{ \underbrace{\bar{a}_{12}, \ldots, \bar{a}_{1n}} , \underbrace{\bar{a}_{23}, \ldots, \bar{a}_{2n}}, \ldots, \underbrace{\bar{a}_{r,r+1}, \ldots, \bar{a}_{r,n}} \right\},
\end{eqnarray*}
where $\sigma\left\{ X_1, \ldots, X_t\right\}$ denotes the $\sigma$-field generated by random variables $X_1, \ldots, X_t$.
That is, $\mathcal{F}_t$ is the $\sigma$-field generated by elements of the first $t$ rows of the upper triangular matrix of $A$.
Observe that
\begin{eqnarray*}
 \sum_{i=1}^r \sum_{j=1,j\neq i}^n \sum_{k=1, k\neq i,j}^n \frac{ \bar{a}_{ij} \bar{a}_{ik} }{ v_{ii} }
  =  2 \sum_{i=1}^r \sum_{ \begin{smallmatrix} 1\le j < k \le n\\ j,k \neq i \end{smallmatrix} }
 \frac{\bar{a}_{ij} \bar{a}_{ik}}{v_{ii}}.
\end{eqnarray*}
For $i=3, \ldots, r-1$, we divide $\sum_{ \begin{smallmatrix} 1\le j < k \le n\\ j,k \neq i \end{smallmatrix} }
 \bar{a}_{ij} \bar{a}_{ik} $ into two parts:
\begin{eqnarray*}
\sum_{ \begin{smallmatrix} 1\le j < k \le n\\ j,k \neq i \end{smallmatrix} }
 \bar{a}_{ij} \bar{a}_{ik} =
 \underbrace{ \sum_{j_1=1}^{i-2} \sum_{j_2=j_1+1}^{i-1} \bar{a}_{i,j_1} \bar{a}_{i,j_2} }_{ X_i }
 + \underbrace{ \sum_{j_1=1}^{n-1} \sum_{j_2=i+1}^n \bar{a}_{i,j_1} \bar{a}_{i,j_2} }_{ Y_i }.
\end{eqnarray*}
In view of that $\bar{a}_{ij}=\bar{a}_{ji}$, $X_i$ can be rewritten as
\[
X_i  =  \left\{\bar{a}_{1i}(\bar{a}_{2i} + \ldots + \bar{a}_{i-1,i})\right\} + \left\{ \bar{a}_{2i}( \bar{a}_{3i} + \ldots + \bar{a}_{i-1,i}) \right\}
+ \cdots + \bar{a}_{i-2,i} \bar{a}_{i-1,i}.
\]
(In paired comparisons data, we let $\bar{a}_{ji}=-\bar{a}_{ij}$, $j=1,\ldots, i-2$, because $\bar{a}_{ij}+\bar{a}_{ji}=0$.)
Define
\[
Y_1 = \sum_{ \begin{smallmatrix} 1\le j < k \le n\\ j,k \neq 1 \end{smallmatrix} }
 \bar{a}_{1j} \bar{a}_{1k},~~
Y_2=\sum_{ \begin{smallmatrix} 1\le j < k \le n\\ j,k \neq 2 \end{smallmatrix} }
 \bar{a}_{2j} \bar{a}_{2k},
\]
Let $Z_1=0$ and
\begin{eqnarray*}
Z_2 & = & \bar{a}_{13}\bar{a}_{23} + \bar{a}_{14}\bar{a}_{24} + \cdots + \bar{a}_{1r}\bar{a}_{2r}, \\
Z_3 & = & \bar{a}_{34}(\bar{a}_{14} + \bar{a}_{24} ) + \bar{a}_{35}(\bar{a}_{15}+\bar{a}_{25}) + \bar{a}_{36}(\bar{a}_{16}+\bar{a}_{26}) + \cdots + \bar{a}_{3r}(\bar{a}_{1r}+\bar{a}_{2r}), \\
Z_4 & = & \bar{a}_{45}( \bar{a}_{15}+\bar{a}_{25}+\bar{a}_{35}) + \bar{a}_{46}(\bar{a}_{16}+\bar{a}_{26}+\bar{a}_{36})+\cdots + \bar{a}_{4r}(\bar{a}_{1r}+\bar{a}_{2r}+\bar{a}_{3r}),\\
& \vdots & \\
Z_{r-1} & = & \bar{a}_{r-1,r}( \bar{a}_{1r} + \bar{a}_{2r} + \cdots + \bar{a}_{r-2,r}), \\
Z_r & = & 0.
\end{eqnarray*}
It is clear that
\[
\E( Y_i + Z_i | \mathcal{F}_{i-1} ) = 0.
\]
and
\[
\sum_{i=1}^r \frac{(Y_i + Z_i)}{v_{ii}} = \sum_{i=1}^r \sum_{ \begin{smallmatrix} 1\le j < k \le n\\ j,k \neq i \end{smallmatrix} }
 \frac{\bar{a}_{ij} \bar{a}_{ik}}{v_{ii}}.
\]
Therefore, $(Y_i+Z_i)/v_{ii}$, $i=1,\ldots, r$ is a martingale difference.
We shall apply \citeauthor{Brown1971}'s (\citeyear{Brown1971}) Martingale limit theorem   to show the asymptotic normality of $\sum_{i=1}^r (Y_i +Z_i)$. 
This requires us to check two conditions:
\begin{equation}\label{condition-a}
\frac{1}{r} \sum_{i=1}^r \E \left\{ (\frac{Y_i+Z_i}{v_{ii} })^2 1( | \frac{ Y_i + Z_i}{v_{ii}} | > r^{1/2} \epsilon  ) \right\}  \to 0,
\end{equation}
as $n\to\infty$ for each $\epsilon>0$, and
\begin{equation}\label{condition-b}
\frac{1}{2r} \sum_{i=1}^r \E\left\{ (\frac{Y_i + Z_i}{v_{ii}})^2 | \mathcal{F}_{i-1} \right\} \to 1 \quad \mbox{in probability},
\end{equation}
as $n\to\infty$. They are shown in two steps below.

Step 1. We show \eqref{condition-a}. It is sufficient to demonstrate
\begin{equation}\label{condition-a-check}
\frac{1}{r^2} \sum_{t=1}^r \E \{ (\frac{Y_t+Z_t}{v_{tt} })^4 \} \to 0, \quad r\to\infty.
\end{equation}
Because $v_{ii} \ge (n-1)/b_n$, it in turn requires us to demonstrate
\begin{equation}\label{condition-a-check2}
\frac{b_n^4}{r^2n^4} \sum_{t=1}^r \E \{ (Y_t+Z_t)^4 \} \to 0, \quad r\to\infty.
\end{equation}
The Chauchy-Schwarz inequality gives that
\[
2Y_t^2Z_t^2 \le  Y_t^4 +  Z_t^4, ~~ 4 Y_t^3Z_t \le 2 Y_t^4 + 2 Y_t^2Z_t^2 \le 3 Y_t^4 + Z_t^4.
\]
It follows that
\begin{eqnarray}
\nonumber
\E (Y_t + Z_t)^4 & = & \E (Y_t^4 + 4Y_t^2Z_t^2 + Z_t^4 + 4Y_t^3Z_t + 2Y_t^2Z_t^2 + 4Y_tZ_t^3 ) \\
\label{eq-EYZt4-7}
& \le & 7\E (Y_t^4 +Z_t^4).
\end{eqnarray}
To show \eqref{condition-a-check2}, we shall derive the upper bounds of $\E Y_t^4$ and $\E Z_t^4$. This is done in two sub-steps.\\
Step 1(a). We first derive the upper bound of $\E Y_t^4, t=1, \ldots, r$.
To gain some intuitions, we write detailed expressions of several $Y_t$ below:
\begin{eqnarray*}
Y_1 & = & \bar{a}_{12} ( \bar{a}_{13} + \cdots + \bar{a}_{1n} ) + \bar{a}_{13}( \bar{a}_{14} + \cdots + \bar{a}_{1n}) + \cdots + \bar{a}_{1,n-1} \bar{a}_{1n}, \\
Y_2 & = & \bar{a}_{21} ( \bar{a}_{23}+ \cdots + \bar{a}_{2n} ) + \bar{a}_{23}( \bar{a}_{24} + \cdots + \bar{a}_{2n}) + \cdots + \bar{a}_{2,n-1} \bar{a}_{2n}, \\
Y_3 & = & (\bar{a}_{31}+\bar{a}_{32})(\bar{a}_{34} + \cdots + \bar{a}_{3n})  + \left\{ \bar{a}_{34}( \bar{a}_{35} + \cdots + \bar{a}_{3n} ) \right. \\
    &&\left.+ \bar{a}_{35}( \bar{a}_{36} + \cdots + \bar{a}_{3n} ) + \cdots + \bar{a}_{3,n-1} \bar{a}_{3n} \right\}, \\
Y_4 & = & (\bar{a}_{41}+\bar{a}_{42}+\bar{a}_{43})(\bar{a}_{45} + \cdots + \bar{a}_{4n})  + \left\{ \bar{a}_{45}( \bar{a}_{46} + \cdots + \bar{a}_{4n} ) \right. \\
    &&\left.+ \bar{a}_{46}( \bar{a}_{47} + \cdots + \bar{a}_{4n} ) + \cdots + \bar{a}_{4,n-1} \bar{a}_{4n} \right\}, \\
    & \vdots & \\
Y_{r-1} & = &  (\bar{a}_{r-1,1} + \cdots + \bar{a}_{r-1,r-2} )(\bar{a}_{r-1, r} + \cdots + \bar{a}_{r-1,n})  + \left\{ \bar{a}_{r-1,r}( \bar{a}_{r-1,r+1} + \cdots + \bar{a}_{r-1,n} ) \right. \\
    &&\left.+ \bar{a}_{r-1,r+1}( \bar{a}_{r-1, r+2} + \cdots + \bar{a}_{r-1,n} ) + \cdots + \bar{a}_{r-1,n-1} \bar{a}_{r-1,n} \right\}, \\
Y_{r} & = &  (\bar{a}_{r,1} + \cdots + \bar{a}_{r,r-1} )(\bar{a}_{r, r+1} + \cdots + \bar{a}_{r,n})  + \left\{ \bar{a}_{r,r+1}( \bar{a}_{r,r+2} + \cdots + \bar{a}_{r,n} ) \right. \\
    &&\left.+ \bar{a}_{r,r+2}( \bar{a}_{r, r+3} + \cdots + \bar{a}_{r,n} ) + \cdots + \bar{a}_{r,n-1} \bar{a}_{r,n} \right\}.
\end{eqnarray*}
As we can see, $Y_t$ can be divided into two parts:
\begin{equation}
Y_t = \underbrace{ ( \sum_{i_1=1}^{t-1} \bar{a}_{t,i_1} )( \sum_{j_1=t+1}^n \bar{a}_{t,j_1} )}_{ Y_{t1} } +
\underbrace{ \sum_{i_1=t+1}^{n-1} \sum_{j_1=i_1+1}^n \bar{a}_{t,i_1} \bar{a}_{t,j_1} }_{ Y_{t2} }.
\end{equation}
Therefore, we have
\begin{equation}\label{eq-EYt-4-a}
\E Y_t^4 = \E ( Y_{t1}^4 + Y_{t2}^4 + 4Y_{t1}^3 Y_{t2} + 4 Y_{t1}Y_{t2}^3 + 6Y_{t1}^2 Y_{t2}^2 ).
\end{equation}
Because $\bar{a}_{t,i_1}, i_1=1,\ldots, t-1$ are independent of $\bar{a}_{t,j_1}, j_1=t+1, \ldots, n$, we have
\begin{equation}\label{eq-EYt-4-b}
\E Y_{t1}Y_{t2}^3 = \E ( \sum_{i_1=1}^{t-1} \bar{a}_{t,i_1} ) \E \{(\sum_{j_1=t+1}^n \bar{a}_{t,j_1}) Y_{t2}^3\}=0,
\end{equation}
and
\begin{equation}\label{eq-EYt-4-12}
\E Y_{t1}^3 Y_{t2} = \E ( \sum_{i_1=1}^{t-1} \bar{a}_{t,i_1} )^3 \E \{(\sum_{j_1=t+1}^n \bar{a}_{t,j_1})^3 Y_{t2}\}.
\end{equation}
Because $\bar{a}_{t,i_1}, i_1=1, \ldots, t-1$ are independent and $\bar{a}_{t,i_1}=0$, we have
\begin{equation}\label{eq-EYt-4-b2}
\E ( \sum_{i_1=1}^{t-1} \bar{a}_{t,i_1} )^3 = \sum_{i_1=1}^{t-1} \E (  \bar{a}_{t,i_1}^3 )\le \frac{ (t-1)}{ c_n}.
\end{equation}
Note that
\[
\E (Y_{t1})^3 Y_{t2} = \sum_{i_1=t+1}^{n-1} \sum_{j_1=i_1+1}^n \sum_{i_2,i_3,i_4=t+1}^n \bar{a}_{t,i_1}\bar{a}_{t,j_1}
\bar{a}_{t,i_2} \bar{a}_{t,i_3} \bar{a}_{t,i_4}
\]
If the product $\bar{a}_{t,i_1}
\bar{a}_{t,i_2} \bar{a}_{t,i_3} \bar{a}_{t,i_4}\bar{a}_{t,i_5}$ is not equal to $0$, it must be in the forms of
$\bar{a}_{t,i_1}^5$ or $\bar{a}_{t,i_1}^2 \bar{a}_{t,i_2}^3$. Therefore,
\begin{eqnarray*}
\E (\sum_{j_1=t+1}^n \bar{a}_{t,j_1})^3 Y_{t2} & = & \sum_{i_1=t+1}^{n-1} \sum_{j_1=i_1+1}^n \sum_{i_2,i_3,i_4=t+1}^n \bar{a}_{t,i_1}\bar{a}_{t,j_1}
\bar{a}_{t,i_2} \bar{a}_{t,i_3} \bar{a}_{t,i_4} \\
& = & 3\sum_{i_1=t+1}^{n-1} \sum_{j_1=i_1+1}^n \E \bar{a}_{t,i_1}^3 \bar{a}_{t,j_1}^2  \\
& \le & \frac{ 3(n-t)(n-t-1) }{ 2c_n^2 }.
\end{eqnarray*}
In view of \eqref{eq-EYt-4-12} and \eqref{eq-EYt-4-b2}, we have
\begin{equation}\label{ineq-EYt13-Yt2}
\E Y_{t1}^3 Y_{t2} \le \frac{ 3(t-1)(n-t)(n-t-1) }{ c_n^3 }.
\end{equation}
Next, we calculate $\E Y_{t1}^4$. Because $\sum_{i_1=1}^{t-1} \bar{a}_{t,i_1}$ is independent of $\sum_{j_1=t+1}^n \bar{a}_{t,j_1}$, we have
\[
\E Y_{t1}^4 = \E ( \sum_{i_1=1}^{t-1} \bar{a}_{t,i_1} )^4 \E( \sum_{j_1=t+1}^n \bar{a}_{t,j_1} )^4.
\]
Because $\bar{a}_{t,i_1}, i_1=1, \ldots, t-1$ are independent and $\E \bar{a}_{i,i_1}=0$, we have
\begin{eqnarray*}
\E ( \sum_{i_1=1}^{t-1} \bar{a}_{t,i_1} )^4 & = & \sum_{i_1=1}^{t-1} \E \bar{a}_{t,i_1}^4 + 3\sum_{i_1=1}^{t-1} \sum_{i_2=1, i_2\neq i_1}^{t-1}\E \bar{a}_{t,i_1}^2\E \bar{a}_{t,i_2}^2 \\
 & \le & \frac{(t-1)}{c_n} + \frac{3(t-1)(t-2)}{c_n^2},
\end{eqnarray*}
and
\[
 \E( \sum_{j_1=t+1}^n \bar{a}_{t,j_1} )^4 \le \frac{(n-t)}{c_n} + \frac{3(n-t)(n-t-1)}{c_n^2}.
\]
It follows that
\begin{equation}\label{ineq-EYt1-4}
\E Y_{t1}^4 \le \frac{ n-1}{ c_n} + \frac{ 3((t-1)^2 + (n-t)^2 )}{ c_n^2 }.
\end{equation}
Now, we calculate $\E Y_{t2}^4$:
\[
\E Y_{t2}^4 = \E \left( \sum_{i_1=t+1}^{n-1} \sum_{j_1=i_1+1}^n \bar{a}_{t,i_1} \bar{a}_{t,j_1} \right)^4.
\]
It has $8$ summarizations:
\[
 \sum_{i_1=t+1}^{n-1} \sum_{j_1=i_1+1}^n \sum_{i_2=t+1}^{n-1} \sum_{j_2=i_2+1}^n\sum_{i_3=t+1}^{n-1} \sum_{j_3=i_3+1}^n
  \sum_{i_4=t+1}^{n-1} \sum_{j_4=i_4+1}^n \bar{a}_{t,i_1} \bar{a}_{t,j_1}\bar{a}_{t,i_2} \bar{a}_{t,j_2}\bar{a}_{t,i_3} \bar{a}_{t,j_3}\bar{a}_{t,i_4} \bar{a}_{t,j_4}.
\]
Observe that for $i_1<j_1$, $i_2<j_2$, $i_3<j_3$, $i_4<j_4$, if
$\bar{a}_{t,i_1} \bar{a}_{t,j_1} \bar{a}_{t,i_2} \bar{a}_{t,j_2} \bar{a}_{t, i_3} \bar{a}_{t,j_3} \bar{a}_{t,i_4} \bar{a}_{t,j_4}\neq 0$,
it must belongs to one of the four forms:
\begin{align*}
(\bar{a}_{t,k_1})^4 (\bar{a}_{t,k_2})^4, ~~(\bar{a}_{t,k_1})^4 (\bar{a}_{t,k_2})^2(\bar{a}_{t,k_3})^2,
\\
(\bar{a}_{t,k_1})^2(\bar{a}_{t,k_2})^2(\bar{a}_{t,k_3})^2(\bar{a}_{t,k_4})^2, ~~
(\bar{a}_{t,k_1})^3(\bar{a}_{t,k_2})^3(\bar{a}_{t,k_3})^2,
\end{align*}
where $k_1, k_2, k_3, k_4$ are four distinct values. \\
(Case 1) For the type of $(\bar{a}_{t,k_1})^4 (\bar{a}_{t,k_2})^4$, it must have
$i_1=i_2=i_3=i_4$ and $j_1=j_2=j_3=j_4$ and the number of such terms is at most
\[
(n-t-1)+(n-t-2)+\cdots +1=\frac{1}{2}(n-t)(n-t-1).
\]
(Case 2) For the type of $(\bar{a}_{t,k_1})^4 (\bar{a}_{t,k_2})^2(\bar{a}_{t,k_3})^2$, it must have $i_1=i_2=i_3=i_4$ or $j_1=j_2=j_3=j_4$.
If $i_1=i_2=i_3=i_4$, then the number of such terms is at most
\[
6\{(n-t-1)^2+(n-t-2)^2+\cdots +1\} = (n-t-1)(n-t)(2(n-t-1)+1).
\]
If $j_1=j_2=j_3=j_4$, then the number of such terms is at most at most $3(n-1-t)(n-t)(n-t-1)$. \\
(Case 3) For the type of $(\bar{a}_{t,k_1})^2(\bar{a}_{t,k_2})^2(\bar{a}_{t,k_3})^2(\bar{a}_{t,k_4})^2$, it has
at most $c_2 (n-t-1)^4$ such terms, where $c_2$ is an absolute constant. \\
(Case 4) For the type of $(\bar{a}_{t,k_1})^3(\bar{a}_{t,k_2})^3(\bar{a}_{t,k_3})^2$, it has at most
$c_3 (n-t-1)^3$ such terms, where $c_3$ is an absolute constant. \\
As a result, we have
\[
\E Y_{t2}^4 \lesssim \frac{ (n-r)^2}{c_n^2 } + \frac{ (n-r)^3 }{ c_n^3 } +  \frac{ (n-r)^4 }{ c_n^4}.
\]
In view of \eqref{ineq-EYt1-4}, we have
\begin{equation}\label{ineq-EYt4}
\E Y_t^4 \lesssim \frac{ n-1}{ c_n} + \frac{ 3((t-1)^2 + (n-t)^2 )}{ c_n^2 } + \frac{ t^2 }{ c_n^2 } + \frac{ (n-t)^4 }{ c_n^4 }.
\end{equation}
Step 1 (b), we calculate $\E Z_t^4$. Note that for $t=2, \ldots,r-1$, we have
\[
\E Z_t^4 = \E \left\{ \sum_{i_1=t+1}^r \bar{a}_{t, i_1} ( \sum_{i_2=1}^{t-1} \bar{a}_{i_2, i_1}) \right\}^4.
\]
Because $\bar{a}_{t, i_1} ( \sum_{i_2}^{t-1} \bar{a}_{i_2, i_1})$, $i_1=t+1, \ldots, r$ are independent, we have
\begin{eqnarray*}
\E Z_t^4 & = & \sum_{i_1=t+1}^r \E \left\{  \bar{a}_{t, i_1} ( \sum_{i_2=1}^{t-1} \bar{a}_{i_2, i_1}) \right\}^4 \\
&&+\sum_{i_1, j_1=t+1,i_1\neq j_1}^r  \E \left\{  \bar{a}_{t, i_1} ( \sum_{i_2=1}^{t-1} \bar{a}_{i_2, i_1}) \right\}^2 \E \left\{  \bar{a}_{t, j_1} ( \sum_{i_2=1}^{t-1} \bar{a}_{i_2, j_1}) \right\}^2.
\end{eqnarray*}
For the first term in the above equation, we have
\begin{eqnarray*}
&&\E \left\{  \bar{a}_{t, i_1} ( \sum_{i_2=1}^{t-1} \bar{a}_{i_2, i_1}) \right\}^4  \\
& = & \sum_{i_2=1}^{t-1} \E \left\{  \bar{a}_{t, i_1} \bar{a}_{i_2, i_1}) \right\}^4
+ \sum_{i_2, i_3=1, i_2\neq i_3}^{(t-1)} \E \left\{  \bar{a}_{t, i_1} \bar{a}_{i_2, i_1}) \right\}^2\E \left\{  \bar{a}_{t, i_1} \bar{a}_{i_3, i_1}) \right\}^2 \\
& \le & \frac{ t-1}{ c_n } + \frac{ (t-1)^2 }{ c_n^2 }.
\end{eqnarray*}
It follows that
\begin{equation*}\label{ineq-EZt-4}
\E Z_t^4 \le (r-t-1)\times \left( \frac{ (t-1)}{ c_n } + \frac{ (t-1)^2 }{ c_n^2 } \right) + (r-t-1)^2 \times \frac{ (t-1)^2 }{ c_n^4 }.
\end{equation*}
By combining \eqref{eq-EYZt4-7}, \eqref{ineq-EYt4} and \eqref{ineq-EZt-4}, if $b_n^4/c_n^4=o(r)$, we have
\begin{eqnarray*}
 \frac{b_n^4}{r^2n^4} \sum_{i=1}^r \E \left( Y_i + Z_i \right)^4  \lesssim \frac{ b_n^4 }{ r^2 n^4 } \times \frac{ rn^4+r^5 }{ c_n^4} \lesssim \frac{ b_n^4}{ rc_n^4 } \to 0,
\end{eqnarray*}
which shows \eqref{condition-a-check2}.

Step 2. We show \eqref{condition-b}. We first show
\begin{equation}\label{eq-lemma1-d}
\lim_{r\to\infty} \frac{1}{2r} \sum_{i=1}^r \E\left(\frac{Y_i + Z_i}{v_{ii}}\right)^2 \to 1.
\end{equation}
In view of \eqref{eq-lemma1-a},  \eqref{eq-lemma1-b} and \eqref{eq-lemma1-c}, it is sufficient to demonstrate
\begin{equation}\label{eq-lemma1-e}
\lim_{r\to\infty} \frac{1}{r} \mathrm{Cov}\left( \sum_{i=1}^r \sum_{j=1}^n \frac{ (\bar{a}_{ij}^2 - \E \bar{a}_{ij}^2 ) }{ v_{ii} }
, \sum_{i=1}^r \sum_{j=1,j\neq i}^n \sum_{k=1, k\neq i,j}^n \frac{ \bar{a}_{ij} \bar{a}_{ik} }{ v_{ii} } \right) = 0
\end{equation}
Note that $\E\bar{a}_{i_2,j_2}\bar{a}_{i_2,j_3} )= 0$ for $j_2\neq j_3$.
If $\mathrm{Cov}( \bar{a}_{i_1,j_1}^2, \bar{a}_{i_2,j_2}\bar{a}_{i_2,j_3} )\neq 0$ for $j_2\neq j_3$, it must have $i_1=i_2$, $j_2=j_3=0$.
Therefore, we have
\begin{eqnarray*}
&&\sum_{i_1=1}^r \sum_{j_1=1}^n \sum_{i_2=1}^r \sum_{j_2=1,j_2\neq i_2}^n \sum_{j_3=1, j_3\neq i_2,j_2}^n \mathrm{Cov}( \bar{a}_{i_1,j_1}^2, \bar{a}_{i_2,j_2}\bar{a}_{i_2,j_3} ) \\
& = & \sum_{i_1=1}^r \sum_{j_1=1}^n  \sum_{j_2=1,j_2\neq i_2}^n \sum_{j_3=1, j_3\neq i_2,j_2}^n \mathrm{Cov}( \bar{a}_{i_1,j_1}^2, \bar{a}_{i_1,j_2}\bar{a}_{i_1,j_3} ) \\
& = & \sum_{i_1=1}^r \sum_{j_1=1}^n\mathrm{Cov}( \bar{a}_{i_1,j_1}^2, \bar{a}_{i_1,j_1}^2 ) \\
& \lesssim & \frac{rn}{c_n}.
\end{eqnarray*}
It follows that we have \eqref{eq-lemma1-e} if $b_n^2/c_n=o(n)$.
Therefore,  it is sufficient to demonstrate
\begin{equation}\label{eq-lemma1-f}
\frac{1}{r^2} \mathrm{Var}\left( \sum_{i=1}^r \E\{ (\frac{Y_i + Z_i}{v_{ii}})^2 | \mathcal{F}_{i-1} \} \right) \to 0.
\end{equation}
to show \eqref{condition-b}.
It  essentially requires us to  calculate the variance:
\begin{eqnarray*}
 && \mathrm{Var}\left( \frac{1}{2r} \sum_{i=1}^r \E\{ (\frac{Y_i + Z_i}{v_{ii}})^2 | \mathcal{F}_{i-1} \}  \} \right) \\
& = & \frac{1}{4r^2} \E \left(  \sum_{i=1}^r \E \left[ \frac{ \{ (Y_i +Z_i)^2 - \E (Y_i +Z_i)^2 \} }{ v_{ii}^2 } \Big| \mathcal{F}_{i-1} \right] \right)^2 \\
& = & \frac{1}{4r^2} \sum_{i=1}^r \sum_{j=1}^r \E \left\{ \left(   \E \left[ \frac{ \{ (Y_i +Z_i)^2 - \E (Y_i +Z_i)^2 \} }{ v_{ii}^2 } \Big| \mathcal{F}_{i-1} \right] \right) \right. \\
&&~~~~~~~~~~~~~~~~~~~~ \times \left. \left(   \E \left[ \frac{ \{ (Y_j +Z_j)^2 - \E (Y_j +Z_j)^2 \} }{ v_{jj}^2 } \Big| \mathcal{F}_{j-1} \right] \right)
\right\}.
\end{eqnarray*}
Therefore, showing \eqref{eq-lemma1-f} is equivalent to showing
\begin{equation}\label{eq-condition-YZ}
\frac{b_n^4}{r^2n^4} \sum_{i=1}^r \E  \left(   \E \left[  \{ (Y_i +Z_i)^2 - \E (Y_i +Z_i)^2 \}  \Big| \mathcal{F}_{i-1} \right] \right)^2\to 0,
\end{equation}
and
\begin{eqnarray}
\nonumber
H&:=&\frac{b_n^4}{r^2n^4}  \sum_{i,j=1; i\neq j}^r \left| \E \left\{ \left(   \E \left[  \{ (Y_i +Z_i)^2 - \E (Y_i +Z_i)^2 \}  \Big| \mathcal{F}_{i-1} \right] \right) \right. \right. \\
\label{eq-condition-YZ2}
&&~~~~~~~~~~~~~~~~~~~~ \times \left.\left. \left(   \E \left[  \{ (Y_j +Z_j)^2 - \E (Y_j +Z_j)^2 \}   \Big| \mathcal{F}_{j-1} \right] \right)
\right\}\right|\to 0.
\end{eqnarray}
This is done in two steps.

Step 3. We show \eqref{eq-condition-YZ}.
We derive the explicit expression of the condition expectation:
\begin{eqnarray*}
\E \left[  (Y_t +Z_t)^2 | \mathcal{F}_{t-1} \right]= \E \left(  Y_t^2 | \mathcal{F}_{t-1} \right)
+ \E \left(  Z_t^2 | \mathcal{F}_{t-1} \right) + 2 \E \left(  Y_tZ_t | \mathcal{F}_{t-1} \right).
\end{eqnarray*}
Recall that
\[
Y_t=\underbrace{ ( \sum_{i_1=1}^{t-1} \bar{a}_{t,i_1} )( \sum_{j_1=t+1}^n \bar{a}_{t,j_1} )}_{ Y_{t1} } +
\underbrace{ \sum_{i_1=t+1}^{n-1} \sum_{j_1=i_1+1}^n \bar{a}_{t,i_1} \bar{a}_{t,j_1} }_{ Y_{t2} },
\]
and
\[
Z_1=0, ~~Z_t =  \sum_{i_1=t+1}^r \bar{a}_{t, i_1} ( \sum_{i_2=1}^{(t-1)} \bar{a}_{i_2, i_1}), t=2,\ldots, r-1,~~Z_r=0.
\]
It is easy to see that
\[
\E (Y_1Z_1) =0, \E [(Y_2Z_2)| \mathcal{F}_1 ] =0.
\]
The conditional expectation of $Y_tZ_t$ is
\begin{eqnarray}
\nonumber
\E ( Y_tZ_t | \mathcal{F}_{t-1} ) &  =  & \E \left[ ( \sum_{i_1=1}^{t-1} \bar{a}_{t,i_1} )( \sum_{j_1=t+1}^n \bar{a}_{t,j_1} )
\{\sum_{i_3=t+1}^r \bar{a}_{t, i_3} ( \sum_{i_4=1}^{(t-1)} \bar{a}_{i_4, i_3})\} \big| \mathcal{F}_{t-1} \right] \\
\nonumber
&& + \E \left[ \sum_{i_1=t+1}^{n-1} \sum_{j_1=i_1+1}^n \bar{a}_{t,i_1} \bar{a}_{t,j_1}\{\sum_{i_3=t+1}^r \bar{a}_{t, i_3} ( \sum_{i_4=1}^{(t-1)} \bar{a}_{i_4, i_3})\} \big| \mathcal{F}_{t-1} \right] \\
\nonumber
&=& \sum_{i_1=1}^{t-1}\sum_{j_1=t+1}^n \sum_{i_3=t+1}^r\sum_{i_4=1}^{(t-1)} \bar{a}_{t,i_1} \bar{a}_{i_4, i_3}\E \bar{a}_{t,j_1}\bar{a}_{t, i_3} \\
\label{condition-expe-YtZt}
&&+ \sum_{i_1=t+1}^{n-1} \sum_{j_1=i_1+1}^n \sum_{i_3=t+1}^r\sum_{i_4=1}^{(t-1)} \bar{a}_{i_4, i_3} \E \bar{a}_{t,i_1}\bar{a}_{t,j_1}\bar{a}_{t, i_3}.
\end{eqnarray}
Therefore, we have
\[
\E \{ \E ( Y_tZ_t | \mathcal{F}_{t-1} ) \} =0.
\]
The conditional expectation of $Y_t^2$ is
\begin{eqnarray}
\nonumber
&&\E \left(  Y_t^2 | \mathcal{F}_{t-1} \right) \\
\nonumber
& = & \E \left(  Y_{t1}^2 | \mathcal{F}_{t-1} \right)
+ 2\E \left(  Y_{t1} | \mathcal{F}_{t-1} \right) \E \left(  Y_{t2} | \mathcal{F}_{t-1} \right)
+ \E \left(  Y_{t2}^2 | \mathcal{F}_{t-1} \right) \\
\label{condition-expe-Yt2}
& = &  ( \sum_{i_1=1}^{t-1} \bar{a}_{t,i_1} )^2\E( \sum_{j_1=t+1}^n \bar{a}_{t,j_1} )^2 +
2( \sum_{i_1=1}^{t-1} \bar{a}_{t,i_1} )\E\left\{( \sum_{j_1=t+1}^n \bar{a}_{t,j_1} )Y_{t2}\right\}
 + \E Y_{t2}^2.
\end{eqnarray}
The conditional expectation of $Z_t^2$ is
\begin{eqnarray}
\nonumber
&& \E \left( Z_t^2 | \mathcal{F}_{t-1} \right) \\
\nonumber
& = & \E \left\{ \sum_{i_1=t+1}^r \bar{a}_{t,i_1} ( \sum_{i_2=1}^{t-1} \bar{a}_{i_2,i_1} ) \cdot
\sum_{j_1=t+1}^r \bar{a}_{t,j_1} (\sum_{j_2=1}^{t-1} \bar{a}_{j_2,j_1} ) | \mathcal{F}_{t-1} \right\} \\
\nonumber
& = & \sum_{i_1=t+1}^r \sum_{i_2=1}^{t-1}\sum_{j_1=t+1}^r\sum_{j_2=1}^{t-1} \bar{a}_{i_2,i_1}\bar{a}_{j_2,j_1}\E (\bar{a}_{t,i_1}\bar{a}_{t,j_1} ) \\
\label{condition-expe-Zt2}
& = & \sum_{i_1=t+1}^r \sum_{i_2=1}^{t-1}\sum_{j_2=1}^{t-1} \bar{a}_{i_2,i_1}\bar{a}_{j_2,i_1}\E (\bar{a}_{t,i_1}^2 ).
\end{eqnarray}
By the Cauchy--Schwarz inequality, we have
\begin{eqnarray*}
&&\E \left[ \{ \E(Y_i^2|\mathcal{F}_{i-1}) - \E Y_i^2 \} + \{\E(Z_i^2|\mathcal{F}_{i-1}) - \E Z_i^2 \}
+ 2\E(Y_iZ_i|\mathcal{F}_{i-1}) \right]^2
\\
&\le & 2\E \{ \E(Y_i^2|\mathcal{F}_{i-1}) - \E Y_i^2 \}^2 + 2\E \{\E(Z_i^2|\mathcal{F}_{i-1}) - \E Z_i^2 \}^2
+  4\E \{\E(Y_iZ_i|\mathcal{F}_{i-1}) \}^2.
\end{eqnarray*}
The proof of \eqref{eq-condition-YZ} is divided into three sub-steps.
Step 3(a). We derive the upper bound of $\E \{\E(Y_tZ_t|\mathcal{F}_{t-1}) \}^2$. Note that
\begin{eqnarray*}
\E \{\E(Y_tZ_t|\mathcal{F}_{t-1}) \}^2
&\le& 2\E \left( \sum_{i_1=1}^{t-1}\sum_{j_1=t+1}^n \sum_{i_3=t+1}^r\sum_{i_4=1}^{t-1} \bar{a}_{t,i_1} \bar{a}_{i_4, i_3}\E \bar{a}_{t,j_1}\bar{a}_{t, i_3} \right)^2 \\
&&+ 2\E \left( \sum_{i_1=t+1}^{n-1} \sum_{j_1=i_1+1}^n \sum_{i_3=t+1}^r\sum_{i_4=1}^{t-1} \bar{a}_{i_4, i_3} \E \bar{a}_{t,i_1}\bar{a}_{t,j_1}\bar{a}_{t, i_3}\right)^2.
\end{eqnarray*}
Because $j_1>i_1$ and $\bar{a}_{t,s}=0$ for any pair $(t,s)$, we have
\[
\E \bar{a}_{t,i_1}\bar{a}_{t,j_1}\bar{a}_{t, i_3}=0.
\]
It follows that
\begin{eqnarray*}
&&\E \{\E(Y_tZ_t|\mathcal{F}_{t-1}) \}^2 \\
& \le & 2\E \left( \sum_{i_1=1}^{t-1} \sum_{i_3=t+1}^r\sum_{i_4=1}^{t-1} \bar{a}_{t,i_1} \bar{a}_{i_4, i_3}\E \bar{a}_{t, i_3}^2 \right)^2\\
& = & 2\sum_{i_1=1}^{t-1} \sum_{j_1=1}^{t-1} \sum_{i_4=1}^{ t-1} \sum_{j_4=1}^{t-1} \sum_{i_3=t+1}^r \sum_{j_3=r+1}^r
\E\bar{a}_{t,i_1} \bar{a}_{i_4,i_3} \bar{a}_{t,j_1} \bar{a}_{j_4,j_3} \E \bar{a}_{t,i_3}^2 \E \bar{a}_{t,j_3}^2.
\end{eqnarray*}
If $\E\bar{a}_{i_1,i_2} \bar{a}_{i_3,i_4} \bar{a}_{i_5,i_6} \bar{a}_{i_7, i_8}$ is not zero, it must be in the forms of
$\E\bar{a}_{ij}^4$ or $\E\bar{a}_{ij}^2\bar{a}_{kl}^2$. Because
$t$ is fixed in $\bar{a}_{t,i_1} \bar{a}_{i_4,i_3} \bar{a}_{t,j_1} \bar{a}_{j_4,j_3}$, we have
\begin{eqnarray}
\nonumber
&&\frac{1}{2}\E \{\E(Y_tZ_t|\mathcal{F}_{t-1}) \}^2 \\
\nonumber
&\le & \sum_{i_1=1}^{t-1}  \sum_{i_4=1}^{ t-1}  \sum_{i_3=t+1}^r
\E\bar{a}_{t,i_1}^2 \bar{a}_{i_4,i_3}^2  \E \bar{a}_{t,i_3}^2 \E \bar{a}_{t,j_3}^2 \\
\label{ineq-con-YtZt-2}
&\le & \frac{ (t-1)^2(r-t) }{ c_4^2}.
\end{eqnarray}

Step 3(b). We derive the upper bound of $\E \{ \E(Y_t^2|\mathcal{F}_{t-1}) - \E Y_t^2 \}^2$.
Note that
\begin{eqnarray}
\nonumber
&&\E \left(  Y_t^2 | \mathcal{F}_{t-1} \right) \\
\nonumber
& = & \E \left(  Y_{t1}^2 | \mathcal{F}_{t-1} \right)
+ 2\E \left(  Y_{t1} | \mathcal{F}_{t-1} \right) \E \left(  Y_{t2} | \mathcal{F}_{t-1} \right)
+ \E \left(  Y_{t2}^2 | \mathcal{F}_{t-1} \right) \\
\label{eq-condition-expe-Yt2}
& = &  ( \sum_{i_1=1}^{t-1} \bar{a}_{t,i_1} )^2\E( \sum_{j_1=t+1}^n \bar{a}_{t,j_1} )^2 +
2( \sum_{i_1=1}^{t-1} \bar{a}_{t,i_1} )\E\left\{( \sum_{j_1=t+1}^n \bar{a}_{t,j_1} )Y_{t2}\right\}
 + \E Y_{t2}^2.
\end{eqnarray}
It follows that
\begin{eqnarray}
\nonumber
& &\E \{ \E(Y_t^2|\mathcal{F}_{t-1}) - \E Y_t^2 \}^2 \\
\nonumber
&= & \E \left[ \left\{( \sum_{i_1=1}^{t-1} \bar{a}_{t,i_1} )^2-\E( \sum_{i_1=1}^{t-1} \bar{a}_{t,i_1} )^2\right\}\E( \sum_{j_1=t+1}^n \bar{a}_{t,j_1} )^2 \right.\\
&& + \left. 2( \sum_{i_1=1}^{t-1} \bar{a}_{t,i_1} )\E\left\{( \sum_{j_1=t+1}^n \bar{a}_{t,j_1} )Y_{t2}\right\} \right]^2 \\
\nonumber
& \le & 2\E \left[ \left\{( \sum_{i_1=1}^{t-1} \bar{a}_{t,i_1} )^2-\E( \sum_{i_1=1}^{t-1} \bar{a}_{t,i_1} )^2\right\}\E( \sum_{j_1=t+1}^n \bar{a}_{t,j_1} )^2\right]^2 \\
\nonumber
&& + 4\E\left[( \sum_{i_1=1}^{t-1} \bar{a}_{t,i_1} )\E\left\{( \sum_{j_1=t+1}^n \bar{a}_{t,j_1} )Y_{t2}\right\}
\right]^2 \\
\nonumber
& \le & \frac{2(n-t)^2}{c_n^2} \E \left\{( \sum_{i_1=1}^{t-1} \bar{a}_{t,i_1} )^2-\E( \sum_{i_1=1}^{t-1} \bar{a}_{t,i_1} )^2\right\}^2 \\
\label{ineq-cond-expe-Yt2}
&& + \frac{4(t-1)}{ c_n} \E\left\{( \sum_{j_1=t+1}^n \bar{a}_{t,j_1} )Y_{t2}\right\}^2.
\end{eqnarray}
The upper bounds of two expectations in the above last inequality are derived as follows.
Note that
\begin{eqnarray*}
 &&\E\left\{( \sum_{j_1=t+1}^n \bar{a}_{t,j_1} )Y_{t2}\right\}^2 \\
 & = & \E \left( \sum_{i_1=t+1}^{n-1} \sum_{j_1=t+1}^n \bar{a}_{t,i_1} \bar{a}_{t,j_1}\right)^2 \left( \sum_{k_1=t+1}^n \bar{a}_{t,k_1}\right)^2 \\
 & = & \sum_{i_1=t+1}^{n-1}\sum_{j_1=i_1+1}^n \sum_{i_2=t+1}^{n-1} \sum_{j_2=i_2+1}^n \sum_{k_1=t+1}^n \sum_{k_2=t+1}^n
 \E \bar{a}_{t,i_1} \bar{a}_{t,j_1} \bar{a}_{t,i_2} \bar{a}_{t,j_2} \bar{a}_{t,k_1}\bar{a}_{t,k_2}
\end{eqnarray*}
If $\E \bar{a}_{t,i_1} \bar{a}_{t,j_1} \bar{a}_{t,i_2} \bar{a}_{t,j_2} \bar{a}_{t,k_1}\bar{a}_{t,k_2}$ is not zero, it must be
one of the forms: $\bar{a}_{i_1,j_1}^6$, $\bar{a}_{i_1,j_1}^3\bar{a}_{i_2,j_2}^3$, $\bar{a}_{i_1,j_1}^4 \bar{a}_{i_1,j_1}^2$
and $\bar{a}_{i_1,j_1}^2 \bar{a}_{i_2,j_2}^2 \bar{a}_{i_3,j_3}^2$ for three distinct random variables $\bar{a}_{i_1,j_1}$, $\bar{a}_{i_2,j_2}$
and $\bar{a}_{i_3,j_3}$. Therefore, we have
\begin{equation}\label{ineq-cond-expe-Yt2-a}
\E\left\{( \sum_{j_1=t+1}^n \bar{a}_{t,j_1} )Y_{t2}\right\}^2 \lesssim \frac{ (n-t)^2}{ c_n^2} + \frac{ (n-t)^3 }{ c_n^3}.
\end{equation}
Note that
\begin{eqnarray*}
\E \left\{( \sum_{i_1=1}^{t-1} \bar{a}_{t,i_1} )^2-\E( \sum_{i_1=1}^{t-1} \bar{a}_{t,i_1} )^2\right\}^2
 =  \E ( \sum_{i_1=1}^{t-1} \bar{a}_{t,i_1} )^4- \left\{\E( \sum_{i_1=1}^{t-1} \bar{a}_{t,i_1} )^2\right\}^2,
\end{eqnarray*}
and
\begin{eqnarray*}
\E(\sum_{i_1=1}^{t-1} \bar{a}_{t,i_1} )^4 & = & \sum_{i_1=1}^{t-1} \sum_{i_2=1}^{t-1} \sum_{i_3=1}^{t-1} \sum_{i_4=1}^{t-1}
\bar{a}_{t,i_1} \bar{a}_{t,i_2} \bar{a}_{t,i_3} \bar{a}_{t,i_4} \\
& = &  \sum_{i_1=1}^{t-1}\E \bar{a}_{t,i_1}^4 +
3\sum_{i_1=1}^{t-1} \sum_{i_2=1,i_2\neq i_1}^{t-1} \E \bar{a}_{t,i_1}^2 \E \bar{a}_{t,i_2}^2  \\
& \lesssim & \frac{ (t-1)^2 }{ c_n^2 }
\end{eqnarray*}
In view of \eqref{ineq-cond-expe-Yt2} and \eqref{ineq-cond-expe-Yt2-a}, it follows that
\begin{equation}\label{ineq-con-exp-Yt2}
\E \{ \E(Y_t^2|\mathcal{F}_{t-1}) - \E Y_t^2 \}^2 \lesssim \frac{ (n-t)^2(t-1)^2 }{ c_n^4 } +
\left\{ \frac{(n-t)^2}{c_n^2} + \frac{ (n-t)^3}{ c_n^3}  \right\} \frac{ (t-1)}{ c_n }.
\end{equation}

Step 3(c). We derive the upper bound of $\E \left( Z_t^2 | \mathcal{F}_{t-1} \right)$.
The conditional expectation of $Z_t^2$ is
\begin{eqnarray}
\nonumber
&& \E \left( Z_t^2 | \mathcal{F}_{t-1} \right) \\
\nonumber
& = & \E \left\{ \sum_{i_1=t+1}^r \bar{a}_{t,i_1} ( \sum_{i_2=1}^{t-1} \bar{a}_{i_2,i_1} ) \cdot
\sum_{j_1=t+1}^r \bar{a}_{t,j_1} (\sum_{j_2=1}^{t-1} \bar{a}_{j_2,j_1} ) | \mathcal{F}_{t-1} \right\} \\
\nonumber
& = & \sum_{i_1=t+1}^r \sum_{i_2=1}^{t-1}\sum_{j_1=t+1}^r\sum_{j_2=1}^{t-1} \bar{a}_{i_2,i_1}\bar{a}_{j_2,j_1}\E (\bar{a}_{t,i_1}\bar{a}_{t,j_1} ) \\
\label{eq-conditional-Zt2}
& = & \sum_{i_1=t+1}^r \sum_{i_2=1}^{t-1}\sum_{j_2=1}^{t-1} \bar{a}_{i_2,i_1}\bar{a}_{j_2,i_1}\E (\bar{a}_{t,i_1}^2 ).
\end{eqnarray}
It follows that
\begin{eqnarray}
\nonumber
&&\E \left\{ \E \left( Z_t^2 | \mathcal{F}_{t-1} \right) \right\}^2 \\
\nonumber
& \le & c_n^2
\E \left\{ \sum_{i_1=t+1}^r \sum_{i_2=1}^{t-1}\sum_{j_2=1}^{t-1} \bar{a}_{i_2,i_1}\bar{a}_{j_2,i_1}\right\}^2 \\
\nonumber
& \le & c_n^2 \sum_{i_1=t+1}^{r} \sum_{i_2=t+1}^{t-1} \E \bar{a}_{i_1,i_2}^4
+  \left(\sum_{i_1=t+1}^r \sum_{i_2=1}^{t-1} \E \bar{a}_{i_1,i_2}^2  \right)^2
+ \sum_{i_1=t+1}^r \sum_{i_2=1}^{t-1} \sum_{s_2=1}^{t-1} \E \bar{a}_{i_1,i_2}^2 \E \bar{a}_{i_1,s_2}^2 \\
\label{ineq-upper-EZt2}
& \lesssim & \frac{ (r-t)(t-1) }{ c_n^3 } + \frac{ (r-t)^2(t-1)^2 }{ c_n^4 } + \frac{ (r-t)(t-1)^2 }{ c_n^4 }.
\end{eqnarray}

By combining \eqref{ineq-con-YtZt-2}, \eqref{ineq-con-exp-Yt2} and \eqref{ineq-upper-EZt2}, it yields
\begin{eqnarray*}
&&\frac{b_n^4}{r^2n^4} \sum_{i=1}^r \E  \left(   \E \left[  \{ (Y_i +Z_i)^2 - \E (Y_i +Z_i)^2 \}  \Big| \mathcal{F}_{i-1} \right] \right)^2 \\
& \lesssim & \frac{b_n^4}{r^2n^4}  \sum_{t=1}^r \left(\frac{ t^2(r-t) }{c_n^2} + \frac{ (n-t)^3t }{c_n^4} + \frac{(r-t)^2t}{c_n^4} \right) \\
& \lesssim & \frac{b_n^4}{r^2n^4} \times \frac{ n^3 r^2 }{ c_n^4 } \to 0.
\end{eqnarray*}
This shows \eqref{eq-condition-YZ}.

Step 4. We show \eqref{eq-condition-YZ2}.
This requires us to calculate
\begin{eqnarray*}
&&\mathrm{Cov}\left\{ \E(Y_t^2|\mathcal{F}_{t-1}) + \E(Z_t^2|\mathcal{F}_{t-1}) + \E(Y_tZ_t|\mathcal{F}_{t-1})\right.,\\
&&~~~~~~~~\left.\E(Y_s^2|\mathcal{F}_{s-1}) + \E( Z_s^2|\mathcal{F}_{s-1}) + \E(Y_sZ_s|\mathcal{F}_{s-1} ) \right\}.
\end{eqnarray*}
This is done in six sub-steps. In what follows, we assume $t<s$.\\
Step 4(a). We derive the upper bound of $\mathrm{Cov}\left\{ \E(Y_t^2|\mathcal{F}_{t-1}), \E(Y_s^2|\mathcal{F}_{s-1})\right\}$.
Because $\sum_{i_1=1}^{t-1} \bar{a}_{t,i_1}$ is independent of $\sum_{i_1=1}^{s-1} \bar{a}_{s,i_1}$, we have
\begin{eqnarray}
\nonumber
&&\mathrm{Cov}\left\{ \E(Y_t^2|\mathcal{F}_{t-1}), \E(Y_s^2|\mathcal{F}_{s-1})\right\} \\
\nonumber
&=&\mathrm{Cov}\left\{( \sum_{i_1=1}^{t-1} \bar{a}_{t,i_1} )^2\E( \sum_{j_1=t+1}^n \bar{a}_{t,j_1} )^2 +
2( \sum_{i_1=1}^{t-1} \bar{a}_{t,i_1} )\E\left\{( \sum_{j_1=t+1}^n \bar{a}_{t,j_1} )Y_{t2}\right\}, \right. \\
\nonumber
&&
~~~~~~\left.( \sum_{i_1=1}^{s-1} \bar{a}_{s,i_1} )^2\E( \sum_{j_1=s+1}^n \bar{a}_{s,j_1} )^2 +
2( \sum_{i_1=1}^{s-1} \bar{a}_{s,i_1} )\E\left\{( \sum_{j_1=s+1}^n \bar{a}_{s,j_1} )Y_{s2}\right\} \right\}\\
\label{le1-4a}
& = & 0.
\end{eqnarray}

Step 4(b). We derive the upper bound of $\mathrm{Cov}( \E(Y_t^2|\mathcal{F}_{t-1}) \E ( Z_s^2 | \mathcal{F}_{s-1} ) )$.
Let
\[
\eta_t = \E( \sum_{j_1=t+1}^n \bar{a}_{t,j_1} )^2 \le \frac{ (n-1-t) }{ c_n }.
\]
Then, for $t<s$,
\begin{eqnarray*}
& &  \mathrm{Cov}\left\{( \sum_{i_1=1}^{t-1} \bar{a}_{t,i_1} )^2\eta_t,
( \sum_{i_1=1}^{s-1} \bar{a}_{s,i_1} )^2\eta_s \right\} \\
& = & \eta_t \eta_s \sum_{i_1=1}^{t-1} \sum_{i_2=1}^{t-1} \sum_{j_1=1}^{s-1} \sum_{j_2=1}^{s-1} \mathrm{Cov}( \bar{a}_{t,i_1} \bar{a}_{t,i_2},
\bar{a}_{s,j_1} \bar{a}_{s,j_2} ) \\
& = & \eta_t \eta_s \sum_{i_1=1}^{t-1} \sum_{i_2=1}^{t-1} \sum_{j_1=1}^{t-1} \sum_{j_2=1}^{t-1} \mathrm{Cov}( \bar{a}_{t,i_1} \bar{a}_{t,i_2},
\bar{a}_{s,j_1} \bar{a}_{s,j_2} )
\end{eqnarray*}
Because $\bar{a}_{t,i_1}, i_1=1,\ldots, t-1$ are independent of $\bar{a}_{s,j_1}, j_1=1,\ldots, s-1$, we have
\begin{equation*}
\mathrm{Cov}\left\{( \sum_{i_1=1}^{t-1} \bar{a}_{t,i_1} )^2\eta_t,
( \sum_{i_1=1}^{s-1} \bar{a}_{s,i_1} )^2\eta_s \right\} =0.
\end{equation*}
Recall that
\[
\E(Y_t^2|\mathcal{F}_{t-1})=( \sum_{i_1=1}^{t-1} \bar{a}_{t,i_1} )^2\E( \sum_{j_1=t+1}^n \bar{a}_{t,j_1} )^2 +
2( \sum_{i_1=1}^{t-1} \bar{a}_{t,i_1} )\E\left\{( \sum_{j_1=t+1}^n \bar{a}_{t,j_1} )Y_{t2}\right\},
\]
and
\begin{eqnarray*}
 \E \left( Z_s^2 | \mathcal{F}_{s-1} \right) =  \sum_{i_1=s+1}^r \sum_{i_2=1}^{s-1}\sum_{j_2=1}^{s-1} \bar{a}_{i_2,i_1}\bar{a}_{j_2,i_1}\E (\bar{a}_{s,i_1}^2 )
\end{eqnarray*}
It follows that
\begin{equation}\label{le1-4b}
\mathrm{Cov}\left\{ \E(Y_t^2|\mathcal{F}_{t-1}), \E( Z_s^2|\mathcal{F}_{s-1}) \right\}=0
\end{equation}

Step 4(c). We derive the upper bound of $\mathrm{Cov}\left\{ \E(Y_t^2|\mathcal{F}_{t-1}), \E ( Y_sZ_s | \mathcal{F}_{s-1} ) \right\}$.
Recall that
\begin{eqnarray*}
\E ( Y_sZ_s | \mathcal{F}_{s-1} ) &=& \sum_{i_1=1}^{s-1}\sum_{j_1=s+1}^n \sum_{i_3=s+1}^r\sum_{i_4=1}^{(s-1)} \bar{a}_{s,i_1} \bar{a}_{i_4, i_3}\E \bar{a}_{s,j_1}\bar{a}_{s, i_3} \\
&&+ \sum_{i_1=s+1}^{n-1} \sum_{j_1=i_1+1}^n \sum_{i_3=s+1}^r\sum_{i_4=1}^{(s-1)} \bar{a}_{i_4, i_3} \E \bar{a}_{s,i_1}\bar{a}_{s,j_1}\bar{a}_{s, i_3}.
\end{eqnarray*}
It follows that
\begin{equation}\label{le1-4c}
\mathrm{Cov}\left\{ \E(Y_t^2|\mathcal{F}_{t-1}), \E ( Y_sZ_s | \mathcal{F}_{s-1} ) \right\}=0.
\end{equation}

Step 4(d). We derive the upper bound of $\mathrm{Cov}(\E(Z_t^2|\mathcal{F}_{t-1}),
\E(Y_s^2|\mathcal{F}_{s-1}) )$.
Let
\[
\eta_{t2}=\E\left\{( \sum_{j_1=t+1}^n \bar{a}_{t,j_1} )Y_{t2}\right\}
\]
Then, we have
\begin{eqnarray}
\nonumber
 & & \mathrm{Cov}\left(\E [Z_t^2|\mathcal{F}_{t-1}], \E(Y_s^2|\mathcal{F}_{s-1}) \right) \\
\nonumber
& = &  \mathrm{Cov}\left( \sum_{i_1=t+1}^r \sum_{i_2=1}^{t-1}\sum_{j_2=1}^{t-1} \bar{a}_{i_2,i_1}\bar{a}_{j_2,i_1}\E (\bar{a}_{t,i_1}^2), ( \sum_{i_1=1}^{s-1} \bar{a}_{s,i_1} )^2\eta_s +
2( \sum_{i_1=1}^{s-1} \bar{a}_{s,i_1} )\eta_{s2}  \right) \\
\nonumber
& = & \sum_{i_1=t+1}^r \sum_{i_2=1}^{t-1}\sum_{j_2=1}^{t-1} \sum_{i_3=1}^{s-1} \left( \sum_{i_4=1}^{s-1} \mathrm{Cov}( v_{t,i_1} \bar{a}_{i_2,i_1}\bar{a}_{j_2,i_1}, \eta_s \bar{a}_{s,i_3}\bar{a}_{s,i_4})
\right. \\
\nonumber
&&~~~~+ \left. \mathrm{Cov}(  v_{t,i_1} \bar{a}_{i_2,i_1}\bar{a}_{j_2,i_1}, \bar{a}_{s,i_3})  \right)\\
\nonumber
& \lesssim & \frac{(n-t)}{c_n^2}\sum_{i_2,j_2,i_3,i_4=1}^{s-1} | \mathrm{Cov}( \bar{a}_{s,i_2} \bar{a}_{s,j_2}, \bar{a}_{s,i_3}\bar{a}_{s,i_4} ) | \\
\nonumber
&& + \frac{1}{c_n} \sum_{i_2=1}^{s-1} |\E\bar{a}_{s,i_2}^3| \\
\label{le1-4d}
& \lesssim &  \frac{ (n-t)(s-1)^2 }{ c_n^4 }.
\end{eqnarray}

Step 4(e). We derive the upper bound of $\mathrm{Cov}\left(\E(Z_t^2|\mathcal{F}_{t-1}), \E( Z_s^2|\mathcal{F}_{s-1}) \right)$ and \\
 $\mathrm{Cov}\left( \E(Z_t^2|\mathcal{F}_{t-1}), \E(Y_sZ_s|\mathcal{F}_{s-1} ) \right)$.
Note that $t<s$. Then we have
\begin{eqnarray}
\nonumber
&&\mathrm{Cov}\left(\E(Z_t^2|\mathcal{F}_{t-1}), \E( Z_s^2|\mathcal{F}_{s-1}) \right) \\
\nonumber
& = & \mathrm{Cov}\left( \sum_{i_1=t+1}^r \sum_{i_2=1}^{t-1}\sum_{j_2=1}^{t-1} v_{t,i_1}\bar{a}_{i_2,i_1}\bar{a}_{j_2,i_1},
\sum_{i_3=s+1}^r \sum_{i_4=1}^{s-1}\sum_{j_3=1}^{s-1} v_{s,i_3}\bar{a}_{i_4,i_3}\bar{a}_{j_3,i_3} \right) \\
\nonumber
& = & \sum_{i_1=t+1}^r \sum_{i_2=1}^{t-1}\sum_{j_2=1}^{t-1}\sum_{i_3=s+1}^r \sum_{i_4=1}^{s-1}\sum_{j_3=1}^{s-1} v_{s,i_3}v_{t,i_1}\mathrm{Cov}
(\bar{a}_{i_1, i_2}\bar{a}_{i_1, j_2},  \bar{a}_{i_3, i_4}\bar{a}_{i_3, j_3}) \\
\nonumber
& = &  \sum_{i_2=1}^{t-1}\sum_{j_2=1}^{t-1}\sum_{i_3=s+1}^r \sum_{i_4=1}^{t-1}\sum_{j_3=1}^{t-1} v_{s,i_3}v_{t,i_3}\mathrm{Cov}
(\bar{a}_{i_3, i_2}\bar{a}_{i_3, j_2},  \bar{a}_{i_3, i_4}\bar{a}_{i_3, j_3}) \\
\label{le1-4e}
& \lesssim & \frac{ (t-1)^2(r-s) }{ c_n^4}.
\end{eqnarray}
Recall that
\begin{eqnarray*}
\E ( Y_sZ_s | \mathcal{F}_{s-1} ) &=& \sum_{i_1=1}^{s-1}\sum_{j_1=s+1}^n \sum_{i_3=s+1}^r\sum_{i_4=1}^{s-1} \bar{a}_{s,i_1} \bar{a}_{i_4, i_3}\E \bar{a}_{s,j_1}\bar{a}_{s, i_3} \\
&&+ \sum_{i_1=s+1}^{n-1} \sum_{j_1=i_1+1}^n \sum_{i_3=s+1}^r\sum_{i_4=1}^{s-1} \bar{a}_{i_4, i_3} \E \bar{a}_{s,i_1}\bar{a}_{s,j_1}\bar{a}_{s, i_3}.
\end{eqnarray*}
Because $i_1<j_1$, we have $\E \bar{a}_{s,i_1}\bar{a}_{s,j_1}\bar{a}_{s, i_3}=0$. This leads to
\begin{eqnarray*}
\E ( Y_sZ_s | \mathcal{F}_{s-1} ) & = & \sum_{i_1=1}^{s-1}\sum_{j_1=s+1}^n \sum_{i_3=s+1}^r\sum_{i_4=1}^{s-1} \bar{a}_{s,i_1} \bar{a}_{i_4, i_3}\E \bar{a}_{s,j_1}\bar{a}_{s, i_3}, \\
& = & \sum_{i_1=1}^{s-1} \sum_{i_3=s+1}^r \sum_{i_4=1}^{s-1} \bar{a}_{s,i_1} \bar{a}_{i_4, i_3}\E \bar{a}_{s, i_3}^2.
\end{eqnarray*}
It follows that
\begin{equation}\label{le1-4ee}
\mathrm{Cov}\left( \E(Z_t^2|\mathcal{F}_{t-1}), \E(Y_sZ_s|\mathcal{F}_{s-1} ) \right) =0.
\end{equation}

Step 4(f). We derive the upper bound of
\[
\mathrm{Cov}\left\{ \E(Y_tZ_t|\mathcal{F}_{t-1}),
\E(Y_s^2|\mathcal{F}_{s-1}) + \E( Z_s^2|\mathcal{F}_{s-1}) + \E(Y_sZ_s|\mathcal{F}_{s-1} ) \right\}.
\]
Recall that
\begin{eqnarray*}
\E ( Y_tZ_t | \mathcal{F}_{t-1} ) &=& \sum_{i_1=1}^{t-1} \sum_{i_3=t+1}^r \sum_{i_4=1}^{t-1} \bar{a}_{t,i_1} \bar{a}_{i_4, i_3}\E \bar{a}_{t, i_3}^2,
\end{eqnarray*}
and
\begin{eqnarray*}
\E \left(  Y_s^2 | \mathcal{F}_{s-1} \right)
& = &  ( \sum_{i_1=1}^{s-1} \bar{a}_{s,i_1} )^2\E( \sum_{j_1=s+1}^n \bar{a}_{s,j_1} )^2 +
2( \sum_{i_1=1}^{s-1} \bar{a}_{s,i_1} )\E\left\{( \sum_{j_1=s+1}^n \bar{a}_{s,j_1} )Y_{s2}\right\}
 + \E Y_{s2}^2.
\end{eqnarray*}
Because $\bar{a}_{t,i_1}$ is independent of $\bar{a}_{i_4, i_3}$, we have
\[
\E\sum_{i_3=t+1}^r \sum_{i_4=1}^{t-1} \bar{a}_{t,i_1} \bar{a}_{i_4, i_3}\E \bar{a}_{t, i_3}^2=0.
\]
It follows that
\begin{eqnarray*}
&&\mathrm{Cov}\left\{ \sum_{i_1=1}^{t-1} \sum_{i_3=t+1}^r \sum_{i_4=1}^{t-1} \bar{a}_{t,i_1} \bar{a}_{i_4, i_3}\E \bar{a}_{t, i_3}^2,
( \sum_{i_1=1}^{s-1} \bar{a}_{s,i_1} )^2\E( \sum_{j_1=s+1}^n \bar{a}_{s,j_1} )^2 \right\} \\
& = & \E \left\{ \sum_{i_1=1}^{t-1} \sum_{i_3=t+1}^r \sum_{i_4=1}^{t-1} \bar{a}_{t,i_1} \bar{a}_{i_3,i_4}\E \bar{a}_{t, i_3}^2
( \sum_{i_1=1}^{s-1} \bar{a}_{s,i_1} )^2\E( \sum_{j_1=s+1}^n \bar{a}_{s,j_1} )^2 \right\} \\
& = & 0,
\end{eqnarray*}
which is due to that $\bar{a}_{t,i_1}$ is independent of $\bar{a}_{i_3,i_4}, \bar{a}_{s,i_1}$. Therefore,
\begin{equation}\label{le1-4fa}
\mathrm{Cov}( \E ( Y_tZ_t | \mathcal{F}_{t-1} ), \E \left(  Y_s^2 | \mathcal{F}_{s-1} \right) ) =0.
\end{equation}
Similarly, we have
\begin{eqnarray}
\nonumber
&&\mathrm{Cov}( \E ( Y_tZ_t | \mathcal{F}_{t-1} ), \E ( Y_sZ_s | \mathcal{F}_{s-1} ) ) \\
\nonumber
& = & \mathrm{Cov}\left( \sum_{i_1=1}^{t-1} \sum_{i_3=t+1}^r \sum_{i_4=1}^{t-1} \bar{a}_{t,i_1} \bar{a}_{i_4, i_3}\E \bar{a}_{t, i_3}^2, \sum_{i_1=1}^{s-1} \sum_{i_3=s+1}^r \sum_{i_4=1}^{s-1} \bar{a}_{s,i_1} \bar{a}_{i_4, i_3}\E \bar{a}_{s, i_3}^2 \right) \\
\nonumber
& = &\mathrm{Cov}\left( \sum_{i_1=1}^{s-1} \sum_{i_3=s+1}^r \sum_{i_4=1}^{s-1} \bar{a}_{t,i_1} \bar{a}_{i_3, i_4}\E \bar{a}_{t, i_3}^2,
\sum_{j_1=1}^{s-1} \sum_{j_3=s+1}^r \sum_{j_4=1}^{s-1} \bar{a}_{s,j_1} \bar{a}_{j_3, j_4}
\E \bar{a}_{s, j_3}^2 \right) \\
\label{le1-4fb}
& = & \E \left\{ \sum_{i_1=1}^{s-1}\sum_{j_1=1}^{s-1}\bar{a}_{t,i_1}\bar{a}_{s,j_1} \left( \sum_{i_3=s+1}^r \sum_{i_4=1}^{s-1}  \bar{a}_{i_3, i_4}\E \bar{a}_{t, i_3}^2\right) \cdot
\left(\sum_{j_3=s+1}^r \sum_{j_4=1}^{s-1}  \bar{a}_{j_3, j_4} \E \bar{a}_{s, j_3}^2\right) \right\}=0,
\end{eqnarray}
and
\begin{eqnarray}
\nonumber
&&\mathrm{Cov}\left\{ \E(Y_tZ_t|\mathcal{F}_{t-1}), \E( Z_s^2|\mathcal{F}_{s-1}) \right\} \\
\label{le1-4fc}
& = & \E \sum_{i_1=1}^{t-1} \sum_{i_3=t+1}^r \sum_{i_4=1}^{t-1} \bar{a}_{t,i_1} \bar{a}_{i_4, i_3}\E \bar{a}_{t, i_3}^2
\sum_{i_5=s+1}^r \sum_{i_6=1}^{s-1}\sum_{j_2=1}^{s-1} \bar{a}_{i_6,i_5}\bar{a}_{j_2,i_5}\E (\bar{a}_{s,i_5}^2 )=0.
\end{eqnarray}

By combining \eqref{le1-4a}--\eqref{le1-4fc}, $H$ in \eqref{eq-condition-YZ2} can be bounded above by
\[
H \lesssim \frac{ b_n^4}{r^2n^4} \sum_{t,s=1,t\neq s}^r \left(\frac{ (n-t)(s-1)^2 }{c_n^4 } + \frac{ (t-1)^2(r-s) }{ c_n^4 }\right)
\lesssim \frac{ b_n^4 }{ nc_n^4} \to 0.
\]
This shows \eqref{eq-condition-YZ2}.
\end{proof}

\section{Bernstein's inequality and Martingale central limit theorem}
\label{section-bernstein}

This section collects a user-friendly version of the Bernstein inequality on bounded random variables and a Martingale central limit theorem.

The following Bernstein inequality can be easily found in textbooks such as \cite{Boucheron2013book}. The proof is omitted.

\begin{lemma}[Bernstein's inequality]\label{lemma:bernstein}
Suppose $n$ independent random variables $x_{i}$ ($1\le i \le n$)
each satisfying $\left|x_{i}\right|\leq B$. For any $a\geq2$, one
has
\[
\left|\sum_{i=1}^{n}x_{i}-\E\left[\sum_{i=1}^{n}x_{i}\right]\right|\leq\sqrt{2a\log n\sum_{i=1}^{n}\E\left[x_{i}^{2}\right]}+\frac{2a}{3}B\log n
\]
with probability at least $1-2n^{-a}$.
\end{lemma}

Next, we present the martingale central limit theorem by \cite{Brown1971}.

\begin{lemma}[\cite{Brown1971}]\label{lemma-martingale-clt}
Let $\{ S_n, \mathcal{F}_n, n=1, 2, \ldots \}$ be a martingale on the probability space $\{\Omega, \mathcal{F}, \P\}$, with $S_0=0$, and $X_n = S_n - S_{n-1}$, $n=1,2,\ldots$.
Define
\[
\sigma_n^2 = \E( X_n^2|\mathcal{F}_{n-1}), \quad V_n^2 = \sum_{j=1}^n \sigma_j^2, \quad s_n^2 = \E V_n^2 = \E S_n^2.
\]
If the condition
\begin{equation*}
\frac{ V_n^2 }{s_n^2} \stackrel{p.}{\to} 1,\quad, n\to\infty,
\end{equation*}
and the Lindeberg condition
\begin{equation*}
\frac{1}{s_n^2} \sum_{j=1}^n \E [ X_j^2 1(|X_j| \ge \epsilon s_n ) ] \stackrel{p.}{\to} 0, \quad, n\to\infty
\end{equation*}
hold, then $S_n/s_n$ converges in distribution to the standard normal distribution as $n\to\infty$.
\end{lemma}

\end{document}